\documentclass[a4paper]{report}

\usepackage[english]{babel}
\usepackage[utf8x]{inputenc}
\usepackage[T1]{fontenc}

\usepackage[a4paper,top=2cm,bottom=2.5cm,left=2.3cm,right=2.3cm,marginparwidth=1.75cm]{geometry}

\usepackage[colorinlistoftodos]{todonotes}
\usepackage{hyperref}
\hypersetup{%
  colorlinks = true,
  allcolors=blue
}
\usepackage{xy}
\usepackage[affil-it]{authblk}
\usepackage{amssymb}
\usepackage{cancel}
\usepackage{epstopdf}
\usepackage{verbatim}
\usepackage{amsmath,amscd}
\usepackage{graphicx}
\usepackage{subcaption}
\usepackage{amsthm}
\usepackage{tikz}
\usepackage{pdfpages}
\usepackage{pgf}
\usepackage{mathrsfs}
\usetikzlibrary{arrows}
\usepackage[toc,page]{appendix}
\usepackage{caption}
\usepackage[colorinlistoftodos]{todonotes}

\newtheorem{prop}{Proposition}

\makeatletter
\newtheorem*{rep@theorem}{\rep@title}
\newcommand{\newreptheorem}[2]{%
\newenvironment{rep#1}[1]{%
 \def\rep@title{#2 \ref{##1}}%
 \begin{rep@theorem}}%
 {\end{rep@theorem}}}
\makeatother

\newtheorem{theorem}{Theorem}[section]
\newtheorem{lemma}{Lemma}[section]
\newreptheorem{theorem}{Theorem}
\newtheorem{pro}{Proposition}[section]

\theoremstyle{definition}
\newtheorem{mydef}{Definition}[section]
\newtheorem{algo}{Algorithm}[section]
\newtheorem{rem}{Remark}[section]
\newtheorem{exa}{Example}[section]
\newtheorem{cor}{Corollary}[section]

\title{
	{Tropicalizing abelian covers of algebraic curves}\\
}

	\author{Paul Alexander Helminck}

\begin{document}

\includepdf[noautoscale]{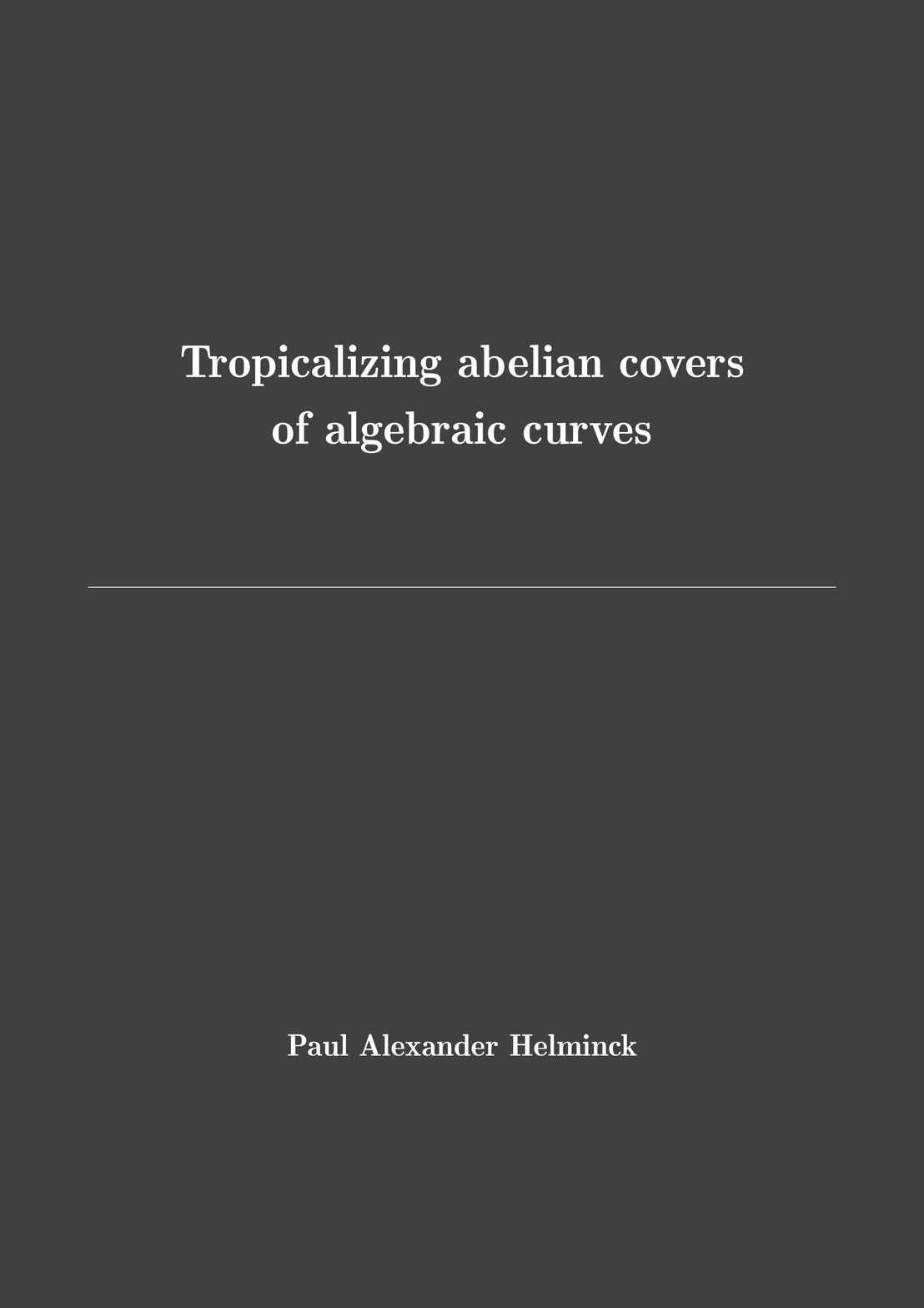}

\newpage\null\thispagestyle{empty}\newpage

\begin{titlepage}
    \begin{center}
        \vspace*{1cm}
        
        \Huge
        \textbf{Tropicalizing abelian covers of algebraic curves}
        
        \vspace{0.5cm}
        \LARGE
      
        \vspace{1.5cm}
        
        \textbf{Paul Alexander Helminck}

        \vfill
        
                 \includegraphics[width=0.4\textwidth]{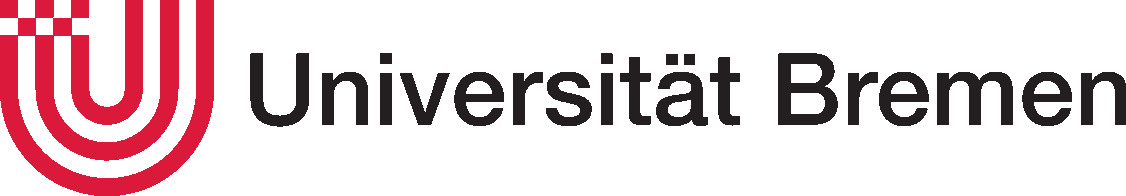}
        
        \vspace{0.5cm}
        
        Dissertation
        
        \vspace{0.8cm}
        
       Zur Erlangung des akademischen Grades\\
        Doktor der Naturwissenschaften\\
         (Dr. rer. nat.)

        \vspace{0.8cm}

        \Large
       Universit\"{a}t Bremen\\
       Fachbereich 3\\
        Deutschland\\
        11-12-2017\\
                        \includegraphics[width=0.7\textwidth]{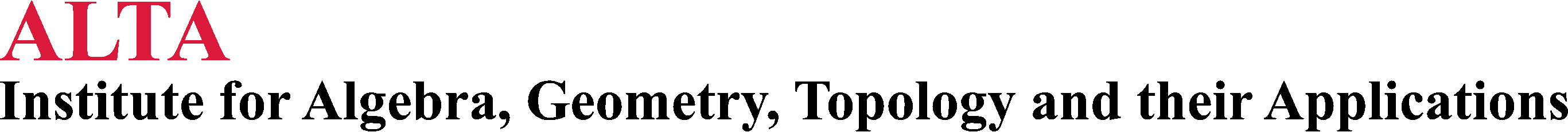}

    \end{center}
\end{titlepage}

\chapter*{Referees}
\begin{enumerate}
\item First referee: Professor Eva Maria Feichtner, University of Bremen.
\item Second referee: Associate Professor Joseph Rabinoff, Georgia Institute of Technology.
\end{enumerate}



\chapter*{Abstract}
In this thesis, we study the Berkovich skeleton of an algebraic curve over a discretely valued field $K$. We do this using coverings $C\rightarrow{\mathbb{P}^{1}}$ of the projective line. To study these coverings, we take the Galois closure of the corresponding injection of function fields $K(\mathbb{P}^{1})\rightarrow{K(C)}$, giving a Galois morphism $\overline{C}\rightarrow{\mathbb{P}^{1}}$. A theorem by Liu and Lorenzini tells us how to associate to this morphism a Galois morphism of semistable models $\mathcal{C}\rightarrow{\mathcal{D}}$. That is, we make the branch locus disjoint in the special fiber of $\mathcal{D}$ and remove any vertical ramification on the components of $\mathcal{D}_{s}$.  
This morphism $\mathcal{C}\rightarrow{\mathcal{D}}$ then gives rise to a morphism of intersection graphs $\Sigma(\mathcal{C})\rightarrow{\Sigma(\mathcal{D})}$. Our goal is to reconstruct $\Sigma(\mathcal{C})$ from $\Sigma(\mathcal{D})$ and 
we will do this by giving a set of covering and twisting data. 
These then give algorithms for finding the Berkovich skeleton of a curve $C$ whenever that curve has a morphism $\overline{C}\rightarrow{\mathbb{P}^{1}}$ with a solvable Galois group. In particular, this gives an algorithm for finding the Berkovich skeleton of any genus three curve. These coverings also give a new proof of a classical result on the semistable reduction type of an elliptic curve, saying that an elliptic curve has potential good reduction if and only if the valuation of the $j$-invariant is positive.   




\chapter*{Acknowledgements}
The author would like to thank his supervisor Professor Eva Maria Feichtner and second referee Associate Professor Joseph Rabinoff for making this thesis possible. Furthermore, the author would like to thank Professor Bernd Sturmfels for his enthusiasm and encouragement throughout the project. The author would also like to thank Madeline Brandt, dr. Martin Ulirsch and Professor Matt Baker for reading through early versions of this thesis and for their helpful remarks. 

This thesis could not have been made without the help of the author's family and friends during this project, for which the author is extremely grateful. The author would also like to thank the Max-Planck Institute for Mathematics in the Sciences in Leipzig for their hospitality during the summer, where he wrote a paper with Madeline Brandt about tropical superelliptic curves. Furthermore, the author would like to thank the Technische Universit\"{a}t Berlin and the Universi\"{a}t Regensburg for their hospitality during two conferences on tropical and non-archimedean geometry, where the author gave two talks on this thesis. 


\begingroup
\color{black}
\tableofcontents
\endgroup


\chapter{Introduction}\label{Introduction}

In this thesis, we will be studying the Berkovich skeleton of an algebraic curve $C$ over a discretely valued field $K$ with uniformizer $\pi$ and residue field $k$. Informally speaking, we view $C/K$ as a family of curves (where the uniformizer is the parameter) and assign a 
combinatorial limit object $\Sigma(C)$ (a weighted metric graph) that retains information about the original family $C/K$. The goal of this thesis is to explicitly find $\Sigma(C)$ for a given $C/K$. The idea is to start with a morphism $C\rightarrow{\mathbb{P}^{1}}$ and then to consider the Galois closure $\overline{\phi}:\overline{C}\rightarrow{}\mathbb{P}^{1}$ of this morphism. We then find the Berkovich skeleton of the Galois closure and take an appropriate quotient to obtain the Berkovich skeleton of $C$.

Throughout this thesis, we will be using the language of semistable models to find the Berkovich skeleton of a curve. An important theorem in this area is the {\it{semistable reduction theorem}} by Mumford and Deligne in \cite{Deligne1969}, which says that over a finite extension $K'$ of $K$, every curve admits a semistable model. In proving this theorem, they used a result by Grothendieck saying that abelian varieties become stable after a finite extension. This leads to an algorithm for finding semistable models, at least in principle. One considers the Jacobian $J(C)$ of a curve $C$ and its $\ell$-torsion $J(C)[\ell]$ for a prime $\ell$. Taking $\ell\geq{3}$ and coprime to the characteristic of the residue field, one extends the base field $K$ so that the $\ell$-torsion becomes rational. Taking the minimal desingularization of $C$ over this field then yields a semistable model.

A lot of the steps given above are hard to perform in practice. We first have to find the Jacobian as an embedded variety in some $\mathbb{P}^{n}$, write down the addition formulas and then find the equations for the $\ell$-torsion. We then take any model $\mathcal{C}$ of $C$ over $K'$ and then desingularize this model. This desingularization process is a fairly easy local computation, consisting of several blow-ups and normalizations. A problem with this approach is that it doesn't tell us what happens to the Berkovich skeleton if we change the curve slightly. In this case, we would have to restart the entire process of finding the Jacobian, adding the torsion points and desingularizing our models. 

The reason that we're interested in varying the curve is as follows. Let $E$ be an elliptic curve over $K$. We then have the following result:
\begin{equation}\label{EllCurvesResult}
E\text{ has potential good reduction if and only if }v(j)>0.
\end{equation}
Here $j$ is the $j$-invariant of the elliptic curve. For more on this, see \cite[Chapter VII]{Silv1}. What this says is that the intersection graph $\Sigma(\mathcal{E})$ of a semistable model $\mathcal{E}$ of $E$ over an extension of $K$ contains a cycle if and only if the valuation of the $j$-invariant is negative. In fact, the length of this cycle is then equal to $-v(j)$. The idea of the proof is as follows. We assume that $\text{char}(k)\neq{2}$. In this case, we can find a Legendre equation for $E$
\begin{equation}
y^2=x(x-1)(x-\lambda)
\end{equation} 
over some finite extension $K'$ of $K$, where $\lambda\in{K'}$. One then explicitly finds the reduction type in terms of the branch points $\{0,1,\lambda,\infty\}$ of the natural morphism given on affines by $(x,y)\mapsto{x}$. In this last step, it is important to know how the reduction type of $E$ changes when we vary $\lambda$.  For more general morphisms, a theorem by Liu and Lorenzini (Theorem \ref{MaintheoremSemSta}) tells us that there is a connection between the branch locus of a morphism $\phi: C\rightarrow{\mathbb{P}^{1}}$ and the Berkovich skeleton of $C$. 
 To obtain results similar to the one in Equation \ref{EllCurvesResult}, we then see that we need to have a good idea of how the branch points of $\phi$ contribute to the reduction type of $C$. 


A similar result for curves of genus two was obtained in \cite{liu} using Igusa invariants. There are six reduction types in this case and the criteria given there are in terms of the valuations of these Igusa invariants. These were cast into a tropical form in \cite{Igusa}. We note that in this case, the criteria for curves over a field with residue characteristic greater than $3$ are different from the criteria for characteristic $2$ and $3$, in contrast to the elliptic curve case.

We are now quite naturally led to the following case: curves of genus three. We then immediately encounter a problem that was not present in the previous two cases. There are curves of genus three that do \emph{not} admit a hyperelliptic covering to the projective line, that is, a degree two morphism $C\rightarrow{\mathbb{P}^{1}}$. Using the canonical embedding, one then finds that the curve can be given as a quartic in $\mathbb{P}^{2}$. Projecting onto a point $P\in{C(K)}$ (which certainly exists after a finite extension), we then obtain a degree three morphism to the projective line.

These degree three morphisms are quite often not Galois, in the sense that the extension of function fields $K(\mathbb{P}^{1})\rightarrow{K(C)}$ is not normal. If it is normal, then $\phi$ is an abelian covering of order three, which we will study in Chapter \ref{Abelian}. Now suppose that $\phi$ is not Galois. We take the Galois closure of this morphism to obtain a degree six Galois covering $\overline{C}\rightarrow{\mathbb{P}^{1}}$. If $\overline{C}$ is not geometrically irreducible, then $\phi$ becomes an abelian morphism after a degree two extension of $K$. We now assume that $\overline{C}$ is geometrically irreducible. Then $\overline{\phi}$ is Galois with Galois group $S_{3}$. Note that this group is solvable, with subnormal series $(1)\vartriangleleft{H}\vartriangleleft{S_{3}}$, where $H$ is the normal subgroup of order three. We can then use our techniques for solvable Galois coverings to find the Berkovich skeleton of $\overline{C}$, see Chapter \ref{Solvable}.  
 

Throughout this thesis, we'll be using a result by Q. Liu and D. Lorenzini in \cite{liu1} for Galois coverings $C\rightarrow{D}$ such that the order of the Galois group is coprime to the characteristic of the residue field. These are also known as tame coverings. This result says that if we add the coordinates of the branch points to our base field $K$, take some model $\mathcal{D}$ (whose construction will be explained in Chapter \ref{Appendix2}) where the branch points reduce to disjoint smooth points on the special fiber and then extend the base field to eliminate any vertical ramification, we obtain a morphism of semistable models
\begin{equation}
\mathcal{C}\rightarrow{\mathcal{D}}.
\end{equation}
This will be the main ingredient in this thesis. A similar result was also obtained analytically in \cite{ABBR1}. 

To reconstruct the Berkovich skeleton of the Galois closure $\overline{C}$, we will use two concepts: \emph{covering data} and \emph{twisting data}. We start with a canonical semistable model of $\mathbb{P}^{1}$ for the morphism $\phi$. Its intersection graph is also known as the tropical separating tree. For every edge and vertex in this graph, we will give a formula for the number of elements in the pre-image of this edge or vertex. This is what we call the covering data. The formulas and proofs will be given in Chapter \ref{Inertiagroups} and the algorithm for the covering data will be given in Chapter \ref{Coveringdata}. We note that the algorithm for the covering data works for general Galois coverings, not just solvable ones.

For many small examples, this covering data is sufficient to determine the Berkovich skeleton of $\overline{C}$. In general, some additional data is needed to completely determine the covering graph. For a cyclic abelian covering $C\rightarrow{D}$, this will be given as a $2$-cocycle on $\Sigma(\mathcal{D})$ in terms of graph cohomology. This will tell us how to glue together any edges and vertices we obtain from the covering data. We will call this additional data \emph{twisting data}. It will be given in Chapter \ref{Twistingdata}. This is the part of the thesis where we use the assumption that the Galois group is abelian.  

\section{Notation}

We will use the following standard notation throughout this thesis:

\begin{itemize}
\item $K$ is a discretely valued complete field of characteristic zero with valuation $v:K^{*}\rightarrow{\mathbb{Z}}$,
\item $R=\{x\in{K}:v(x)\geq{0}\}$ is the valuation ring of $K$,
\item $R^{*}=\{x\in{K}:v(x)=0\}$ is the unit group,
\item $\mathfrak{m}=\{x\in{K}:v(x)>0\}$ is the unique maximal ideal in $R$,
\item $\pi$ is a uniformizer for $v$, i.e. $\pi{R}=\mathfrak{m}$,
\item $k:=R/\mathfrak{m}$ is the residue field of $R$.
\end{itemize}
We will assume that $v$ is normalized so that $v(\pi)=1$. For simplicity, we will also assume that the residue field $k$ is algebraically closed. In practice, it will be sufficient to assume that the residue field is large enough to contain the coordinates of all the branch and ramification points. 
For any finite extension $K'$ of $K$, we let $R'$ be a discrete valuation ring in $K'$ dominating $R$.

For a scheme $X$, we let $\mathcal{O}_{X}$ be its structure sheaf. For any point $x\in{X}$ we let $\mathcal{O}_{X,x}$ be the stalk of $\mathcal{O}_{X}$ at $x$. It is a local ring with maximal ideal $\mathfrak{m}_{x}$. A \emph{generic point} of an irreducible component $\Gamma$ is a point $y\in{X}$ such that $\overline{\{y\}}=\Gamma$. A point $x$ is a \emph{specialization} of a point $y$ if $x\in\overline{\{y\}}$.  
 For any Noetherian local ring $A$ with maximal ideal $\mathfrak{m}_{A}$, we let $\hat{A}$ be its $\mathfrak{m}_{A}$-adic completion, as in \cite[Section 1.3]{liu2}.
 
 For graphs, we use the definition found in \cite[Section 2.1]{Serre1980}. A graph $\Sigma$ consists of a set $E$ and a set $V$, together with two maps 
 \begin{align*}
E&\rightarrow{V\times{V}}           &  e &\mapsto{}(o(e),t(e))              &  {}&{}\\
E&\rightarrow{E}         &  e&\mapsto{\overline{e}}   &  {}&{},
\end{align*}
 which satisfy the following condition: for every $e$ in $E$, we have $\overline{\overline{e}}=e$, $\overline{e}\neq{e}$ and $o(e)=t(\overline{e})$. The set $E$ is known as the edge set, the set $V$ as the vertex set, $o(e)$ as the outgoing vertex of $e$ and $t(e)$ as the target vertex of $e$. An orientation of the graph is a subset $Y_{+}$ of $Y$ such that $Y$ is the disjoint union of $Y_{+}$ and $\overline{Y}_{+}$. When we're not interested in the orientation, we will refer to the set $\{e,\overline{e}\}$ as one edge. 


\section{Curves and fibered surfaces}

An algebraic variety over $K$ is a scheme of finite type over $\text{Spec }{K}$ and a curve over $K$ is an algebraic variety whose irreducible components are of dimension 1. For integral algebraic varieties $X$ over $K$, we denote their function fields by $K(X)$. We then say that $X$ is geometrically irreducible if the base change of $X$ to the algebraic closure of $K$ is irreducible.
\begin{exa}
Let $X$ be given by $\text{Spec}(\mathbb{Q}_{2}[x,y]/(x^2-2y^2))$, where $\mathbb{Q}_{2}$ is the field of $2$-adic numbers. Then $X$ is irreducible, but not geometrically irreducible, since the base change to $\mathbb{Q}_{2}(\sqrt{2})$ gives two irreducible components with generic points $(x\pm\sqrt{2}y)$. 
\end{exa}
\begin{lemma}\label{GeometricallyIrreducible}
An integral algebraic variety over $K$ with function field $K(X)$ is geometrically irreducible if and only if $K(X)\cap{K^{s}}=K$, where $K^{s}$ is the separable closure of $K$.
\end{lemma}
\begin{proof}
See \cite[Chapter 3, Corollary 2.14]{liu2}. 
\end{proof}
We say that an algebraic variety over $K$ is smooth at a point $x\in{X}$ if the points of $X_{\overline{K}}$ lying above $x$ are regular points of $X_{\overline{K}}$. We then say that $X$ is smooth over $K$ if it is smooth at all of its points. 
We now define the arithmetic genus of a projective curve. We start with the definition of the Euler-Poincar\'{e} characteristic of a coherent sheaf of a projective variety over a field. So let $X$ be a projective variety over a field $K$ and let $\mathcal{F}$ be a coherent sheaf. We then define
\begin{equation}
\chi_{K}(\mathcal{F})=\sum_{i\geq{0}}(-1)^{i}\text{dim}_{K}H^{i}(X,\mathcal{F}),
\end{equation} 
where the $H^{i}(X,\mathcal{F})$ are the \v{C}ech cohomology groups of $X$. 
We have that $H^{i}(X,\mathcal{F})=0$ for $i>\text{dim}(X)$, so the above sum is finite. We now define the \emph{arithmetic genus} of a curve $X$ over a field to be
\begin{equation}
p_{a}(X):=1-\chi_{K}(\mathcal{O}_{X}).
\end{equation} 
We will also refer to this integer as the genus of the curve $X$, where we sometimes write $g(X):=p_{a}(X)$.

We now move from algebraic varieties over a field $K$ to schemes over the discrete valuation ring $R$. We will mostly follow Chapters 8,9 and 10 in \cite{liu2}.
   A {\bf{fibered surface}} over $S:=\text{Spec }R$ (in short: over $R$) is an integral, projective, flat $R$-scheme $\tau:\mathcal{C}\longrightarrow{S}$ of dimension 2. The generic fiber of $\mathcal{C}$ will be denoted by $\mathcal{C}_{\eta}$ and the special fiber by $\mathcal{C}_{s}$. An {\bf{arithmetic surface}} is a fibered surface over $S$ that is regular.  A {\bf{model}} of a curve $C$ over $K$ is a normal fibered surface $\mathcal{C}\longrightarrow{S}$ together with an isomorphism $f:\mathcal{C}_{\eta}\simeq{C}$.
   Let $z$ be a closed point in $\mathcal{C}$. We say that $z$ is an {\bf{ordinary double point}} if
   \begin{equation}
   \hat{\mathcal{O}}_{\mathcal{C},z}\simeq{R[[x,y]]/(xy-\pi^{n})}
   \end{equation}
   for some $n\in\mathbb{N}$. We call the integer $n$ the \emph{thickness} or \emph{length} of $z$. 
  A model $\mathcal{C}$ is said to be {\bf{semistable}} if the special fiber $\mathcal{C}_{s}$ is reduced and has only ordinary double points as its singularities. We will adopt the terminology of \cite{baker} and say that the model $\mathcal{C}$ is {\bf{strongly semistable}} if in addition to semistability the irreducible components of $\mathcal{C}_{s}$ are all smooth. 

\begin{exa}\label{ExampleChapter1}
We illustrate some of the local properties in the above definitions.
Let $A:=R[x,y]/I$, with $I$ specified below.  
We assume that $\text{char}(k)\neq{2}$. 
\begin{enumerate}
\item ({\it{Flatness}}) Let $I=(\pi(y^2-x^3-1))$. Then the generic fiber is an elliptic curve and the special fiber is $k[x,y]$. 
The ring $A$ is not flat over $R$, since it contains torsion. 
\item ({\it{Ordinary double point with a non-smooth component}}) 
Take $I=(y^2-x^3-x^2-\pi)$. The special fiber is then given by $y^2=x^3+x^2$, which is not smooth, since the point $\mathfrak{p}=(x,y)$ is not regular. 
\item \label{ExaS2} ({\it{Ordinary double point with two smooth components}}) Take $I=(y^2-f)$, where 
\begin{equation*}
f=x(x-\pi)(x+1)(x+1-\pi)(x+2)(x+2-\pi).
\end{equation*} 
The special fiber then consists of two irreducible components, given by $y=\pm{x(x+1)(x+2)}$. These intersect each other transversally in the three points $(0),(-1),(-2)$.
\end{enumerate} 
\end{exa}

Let us first define some properties of morphisms of curves over $\text{Spec}(K)$. 
 Let $\phi:C\rightarrow{D}$ be a finite morphism of smooth, projective, geometrically irreducible curves over $K$. We say that $\phi$ is Galois with Galois group $G$ if the corresponding injection of function fields $K(D)\rightarrow{K(C)}$ is Galois with Galois group $G$. We say that $\phi$ is separable if the corresponding injection of function fields is separable. The degree of $\phi$ is defined to be the degree of the field extension $K(D)\subseteq{K(C)}$. 


For a morphism of curves as defined above, we then have the following version of the {\it{Riemann-Hurwitz}} formula. 
\begin{theorem}{\bf{[Riemann-Hurwitz formula]}}\label{RiemannHurwitz}
Let $\phi:C\rightarrow{D}$ be a finite, separable morphism of smooth projective curves over $K$. Then
\begin{equation}
2p_{a}(C)-2=\text{deg}(\phi)(2p_{a}(D)-2)+\sum_{P\in{C}}(e_{P}-1).
\end{equation}
Here $e_{P}$ is the ramification index of $\phi$ at $P$.
\end{theorem}
\begin{proof}
See \cite[Chapter 7, Theorem 4.16]{liu2}. 
\end{proof}

Now let $\mathcal{C}$ and $\mathcal{D}$ be models for $C$ and $D$ respectively. A \emph{finite morphism of models for} $\phi$ is a finite morphism $\phi_{\mathcal{C}}:\mathcal{C}\rightarrow{\mathcal{D}}$ over $\text{Spec}(R)$ such that the base change to $\text{Spec}(K)$ gives $\phi:C\rightarrow{D}$.

\section{Intersection graphs and Berkovich skeleta}\label{IntersectionGraphIntro}

Let $\mathcal{C}$ be a strongly semistable model. In this section, we define the intersection graph of $\mathcal{C}$. We furthermore relate these graphs to the main object in this thesis: the Berkovich skeleton.

\begin{mydef}({\bf{Dual Intersection Graph}})
Let $\mathcal{C}$ be a strongly semistable model for a curve $C$ over $K$. Let $\{\Gamma_{1},...,\Gamma_{r}\}$ be the set of irreducible components. We define the dual intersection graph $\Sigma(\mathcal{C})$ of $\mathcal{C}_{}$ to be the finite graph whose vertices $v_{i}$ correspond to the irreducible components $\Gamma_{i}$ of $\mathcal{C}_{s}$ and whose edges correspond to intersections between components. The latter means that we have one edge for every point of intersection. 
We write $V(\Sigma(\mathcal{C}))$ for the vertex set and $E(\Sigma(\mathcal{C}))$ for the edge set of $\Sigma(\mathcal{C})$. 
\end{mydef}

\begin{exa}
In Example \ref{ExampleChapter1}.\ref{ExaS2},  the intersection graph consists of two vertices with three edges between them. One can find the graph in Figure \ref{Eersteechteplaatje}. The morphism $C\rightarrow{\mathbb{P}^{1}}$ collapses the three edges to smooth points on the only component of the semistable model $\mathbb{P}_{R}^{1}$. 
\begin{figure}[h!]
\centering
\includegraphics[scale=0.6]{{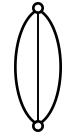}}
\caption{\label{Eersteechteplaatje} The intersection graph in Example \ref{ExampleChapter1}.\ref{ExaS2}.} 
\end{figure}
\end{exa}

We will also want to keep track of the genera of the components. We will do this by assigning to every vertex in the dual intersection graph its associated genus. We define 
\begin{equation*}
w(v_{i}):=g(\Gamma_{i}).
\end{equation*} 
Whenever we draw the graph of a certain curve, we will write the genera next to the components in question. Whenever the component has genus 0, we will omit the zero. This function
\begin{equation}
w:V(\Sigma)\rightarrow{\mathbb{N}}
\end{equation}
now turns the intersection graph into a \emph{weighted graph}. 
We have the following
\begin{theorem}\label{TheoremGenus}
Let $\mathcal{C}$ be a strongly semistable model for a smooth curve $C$ over $K$ with intersection graph $G$. Let $\beta(G)$ be the Betti number of $G$ and let $p_{a}(\mathcal{C}_{s})$ be the arithmetic genus of $\mathcal{C}_{s}$. We then have
\begin{equation*}
p_{a}(\mathcal{C}_{s})=\beta(G)+\sum_{1\leq{i}\leq{r}}w(v_{i}).
\end{equation*}
\end{theorem}
\begin{proof}
See \cite[Page 511]{liu2}.
\end{proof}
Let us now define the notion of a {\it{weighted metric graph}}. 
\begin{mydef}
A {\bf{weighted metric graph}} is a triple $(\Sigma,w(\cdot{}),l(\cdot{}))$, where 
\begin{itemize}
\item $\Sigma$ is a finite graph, 
\item $w(\cdot{})$ a function $w:V(\Sigma)\rightarrow{\mathbb{N}}$,
\item $l(\cdot{})$ a function $l: E(\Sigma)\rightarrow{\mathbb{N}}$. 
\end{itemize}
We refer to $w(\cdot{})$ as the {\it{weight}} function and $l(\cdot{})$ as the {\it{length}} function associated to $\Sigma$. 
\end{mydef}
We now turn our weighted intersection graph $(\Sigma(\mathcal{C}),w)$ into a {\it{weighted metric graph}}.
To do this, we need to assign a notion of length to our edges. Let $e$ be an edge in $\Sigma(\mathcal{C})$, corresponding to an intersection point $z\in\mathcal{C}$. Recall that we have the following isomorphism for the completed local ring of $z$:
\begin{equation}
 \hat{\mathcal{O}}_{\mathcal{C},z}\simeq{R[[x,y]]/(xy-\pi^{n})}.
\end{equation} 
We then define the {\bf{length function}} $l:E(\Sigma(\mathcal{C}))\rightarrow{\mathbb{N}}$ by
\begin{equation}
l(e)=n.
\end{equation}
Different semistable models can give rise to subdivisions of our graph $\Sigma(\mathcal{C})$, so we need to define the notion of refinements. To obtain the \emph{minimal} Berkovich skeleton, we also need to do some pruning and delete the leaves. 

\begin{mydef}\label{Refinement}
 A {\bf{refinement}} of $\Sigma(\mathcal{C})$ is a graph obtained from $\Sigma(\mathcal{C})$ by subdividing the edges of $\Sigma(\mathcal{C})$ in a length-preserving fashion. Here we only allow subdivisions where the edges have integer lengths and the new vertices have weight zero.
  We say that two weighted metric graphs $\Sigma$ and $\Sigma'$ are \emph{equivalent} if they admit a common refinement $\tilde{\Sigma}$. We write $\Sigma\sim\Sigma'$. 
 \end{mydef}
 \begin{rem}
 Every weighted metric graph as defined above now has a "maximal refinement", in the sense that we can subdivide any edge of length $n$ with vertices $v_{1}$ and $v_{n+1}$ into a graph with $n+1$ vertices $\{v_{1},v_{2},...,v_{n+1}\}$ and edges $e_{i,i+1}$ of length $1$. 
 \end{rem}


\begin{mydef}
Let $\Sigma(\mathcal{C})$ be as above. 
A {\bf{leaf}} of $\Sigma(\mathcal{C})$ is a subgraph $L$ of $\Sigma(\mathcal{C})$ with vertex set $\{v\}$ and edge set $\{e\}$, where $v\in\Sigma(\mathcal{C})$ has valency one, genus zero and $e$ is the edge connected to $v$. A weighted metric graph without leaves is called leafless. 
\end{mydef}

\begin{mydef}
Let $\Sigma(\mathcal{C})$ be as above. Consider the subgraph $\Sigma(C)$ obtained from $\Sigma(\mathcal{C})$ by deleting all the leaves. The equivalence class of this graph $\Sigma(C)$ under refinements 
 of leafless weighted metric graphs is the {\bf{Berkovich skeleton}} of $C$. 
\end{mydef}

\begin{rem}
This graph can also be obtained algebraically: we take the semistable model $\mathcal{C}$ and contract all the \emph{exceptional} divisors $E$ which have self-intersection $-1$, see Chapter \ref{IntersectionTheory} and \cite[Chapter 9.3.1]{liu2}. 
The desingularization of this model is then the minimal regular model for curves of genus $\geq{1}$. The intersection graph of this minimal regular model is then exactly the leafless maximal refinement.  
\end{rem}

\begin{rem}
This definition makes no reference to Berkovich spaces, but it gives the same skeleton as defined in that context, see \cite{berkovich2012}. In terms of semistable vertex sets (see \cite{BPRa1}), this skeleton is known as the \emph{minimal} Berkovich skeleton. 
\end{rem}

\begin{rem}
In Section \ref{Metrizedcomplex}, we will enhance the weighted metric graph $\Sigma$ with additional data in the form of an explicit curve $C_{v}/k$ for every vertex $v\in\Sigma$. This will turn $\Sigma$ into a metrized complex of $k$-curves. 
\end{rem}

\section{Main problems}

We now give a summary of the main problems we wish to address in this thesis. They are as follows:

\begin{enumerate}
\item There exist criteria for the Berkovich skeleta of elliptic curves and genus two curves in terms of coordinates on their coarse moduli spaces, see \cite[Chapter VII]{Silv1} and \cite{Igusa}. Can they be generalized to curves of higher genus? 
\item Is there a fast algorithm for finding the Berkovich skeleton of a genus three curve?
\item Are there fast algorithms for finding the Berkovich skeleton of other types of curves?  
\end{enumerate}

We will answer these questions in Chapter \ref{Conclusion}. 

To answer these questions, we used the analogy between coverings of curves and finite extensions of number fields as our motivation.
It is in this theory of number fields that one quite quickly sees that it is better to consider the fully symmetric version, the Galois closure, of a finite extension of number fields to study the decomposition of primes. This then also yields the decomposition for the subfields by taking an appropriate quotient. The idea in this thesis is to view the vertices and edges of a Berkovich skeleton as the primes in a number field and then to find the decomposition groups of these primes. 
This then locally gives the Berkovich skeleton of the curve lying above it and in order to give the full skeleton some additional data has to be added. We call this the twisting data of the covering. Something similar happens for number fields: knowing the decomposition of primes for a covering $L\supset{K}$ doesn't directly give any \emph{global} information like the class number $h_{L}$ of the number field $L$. We will view the Berkovich skeleton of a curve $C$ as an analogue of the class group/class number in number theory.  

\chapter{Divisors on curves and graphs}\label{Divisors}

The problem we wish to address here is as follows: we wish to {\it{transport}} divisors from a curve $C$ to a strongly semistable regular model $\mathcal{C}$ and then to its intersection graph $\Sigma(\mathcal{C})$. This will require some notions from graph theory and intersection theory. In each of the three settings we have a notion of a {\it{principal}} divisor. This will then give us the notion of a "Jacobian" in each scenario.\\
We will start with intersection graphs and Jacobians on these intersection graphs. Here we will introduce the Laplacian operator. We will then move to intersection theory on $\mathcal{C}$, where we will show how to move from divisors on $\mathcal{C}$ to divisors on the intersection graph. Lastly, we will study how the N\'{e}ron model of the Jacobian of $C$ fits into all of this and how we can make sense of the identity component of that N\'{e}ron model in terms of graph cohomology. 
\section{Divisors on graphs and Laplacians}\label{DivisorsLaplacians}
So let $G$ be a graph, which we will assume to be finite, connected and without loop edges. Let $V(G)$ be its vertices and $E(G)$ its edges. We define $\text{Div}(G)$ to be the free abelian group on the vertices $V(G)$ of $G$. 
Writing $D\in{\text{Div}(G)}$ as $D=\sum_{v\in{V(G)}}c_{v}(v)$, we define the degree map as $\text{deg}(D)=\sum_{v\in{V(G)}}c_{v}$. We let $\text{Div}^{0}(G)$ be the group of divisors of degree zero on $G$.\\ Now let $\mathcal{M}(G)$ be the group of $\mathbb{Z}$-valued functions on $V(G)$. Define the {\bf{Laplacian operator}} $\Delta:\mathcal{M}(G)\longrightarrow\text{Div}^{0}(G)$ by 
\begin{equation*}
\Delta(\phi)=\sum_{v\in{V(G)}}\sum_{e=vw\in{E(G)}}(\phi(v)-\phi(w))(v).
\end{equation*}   
We then define the group of principal divisors to be the image of the Laplacian operator:
\begin{equation*}
\text{Prin}(G):=\Delta(\mathcal{M}(G)).
\end{equation*}


\begin{mydef}[{\bf{Tropical Jacobians}}]
We define the {\it{tropical Jacobian}} of $G$ to be the group
\begin{equation}
J(G):=\text{Div}^{0}(G)/\text{Prin}(G).
\end{equation}
\end{mydef}

\begin{exa}\label{FirstExa}
Suppose we take $\text{Proj }R[X,Y,W]/(XY-\pi{W}^2)$ with its usual grading. As before, we have two components intersecting each other in one point. It is now quite easy to see that every divisor of degree zero is in fact principal. Take any $D$ of degree zero. Then $D(\Gamma_{1})=-D(\Gamma_{2})$. Let us define
\begin{eqnarray*}
\phi(\Gamma_{1})&=&0,\\
\phi(\Gamma_{2})&=&D(\Gamma_{2}).
\end{eqnarray*}
Then $\phi$ has the right divisor and as such every divisor is principal.
\end{exa}
\begin{exa}\label{SecExa2}
\begin{figure}[h!]
\centering
\includegraphics[scale=0.44]{{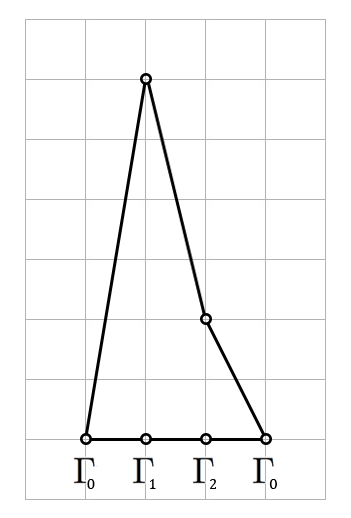}}
\caption{\label{16eplaatjeExtra} {\it{The graph of the function $\phi$ considered in Example \ref{SecExa2}.}}}
\end{figure}
Throughout this thesis, we will connect the values of $\phi$ by the unique line between them. 
An example of a Laplacian can be found in Figure \ref{16eplaatjeExtra}. The graph in question is given in Figure \ref{27eplaatje}.
\begin{figure}[h!]
\centering
\includegraphics[scale=0.4]{{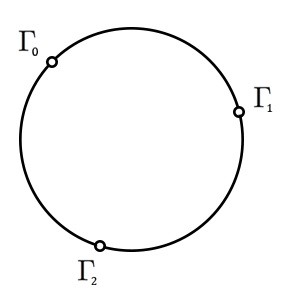}}
\caption{\label{27eplaatje} {\it{The graph considered in Example \ref{SecExa2}.}}}
\end{figure}
 The divisor of the Laplacian in this figure is 
\begin{equation*}
\Delta(\phi)=-8(\Gamma_{0})+10(\Gamma_{1})-2(\Gamma_{2}).
\end{equation*}
\end{exa}

We would like to quickly mention a connection between this tropical Jacobian and a well-known theorem on spanning trees in a graph: {\bf{Kirchhoff's Theorem}}.
\begin{theorem}
Let $G$ be a finite connected graph. Then the order of the tropical Jacobian of $G$ is equal to the number of maximal spanning trees in $G$.
\end{theorem} 

\begin{exa}
Let us take the graph from Example (\ref{ExampleChapter1}.\ref{ExaS2}).
 Then there are three maximal spanning trees, and so the order of tropical Jacobian is three. This of course also means that $\text{Jac}(G)\simeq{\mathbb{Z}/3\mathbb{Z}}$. 
\end{exa}

\begin{rem}
 We will later see that the tropical Jacobian is canonically isomorphic to the component group of the N\'{e}ron model of the Jacobian of $C$, see Section \ref{JacNer} or \cite[Page 24]{baker}. Using Kirchhoff's theorem we can say that the order of this component group is then equal to the number of maximal spanning trees. 
\end{rem}

\begin{rem}
As in the case of algebraic curves, one has multiple ways of constructing a "{\it{Tropical Jacobian}}". In \cite[Page 203]{MZ1},
 a tropical Jacobian is constructed using differential forms: one takes the dual $\Omega(C)^{*}=\text{Hom}(\Omega(C),\mathbb{R})$ of the space of holomorphic differentials $\Omega(C)$, where $C$ is a tropical curve. By integration, one obtains a lattice $H_{1}(\Gamma,\mathbb{Z})$ in this vector space and one then sets
\begin{equation*}
J(C):=\Omega(C)^{*}/H_{1}(\Gamma,\mathbb{Z}).
\end{equation*}
After chosing a basis of $\Omega(C)^{*}$, one then obtains a noncanonical isomorphism $J(C)\simeq{}(\mathbb{R}/\mathbb{Z})^{g}$.
This Jacobian can then be described entirely in terms of the associated intersection graph, as in \cite[Page 35, Section 5]{BBC1}.
 This is already much closer to our approach. \\One obvious difference between this approach and our approach is that our tropical Jacobian is {\it{finite}}. As noted in \cite[Remark A.11]{baker}, we can get somewhat closer by considering the limit over finite extensions $K'\supset{K}$ to obtain a $\mathbb{Q}$-rational tropical Jacobian $J_{\mathbb{Q}}(\Gamma)$ which is noncanonically isomorphic to $(\mathbb{Q}/\mathbb{Z})^{g}$.\\
Let us describe these phenomena in a particular case: an elliptic curve $E$ with multiplicative reduction. Over a discretely valued field $K$ with $v(\pi)=1$, one then obtains an isomorphism with the Tate curve $E(K)\simeq{}K^{*}/<q>$ for some $q$ with positive valuation equal to $-v(j)$. One can then define a "naive" tropicalization map
\begin{eqnarray*}
\text{trop}:(K)^{*}/<q>&\longrightarrow&{\mathbb{Z}/v(q)\mathbb{Z}},\\
\lbrack{x}\rbrack{}&\longmapsto&\lbrack{v(x)}\rbrack{}.
\end{eqnarray*}
This is easily seen to be well-defined. To study the passage to finite extensions of $K$, let us consider the easy example of a ramified extensions of degree $n$ given by $K\subset{}K':=K(\pi^{1/n})$. We extend the valuation on $K$ by $v(\pi^{1/n})=1/n$. As before, one has an isomorphism $E(K')={(K')}^{*}/<q>$. See [\cite{Silv2}, Chapter V] for this. If we take a similar naive tropicalization as before, one obtains $\text{trop}(E(K'))=(\dfrac{1}{n}\mathbb{Z})/(v(q)\mathbb{Z})\simeq{\mathbb{Z}/(n\cdot{}v(q)\mathbb{Z})}$. Taking this argument further to an algebraic closure of $K$, we then easily obtain $\text{trop}(E(\bar{K}))=\mathbb{Q}/\mathbb{Z}$. 

\end{rem}

\section{Intersection theory on $\mathcal{C}$}\label{IntersectionTheory}

Here we will start transporting divisors. Let us suppose now that we have a strongly semistable {\it{regular}} model $\mathcal{C}$. As before we will consider its intersection graph $\Sigma(\mathcal{C})$ and the irreducible components $\{\Gamma_{1},...,\Gamma_{n}\}$. Let $\text{Div}(C)$ (resp. $\text{Div}(\mathcal{C})$) be the group of Cartier divisors on $C$ (resp. $\mathcal{C}$). Since both $C$ and $\mathcal{C}$ are {\it{regular and integral}}, we have by \cite[Page 271]{liu2} that these Cartier divisors correspond to Weil divisors. Similarly, we will let $\text{Prin}(C)$ (resp. $\text{Prin}(\mathcal{C})$) be the group of principal Cartier divisors on $C$ (resp. $\mathcal{C}$). Note also that we have that $\mathcal{C}$ is {\it{normal}} (because $\mathcal{C}$ is regular or because $\mathcal{C}_{s}$ is reduced and $\mathcal{C}_{\eta}$ is normal), so we can talk about valuations at codimension one primes.

The intersection theory that we now need is described in \cite[Page 381]{liu2} and \cite[Page 7]{baker}. We will give a quick summary and refer the reader to the aforementioned sources for the details. Let $\text{Div}_{s}(\mathcal{C})$ be the set of Cartier divisors on $\mathcal{C}$ with support in $\mathcal{C}_{s}$. These are known as the {\bf{vertical divisors}}. This group has the $\Gamma_{i}$ as a $\mathbb{Z}$-basis. (We will later also define the {\it{horizontal divisors}}). At any rate, there exists a bilinear map (the intersection map)
\begin{equation*}
\text{Div}(\mathcal{C})\times{\text{Div}_{s}(\mathcal{C})}\longrightarrow{\mathbb{Z}},
\end{equation*} 
which we will write as $\mathcal{D}\cdot{E}$ for Cartier divisors $\mathcal{D}$ and $E$, where $E\subseteq\mathcal{C}_{s}$. This can then be computed as 
\begin{equation*}
\mathcal{D}\cdot{E}=\text{deg }\mathcal{O}_{X}(\mathcal{D})|_{E}.
\end{equation*}
One special case that needs attention is the {\it{self-intersection}} of elements of $\text{Div}_{s}(\mathcal{C})$. Suppose we have $E\subseteq{\mathcal{C}_{s}}$. The number $E\cdot{E}$ is called the {\it{self-intersection}} of $E$ and is denoted by $E^2$. We then have the following proposition that gives us the self-intersection numbers:

\begin{pro}
Let $\mathcal{C}\longrightarrow{S}$ be as above. The following properties are then true.
\begin{enumerate}
\item For any $E\in\text{Div}_{s}(\mathcal{C})$, we have $\mathcal{C}_{s}\cdot{E}=0$.
\item Let $\Gamma_{i}$ be the irreducible components of $\mathcal{C}_{s}$. Then for any $i\leq{r}$, we have
\begin{equation*}
\Gamma_{i}^{2}=-\sum_{j\neq{i}}\Gamma_{i}\cdot{\Gamma_{j}}.
\end{equation*}  
\end{enumerate} 
\end{pro}
\begin{proof}
This is Chapter 9, Proposition 1.21 in \cite{liu2}. Note that the multiplicities in our case are all 1, so the formula simplifies. 
\end{proof}

\begin{rem}
In the semistable case, all intersections will be {\it{transversal}}, meaning that 
\begin{equation*}
\Gamma_{i}\cdot\Gamma_{j}=\#\{\text{intersection points of } \Gamma_{i}\text{ and }\Gamma_{j}\}.  
\end{equation*}
This means that the self-intersection number of any $\Gamma_{i}$ is just the total number of intersections with other components.
\end{rem}

\begin{exa}
\begin{enumerate}
\item Let us take $\mathcal{C}=\text{Proj}R[X,Y,W]/(XY-\pi{}W^2)$ with affine chart 
\begin{equation*}
A=R[x,y]/(xy-\pi),
\end{equation*} where $x=\dfrac{X}{W}$ and $y=\dfrac{Y}{W}$. Then $\Gamma_{1}=\overline{(x)}$ and $\Gamma_{2}=\overline{(y)}$. Then $\Gamma_{1}\cdot\Gamma_{2}=1$ and $\Gamma_{i}^2=-1$.  
\item Let us consider Example \ref{ExampleChapter1}.\ref{ExaS2} 
again. We have two components $\Gamma_{1}$ and $\Gamma_{2}$. Then $\Gamma_{1}\cdot\Gamma_{2}=3$ and as such we have $\Gamma_{i}^2=-3$.
\end{enumerate}
\end{exa}

\subsection{From $\text{Div}(\mathcal{C})$ to $\text{Div}(\Sigma(\mathcal{C}))$}

Using the intersection theory above, we can now transport our divisors from $\mathcal{C}$ to $\Sigma(\mathcal{C})$. We define a homomorphism $\rho: \text{Div}(\mathcal{C})\longrightarrow{\text{Div}(\Sigma(\mathcal{C}))}$ with
\begin{equation*}
\rho(\mathcal{D})=\sum_{v_{i}\in{\Sigma(\mathcal{C})}}({\mathcal{D}}\cdot\Gamma_{i})(v_{i}).
\end{equation*}
We call this map the {\it{specialization map}}.

\begin{exa}
Suppose we take $\mathcal{C}=\text{Proj}R[X,Y,W]/(XY-\pi{}W^2)$ again. Then 
\begin{equation*}
\rho(\Gamma_{1})=\Gamma_{1}^2(v_{1})+(\Gamma_{1}\cdot\Gamma_{2})(v_{2})=-1\cdot{(v_{1})}+1\cdot(v_{2}).
\end{equation*}
\end{exa}
\begin{exa}
Let us consider Example \ref{ExampleChapter1}.\ref{ExaS2} again. 
We then see that
\begin{equation*}
\rho(\Gamma_{1})=-3\cdot(v_{1})+3\cdot(v_{2}).
\end{equation*}
Note that this divisor is actually trivial in the tropical Jacobian. We have that the negative of the characteristic function of the vertex $v_1$ has divisor equal to $-3(v_{1})+3(v_{2})$, so $\rho(\Gamma_{1})$ is in the image of $\Delta$ (the Laplacian). This happens in general: a multiple of the negative of the characteristic function at a vertex $v_{i}$ is equal to $\rho(\Gamma_{i})$, see Lemma \ref{VertInd1}.  

Thus the image of any vertical divisor in the tropical Jacobian is {\it{zero}}. If we want nontrivial examples of elements of $J(\Sigma(\mathcal{C}))$, we have to look elsewhere. This is given by the {\it{horizontal divisors}}, which we will discuss in Section \ref{Horizontaldivisors}.
\end{exa}


\subsection{The intersection matrix}

We will now associate a matrix $A$ to the restriction of the intersection pairing $i_{s}(\cdot{},\cdot{})$ to the special fiber, known as the {\it{intersection matrix}}. A good reference for the material below is \cite[Chapter 9, page 272]{Bosch1990}. The finitely generated, torsion-free $\mathbb{Z}$-module $\text{Div}_{s}(\mathcal{C})$ has the irreducible components $\Gamma_{i}$ as a basis. We then construct the intersection matrix by
\begin{equation}
a_{i,j}:=(\Gamma_{i}\cdot{\Gamma_{j}}).
\end{equation} 
We can view it as a linear map $\mathbb{Z}^{n}\rightarrow{\mathbb{Z}^{n}}$, where $n$ is the number of irreducible components $\Gamma_{i}$ in $\mathcal{C}_{s}$. We let $e_{i}$ be the standard basis of $\mathbb{Z}^{n}$, so that $e_{i}$ corresponds to $\Gamma_{i}$. Note that the intersection pairing $i_{s}(\cdot{},\cdot{})$ is now given by the bilinear form $<v,w>=v^{T}\cdot{(Aw)}$ induced by $A$.
We would now like to know the rank of $A$. It is given by
\begin{pro}\label{RankIntersect}
The rank of $A$ is $n-1$. Its kernel is generated by the element $c:=(1,1,...,1)$, corresponding to $\sum_{i}\Gamma_{i}=\mathcal{C}_{s}$.
\end{pro}
\begin{proof}
Suppose that $v\in\text{Ker}(A)$. In particular, we then have $v^{T}Av=0$. By \cite[Chapter 9, Theorem 1.23]{liu2}, we then find that $v\in{c\cdot{\mathbb{R}}}$. This then easily implies that $v=n\cdot{c}$ for some $n\in\mathbb{Z}$. 

Conversely, consider the vector $Ac$. For every basis vector $e_{i}$, we calculate
\begin{equation}
{e_{i}}^{T}\cdot{Ac}=i_{s}(\Gamma_{i},\mathcal{C}_{s})=0,
\end{equation}
where the last equality can be found in \cite[Chapter 9, Proposition 1.21]{liu2}. This then implies that $Ac=0$, as desired. 
\end{proof}

\begin{cor}\label{KernelRho}
Consider the restriction $\rho_{\mathcal{C}_{s}}$ of $\rho$ to the divisors with support in the special fiber $\mathcal{C}_{s}$. Then 
\begin{equation}
\text{Ker}(\rho_{\mathcal{C}_{s}})=<\mathcal{C}_{s}>.
\end{equation} 
\end{cor}





\section{Transporting divisors from $C$ to $\mathcal{C}$}\label{Horizontaldivisors}
Now we would like to transport divisors from $\text{Div}(C)$ to $\text{Div}(\mathcal{C})$. Suppose we have any divisor $D\in\text{Div}(C)$. We can now take the closure $\mathcal{D}$ of $D$ inside $\mathcal{C}$. This naturally gives a Cartier divisor of $\mathcal{C}$. These are known as the {\bf{horizontal divisors}}. We will associate a function to the above transportation. Define $\psi:\text{Div}(C)\longrightarrow{\text{Div}(\mathcal{C})}$ by
\begin{equation*}
\psi(D)=\overline{D},
\end{equation*} where the closure is in $\mathcal{C}$.
We will make ths process a bit more explicit using the \emph{reduction map}. 
 \begin{mydef}\label{ReductionMap11}
Let $S$ be the spectrum of a Henselian discrete valuation ring $R$. Let $\mathcal{X}\rightarrow{}S$ be surjective and proper, with generic fiber $X$. Let $X^{0}$ denote the set of closed points of $X$. We define the map $r_{\mathcal{X}}:X^{0}\rightarrow{\mathcal{X}_{s}}$ by
\begin{equation}
r_{\mathcal{X}}(x)=\overline{\{x\}}\cap{\mathcal{X}_{s}}.
\end{equation}
We call $r_{\mathcal{X}}$ the reduction map associated to $\mathcal{X}$. We then say that $x$ reduces to $r_{\mathcal{X}}(x)$. 
 \end{mydef}
 \begin{rem}
We note first that $r_{\mathcal{X}}$ is surjective by \cite[Proposition 1.36, Page 468]{liu2}. Note also that in the definition of the reduction map, one needs the ring $R$ to be Henselian
because otherwise there could be multiple reduction points. One can consider the example
\begin{equation*}
\mathcal{X}=\text{Spec}(\mathbb{Z}_{(5)}[x]/(x^2+1))\longrightarrow{\text{Spec}(\mathbb{Z}_{(5)}[x]/(x^2+1))},
\end{equation*}
where $\mathbb{Z}_{(5)}$ is the localization of $\mathbb{Z}$ at $(5)$. Consider the closed point $(0)$ of the generic fiber.  
There are then two possible reductions: $(x-1,5)$ and $(x-2,5)$. Note that if we instead take the $5$-adic ring in the above example, then $\mathcal{X}$ has two connected components. 
\end{rem}
\begin{mydef}
Let $\mathcal{X}\longrightarrow{\text{Spec}(R)}$ be irreducible, surjective and proper.  Let $\tilde{x}$ be a closed point of $\mathcal{X}_{s}$. Define
\begin{equation*}
X_{+}(\tilde{x}):=r_{\mathcal{X}}^{-1}(\tilde{x}).
\end{equation*}
This is known as the {\it{formal fiber}} of $\tilde{x}$.
\end{mydef}
\begin{rem}
For semistable models, these formal fibers are naturally isomorphic to open annuli and spheres, where one takes an absolute value corresponding to the valuation on $R$. These notions play an important role in analytic theories of semistability, to name a few: Rigid geometry, Formal $R$-schemes and Berkovich spaces. In the Berkovich theory one also has formal fibers for points that are not necessarily closed in $\mathcal{X}_{s}$: for instance a generic point of a component. These are known as the type $2$ points for curves. 
\end{rem}
\begin{exa}
Let $\mathcal{C}=\text{Proj}R[X,T,W]/(XT-\pi^{n}{W^2})$ with open affine $U=\text{Spec}(R[x,t]/(xt-\pi^{n})$ where $x=\dfrac{X}{W}$ and $t=\dfrac{T}{W}$. Let $C$ be its generic fiber. Let $\tilde{x}=(x,t,\pi)$. Note that $\tilde{x}$ is not a regular point. We then have that
\begin{equation*}
C_{+}(\tilde{x})(K)=\{a\in{K}:\,|\pi|^{n}<|a|<1\}.
\end{equation*}
That is, it is an open annulus.
See \cite[Page 471]{liu2} for the details.
\end{exa}
 
Let us return to our transportation morphism $\psi:\text{Div}(C)\rightarrow{\text{Div}(\mathcal{C})}$. Consider the divisor $D=P$, where $P$ is some point in $C(K)$. Then $P$ specializes to a well-defined point $r_{\mathcal{C}}(P)$ that lies in the smooth locus of $\mathcal{C}_{s}$, see \cite[Corollary 9.1.32]{liu2}. Note that we use the regularity of $\mathcal{C}$ here, see Example \ref{regnec} below for a simple counterexample. At any rate, the point $P$ reduces to a smooth point and as such it reduces to a unique irreducible component of $\mathcal{C}_{s}$. We will denote this component by $c(P)$. We then have $\overline{\{P\}}=\{P,r_{\mathcal{C}}(P)\}$ and $\rho(\mathcal{D})=c(P)$.

\begin{exa} Consider the affine scheme defined by $R[x,y]/(xy-\pi)$. It has generic fiber $K[x,y]/(xy-\pi)$ and special fiber $k[x,y]/(xy)$. Consider the point defined by the prime ideal $\mathfrak{p}=(x-\pi,y-1)$. This corresponds to the point on the generic fiber $"(\pi,1)"$. There is exactly one maximal ideal lying above $\mathfrak{p}$, namely $\mathfrak{m}=(x-\pi,y-1,\pi)$ (which corresponds to $"(0,1)"$ on the special fiber). The closure of the prime ideal $\mathfrak{p}$ is then $\{\mathfrak{p},\mathfrak{m}\}$. The point $P$ reduces to a {\it{unique}} component, namely the one defined by the prime ideal $(x)$. 
\end{exa}

\begin{exa} ({\it{Regularity}})\label{regnec}
Suppose we now have the affine scheme defined by 
\begin{equation*}
A:=R[x,y]/(xy-\pi^2).
\end{equation*} It has generic fiber $K[x,y]/(xy-\pi^2)$ and special fiber $k[x,y]/(xy)$. This scheme is however not regular: the tangent space at $\mathfrak{m}=(x,y,\pi)$ is 3-dimensional, which is strictly higher than the dimension of the ring $A$ (which is 2). \\
Consider the prime ideal defined by $\mathfrak{p}=(x-\pi,y-\pi)$. This corresponds to the point $"(\pi,\pi)"$ on the generic fiber. There is exactly one maximal ideal above it (this holds for any proper morphism of schemes $\mathcal{X}\longrightarrow{S}$ where $S$ is the spectrum of a complete d.v.r.), but there is no {\it{unique}} component that it reduces to. Here $\mathfrak{p}\subseteq{\mathfrak{m}}=(x,y,\pi)$, which corresponds to the origin of the coordinate axes. We have that both $\Gamma_{1}:=\overline{(x)}$ and $\Gamma_{2}:=\overline{(y)}$ contain this point.
\end{exa}
\begin{rem}
To actually define a reduction for the point in the last example, one can blow-up the point $\mathfrak{m}$ to obtain a regular model. This works in general, see for instance \cite[Page 404]{liu2}. We will see many examples of this phenomenon later on. 
\end{rem}

\begin{rem}
[{\bf{Conventions on divisors}}]
As noted earlier, since both $C$ and $\mathcal{C}$ are {\it{regular and integral}}, we have by \cite[Page 271]{liu2} that the Cartier divisors correspond to Weil divisors. We will therefore write every Cartier divisor as a {\it{Weil divisor}}, i.e. as finite sums of irreducible closed subsets of codimension 1.\\
Let us give one more notational device regarding principal divisors. Let $K(\mathcal{C})$ be the function field of $\mathcal{C}$. It is equal to the function field of $C$. If we have an element $f\in{K(\mathcal{C})}$, we can consider its divisor in both $C$ and in $\mathcal{C}$. To avoid any ambiguity, we will write $\text{div}(f)$ or $(f)$ for the divisor in $\mathcal{C}$ and $\text{div}_{\eta}(f)$ or $(f)_{\eta}$ for the divisor in $C$.
\end{rem}

We will now consider the {\bf{principal divisors}} of $C$ and we will see what happens to them under this map $\psi$. Unfortunately, if we take a principal divisor $(f)$ and consider its closure in $\mathcal{C}$, then the resulting divisor in $\text{Div}(\mathcal{C})$ can be nonprincipal.  Let us see why this happens.

\begin{exa}
Suppose we take $\mathcal{C}=\text{Proj}R[X,Y,W]/(XY-\pi{}W^2)$ again with affine patch 
\begin{equation*}
A_{1}=R[x,y]/(xy-\pi).
\end{equation*} It has generic fiber $K[x,y]/(xy-\pi)$. Let us take $x$ in the function field of $C$. Then
\begin{equation*}
\text{div}_{\eta}(x)=(0)-(\infty).
\end{equation*}
Note that these points actually don't lie in the affine patch $A_{1}$; they lie in the other patches determined by $D^{+}(X)$ and $D^{+}(Y)$ (where the current patch $A_{1}$ corresponds to $D^{+}(W)$).\\
The function $x$ can also be considered as an element of the function field of $\mathcal{C}$ (they are the same after all). To determine this divisor in $\mathcal{C}$, we have to know at which codimension 1 primes $x$ has nonzero valuation. Consider $\mathfrak{p}_{1}=(x,\pi)=(x)$. The local ring $A_{1,\mathfrak{p}_{1}}$ is a discrete valuation ring with generator $x$. Thus $x$ has valuation 1 here. For $\Gamma_{2}$ we have the local ring $A_{1,\mathfrak{p}_{2}}$ where $\mathfrak{p}_{2}=(y,\pi)$. The element $x$ is invertible in this ring, so it has zero valuation. We then in fact have that
\begin{equation*}
\text{div}(x)=\overline{\{P\}}-\overline{\{\infty\}}+(\Gamma_{1}).
\end{equation*}
Note that the closure of $\text{div}_{\eta}(x)$ in $\mathcal{C}$ only contains the first two. In general, for any nonzero element $f$ of the function field of $\mathcal{C}$ we can write
\begin{equation*}
\text{div}(f)=\overline{\text{div}_{\eta}(f)}+V,
\end{equation*}
where $V$ is a vertical divisor (that is defined by the valuations of $f$ at the vertical divisors).\\
In fact, if we now have any divisor $D\in\text{Div}(C)$ of the form $D=\sum_{P\in{C(K)}}n_{P}(P)$, then we can take the closure $\mathcal{D}$ of $D$ in $\mathcal{C}$ and obtain a divisor there. We have
\begin{equation*}
\mathcal{D}=\sum_{P\in{C(K)}}n_{P}\overline{(P)}+\sum_{i}c_{i}(\Gamma_{i}),
\end{equation*}
where $c_{i}$ is the valuation of $\mathcal{D}$ at $\Gamma_{i}$.

\end{exa}

Luckily we have the following proposition, which tells us that principal divisors on $C$ map down to principal divisors on $\Sigma(\mathcal{C})$. 
\begin{pro}\label{Principaldivisors}
The specialization map $\rho$ induces a map
\begin{equation}
\text{Prin}({C})\rightarrow{\text{Prin}(\Sigma(\mathcal{C}))}.
\end{equation}
\end{pro}
\begin{proof}
See \cite[Lemma 2.1]{baker}. 
\end{proof}
We will use Proposition \ref{Principaldivisors} in Section \ref{JacNer} to construct a map from the Jacobian of $C$ to the tropical Jacobian of $\mathcal{C}$.

\section{Jacobians and N\'{e}ron models}\label{JacNer}

In this section we take the two transporting maps from $\text{Div}(C)$ to $\text{Div}(\mathcal{C})$ and from $\text{Div}(\mathcal{C})$ to $\text{Div}(\Sigma(\mathcal{C}))$ and consider the maps on the Jacobians. There is a description of this map in terms of the N\'{e}ron model of the Jacobian of $C$, which we will present here.\\
Let $\text{Div}(C)$ and $\text{Div}(\mathcal{C})$ be as before. Let $\text{Div}^{0}(C)$ be the subgroup of Cartier divisors of degree zero on $C$. We further define $\text{Div}^{(0)}(\mathcal{C})$ to be the subgroup of $\text{Div}(\mathcal{C})$ consisting of the Cartier divisors such that the restriction of the associated line bundle $\mathcal{O}_{\mathcal{C}}(\mathcal{D})$ to each irreducible component of $\mathcal{C}_{s}$ has degree zero. This last condition can be translated to
\begin{equation*}
\text{deg}(\mathcal{O}_{\mathcal{C}}(\mathcal{D})|_{\Gamma_{i}})=0
\end{equation*} 
for every $\Gamma_{i}$. Using our specialization map $\rho$ from before, we can write
\begin{equation*}
\text{Div}^{(0)}(\mathcal{C})=\text{Ker}(\rho).
\end{equation*}
We now let 
\begin{equation*}
\text{Div}^{(0)}(C)=\{D\in\text{Div}^{0}(C):\psi(D)\in\text{Ker}(\rho)\}.
\end{equation*}
As such, it is the inverse image of $\text{Ker}(\rho)$ under $\psi$.\\
Let us consider the associated Jacobians. Let $J(C)$ be the Jacobian of $C$ over $K$, that is: $\text{Div}^{0}(C)/\text{Prin}(C)$. Now let $\mathcal{J}$ be its N\'{e}ron model over $\text{Spec}(R)$. We direct the reader unfamiliar with N\'{e}ron models to \cite{liu2}, \cite{Silv2} and \cite{Bosch1990} 
for introductions to the subject. We let $\mathcal{J}^{0}$ be the connected component of the identity in $\mathcal{J}$. We denote by $\Psi=\mathcal{J}_{s}/{\mathcal{J}_{s}}^{0}$ the group of connected components of the special fiber $\mathcal{J}_{s}$ of $\mathcal{J}$. This is in fact a finite group that is isomorphic to the tropical Jacobian we defined earlier.  See \cite[Page 24]{baker} for the details.
\begin{exa}
Let us take an elliptic curve $E$ with split multiplicative reduction. Its reduction type is thus $I_{n}$ and we have that the intersection graph is just a cycle with $n$ vertices, where $n=-v(j)$, where $j$ is the j-invariant of $E$. We have that $E$ is canonically isomorphic to its own Jacobian. The N\'{e}ron model of $E$ in this case is obtained as follows: one takes the minimal regular model $\mathcal{C}$. 
One then considers the closed subscheme $S$ consisting of all the intersection points of $\mathcal{C}_{s}$. We give it the reduced induced subscheme structure. The open subscheme $\mathcal{E}:=\mathcal{C}\backslash{S}$ is then the N\'{e}ron model of $E$.  
It is a $\text{Spec}(R)$-scheme that is {\it{not proper}}, but it is a group scheme over $\text{Spec}(R)$. Its component group is then equal to $\mathbb{Z}/n\mathbb{Z}$. The details can be found in \cite[Page 492]{liu2}. 

The corresponding analytic version might be useful to have in mind as well. We will follow \cite[Chapter V]{Silv2}. 
Since $E$ has split multiplicative reduction, we have an analytic isomorphism 
\begin{equation*}
E(K)\simeq{K^{*}/(q)}
\end{equation*} 
for some $q\in{K^{*}}$ with $\text{val}(q)=n$.
We have a natural map
\begin{equation*}
i:R^{*}\longrightarrow{K^{*}/(q)},
\end{equation*}
where the image of $R^{*}$ in $E(K)$ is equal to the $R$-points of the connected component of the identity $\mathcal{E}^{0}$:
\begin{equation*}
i(R^{*})=\mathcal{E}^{0}(R).
\end{equation*}
We then quite easily see that
\begin{equation*}
\Psi=(K^{*}/(q))/(i(R^{*}))\simeq{\mathbb{Z}/n\mathbb{Z}}.
\end{equation*}
\end{exa}
Let us now return to the more general case of Jacobians and their N\'{e}ron models. We can ask for a concrete description of the $R$-points of the connected component of the identity and this is given by the following isomorphism:
\begin{equation}\label{RaynaudIsom}
J^{0}(K):=\mathcal{J}^{0}(R)\simeq{\text{Div}^{(0)}(C)/\text{Prin}^{(0)}(C)},
\end{equation}
where 
\begin{equation*}
\text{Prin}^{(0)}(C):=\text{Div}^{(0)}(C)\cap{\text{Prin(C)}}.
\end{equation*}
In other words, if we let $j$ be the injection $\text{Prin}(C)\longrightarrow{\text{Div}(C)}$, then
\begin{equation*}
\text{Prin}^{(0)}(C)=(\psi\circ{j})^{-1}(\text{Ker}(\rho)).
\end{equation*}

The isomorphism in Equation \ref{RaynaudIsom} comes from a theorem by Raynaud, which states that $\mathcal{J}^{0}=\text{Pic}_{\mathcal{C}/R}^{0}\label{Pick}$ represents the functor of "isomorphism classes of line bundles whose restriction to each element of $\mathcal{C}$ has degree zero". A quick sidenote to clarify this functorial approach: the entities above are considered to be functors from $(\text{Sch})\longrightarrow(\text{Sets})$. This identity of functors then means for instance that if we plug in the spectrum of the residue field $k$ as a scheme, we obtain the identity
\begin{equation}
\mathcal{J}^{0}(k)=\text{Pic}^{0}(\mathcal{C}_{s})(k).
\end{equation}
We will study the entity on the right hand side in the next section.

We note that we now have a natural map from the Jacobian of a curve $J(C)$ to the tropical Jacobian $J(\Sigma(\mathcal{C}))$. Let $P\in{J(C)(K)}$ and let $D\in\text{Div}^{0}(C)$ be any representative of $P$. We then define $\tilde{\psi}(D)=\rho(\psi(D))$. By Proposition \ref{Principaldivisors}, we then see that this is well-defined and from 
\cite[Diagram A.6, Page 25]{baker} we see that the kernel of this map is in fact $J^{0}(K)$.  

Let us now review some facts about the {\bf{torsion}} in the Jacobian of a curve. 
\begin{theorem}
Let $C$ be a smooth, connected, projective curve of genus $g$ over an algebraically closed field $K$ . Let $n\in\mathbb{N}$ be non-zero.
\begin{enumerate}
\item If $(n,\text{char}(K))=1$, then $J(C)[n]\simeq{(\mathbb{Z}/n\mathbb{Z})^{2g}}$.
\item If $\text{char}(K)=p$, then there exists an $0\leq{h}\leq{g}$ such that for any $n=p^{m}$ we have $J(C)[n]\simeq{(\mathbb{Z}/n\mathbb{Z})^{h}}$.
\end{enumerate}
\end{theorem}
\begin{proof}
This can be found in \cite[Theorem 4.38, Page 299]{liu2} 
or \cite[Corollary 2.3.2]{KatzMazur} and \cite[Chapter 3, Corollary 6.4]{Silv1} for elliptic curves. 
\end{proof}
In the rest of the thesis, we will mainly be dealing with the first case of the theorem. 

\section{Decomposition of $\mathcal{J}^{0}(k)$} 

In this section we will further study the $\mathcal{J}^{0}(k)$ introduced in the previous section. In fact, we will only study the group $\text{Pic}^{0}(X_{k})$ for a curve (not necessarily irreducible) over $k$ (reminder: this is the residue field of $R$, which we assume is algebraically closed). We have a natural identification
\begin{equation*}
\mathcal{J}^{0}(k)=\text{Pic}^{0}(\mathcal{C}_{s})(k).
\end{equation*}
from Section \ref{Pick} and as such we have a description of $\mathcal{J}^{0}(k)$.\\
So consider a connected projective curve $X$ over $k$ with {\it{smooth}} irreducible components $X_{1},...,X_{r}$. We will follow \cite[Chapter 7, Section 5]{liu2} with some extra assumptions for the scenario we're interested in. Let us suppose that $X$ is reduced and that it only has ordinary double points as its singularities (which is the case we're most interested in, the {\it{semistable}} case). Let $X':=\coprod_{1\leq{i}\leq{r}}X_{i}$ be the {\it{normalization}} of $X$. We have a surjective integral morphism $\pi:X'\longrightarrow{X}$.
\begin{mydef}
$\text{Pic}^{0}(X)$ is the set of isomorphism classes of invertible sheaves $\mathcal{L}$ such that $\text{deg}(\mathcal{L}|_{X_{i}})=0$ for every $1\leq{i}\leq{r}$.
\end{mydef}
Let $G$ be the intersection graph of $X$, as in \cite[Chapter 10, Definition 1.48]{liu2}. 
The structure of $\text{Pic}^{0}(X)$ is given by the following theorem. 

\begin{theorem}\label{ToricPic}
Let $X$ be as above (i.e., semistable). Let $t=\beta(G)$ be the Betti number of $G$. The following properties are then true.
\begin{enumerate}
\item[a)] The morphism $\pi$ induces a canonical surjective homomorphism
\begin{equation}\label{CohPic1}
\pi_{\text{Pic}^{0}}:\text{Pic}^{0}(X)\longrightarrow{\prod_{1\leq{i}\leq{r}}\text{Pic}^{0}(X_{i})}.
\end{equation}
\item[b)] Let $L=\text{Ker}(\pi_{\text{Pic}}^{0})$. Then $L\simeq{(k^{*})^{t}}$.
\end{enumerate}
\end{theorem}
\begin{proof}
(See \cite[Page 313]{liu2}, the following is a sketch)  Consider the exact sequence of sheaves of abelian groups
\begin{equation}\label{CohPic2}
0\longrightarrow{\mathcal{O}^{*}_{X}}\longrightarrow{\pi_{*}\mathcal{O}^{*}_{X'}}\longrightarrow{\mathcal{G}}\longrightarrow{0},
\end{equation}
where $\mathcal{G}$ is a skyscraper sheaf concentrated at the intersection points of the components of $X$. Let $S:=\{\text{the intersection points of components of $X$}\}$. For any intersection point $x\in{S}$ we have the identity on stalks
\begin{equation*}
\mathcal{G}_{x}=(\pi_{*}\mathcal{O}^{*}_{X'})_{x}/\mathcal{O}^{*}_{X,x}\simeq{k^{*}}
\end{equation*}
(\cite[Lemma 5.12, Page 309]{liu2}). We can take Cech cohomology of sequence (\ref{CohPic2}) to obtain the exact sequence
\begin{equation}\label{CohPic3}
0\longrightarrow{k^{*}}\longrightarrow{(k^{*})^{r}}\longrightarrow{\prod_{x\in{S}}k^{*}}\longrightarrow{\text{Pic}(X)}\longrightarrow{\text{Pic}(X')},
\end{equation}
where we used the identification $H^{1}(X,\mathcal{O}^{*}_{X})=\text{Pic}(X)$ (which is in \cite[Exercise 5.1.2.7]{liu2}). The last homomorphism in (\ref{CohPic3}) coincides with the usual homomorphism $\pi_{\text{Pic}}:\text{Pic}(X)\longrightarrow{\text{Pic}(X')}$, which takes $[\mathcal{L}]$ to $[\pi_{*}\mathcal{L}]$. The theorem now follows from the following observations: 
\begin{enumerate}
\item $\pi_{\text{Pic}}$ is surjective,
\item $[\mathcal{L}]\in\text{Pic}^{0}(X)$ if and only if $[\pi_{*}\mathcal{L}]\in\text{Pic}^{0}(X')$ (this with the previous statement gives (a)),
\item Exactness of the cohomology sequence (\ref{CohPic3}) (which gives (b)).
\end{enumerate}
\end{proof}
\begin{rem}
We will refer to the kernel of $\pi_{\text{Pic}^{0}}$ as the {\bf{toric part}} of $\mathcal{J}^{0}$. It will be denoted by
\begin{equation*}
\mathcal{J}^{0}_{T}:=\text{ker}(\pi_{\text{Pic}^{0}}).
\end{equation*}
The elements of $\mathcal{J}^{0}$ reducing to nontrivial elements under the map $\pi_{\text{Pic}^{0}}$ will be said to belong to the {\bf{abelian part}} of $\mathcal{J}^{0}$. 
\end{rem}
\section{Graph cohomology and the toric part of $\mathcal{J}^{0}(k)$}
\label{TorExtJac1}

From Theorem \ref{ToricPic}, we see that the degree zero line bundles consist of an abelian part and a toric part. We will now give a very explicit way to think about these line bundles that come from the toric part in terms of graphs. The reader that is interested in more of this is directed to \cite{Ulm}. 
 We will mostly follow her presentation of the material, albeit in an algebraic way.
 
So let $G(V,E)$ be a finite connected graph with vertex set $V$ and edge set $E$. We will review \v{C}ech cohomology for this graph with values in an abelian group $A$ (which for us will be $k^{*}$). 

\begin{mydef}
A graph $G(V',E')$ with $V'\subset{V}$ and $E'\subset{E}$, where every edge of $E'$ has source or target in $V'$ is called a subgraph of $G(V,E)$. A subgraph is called complete, if $E'$ contains all edges of $E$ with source and target in $V'$.
\end{mydef}
We can now define a topology on $G$ as follows: the open sets are the complete subgraphs of $G$. With this topology we can now define \v{C}ech cohomology for graphs. Let $e$ be any edge of $G$ and let $G_{e}$ be the (complete) subgraph of $G$ consisting of the edge and the two vertices it joins. We then have the open covering of $G$
\begin{equation*}
\mathfrak{B}=\{G_{e}:e\in{E(G)}\}.
\end{equation*}
As with normal \v{C}ech cohomology, we now define
\begin{equation*}
\check{C}^{q}(\mathfrak{B},A)=\prod_{(e_{0},...,e_{q})\in{E(G)^{q+1}}}A(G_{e_{0}}\cap{\cdot\cdot\cdot}\cap{G_{e_{q}}})
\end{equation*} 
and
\begin{equation*}
d_{q}:\check{C}^{q}\longrightarrow{\check{C}^{q+1}};\,\, \alpha\mapsto{(\prod_{k=0}^{p+1}(-1)^{k}\alpha_{i_{0},...,i_{k-1},i_{k+1},...,i_{q+1}})_{i_{0},...,i_{p+1}}}
\end{equation*}
We then have cohomology groups
\begin{equation*}
\check{H}^{q}(G,A)=\text{ker }d_{q}/\text{im }d_{q-1},
\end{equation*}
which are trivial for $q\geq{2}$ (since we're working with graphs). Let us describe $\check{H}^{1}(G,A)$. The elements of $\text{ker }d_{1}$ are the elements of $C^{1}$ that satisfy the cocycle relations 
\begin{equation*}
\alpha_{e_{i},e_{j}}=\alpha_{e_{i},e_{k}}\cdot{\alpha_{e_{k},e_{j}}}
\end{equation*} for three edges sharing a vertex $v$. The coboundaries of $\text{im }d_{0}$ can then be described by
\begin{equation*}
\alpha_{e_{i},e_{j}}=\beta_{e_{j}}\beta^{-1}_{e_{i}}
\end{equation*}
for a 0-cocycle $(\beta_{e})_{e\in{E}}$. 
For all the proofs involved, the reader is directed to \cite{Ulm}. 

We will now say that an edge {\it{ends in a vertex}}, if said vertex is either target or source of the edge.
\begin{mydef}
Let $e$ be an arbitrary edge with target vertex $v$ and an element $a\in{A}$, we define the weighted cocycle $\alpha(e,a)=(\alpha_{e_{i},e_{j}})_{e_{i},e_{j}\in{E^{2}}}$ by setting
\begin{equation*}
\alpha_{e_{i},e_{j}}=\begin{cases}
a & \text{if }e_{i}=e,e_{j}\neq{e}\text{ and }e_{j}\text{ ends in }v,\\
a^{-1} & \text{if }e_{j}=e,e_{i}\neq{e}\text{ and }e_{i}\text{ ends in }v,\\
1 & \text{otherwise}.
\end{cases}
\end{equation*}
\end{mydef}

That concludes our short review of graph \v{C}ech cohomology on graphs. Let us now return to the scenario of Theorem \ref{ToricPic}. So consider the surjective homomorphism
\begin{equation*}
\pi_{\text{Pic}^{0}}:\text{Pic}^{0}(X)\longrightarrow{\text{Pic}^{0}(X')}=\prod_{1\leq{i}\leq{r}}\text{Pic}^{0}(X_{i})
\end{equation*}
where the $X_{i}$ are the irreducible components of $X$. This homomorphism can be made quite explicit: one takes a divisor class $[D]$ on $X$ and restricts it to all its components:
\begin{equation*}
[D]\longmapsto ([D|_{X_{i}}])_{i}
\end{equation*}
If we now have a divisor class in the kernel of this map, then this means that for every component $X_{i}$, we can write 
\begin{equation*}
D|_{X_{i}}=(f_{i})
\end{equation*}
where $f_{i}\in{k(X_{i})}$, the function field of $X_{i}$. Suppose now that we have two intersection points $x_{j}$ and $x_{k}$ on the same component $X_{i}$ of $X$. Let the corresponding edges in the intersection graph be given by $e_{j}$ and $e_{k}$. We define
\begin{equation*}
\alpha_{e_{j}}=f_{i}(x_{j})	
\end{equation*}
and
\begin{equation*}
\alpha_{e_{j},e_{k}}=\alpha_{e_{j}}/\alpha_{e_{k}}
\end{equation*}
Evaluating this for all edges (or: intersection points) gives a {\it{weighted cocycle}} on the intersection graph that corresponds to the element of $\check{H}^{1}(G,k^{*})=H^{1}(X,\mathcal{O}^{*}_{X})=\text{Pic}^{0}(X)$ (the first equality follows from 
\cite[Proposition 4.2.5]{Ulm}). 

\begin{rem}
In Section \ref{Twistingdata}, we will see a modified version of this $2$-cocycle. It will be used for coverings that are unramified on a subgraph of $\Sigma(\mathcal{C})$. 
\end{rem}


\chapter{The Poincar\'{e}-Lelong formula}\label{Poincare}

In this chapter, we will give an algebraic proof of the {\it{Poincar\'{e}-Lelong formula}}. This formula tells us that the order of a reduced function $f$ at an edge $e$ is given by the slope of the Laplacian $\phi_{f}$ on that edge. In Chapter \ref{Inertiagroups}, we will use this formula to give the order of the inertia group $I_{e}$ for an edge $e\in\Sigma(\mathcal{C})$ and a disjointly branched morphism $\phi_{\mathcal{C}}:\mathcal{C}\rightarrow{\mathcal{D}}$.

Pierre Lelong first studied the {\it{"Poincar\'{e}-Lelong"}} differential equation in 1964 in \cite{Lelong1964}. There, it appeared in the form
\begin{equation}\label{OriginalEquation}
2id_{z}d_{\overline{z}}(V)=\theta,
\end{equation}
where $d_{z}$ and $d_{\overline{z}}$ are complex differentials, $\theta$ is some given entire function (called {\it{"courante"}}), $i$ is the imaginary unit and $V$ is the sought-for function. We will be interested in the {discrete}, nonarchimedean variant of this differential equation on graphs, which is given in its simplest form by
\begin{equation}\label{OurPoincare}
\Delta(\phi)=D,
\end{equation}
where $\phi$ is a $\mathbb{Z}$-valued function on the vertices $V(G)$ of a graph $G$, $\Delta$ the Laplacian operator on $\mathcal{M}(G)$ and $D$ a divisor of degree zero on $V(G)$.
Note that we already met this equation in Section \ref{DivisorsLaplacians}. 
 An introduction to this non-archimedean variant can be found in \cite{bakerfaber} and in \cite{baker}.  

The Poincar\'{e}-Lelong formula that we have in mind can be found in \cite[Theorem 5.15, part 5]{BPRa1}, where it is called the {\it{Slope formula}}. We state it here for the convenience of the reader. Let $X$ be a smooth, proper, connected algebraic curve over a valued field $K$, as in \cite{BPRa1}. 
\begin{theorem}\label{PoincareOriginal}{\bf{[Poincar\'{e}-Lelong formula, analytic version]}}
Let $f$ be an algebraic function on $X$ with no zeros or poles and let
\begin{equation}
F = -\text{log } |f| : X_{an} \rightarrow \mathbb{R}.
\end{equation} Let V be a semistable vertex set of X and let $\Sigma=\Sigma(X,V)$. If $x$ is a type-$2$ point of $X_{an}$ and $v\in{T_{x}}$, then 
\begin{equation}
d_{v}(F(x))=ord_{v}(\tilde{f}_{x}),
\end{equation}
where $\tilde{f}$ is the reduction of $c^{-1}f$ to the residue field of $x$, for $c$ an element in $K$ such that $|f(x)|=c$. 
\end{theorem}
This function $\text{log }|f|$ then satisfies a variant of Equation \ref{OriginalEquation}, namely
\begin{equation}\label{OriginalThuillier}
dd^{c}(\text{log }|f|)=\delta_{\text{div}(f)}.
\end{equation}
Here $dd^{c}$ is a nonarchimedean analogue of the usual {\it{Laplacian}} operator for Riemann surfaces (see Equation \ref{OriginalEquation}), which is defined in \cite[Proposition 3.3.15]{thuil1}. Moreover, $\delta_{\text{div}(f)}$ is the discrete distribution associated to $f$, as in \cite[Section 1.2.5, page 12]{thuil1} (it is called $\mu_{d}(f)$ there). 

We will prove a purely algebraic version of Theorem \ref{PoincareOriginal} using intersection theory on strongly semistable regular models $\mathcal{D}$. 
Our version is then as follows. 

\begin{reptheorem}{ValCor1}
{\bf[Poincar\'{e}-Lelong formula, algebraic version]}
Let $\mathcal{D}$ be a strongly semistable regular model of a curve $D$ with $f\in{K(\mathcal{D})}$.
 Let $\tilde{x}$ be an intersection point of an irreducible component $\Gamma$ with another irreducible component $\Gamma'$. Then
\begin{equation}
v_{\tilde{x}}(\overline{f^{\Gamma}})=\phi(v')-\phi(v),
\end{equation}
where $\phi$ is the Laplacian associated to $f$. 
\end{reptheorem} 


\section{Reducing Cartier divisors}
Let $X$ be a locally Noetherian scheme and $D$ be a Cartier divisor on $X$. We first give some background for studying the reduction of a Cartier divisor. 
\begin{mydef}
The {\it{support}} of $D$, denoted by $\text{Supp }D$, is the set of points $x\in{X}$ such that $D_{x}\neq{1}$. The set $\text{Supp }D$ is then a closed subset of $X$. 
\end{mydef}
\begin{rem}
Recall that the group of Cartier divisors is defined to be $H^{0}(X,\mathcal{K}^{*}_{X}/\mathcal{O}^{*}_{X})$, so $D_{x}$ is the image of $D$ in the stalk of the quotient sheaf $\mathcal{K}^{*}_{X}/\mathcal{O}^{*}_{X}$ in the point $x$.
\end{rem}

\begin{exa}\label{ExamplePoincare1}
Let $\mathcal{C}=\text{Proj}R[X,Y,W]/(XY-\pi{}W^2)$ with $\Gamma_{1}=\overline{\{(x)\}}$ and $\Gamma_{2}=\overline{\{(y)\}}$ as before. Consider the Cartier divisor defined by the element $x$. As before, we have that 
\begin{equation*}
\text{div}(x)=\overline{\{P\}}-\overline{\{\infty\}}+(\Gamma_{1}).
\end{equation*}
We then have
\begin{equation*}
\text{Supp}(\text{div}(x))=\overline{\{P\}}\cup\overline{\{\infty\}}\cup\Gamma_{1}.
\end{equation*}
\end{exa}

Recall that for a locally Noetherian scheme $X$, we have a notion of {\it{associated primes}}. These are defined by 
\begin{equation*}
\text{Ass}(\mathcal{O}_{X}):=\{x\in{X}:\mathfrak{m}_{x}\in\text{Ass}_{\mathcal{O}_{X,x}}(\mathcal{O}_{X,x})\}.
\end{equation*} 
\begin{theorem}\label{RedCar}
Let $X$ be a closed subscheme of a locally Noetherian scheme $Y$. Let $i:X\longrightarrow{Y}$ be the canonical injection.
\begin{enumerate}
\item The set $G_{X/Y}$ of Cartier divisors $E$ on $Y$ such that 
\begin{equation*}
(\text{Supp}(E))\cap{\text{Ass}(\mathcal{O}_{X})}=\emptyset
\end{equation*}
is a subgroup of $\text{Div}(Y)$.
\item There exists a natural homomorphism $G_{X/Y}\longrightarrow{\text{Div}(X)}$, denoted by $E\longmapsto{E|_{X}}$, compatible with the homomorphism $\mathcal{O}_{Y}\longrightarrow{i_{*}\mathcal{O}_{X}}$. Moreover, we have a canonical isomorphism
\begin{equation*}
\mathcal{O}_{Y}(E)|_{X}\simeq{\mathcal{O}_{X}(E|_{X})}
\end{equation*}
and
\begin{equation*}
\text{Supp}(E|_{X})=\text{Supp}(E)\cap{X}.
\end{equation*}
If $E>0$, then $E|_{X}\geq{0}$. The image of a principal divisor is a principal divisor.
\end{enumerate}
\end{theorem}
\begin{proof}
The details can be found in \cite[Page 261]{liu2}. We will outline the construction of the divisor $E|_{X}$. Let $E$ be represented by $\{U_{i},f_{i}\}$, where the $U_{i}$ are open in $Y$, and $f_{i}\in\mathcal{K}^{*}_{Y}(U_{i})$. Let
\begin{equation*}
\overline{U}_{i}=X\cap{U_{i}}.
\end{equation*}
From the surjective morphism
\begin{equation*}
\mathcal{O}_{Y}\longrightarrow{\pi_{*}(\mathcal{O}_{X})},
\end{equation*}
we obtain a surjective morphism
\begin{equation*}
\mathcal{O}_{Y}(U_{i})\longrightarrow{\mathcal{O}_{X}(\overline{U_{}}_{i})},
\end{equation*}
which we denote on the element $f_{i}$ as $\overline{f}_{i}$. One can now show that $\overline{f}_{i}$ is actually an element of $\mathcal{K}^{*}_{X}(\overline{U}_{i})$, see \cite[Page 261]{liu2} for the details. This then gives a Cartier divisor represented by $\{(\overline{U}_{i},\overline{f}_{i})\}_{i\in{I}}$.
\end{proof}

\section{Reducing Cartier divisors on regular semistable models}

We now specialize to the case of arithmetic surfaces.  Recall that an arithmetic surface is by definition a regular fibered surface. A regular surface is automatically normal (by \cite[Chapter 4, Theorem 2.16]{liu2}), so we have two notions ready: valuations at codimension 1 primes 
and intersection theory (see Section \ref{IntersectionTheory}). 

We would now like to reduce principal divisors to components of the special fiber. Let $f\in{K(\mathcal{D})}$ be an element of the function field of $\mathcal{D}$. 
 As we saw in Example \ref{ExamplePoincare1},  we cannot always restrict the divisor of this element to an irreducible component of the special fiber, since the restricted element might be completely contained in the vanishing set of that component (or in other words, there is a nonempty intersection of the divisor of $f$ with the associated primes of $\mathcal{D}_{s}$).\\
We will therefore modify our $f$ for various irreducible components $\Gamma\subset{\mathcal{D}_{s}}$. Let $y$ be a generic point for $\Gamma$. 
The local ring $\mathcal{O}_{\mathcal{D},y}$ is then a discrete valuation ring. Indeed, it is normal and it has dimension one 
by the fact that $\mathcal{D}_{s}$ is equidimensional of dimension $1$, see \cite[Chapter 4, Proposition 4.16]{liu2}. Here, equidimensional means that all irreducible components have the same dimension. We denote the corresponding valuation by $v(\cdot{})$ in this section.  The uniformizer $\pi$ of $R$ in fact has valuation $v(\pi)=1$, since $\mathcal{D}_{s}$ is assumed to be reduced and $\Gamma$ is contained in the special fiber. Suppose that $v(f)=k$.
\begin{mydef}
The $\Gamma$-modified form of $f$ is defined to be
\begin{equation*}
f^{\Gamma}:=\dfrac{f}{\pi^{k}}.
\end{equation*}
\end{mydef}
By definition, we then have $v(f^{\Gamma})=0$. If we then consider the natural map
\begin{equation*}
\mathcal{O}_{\mathcal{D},y}\longrightarrow{\mathcal{O}_{\mathcal{D},y}/\mathfrak{m}_{y}\mathcal{O}_{\mathcal{D},y}},
\end{equation*}
we see that $f^{\Gamma}$ naturally gives a nonzero element in the residue field, which we denote by $\overline{f^{\Gamma}}$. Note that the residue field at $y$ is the {\it{function field}} of the component $\Gamma$. 

\begin{lemma}\label{RedDiv1}
Let $f$ and $f^{\Gamma}$ be as above. We have
\begin{equation*}
\text{div}_{\Gamma}(\overline{f^{\Gamma}})=(\text{div}_{\mathcal{Y}}(f^{\Gamma}))|_{\Gamma}.
\end{equation*}
\end{lemma}
\begin{proof}

We have that the divisor is represented by $\{\mathcal{D},f^{\Gamma}\}$, which is then reduced to 
\begin{equation*}
\{\mathcal{D}\cap{\Gamma},\overline{f^{\Gamma}}\}=\{\Gamma,\overline{f^{\Gamma}}\}.
\end{equation*} This is exactly the Cartier divisor 
$\text{div}_{\Gamma}(\overline{f^{\Gamma}})$, as desired. 

\end{proof}
 
Let $V_{f}$ and $V_{f^{\Gamma}}$ be the vertical divisors of $f$ and $f^{\Gamma}$ respectively. We have that $V_{f^{\Gamma}}=V_{f}-k\cdot{\mathcal{D}_{s}}$. Let $D^{0}$ be the closed points of the generic fiber. Recall that we have a natural reduction map
\begin{equation*}
r_{\mathcal{D}}:D^{0}\longrightarrow{\mathcal{D}_{s}},
\end{equation*} 
which associates to every closed point $x$ in $D$ the point $\overline{\{x\}}\cap{\mathcal{D}_{s}}$, see Definition \ref{ReductionMap11}.  We now have
\begin{pro}\label{RedDiv2}
Consider the divisor $\text{div}_{\eta}(f)=\sum_{P}n_{P}(P)$ with corresponding $\Gamma$-modified surface divisor
\begin{equation*}
\text{div}(f^{\Gamma})=\sum_{P}n_{P}\overline{\{P\}}+V_{f^{\Gamma}}.
\end{equation*}
For $\tilde{x}$ in the nonsingular locus of $\mathcal{D}_{s}$, consider the formal fiber $D_{+}(\tilde{x})$. Then
\begin{equation*}
v_{\tilde{x}}(\overline{f^{\Gamma}})=\sum_{P\in{D_{+}(\tilde{x})}}n_{P}.
\end{equation*}
For $\tilde{x}$ an intersection point of $\Gamma$ and $\Gamma'$, we have
\begin{equation*}
v_{\tilde{x}}(\overline{f^{\Gamma}})=v_{\Gamma'}(f^{\Gamma}).
\end{equation*}

\end{pro}
\begin{proof}
The idea of the proof is to write out the equality in Lemma \ref{RedDiv1} in terms of valuations. 
For $\tilde{x}$ where $\overline{f^{\Gamma}}$ has positive valuation, the valuation can be found by
\begin{equation*}
v_{\tilde{x}}(\overline{f^{\Gamma}})=\text{length}(\mathcal{O}_{\Gamma,\tilde{x}}/(\overline{f_{\tilde{x}}^{\Gamma}}))
\end{equation*}
(the case with negative valuation is similar). Let $t$ be a local uniformizer of $\Gamma$, so that
\begin{equation*}
\mathcal{O}_{\Gamma,\tilde{x}}=\mathcal{O}_{\mathcal{D},\tilde{x}}/t\mathcal{O}_{\mathcal{D},\tilde{x}}.
\end{equation*}

We have the equality 
\begin{equation*}
\mathcal{O}_{\Gamma,\tilde{x}}/(\overline{f_{\tilde{x}}^{\Gamma}})=\mathcal{O}_{\mathcal{D},\tilde{x}}/(t\mathcal{O}_{\mathcal{D},\tilde{x}}+f^{\Gamma}_{\tilde{x}}\mathcal{O}_{\mathcal{D},\tilde{x}}).
\end{equation*}
But the length of this last ring is exactly the local intersection number, so that
\begin{equation*}
v_{\tilde{x}}(\overline{f^{\Gamma}})=(\Gamma\cdot{\text{div}(f^{\Gamma})})_{\tilde{x}}.
\end{equation*}
Writing this condition in terms of the horizontal and the vertical divisors gives us both statements of the proposition. 
\end{proof}

Proposition \ref{RedDiv2} allows us to calculate the reduced divisor of $f$ directly in terms of the horizontal  and the vertical divisor of $f$. 

\section{Vertical divisors, Laplacians and the Poincar\'{e}-Lelong formula}

In the last section we saw that to know the reduced divisors for a given element $f\in{K(\mathcal{D})}$, we have to know the horizontal divisor and the vertical divisor of $f$. In this section we will give a way of determining the vertical divisor using the divisors on the intersection graph. To do this, we'll explain in more detail the connection between principal divisors on the intersection graph and vertical divisors.

Suppose we are given an element $f$ of the function field $K(\mathcal{D})$. We have two options: we can consider its divisor in $D$ and in $\mathcal{D}$. The divisor $\text{div}_{\eta}(f)$ is well-defined up to a scaling factor of $K^{*}$ and the divisor $\text{div}(f)$ is well-defined up to a scaling factor of $R^{*}$. Namely, for every element of $K^{*}$ with nonzero valuation we get a shift in the vertical divisor and for every element of $R^{*}$ we obtain the same divisor.

Thus in general it is impossible to reconstruct $\text{div}(f)$ from just the generic divisor $\text{div}_{\eta}(f)$. If we however take the {\it{$\Gamma$-modified}} form $f^{\Gamma}$, we already know that $v_{\Gamma}(f^{\Gamma})=0$. There is then a {\it{unique}} solution $V_{f}=\sum_{i}c_{i}\Gamma_{i}$ such that $V_{f}$ is the vertical divisor corresponding to $\text{div}_{\eta}(f)$ with $c(\Gamma)=0$, by Corollary \ref{KernelRho} for instance. The good news is that we can explicitly give this vertical divisor in terms of the Laplacian operator.
\begin{theorem}\label{MainThmVert}
Let $\rho(\text{div}_{\eta}(f))$ be the induced principal divisor of $f$ on the intersection graph $G$ of $\mathcal{D}$. 
Write
\begin{equation*}
\Delta(\phi)=\rho(\text{div}(f))
\end{equation*}
for some $\phi:\mathbb{Z}^{V}\longrightarrow{\mathbb{Z}}$, where $\Delta$ is the Laplacian operator. Choose $\phi$ such that $\phi(\Gamma)=0$. Then the unique vertical divisor corresponding to $\text{div}_{\eta}(f)$ with $V_{f^{\Gamma}}(\Gamma)=0$ is given by
\begin{equation}\label{ExplVert2}
V_{f^{\Gamma}}=\sum_{i}\phi(\Gamma_{i})\cdot{\Gamma_{i}}.
\end{equation} 
\end{theorem}

Before we start the proof of this theorem, let us point out why Theorem \ref{MainThmVert} implies the Poincar\'{e}-Lelong formula. 
\begin{theorem}\label{ValCor1}{\bf[Poincar\'{e}-Lelong formula]}
Let $f$ and $f^{\Gamma}$ be as before. Let $\tilde{x}$ be an intersection point of $\Gamma$ with another component $\Gamma'$. Then
\begin{equation}
v_{\tilde{x}}(\overline{f^{\Gamma}})=\phi(v')-\phi(v).
\end{equation}
\end{theorem} 
\begin{proof}
By Proposition \ref{RedDiv2}, we find that the valuation of $\overline{f^{\Gamma}}$ at $\tilde{x}$ is equal to the valuation of $f^{\Gamma}$ at $\Gamma'$. By Theorem \ref{MainThmVert} we see that this is the slope of $\phi$ in the direction of $\Gamma$, as desired.
\end{proof} 
Let us start by considering the divisors of the simplest functions.
\begin{lemma}\label{VertInd1}
For any component $\Gamma_{i}$ with corresponding vertex $v_{i}$, we have
\begin{equation*}
\Delta(1_{v_{i}})=-\rho(\Gamma_{i}).
\end{equation*}
\end{lemma}
\begin{proof}
We calculate both sides. Define
\begin{equation*}
B:=\{j:\Gamma_{i}\cap\Gamma_{j}\neq\emptyset\}\backslash{\{i\}}.
\end{equation*}
Let $b=\#{B}$. We have
\begin{equation*}
\Delta(1_{v_{i}})=(\sum_{j\in{B}}-1\cdot(v_{j}))+b\cdot{v_{i}}.
\end{equation*}
 We also have
\begin{equation*}
\rho(\Gamma_{i})=(\sum_{j\in{B}}(\Gamma_{j}\cdot{\Gamma_{i}})(v_{j}))-b\cdot{v_{i}}=(\sum_{j\neq{i}:\Gamma_{i}\cap\Gamma_{j}}(1)(v_{j}))-b\cdot{v_{i}}.
\end{equation*}
We thus see that
\begin{equation*}
\Delta(1_{v_{i}})=-\rho(\Gamma_{i}),
\end{equation*}
as desired.
\end{proof}

\begin{proof}
(of Theorem \ref{MainThmVert}) Let $\phi$ be such that
\begin{equation*}
\Delta(\phi)=\rho(\text{div}_{\eta}(f))
\end{equation*}
and $\phi(\Gamma)=0$. We write
\begin{equation*}
\phi=\sum_{i}\phi(\Gamma_{i})\cdot{}1_{v_{i}}.
\end{equation*}
Taking the Laplacian operator of $\phi$ and using Lemma \ref{VertInd1}, we see that
\begin{equation*}
\Delta(\phi)=-\rho(\sum_{i}\phi(\Gamma_{i})\cdot({\Gamma_{i}})).
\end{equation*}
Writing 
\begin{equation*}
\text{div}(f^{\Gamma})=\text{div}_{\eta}(f^{\Gamma})+\sum_{i}c_{i}(\Gamma_{i})
\end{equation*}
and using that $\rho(\text{div}(f^{\Gamma}))=0$, we see that
\begin{equation*}
\rho(\sum_{i}c_{i}(\Gamma_{i}))=\rho(\sum_{i}\phi(\Gamma_{i})(\Gamma_{i})).
\end{equation*}
Since $c(\Gamma)=0$ and $\phi(\Gamma)=0$, we must have that these vertical divisors are equal by Corollary \ref{KernelRho}. This gives the theorem.
\end{proof}

\begin{exa}
Let us consider the projective line $\mathbb{P}^{1}$ with the function $f=x(x-\pi)$. We take the semistable model
\begin{equation*}
\text{Proj}R[X,T,W]/(XT-\pi{W}^2)
\end{equation*}
with open affine
\begin{equation*}
\text{Spec}R[x,t]/(xt-\pi).
\end{equation*}
We label the components as $\Gamma=Z(x)$ and $\Gamma'=Z(t)$. We see that 
\begin{equation*}
\text{div}_{\eta}(f)=(0)+(\pi)-2(\infty)
\end{equation*}
and that
\begin{equation*}
\rho(\text{div}_{\eta}(f))=2(\Gamma)-2(\Gamma').
\end{equation*}
\begin{figure}[h!]
\centering
\includegraphics[scale=0.6]{{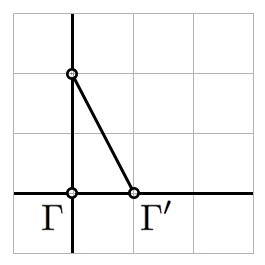}}
\caption{\label{Tweedeplaatje} {\it{The Laplacian of the function $f=x(x-\pi)$.}}} 
\end{figure}
The Laplacian thus has slope $-2$ from $\Gamma$ to $\Gamma'$, as in Figure \ref{Tweedeplaatje}. This means that if we take the $\Gamma$-modified form of $f$, it will have a pole of order 2 at the intersection point. Furthermore, we see that $f^{\Gamma}$ has a zero of order one at $(0)$ and $(\pi)$. This determines $f^{\Gamma}$ up to a constant in $k$.\\
We can also just calculate the modified form. By writing $f=x^2(1-t)$, we easily see that $v_{\Gamma}(f)=2$. Then the $\Gamma$-modified form of $f$ is equal to
\begin{equation*}
f^{\Gamma}=\dfrac{1-t}{t^2}.
\end{equation*}
As expected, this has a pole of order 2 at $t=0$ and a zero of order $1$ at $(\pi)$. Furthermore, it has a zero of order one at $t=\infty$, which corresponds to the point $(0)$, as expected.\\
Let us now determine the $\Gamma'$-modified form of $f$. We have that the Laplacian has slope 2 and thus that $f^{\Gamma'}$ has a zero of order two at the intersection point. Furthermore, we see that $f^{\Gamma'}$ has a pole of order two at $(\infty)$. We calculate the $\Gamma'$-form. Since $v_{\Gamma'}(f)=0$, we can just substitute $t=0$. We then obtain
\begin{equation*}
f^{\Gamma'}=\overline{x}^{2}
\end{equation*} 
which has a zero of order 2 at $x=0$ (which corresponds to the intersection point) and a pole of order two at $\infty$. 
\end{exa}



\chapter{Semistable models and Galois actions}\label{Quotients}

In this chapter we will study a specific type of semistable coverings $\mathcal{C}\rightarrow{\mathcal{D}}$ that we call \emph{disjointly branched morphisms}. We will be particularly interested in the case where the associated morphism of curves is Galois with Galois group $G$. In this case, we obtain a quotient of schemes $\mathcal{C}\rightarrow{\mathcal{C}/G=\mathcal{D}}$. If the order of the Galois group is coprime to the characteristic of the residue field (we say that the covering is \emph{tame}), then we also obtain a quotient on the special fiber $\mathcal{C}_{s}\rightarrow{\mathcal{C}_{s}/G=\mathcal{D}_{s}}$, as we will see in Section \ref{QuotientSpecial}. After defining the notion of a metrized complex of curves, we will see that we also obtain a quotient of metrized complexes in Section \ref{MetrizedQuotient}. These metrized complexes are  enhanced versions of our earlier weighted metric graphs, where every vertex now has an explicit curve attached to it.

We will also define the notions of decomposition and inertia groups here for general quotients of schemes under a finite group action. We will see how this relates to the usual theory of Galois extensions for Dedekind domains, which we use to study our coverings at the codimension one primes. In Chapter \ref{Inertiagroups}, we will study these decomposition and inertia groups for \emph{disjointly branched morphisms}, which are used to give the \emph{covering data} for the Galois coverings of graphs obtained in this section.

\section{Disjointly branched morphisms}\label{DisjointBran11}

Throughout this thesis, we'll be making great use of a classical theorem on coverings of semistable curves. This theorem also gives a practical way of explicitly calculating a lot of semistable reduction graphs. The theorem states that the normalization of a semistable model $\mathcal{D}$ of a curve $D$ that separates the branch points (in the special fiber of $\mathcal{D}$) of a finite cover $C\longrightarrow{D}$, will yield a semistable model $\mathcal{C}$ of $C$ after some finite extensions. This already gives some intuition why tropical geometry comes into play here: 
points that reduce to the same point on the special fiber will have a relative distance with strictly positive valuation, which is a tropical condition.
The theorem is as follows. 
\begin{theorem}\label{MaintheoremSemSta}
\begin{flushleft}
{\bf[{Obtaining semistable models from coverings}]}
\end{flushleft}

 Let $f:C\longrightarrow{D}$ be a finite morphism of smooth, projective geometrically connected curves over $K$. Suppose that $f$ is Galois with group $G$ of order prime to $\text{char}(k)$ and that $D$ admits a semistable model $\mathcal{D}_{0}$ over $R$.	
 Then the potential stable reduction of $C$ can be obtained by following the steps below:
 \begin{enumerate}
 \item {\bf{(Including branch points)}} Let $B\subset{D}$ be the branch locus of $f$. Take a finite separable extension $M/K$ to make the points of $B$ rational over $M$. Replace $\mathcal{D}_{0}$ by $\mathcal{D}_{0}\times_{\text{Spec}(R)}{\text{Spec}(R')}$, where $R'$ is a discrete valuation ring that dominates $R'$ and has field of fractions $M$.  
 \item {\bf{(Separation)}} Let $\mathcal{B}_{0}$ be the closure of $B$ in $\mathcal{D}_{0}$. Perform blow-ups at the closed points of $\mathcal{B}_{0}$ to obtain a birational morphism $\phi:\mathcal{D}\longrightarrow{\mathcal{D}_{0}}$ with $\mathcal{D}$ semistable such that the closure $\mathcal{B}$ of $B$ in $\mathcal{D}$ is a disjoint union of sections contained in the smooth locus of $\mathcal{D}$. 
 \item  {\bf{(Normalization)}} Let $\mathcal{C}_{0}\longrightarrow{D}$ be the normalization of $\mathcal{D}$ in $K(C_{M})$. Let
 \begin{equation*}
 \mathcal{F}=\{\Delta\,:\,\Delta\subseteq{\mathcal{D}_{s}}\text{ such that either }p_{a}(\Delta)\geq{1},\text{ or }\Delta\text{ contains at least three points of }\mathcal{B}\cup{(\mathcal{D}_{s})_\text{sing}}\}.
 \end{equation*}
 Let $e_{\Delta}$ denote the ramification index $e_{\Gamma/\Delta}$ for an irreducible component $\Gamma$ of $(\mathcal{C}_{0})_{s}$ lying above $\Delta$. Set $e=\text{lcm}\{e_{\Delta}\,|\,\Delta\in\mathcal{F}\}$ and $e=1$ if $\mathcal{F}=\emptyset$.
 \end{enumerate}
 Then for any extension of discrete valuation rings $R''\supseteq{R'}$ with $L:=\text{Quot}(R'')$ of ramification index divisible by $e$, the normalization $\mathcal{C}$ of $\mathcal{D}_{R''}$ in $K(C_{L})$ is a semistable model of $C_{L}$.
\end{theorem}
\begin{proof}
See \cite[Chapter 10, Proposition 4.30]{liu2} for the theorem as stated above. A proof can be found in \cite[Theorem 2.3]{liu1}, for instance.
\end{proof}

In other words, if we find a semistable model $\mathcal{D}$ such that the closure of the branch locus consists of disjoint smooth sections over $\text{Spec}(R)$, then the normalization $\mathcal{C}$ of $\mathcal{D}$ in $K(C)$ (see \cite[Section 4.1.2]{liu2} for more background on normalizations) gives a morphism $\mathcal{C}\rightarrow{\mathcal{D}}$ that might be vertically ramified. Removing this vertical ramification by taking a finite extension $K\subset{K'}$ then gives a morphism of semistable models $\mathcal{C}'\rightarrow{\mathcal{D}'}$. We'll call these morphisms "disjointly branched". 

\begin{mydef}\label{disbran}
Let $\phi:C\rightarrow{D}$ be a finite, Galois morphism of curves over $K$ with Galois group $G$.
Let $\phi_{\mathcal{C}}:\mathcal{C}\rightarrow{\mathcal{D}}$ be a finite morphism of models for $\phi$. We say $\phi_{\mathcal{C}}$ is {\bf{disjointly branched}} if the following hold:
\begin{enumerate}
\item The closure of the branch locus in $\mathcal{D}$ consists of disjoint, smooth sections over $\text{Spec}(R)$.
\item Let $y$ be a generic point of an irreducible component in the special fiber of $\mathcal{C}$. Then the induced morphism $\mathcal{O}_{\mathcal{D},\phi(y)}\rightarrow{\mathcal{O}_{\mathcal{C},y}}$ is \'{e}tale. 
\item $\mathcal{D}$ is strongly semistable, meaning that $\mathcal{D}$ is semistable and that the irreducible components in the special fiber are all smooth. 
\end{enumerate}
\end{mydef} 

\begin{cor}
Let $\phi_{\mathcal{C}}:\mathcal{C}\rightarrow{\mathcal{D}}$ be a disjointly branched morphism for $\phi:C\rightarrow{D}$. Then $\mathcal{C}$ is semistable.
\end{cor}
\begin{proof}
The morphism $\phi_{\mathcal{C}}$ satisfies all the properties of Theorem \ref{MaintheoremSemSta}, so we directly find that $\mathcal{C}$ is semistable. 
\end{proof}
 Since we want to use this theorem for intersection graphs, we would like to prove that we can find a morphism of {\it{strongly}} semistable models.
 \begin{pro}\label{PropStrongSemSta1}
 Let $\phi:\mathcal{C}\longrightarrow{\mathcal{D}}$ be as in Theorem \ref{MaintheoremSemSta}. If $\mathcal{D}$ is {\it{strongly semistable}}, then $\mathcal{C}$ is also strongly semistable.
 \end{pro}
 \begin{proof}
The proof is mostly based on the following Lemma:
\begin{lemma}\label{LemmaSmooth2}
Let $\phi:\mathcal{C}\longrightarrow{\mathcal{D}}$ be a disjointly branched morphism. Then the pre-image of a smooth point consists of smooth points. 
\end{lemma}
\begin{proof}
This is the last statement of \cite[Theorem 2.3, page 69]{liu1}.
 \end{proof}
 Lemma \ref{LemmaSmooth2} then implies that
\begin{cor}\label{CorSmooth2}
For every ordinary double point  $x$ of $\mathcal{C}$, 
the image $\phi(x)$ is an ordinary double point of $\mathcal{D}$.
\end{cor}
\begin{proof}
Indeed, if $\pi(x)$ is smooth, then there exists a non-smooth point in the pre-image of $\pi(x)$ (namely $x$). This contradicts Lemma \ref{LemmaSmooth2}. Thus $\pi(x)$ is non-smooth. Since $\mathcal{D}$ is semistable, it must be an ordinary double point, as desired.
\end{proof}
Now for the rest of the proof.  Suppose that $x'$ is an ordinary double point in $\mathcal{C}$. Let $\mathfrak{m'}$ be the corresponding maximal ideal on some open affine $A'$, which is the integral closure of $A$ corresponding to an open affine of $\mathcal{D}$. Then $\phi(x)$ is also an ordinary double point by Corollary \ref{CorSmooth2}. Let $\mathfrak{m}$ be the corresponding maximal ideal. Since $\mathcal{D}$ is assumed to be strongly semistable, we can find two distinct prime ideals $\mathfrak{p}_{1}$ and $\mathfrak{p}_{2}$ (corresponding to two components intersecting each other in $\phi(x)$) in the special fiber such that
\begin{equation*}
\mathfrak{p}_{i}\subset\mathfrak{m}
\end{equation*}
for both $i$. By the going-up theorem (which is applicable because $A\subseteq{A}'$ is integral), we can find $\mathfrak{q}_{i}\subset{\mathfrak{m}'}$ such that $\mathfrak{q}_{i}\cap{A}=\mathfrak{p}_{i}$ (note that they are in the special fiber by this condition). But then $\mathfrak{m'}$ is an intersection point of $\mathfrak{q}_{1}$ and $\mathfrak{q}_{2}$. This proves that $\mathcal{C}$ is strongly semistable, as desired.
 \end{proof}

Let us now make some remarks about Theorem \ref{MaintheoremSemSta} and Proposition \ref{PropStrongSemSta1} that should be kept in mind throughout the thesis.

\begin{rem}[{\it{About the branch points}}]
Since we assumed the residue field $k$ to be algebraically closed, we only have to take {\it{ramified extensions}} of $K$ here. 
For the extensions in the third step, we find that the extensions are in fact {\it{tame}}. Indeed, the ramification index of any component has to divide the order of the Galois group, which is coprime to the characteristic of $k$. 
This then also means that the extensions in the third step are obtained by $K\subseteq{K(\pi^{1/n})}$ for some $n$ with $(n,\text{char}(k))=1$. 
\end{rem}

\begin{rem}[{\it{About the closure of $B$}}]
The closure of $B$ in $\mathcal{D}$ can be computed using the {\it{reduction}} map. See Definition \ref{ReductionMap11}.  For any closed point $x$ of the generic fiber, we have that
\begin{equation*}
\overline{\{x\}}=\{x,r_{\mathcal{D}}(x)\}.
\end{equation*}
That is, we take the point $x$ together with its reduction. The condition here is that the reductions of the branch points are {\it{disjoint}} and that they reduce to nonsingular points.
\end{rem}
\begin{rem}[{\it{Caveat about the characteristic}}]
The condition on the characteristic of the residue field is to avoid issues of separability in the special fiber. 
There is a way to address these cases as well using Artin-Schreier equations, see \cite{KArz1}. We will also encounter these problems in Section \ref{QuotientSpecial}, where we study quotients on the special fiber. 
\end{rem}

\begin{rem}[{\it{A separating semistable model}}]
For Galois coverings $C\rightarrow{\mathbb{P}^{1}}$, there is an explicit semistable model of $\mathbb{P}^{1}$ that separates the branch locus. We will explicitly give this model in Chapter \ref{Appendix2}. The corresponding intersection graph is known as the {\bf{tropical separating tree}}, see \cite{supertrop} and \cite{tropicalbook} for more on the tropical point of view of this graph. 
\end{rem}

\begin{rem}[{\it{Galois action}}]
In Theorem \ref{MaintheoremSemSta}, we take the normalization of a certain normal integral scheme in a Galois extension. The resulting scheme $\mathcal{C}$ then actually has a natural Galois action such that $\mathcal{C}/G=\mathcal{D}$. These Galois actions will be reviewed in Section \ref{GaloisQuotientsSchemes}. 
\end{rem}

\section{Galois quotients for schemes}\label{GaloisQuotientsSchemes}

In Section \ref{DisjointBran11} we defined disjointly branched morphisms. The corresponding semistable models 
have a 
 natural Galois action on them. In this section, we will review some facts about quotient schemes for finite Galois groups that act on a scheme $X$. We will quickly specialize to the strongly semistable case, where we consider the problem of Galois actions on graphs. 
We will follow \cite{SGA1} and \cite[Exercises 2.14, 2.3.21 and 3.3.23]{liu2}. Let $X$ be a scheme with a finite group $G$ acting on $X$. This means that we have a group homomorphism
\begin{equation*}
G\longrightarrow{\text{Aut}(X)}.
\end{equation*}
The {\it{quotient scheme}} is then defined by a universal property that we will repeat here. The quotient scheme of $X$ under $G$ is a scheme $Y$ with the following properties:
\begin{enumerate}
\item There is a morphism $p:X\longrightarrow{Y}$.
\item We have $p=p\circ{\sigma}$ for every $\sigma\in{G}$.
\item Any morphism of schemes $f:X\longrightarrow{Z}$ satisfying $f=f\circ{\sigma}$ for every $\sigma$ factors in a unique way through $p$. This means that there exists a unique morphism $\tilde{f}:Y\longrightarrow{Z}$ such that $f=\tilde{f}\circ{p}$.
\end{enumerate}
In other words, $Y$ represents the functor $Z\mapsto{\text{Hom}(X,Z)^{G}}$, see \cite[Page 87]{SGA1} for this point of view.  
Let us consider the affine case first. The following can be found in \cite[Proposition 1.1 and Corollary 1.2]{SGA1}. 
\begin{pro}\label{QuotientLemma10}{\bf{[Affine quotients]}}
Let $A$ be a ring with a finite group action $G\longrightarrow{\text{Aut}(A)}$. Let $B=A^{G}$ be the invariants, $X=\text{Spec}(A)$, $Y=\text{Spec}(B)$ and $p:X\longrightarrow{Y}$ the canonical morphism. Then
\begin{enumerate}
\item $A$ is integral over $B$ (and the morphism $p$ is thus integral).
\item The morphism $p$ is surjective, its fibers are the orbits under $G$, the topology of $Y$ is the quotient of the topology on $X$.
\item Let $x\in{X}$, $y=p(x)$, $G_{x}$ the stabilizer of $x$. Then $k(x)$ (the residue field of $x$) is a normal algebraic extension of $k(y)$ and the homomorphism $G_{x}\longrightarrow{\text{Gal}(k(x)/k(y))}$ is surjective.
\item $(Y,p)$ is the quotient scheme of $X$ by $G$.
\item The natural morphism
\begin{equation*}
\mathcal{O}_{Y}\longrightarrow{p_{*}(\mathcal{O}_{X})^{G}}
\end{equation*}
of sheaves is an isomorphism.
\end{enumerate}
\end{pro}
\begin{proof}
This is almost a word-by-word translation of \cite[Expos\'{e} V, Proposition 1.1 and Corollary 1.2]{SGA1}.  
\end{proof}

Let us now generalize a little bit. We will consider what Grothendieck calls "admissible actions".

\begin{mydef}
Let $G$ be a finite group acting on a scheme $X$, $p:X\longrightarrow{Y}$ an {\it{affine}} invariant morphism such that
\begin{equation*}
\mathcal{O}_{Y}\longrightarrow{p_{*}(\mathcal{O}_{X})^{G}}
\end{equation*} 
is an isomorphism. 
This action is then called an {\bf{admissible action}}. 
\end{mydef}
\begin{pro}\label{PropositionQuotient1}
Let $G$ give an admissible action on $X$. Then the conclusions from Proposition \ref{QuotientLemma10} are still valid. In particular, we have $Y=X/G$. 
\end{pro}
\begin{proof}
This is \cite[Expos\'{e} V, Proposition 1.3]{SGA1}.
\end{proof}
\begin{cor}\label{CorOpenQuotient}
Let $G$ give an admissible action on $X$. Then for any open set $U\subset{Y}$, we have that $U$ is the quotient of $p^{-1}(U)$ under $G$.
\end{cor}
\begin{proof}
See \cite[Expos\'{e} V, Corollary 1.4]{SGA1}. 
\end{proof}

\begin{pro}
Let $G$ be a finite group acting on a scheme $X$. Then $G$ gives an admissible action if and only if there is an open affine cover $\{U_{i}\}$ of $X$ such that each $U_{i}$ is invariant under $G$.
\end{pro}
\begin{proof}
See \cite[Expos\'{e} V, Proposition 1.8]{SGA1}. 
\end{proof}
Let us now focus on 
{\it{normal integral}} schemes.
\begin{pro} \label{GalExtNormInt}
Let $Y$ and $X$ be normal integral schemes. Suppose that we have a finite surjective integral morphism $f:X\longrightarrow{Y}$ such that $K(X)/K(Y)$ is a finite Galois extension with Galois group $G$. Then $X$ is the normalization of $Y$ in $K(X)$. We have an action of $G$ on $X$ with $X/G=Y$. Furthermore, this action on $X$ is transitive. 
\end{pro}
\begin{proof}
The fact that $X$ is the normalization follows from \cite[Page 120, Proposition 1.22]{liu2}. The normalization naturally comes with a group action, stemming from the fact that on affines we have that if $a'\in{A'}$ is integral over $A$, then $\sigma(a')$ is also integral over $A$ for any $\sigma\in{G}$.  Now consider the chain
\begin{equation*}
A'\supseteq{(A')^{G}}\supseteq{A}.
\end{equation*}
We have that $(A')^{G}=A$. Indeed, if $a'\in{(A')^{G}}$, then $a'\in{K(Y)}$ and $a'$ is integral over $A$. Since $A$ is integrally closed in $K(Y)$, we have that $a'\in{A}$, as desired. This then yields $\mathcal{O}_{Y}=p_{*}(\mathcal{O}_{X})$ and thus $Y=X/G$ by Proposition \ref{PropositionQuotient1}. For transitivity, see \cite[Page 546, Lemma 4.34]{liu2}.
\end{proof}

\section{Decomposition and inertia groups}\label{DefinitionInertia}

Let $Y$ be a normal integral locally Noetherian scheme with function field $K(Y)$ and let $L\supset{K(Y)}$ be a Galois extension with Galois group $G$. We let $X$ be the normalization of $Y$ in $L$. We then have $X/G=Y$. Indeed, the morphism $\phi:X\rightarrow{Y}$ is finite surjective and integral, so we can apply Proposition \ref{GalExtNormInt}. We will now define decomposition and inertia groups for these coverings. 


For any point $x$ of $X$, we define the {\bf{decomposition group}} to be $D_{x,X}:=\{\sigma \in G:\sigma(x)=x\}$, the stabilizer of $x$. 
Every element $\sigma\in{D_{x,X}}$ naturally acts on $\mathcal{O}_{X,x}$ and the residue field $k(x)$, see also Proposition \ref{QuotientLemma10}.
We then define the {\bf{inertia group}} $I_{x,X}$ of $x$ to be the elements of $D_{x}$ that reduce to the identity on $k(x)$. In other words, $\sigma \in I_{x,X}$ if and only if for every $z \in {\mathcal{O}_{X,x}}$, we have $\sigma z \equiv z \mod m_x$, where $m_{x}$ is the unique maximal ideal of $\mathcal{O}_{X,x}$.
We will quite often omit the scheme $X$ in $I_{x,X}$ and $D_{x,X}$ and just write $I_{x}$ and $D_{x}$. 

We again consider the affine case. Suppose we have a normal integral domain $A$ with field of fractions $K(A)$ and a Galois extension $L\supset{K(A)}$. Let $B$ be the normalization of $A$ in $L$. For every subgroup $H$ of $\text{Gal}(K(B)/K(A))$, we can then consider the chain
\begin{equation}
B\supset{B^{H}}\supset{A}.
\end{equation}
Note that $B^{H}$ is again integrally closed in $L^{H}$, so we find that it is the integral closure of $A$ in $L^{H}$. We will now take $H=I_{\mathfrak{q}}$ for some $\mathfrak{q}\in\text{Spec}(B)$ with image $\mathfrak{p}=\mathfrak{q}\cap{A}$. In the sense of the following Lemma, the subextension $L^{I_{\mathfrak{q}}}$ is the largest extension such that $B^{I_{\mathfrak{q}}}\supseteq{A}$ is \'{e}tale at $\mathfrak{q}\cap{B^{I_{\mathfrak{q}}}}$. It is therefore also known as the {\it{"maximal unramified extension"}} of $\mathfrak{p}$. 
\begin{lemma}\label{RamificationInertiaLemma}
Let $B\supset{B^{I_{\mathfrak{q}}}}\supset{A}$ be as before. Then the following hold:
\begin{enumerate}
\item The prime $\mathfrak{q}$ is the only prime in $\text{Spec}(B)$ lying above $\mathfrak{q}\cap{B^{I_{\mathfrak{q}}}}$.
\item Let $k(\mathfrak{q})^{\text{sep}}$ be the separable closure of $k(\mathfrak{q}\cap{A})$ in $k(\mathfrak{q})$. 
Then $(k(\mathfrak{q}))^{\text{sep}}=k(\mathfrak{q}\cap{B^{I_{\mathfrak{q}}}})$. 
\item Suppose that $\text{char}(k(\mathfrak{q}))\nmid{|G|}$. Then $k(\mathfrak{q})=k(\mathfrak{q}\cap{B^{I_{\mathfrak{q}}}})$.
\item $B^{H}\supset{A}$ is \'{e}tale at $\mathfrak{q}\cap{B^{H}}$ if and only if $H\supseteq{I_{\mathfrak{q}}}$. 
\end{enumerate}
\end{lemma}
\begin{proof}
These are given in \cite{supertrop} as: Lemma 3.4, Lemma 3.5, Corollary 3.6 and Proposition 3.7 respectively. We refer the reader to that paper.
\end{proof}

\begin{lemma}\label{Orbitstabilizer}
Let $D_{\mathfrak{q}}$ be the decomposition group of $\mathfrak{q}\in\text{Spec}(B)$. Let $$G(\mathfrak{q})=\{\mathfrak{q}'\in\text{Spec}(B):\mathfrak{q}'=\sigma(\mathfrak{q})\text{ for some }\sigma\in{G}\}.$$ Then 
\begin{equation}
|G(\mathfrak{q})|=\dfrac{|G|}{|D_{\mathfrak{q}}|}.
\end{equation}
\end{lemma}
\begin{proof}
We apply the orbit-stabilizer theorem from group theory and directly obtain the lemma.
\end{proof}

By Proposition \ref{QuotientLemma10}, the action of $G$ is transitive, so for any $\mathfrak{p}\in\text{Spec}(A)$ and $\mathfrak{q}\in\text{Spec}(B)$, we find that the number of primes lying above $\mathfrak{p}$ is equal to $\dfrac{|G|}{|D_{\mathfrak{q}}|}$. In other words, if we know the decomposition group of any prime $\mathfrak{q}$ lying above $\mathfrak{p}$, then we know how many primes there are in the pre-image of $\mathfrak{p}$ for the morphism $\text{Spec}(B)\rightarrow{\text{Spec}(A)}$. 

\subsection{Finite extensions of Dedekind domains}

In this section, we give a summary of the theory of finite extensions of Dedekind domains. There will necessarily be some overlap with the previous sections and we invite the reader to compare definitions and results. 

We first describe how to pass from the ring extension $B\supset{A}$ of Noetherian normal integral domains of the previous section to an extension of Dedekind domains. Let $\mathfrak{p}\in\text{Spec}(A)$ be any prime. 
Recall that $B\supset{A}$ is finite by \cite[Chapter 4, Proposition 1.25]{liu2}. 
We can take the base extension
\begin{equation*}
A_{\mathfrak{p}}\longrightarrow{A_{\mathfrak{p}}\otimes{B}}=B_{\mathfrak{p}}
\end{equation*}
of the map $A\longrightarrow{A_{\mathfrak{p}}}$. Since finiteness is preserved under base extensions, we have that this ring extension is again finite. Taking another base extension, this time corresponding to $A_{\mathfrak{p}}\longrightarrow{A_{\mathfrak{p}}/\mathfrak{p}}$, we obtain that the ring extension
\begin{equation*}
A_{\mathfrak{p}}/\mathfrak{p}\longrightarrow{B_{\mathfrak{p}}/\mathfrak{p}}
\end{equation*}
is again finite. This just means that $B_{\mathfrak{p}}/\mathfrak{p}$ 
 is a finite vector space over the field $A_{\mathfrak{p}}/\mathfrak{p}$. We thus have that it's an Artinian ring, meaning that there are only finitely many prime ideals. We are just expressing the fact that a finite morphism of schemes is quasi-finite of course. These prime ideals of $B_{\mathfrak{p}}/\mathfrak{p}$ now correspond exactly to the prime ideals of $B$ above the prime ideal $\mathfrak{p}$.  A small reminder: localization commutes with normalization, so we can write 
 \begin{equation*}
 B_{\mathfrak{p}}=(A_{\mathfrak{p}})'.
 \end{equation*}
Let us now restrict ourselves to the $\mathfrak{p}$ such that $\text{dim }A_{\mathfrak{p}}=1$. These are known as the primes of codimension one. In this case, we have that $A_{\mathfrak{p}}$ is a normal Noetherian integral domain of dimension 1. Or in other words, we have that $A_{\mathfrak{p}}$ is a {\it{Dedekind domain}}. Since $B_{\mathfrak{p}}$ is finite over $A_{\mathfrak{p}}$, we have that $1=\text{dim }(A_{\mathfrak{p}})=\text{dim }(B_{\mathfrak{p}})$. Since $B$ is normal, any localization of $B$ is also normal. Thus $B_{\mathfrak{p}}$ is also a Dedekind domain. This puts us back into the usual framework of algebraic number theory: finite ring extensions of Dedekind domains. One major difference of course with the number field theory is that the residue fields involved can be nonperfect. Note also that we are dealing with finite extensions of \emph{local} Dedekind domains which have a trivial Picard group, so some of the global aspects of the usual theory are lost. 

We recall some definitions and theorems from algebraic number theory. For the rest of the section, $A$ and $B$ are Dedekind. We say that $\mathfrak{q}\in\text{Spec}(B)$ {\it{divides}} $\mathfrak{p}\in\text{Spec}(A)$ if $\mathfrak{q}\cap{A}=\mathfrak{p}$. We define the ramification index of $\mathfrak{q}$ over $\mathfrak{p}$ by
\begin{equation*}
e_{\mathfrak{q}/\mathfrak{p}}=\text{dim}_{A/\mathfrak{p}}(B_{\mathfrak{q}}/\mathfrak{p}B_{\mathfrak{q}}).
\end{equation*}

We will also write this as $e_{\mathfrak{q}}$. We then have a decomposition of prime ideals
\begin{equation}
\mathfrak{p}=\prod_{i=1}^{r}\mathfrak{q_{i}}^{e_{{\mathfrak{q}_{i}}}}.
\end{equation}
For each $\mathfrak{q}$ dividing $\mathfrak{p}$ we have a finite extension of residue fields
\begin{equation*}
A_{\mathfrak{p}}/\mathfrak{p}\longrightarrow{B_{\mathfrak{q}}/\mathfrak{q}}
\end{equation*}
with finite degree
\begin{equation*}
f_{\mathfrak{q}/\mathfrak{p}}=[B_{\mathfrak{q}}/\mathfrak{q}:A_{\mathfrak{p}}/\mathfrak{p}].
\end{equation*}
We will also let $g_{\mathfrak{p}}$ be the number of primes in $\text{Spec}(B)$ dividing a given $\mathfrak{p}$. For a {\it{Galois}} extension, we then have the following
\begin{pro}\label{SerLocFields}
Suppose that the extension $A\subseteq{B}$ with fraction fields $K(A)\subset{K(B)}$ is Galois with Galois group $G$. Let $n$ be its order. Then the integers $e_{\mathfrak{q}}$ and $f_{\mathfrak{q}}$ depend only on $\mathfrak{p}$. If we denote them by $e_{\mathfrak{p}}$ and $f_{\mathfrak{p}}$, then 
\begin{equation}\label{FundGalEq1}
n=e_{\mathfrak{p}}f_{\mathfrak{p}}g_{\mathfrak{p}}.
\end{equation}
\end{pro}
\begin{proof}
See \cite[Page 20]{Ser1}.
\end{proof}
Let $D=D_{\mathfrak{q}}$ be the decomposition group of $\mathfrak{q}$ and $I=I_{\mathfrak{q}}$, as defined earlier. By the orbit-stabilizer theorem, we have that the index of $D$ in $G$ is equal to the number $g_{\mathfrak{p}}$. Let us write
\begin{eqnarray*}
\overline{L}&=&B_{\mathfrak{q}}/\mathfrak{q},\\
\overline{K}&=&A_{\mathfrak{p}}/\mathfrak{p}.
\end{eqnarray*}

For any $\sigma\in{D}$, we a natural $\overline{K}$-automorphism $\overline{\sigma}$ of $\overline{L}$ by passing to the quotient. This gives a homomorphism
\begin{equation*}
\rho:D\longrightarrow{G(\overline{L}/\overline{K})}
\end{equation*}
whose kernel is by definition the {{inertia group}} of $\mathfrak{q}$. As we saw in Section \ref{DefinitionInertia}, this extension $\overline{L}\supset{\overline{K}}$ is not always Galois. 
\begin{pro}\label{SerExtGal1}The following properties are true.
\begin{enumerate}
\item The residue extension $\overline{L}/\overline{K}$ is normal and the homomorphism
\begin{equation*}
\rho:D\longrightarrow{\text{G}(\overline{L}/\overline{K})}
\end{equation*}
defines an isomorphism $D/I\simeq{\text{G}(\overline{L}/\overline{K})}$.
\item If $\overline{L}/\overline{K}$ is separable, then it is a Galois extension with Galois group $D/I$. We then have $[L:L^{I}] =e$, $[L^{I}:L^{D}]=f$ and $[L^{D}:K]=g$. Here $L^{H}$ for any subgroup $H$ of $G$ means the invariant field under $H$.
\end{enumerate} 
\end{pro}

\begin{proof}
See \cite[Page 23]{Ser1}. 
\end{proof}
\begin{rem}
The residue extension $\overline{L}/\overline{K}$ is separable in the following cases:
\begin{enumerate}
\item $\overline{K}$ is perfect (which will generally not be the case for us, unless the residue field $k$ has characteristic zero).
\item The order of the inertia group $I$ is prime to the characteristic $p$ of the residue field $\overline{K}$.
\end{enumerate}
In the case of a disjointly branched Galois morphism we naturally have case $(2)$. This implies that $\overline{L}/\overline{K}$ is separable and thus we can use Proposition \ref{SerExtGal1}.  
\end{rem}

\begin{rem}
Let $\mathcal{C}\rightarrow{\mathcal{D}}$ be a disjointly branched morphism. By definition, we then have $e_{x/y}=1$ for $x$ and $y$ the generic points of components $\Gamma_{x}$ and $\Gamma_{y}$ in $\mathcal{C}_{s}$ and $\mathcal{D}_{s}$ respectively. 
This is because $\pi$ is a uniformizer in both. We therefore have
\begin{equation*}
|G|=f_{\mathfrak{q}/\mathfrak{p}}\cdot{g_{\mathfrak{q}/\mathfrak{p}}}.
\end{equation*}
\end{rem}

\begin{rem}
We would like to point out to the reader that we now have two notations that are very similar. 
On the one hand for an edge $e$ with corresponding prime $\mathfrak{p}$ of $\mathcal{D}$, we have the "splitting" indices 
\begin{equation*}
g_{\mathfrak{q}/\mathfrak{p}}.
\end{equation*}
On the other hand, in a very natural way we have that our primes are curves. We thus also have the genus
\begin{equation*}
g(\mathfrak{p})
\end{equation*} 
to our disposal. We will sometimes write $a(\mathfrak{p})$ (as in: abelian rank) or $g(\Gamma)$ for the arithmetic genus. Whenever we mean the splitting indices, we will write them as a subscript. 
\end{rem}

\section{Kummer extensions}\label{Kummerextensions}

We quickly recall some Kummer Theory. 
Suppose we have a finite cyclic abelian extension $K\subseteq{L}$ of degree $n$ coprime to the characteristic $p$ of $K$. We will also suppose that $K$ contains a primitive $n$-th root of unity $\zeta_{n}$. If $n$ is not coprime to the characteristic, one has to consider so-called Artin-Schreier type extensions. We will not pursue this path here. For our case, the abelian extension takes a very simple form. 
\begin{pro}\label{AbelExt1} (Kummer Theory) Let $K\subseteq{L}$ be a finite cyclic Galois extension of degree $n$ with $n$ coprime to the characteristic $p$ of $K$. Suppose that $K$ contains a primitive $n$-th root of unity. We then have
\begin{equation*}
L=K[X]/(X^{n}-f)
\end{equation*}
for some $f\in{K}$.
\end{pro}
\begin{proof}
See \cite[Proposition 3.2]{neu}. 
\end{proof}
Suppose now that we have a finite extension of Dedekind domains $A\rightarrow{A'}$ such that the morphism of quotient fields $K\rightarrow{L}$ is cyclic abelian. Let $f$ be as in Proposition \ref{AbelExt1}. Write $f=\dfrac{a}{b}$ for $a,b\in{A}$ and let $\alpha\in{L}$ be any root of $X^{n}-f$. We then see that 
$b\cdot{\alpha}$ satisfies the \emph{integral} equation $X^{n}-b^{n-1}a$. In other words, we can now assume that $f\in{A}$. 
We consider the chain of algebras 
\begin{equation*}
A\subseteq{A[X]/(X^{n}-f)}\subseteq{A'}.
\end{equation*}
We do not always have equality, as the extension given by a certain $f$ might be {\it{nonnormal}}.
We will now look at the local case. So assume that $A$ is local and Dedekind, in other words a discrete valuation ring. Let $\mathfrak{p}$  be the maximal ideal of $A$ and $\pi$ the uniformizer. 
We can then write $f=\pi^{m}u$ where $u$ is a unit and $m\geq{0}$. 
\begin{pro}\label{UnrAbelExt1}
Let $A\subseteq{A[X]/(X^{n}-f)}\subseteq{A'}$ be as above with $f=\pi^{m}u$. Let $\mathfrak{q}$ be any prime of $A'$ lying above $\mathfrak{p}$. We then have that
\begin{equation}
|I_{\mathfrak{q}}|=e_{\mathfrak{q}/\mathfrak{p}}=\dfrac{n}{\text{gcd}(m,n)}.
\end{equation} 
\end{pro}
\begin{proof}
We consider the Newton polygon of $X^{n}-\pi^{m}u$, which is given by a single line segment of slope $-\dfrac{m}{n}$. Clearing the denominator and the numerator, we obtain $n/\text{gcd}(n,m)$ 
 in the denominator. This slope is exactly the ramification index of the extension, so we obtain the Proposition.  
\end{proof}

\begin{cor}\label{AbelExt2}
For any cyclic abelian extension $A\subseteq{A'}$ of Dedekind domains with corresponding extension of fraction fields $K\subseteq{L}=K[x]/(X^{n}-f)$, we have that the extension is unramified above $\mathfrak{p}$ if and only if $v_{\mathfrak{p}}(f)\equiv{0}\mod{n}$. 
\end{cor}

\begin{rem}
We will sometimes abbreviate "cyclic abelian coverings" to "abelian coverings" in this thesis. 
\end{rem}


\section{Quotients on the special fiber}\label{QuotientSpecial}

In Proposition \ref{GalExtNormInt}, 
we obtained a quotient $\mathcal{C}/G=\mathcal{D}$ for any finite Galois morphism of normal integral schemes $\mathcal{C}\rightarrow{\mathcal{D}}$. Suppose now that $\mathcal{D}$ is a scheme over a discrete valuation ring $R$. The composition $\mathcal{C}\rightarrow{\mathcal{D}}\rightarrow{\text{Spec}(R)}$ then also endows $\mathcal{C}$ with the structure of an $R$-scheme. The morphism $\mathcal{C}\rightarrow{\mathcal{D}}$ is then a morphism of $R$-schemes. We then have
\begin{lemma}\label{InjectiveGaloisAction}
\begin{equation}
G\hookrightarrow{\text{Aut}_{R}(\mathcal{C})}.
\end{equation}
\end{lemma} 
\begin{proof}
The fact that every $\sigma\in{G}$ gives an $R$-automorphism of $\mathcal{C}$ follows from our definition of the structural morphism $\mathcal{C}\rightarrow{\text{Spec}(R)}$. 
Since $\text{Aut}_{R}(\mathcal{C})\hookrightarrow{\text{Aut}_{K}(\mathcal{C}_{\eta})}$ and $G\hookrightarrow{\text{Aut}_{K}(\mathcal{C}_{\eta})}$, we obtain the statement of the lemma.
\end{proof}
\begin{rem}
If we start the other way around, say with an $R$-scheme structure on $\mathcal{C}$, then this does not descend to an $R$-scheme structure on the quotient $\mathcal{D}$ in general. For instance, let 
Let $K=\mathbb{Q}_{3}(\sqrt{3})$, $R'=\mathbb{Z}_{3}[\sqrt{3}]$, $\pi=\sqrt{3}$ and $R=\mathbb{Z}_{3}$, where $\mathbb{Z}_{3}$ and $\mathbb{Q}_{3}$ are the ring and field of $3$-adic numbers respectively.  Let $X:=\text{Spec}(A)$, where $A=R'[x,y]/(xy-\pi)$ and let $\sigma$ be the automorphism of $A$ given by
\begin{align*}
\sigma(x)&=x,\\
\sigma(y)&=-y,\\
\sigma(\pi)&=-\pi.
\end{align*}
This is an automorphism of order two. The invariant ring $A^{G}$ is then given by $R[x,z]/(x^2\cdot{z}-3)$, which embeds into $A$ by $z\mapsto{y^2}$. We see that the $R'$-scheme structure on $\text{Spec}(A)$ does not descend to an $R'$-scheme structure on $A^{G}$. We do have an $R$-scheme structure on both schemes and by Lemma \ref{InjectiveGaloisAction} we obtain an 
injective homomorphism $G\hookrightarrow{\text{Aut}_{R}(\text{Spec}(A))}$.
The morphism on the special fiber is then given by
\begin{equation}
\mathbb{F}_{3}[x,z]/(x^2\cdot{z})\rightarrow{(\mathbb{F}_{3}[x,y]/(xy))[w]/(w^2)},
\end{equation}
where the nilpotent $w$ comes from the totally ramified extension $\mathbb{Z}_{3}\rightarrow{\mathbb{Z}_{3}[w]/(w^2-3)}\simeq{\mathbb{Z}_{3}[\sqrt{3}]}$. 
\end{rem}



Let us now study the special fiber of $\mathcal{C}$. We have the exact sequence of sheaves
\begin{equation}\label{Firstexactsequence}
0\rightarrow\mathcal{I}\rightarrow{\mathcal{O}_{\mathcal{C}}}\rightarrow{\mathcal{O}_{\mathcal{C}_{s}}}\rightarrow{0},
\end{equation} 
where $\mathcal{I}$ is the ideal sheaf of $\mathcal{C}_{s}$. That is, $\mathcal{I}=\pi{\mathcal{O}_{\mathcal{C}}}$. These sheaves have a natural action by $G$ on them. Taking the invariants under this action then yields the exact sequence
\begin{equation}\label{ExactQuotientSpecial}
0\longrightarrow{\pi{\mathcal{O}_{\mathcal{D}}}}\rightarrow{\mathcal{O}_{\mathcal{D}}}\rightarrow{(\mathcal{O}_{\mathcal{C}_{s}})^{G}}
\end{equation} 
of sheaves on $\mathcal{D}$. 
Note however that the righthand map is not necessarily surjective.
\begin{exa}
Consider the Galois covering given generically by $K(x)\rightarrow{K(x)[z](z^{p}-x)}$, where $K=\mathbb{Q}_{p}(\zeta_{p})$ with uniformizer $\pi=1-\zeta_{p}$. We then have the exact sequence
\begin{equation}
0\rightarrow{}(\pi)\rightarrow{R[z]}\rightarrow{\mathbb{F}_{p}[z]}\rightarrow{0}.
\end{equation}
Note that the Galois action is given by $z\rightarrow{\zeta_{p}\cdot{}z}$ on the generic fiber. The reduction of $\zeta_{p}$ is equal to $1$ however, so the induced action on $\mathbb{F}_{p}[z]$ is trivial. Taking the invariants then yields
\begin{equation}
0\rightarrow{}(\pi)\rightarrow{R[z^{p}]}\rightarrow{\mathbb{F}_{p}[z]}.
\end{equation}
The last map $R[z^{p}]\rightarrow{\mathbb{F}_{p}[z]}$ is not surjective, since there is no element of $R[z^{p}]$ that maps to $z$. 
\end{exa}

When the order of the Galois group is coprime to the characteristic of the residue field, we can prove that the last morphism in Equation \ref{ExactQuotientSpecial} \emph{is} surjective.
\begin{pro}
Suppose that $\text{char}(k)\nmid{|G|}$. Then 
\begin{equation}
0\longrightarrow{\pi{\mathcal{O}_{\mathcal{D}}}}\rightarrow{\mathcal{O}_{\mathcal{D}}}\rightarrow{(\mathcal{O}_{\mathcal{C}_{s}})^{G}}\rightarrow{0}
\end{equation}
is exact.
\end{pro}
\begin{proof}
It suffices to show that the sequence is exact on the stalks. Let $y\in{D}$ and let $x\in\mathcal{C}$ be any point such that $\phi_{\mathcal{C}}(x)=y$. From Equation \ref{Firstexactsequence}, we obtain the exact sequence of abelian groups
\begin{equation}
0\rightarrow\mathcal{I}_{x}\rightarrow{\mathcal{O}_{\mathcal{C},x}}\rightarrow{\mathcal{O}_{\mathcal{C}_{s,x}}}\rightarrow{0}.
\end{equation}
Taking the long exact sequence of group cohomology (see \cite[Appendix B, Proposition 2.3]{Silv1} or \cite{Ser1}), we then obtain 
\begin{equation}
0\longrightarrow{(\pi{\mathcal{O}_{\mathcal{D}}})_{y}}\rightarrow{\mathcal{O}_{\mathcal{D},y}}\rightarrow{(\mathcal{O}_{\mathcal{C}_{s,y}})^{G}}\rightarrow{H^{1}((\pi{\mathcal{O}_{\mathcal{D}}})_{y},G)}.
\end{equation}
Using \cite[Chapter VIII, Section 2, Corollary 1]{Ser1}, we then see that this last cohomology group is zero (here we use the condition on the order of the Galois group). This finishes the proof. 


\end{proof}

\begin{cor}\label{QuotientSpecialFiber}
\begin{equation}
\mathcal{C}_{s}/G=\mathcal{D}_{s}.
\end{equation}
\end{cor}
We note that the above argument can also partially be found in \cite[Remark 1.7]{liu1}.

\section{Galois actions on intersection graphs}\label{MetrizedQuotient}


In this section, we will show that a disjointly branched morphism $\mathcal{C}\rightarrow{\mathcal{D}}$ with Galois group $G$ gives rise to a Galois action on the intersection graph $\Sigma(\mathcal{C})$. In fact, we will see that $\Sigma(\mathcal{C})/G=\Sigma(\mathcal{D})$. We will also define the notions of decomposition and inertia groups for these graphs.

So let us consider a disjointly branched morphism $\phi:\mathcal{C}\longrightarrow{\mathcal{D}}$ with Galois group $G$, as defined in Section \ref{DisjointBran11}. 
\begin{lemma}\label{LemmaGaloisAction}
For any $\sigma\in{G}$ and $x\in\mathcal{C}_{s}$ an intersection point of $\Gamma_{1}$ and $\Gamma_{2}$, we have that $\sigma(x)$ is an intersection point of $\sigma(\Gamma_{1})$ and $\sigma(\Gamma_{2})$.
\end{lemma}
\begin{proof}
Let $\mathfrak{m}$ be the maximal ideal corresponding to $x$ on some open affine $U$. Correspondingly, let $\mathfrak{p}_{i}$ be the primes corresponding to $\Gamma_{i}$. 
The fact that $x$ is an intersection point of $\Gamma_{1}$ and $\Gamma_{2}$ can be paraphrased by
\begin{equation*}
\mathfrak{m}\supseteq{\mathfrak{p}_{i}}
\end{equation*}
for both $i$. Letting $\sigma$ act on the above, we obtain
\begin{equation*}
\sigma(\mathfrak{m})\supseteq{\sigma(\mathfrak{p}_{i})}
\end{equation*}
for both $i$, meaning that $\sigma(x)$ is an intersection point of $\sigma(\Gamma_{1})$ and $\sigma(\Gamma_{2})$. 
\end{proof}
We thus see that $\sigma$ acts as an automorphism of graphs: if $v$ and $v'$ are joined by an edge $e$, then $\sigma(v)$ and $\sigma(v')$ are joined by the edge $\sigma(e)$ by the above lemma. We therefore see that we have a homomorphism
\begin{equation*}
G\longrightarrow{\text{Aut}(\Sigma(\mathcal{C}))}.
\end{equation*}
We can now define {\it{decomposition groups}} and {\it{inertia groups}} for elements of our graph.
\begin{mydef}
[{\bf{Decomposition and inertia groups for graphs}}]
Let $v$ and $e$ be a vertex and edge respectively of the intersection graph $\Sigma(\mathcal{C})$. Let the corresponding points in $\mathcal{C}$ be given by $x_{v}$ and $x_{e}$. We define the {\it{decomposition group of }}$v$ and $e$ to be $D_{x_{v}}$ and $D_{x_{e}}$ respectively. Similarly, we define the {\it{inertia groups}} of $v$ and $e$ to be $I_{x_{v}}$ and $I_{x_{e}}$. 
\end{mydef}
Recall now that we have a homomorphism $G\longrightarrow{\text{Aut}(\Sigma(\mathcal{C}))}$. We can therefore consider the quotient of graphs
\begin{equation*}
\Sigma(\mathcal{C})\longrightarrow{\Sigma(\mathcal{C})/G}.
\end{equation*}

\begin{rem}\label{NonQuotGraph1}
Unfortunately, we can have that $\Sigma(\mathcal{C})/G\neq\Sigma(\mathcal{D})$ for semistable models $\mathcal{C}$ and $\mathcal{D}$. Indeed, consider the semistable model given by the equation
\begin{equation*}
y^2=x(x-\pi)(x+1)(x+1-\pi)(x+2)(x+2-\pi),
\end{equation*}
which gives an intersection graph with two vertices and three edges between them. We have a natural Galois action on this model, given by
\begin{equation*}
y\longmapsto{-y}
\end{equation*}
on the coordinate rings.
Note that the edges of the intersection graph are invariant under the action, giving
\begin{equation*}
I_{e_{i}}=\mathbb{Z}/2\mathbb{Z}.
\end{equation*}
For the components however, we have that
\begin{equation*}
D_{v_{i}}=(1).
\end{equation*}
We thus see that the quotient graph $\Sigma(\mathcal{C})/G$ consists of one vertex with three loops. The intersection graph of the quotient $\mathcal{C}/G$ consists of just one vertex however.
\end{rem}


We would now like to prove the following theorem:
\begin{theorem}\label{MainQuotientTheorem1}
Let $\phi:\mathcal{C}\longrightarrow{\mathcal{D}}$ be a disjointly branched Galois morphism with Galois group $G$. Then
\begin{equation*}
\Sigma(\mathcal{C})/G=\Sigma(\mathcal{D}).
\end{equation*}	
\end{theorem}
\begin{proof}
The proof relies mostly on Lemma \ref{LemmaSmooth2}, which we will restate here.
\begin{lemma}\label{LemmaSmooth1}
Let $\phi:\mathcal{C}\longrightarrow{\mathcal{D}}$ be a disjointly branched Galois morphism. Then the pre-image of a smooth point consists of smooth points.
\end{lemma}
Let us also restate Corollary \ref{CorSmooth2}:
\begin{cor}\label{CorSmooth1}
For every ordinary double point  $x$ of $\mathcal{C}$, we have that the image $\phi(x)$ is an ordinary double point of $\mathcal{D}$.
\end{cor}
To finish the proof of Theorem \ref{MainQuotientTheorem1}, first note that we already have
\begin{equation*}
\mathcal{C}/G=\mathcal{D},
\end{equation*}
a quotient of schemes. We thus only have to show that vertices are mapped to vertices and edges to edges. The first follows from the fact that $\phi:\mathcal{C}\longrightarrow{\mathcal{D}}$ 
maps a codimension one prime in the special fiber $\mathcal{C}_{s}$ to another codimension one prime in $\mathcal{D}_{s}$ (this follows from \cite[Chapter 4.3, Proposition 3.12]{liu2} for instance). 
The second follows from Corollary \ref{CorSmooth1}. This gives the theorem.
\end{proof}

Let us now consider the following problem. Suppose that we again have a disjointly branched morphism 
\begin{equation*}
\phi:\mathcal{C}\longrightarrow{\mathcal{D}}
\end{equation*}
with Galois group $G$. Let $H$ be any subgroup of $G$. We can consider the subquotient
\begin{equation*}
\mathcal{C}\longrightarrow{\mathcal{C}/H}\longrightarrow{\mathcal{D}},
\end{equation*}
where we define $\psi:\mathcal{C}\longrightarrow{\mathcal{C}/H}$ to be the quotient morphism. Note that this quotient is again semistable, by \cite[Proposition 3.48, Page 526]{liu2}. 
\begin{lemma}
The image $\psi(x)$ of an ordinary double point $x\in\mathcal{C}_{s}$ is an ordinary double point of $\mathcal{C}/H$.
\end{lemma}
\begin{proof}
Suppose that $\psi(x)$ is smooth. Then the image of $\psi(x)$ in $\mathcal{D}$ is also smooth, by \cite[Proposition 1.6., page 16]{liu1}.  This contradicts Corollary \ref{CorSmooth1}, concluding the proof.  
\end{proof}
\begin{lemma}\label{MainQuotientLemma1}
Let $\mathcal{C}$, $\mathcal{D}$ and $H$ be as above. Then
\begin{equation*}
\Sigma(\mathcal{C})/H=\Sigma(\mathcal{C}/H).
\end{equation*}
\end{lemma}
\begin{proof}
As in Theorem \ref{MainQuotientTheorem1}, we have a Galois quotient 
\begin{equation*}
\mathcal{C}\longrightarrow{\mathcal{C}/H}
\end{equation*}
with Galois group $H$. As in that theorem, we see that vertices are mapped to vertices and edges to edges, so the lemma follows.  
\end{proof}
The above lemma allows us to find the intersection graph of quotients by just taking the quotient of the intersection graphs. We will use this quite often in the examples to come. 

\begin{mydef}
Let $\phi:\mathcal{C}\longrightarrow{\mathcal{D}}$ be a disjointly branched morphism. We define $\phi_{\Sigma}:\Sigma(\mathcal{C})\longrightarrow{\Sigma(\mathcal{D})}$ to be the induced morphism on intersection graphs. Note that this is well-defined by Theorem \ref{MainQuotientTheorem1}.
\end{mydef}

\section{Metrized complexes of curves}\label{Metrizedcomplex}

In Section \ref{MetrizedQuotient}, we saw that a disjointly branched morphism $\phi_{\mathcal{C}}:\mathcal{C}\rightarrow{\mathcal{D}}$ gives rise to a morphism $\Sigma(\mathcal{C})\rightarrow{\Sigma(\mathcal{D})}$. In this section, we will see that  $\phi_{\mathcal{}}$ also gives rise to a quotient of {\it{metrized complexes of curves}}. Briefly speaking, these metrized complexes consist of a weighted metric graph with the additional data of an algebraic curve $C_{v}$ for every vertex $v$, where the edges adjacent to $v$ are identified with closed points in $C_{v}$.

We will follow the presentation in \cite{ABBR1} for metrized complexes of $k$-curves. 
\begin{mydef}
A metrized complex of $k$-curves consists of the following data:
\begin{enumerate}
\item A weighted metric graph $(\Sigma, w(\cdot{}), l(\cdot{}))$,
\item A smooth, irreducible projective curve $C_{v}/k$ for every vertex $v\in{V(\Sigma)}$ such that $w(v)=g(C_{v})$.
\item An injective function $\text{red}_{v}:T_{v}\rightarrow{C_{v}(k)}$, where $T_{v}$ is the set of edges connected to $v$.
\end{enumerate}
The metrized complex $k$-curves corresponding to this data will be denoted by $(\Sigma,w(\cdot{}),l(\cdot{}),\text{red}_{v}(\cdot{}))$, or just $\Sigma$ if no confusion can arise.
\end{mydef}

\begin{exa}\label{MetrizedExample}
Let $\mathcal{C}$ be a strongly semistable model over $R$. The weighted metric graph $\Sigma(\mathcal{C})$ can be turned into a metrized complex of $k$-curves as follows:
\begin{enumerate}
\item A vertex $v$ in $\Sigma(\mathcal{C})$ naturally corresponds to an irreducible component $\Gamma/k$. We assign the curve $\Gamma$ to $v$ and see that the weight assigned to $v$ is the genus of this curve.
\item Every edge $e$ corresponds to an intersection point $\tilde{x}$ of $\Gamma$ with another component $\Gamma'$. We assign the point $\tilde{x}$ to the edge $e$. 
\end{enumerate}
\end{exa}

We would now like to define the notion of a morphism of metrized complexes. To do that, we first define what a morphism of weighted metric graphs is. After that, we define a {harmonic morphism} of weighted metric graphs and then proceed to the definition of a harmonic morphism of metrized complexes. The underlying idea in defining these notions is that they should be some sort of discrete variant of the usual notion of a morphism of algebraic curves. In particular, we should have some kind of notion of a \emph{ramification index} and \emph{degree}. 

Let $\Sigma'$ and $\Sigma$ be two weighted metric graphs as defined in Section \ref{IntersectionGraphIntro}. A morphism $\psi:\Sigma'\rightarrow{\Sigma}$ is just a morphism of graphs $\psi:\Sigma'\rightarrow{\Sigma}$. For any edge $e'\in{E(\Sigma')}$, we define the \emph{dilatation factor}
\begin{equation}
d_{e'}(\psi)=\dfrac{l(\phi(e'))}{l(e')}.
\end{equation}
This is the analogue of the ramification index for edges. 
For any edge $e\in\Sigma$, we now define the \emph{total degree} of $\psi$ above $e$ to be
\begin{equation}
\text{deg}_{e}(\psi)=\sum_{e'\in\psi^{-1}(e)}d_{e'}(\psi).
\end{equation}
\begin{rem}
We will see in Proposition \ref{InertiagroupIntersectionPoint1} that for a disjointly branched Galois morphism $\mathcal{C}\rightarrow{\mathcal{D}}$ and corresponding  morphism of weighted metric graphs $\Sigma(\mathcal{C})\rightarrow{\Sigma(\mathcal{D})}$, this dilatation factor is equal to the order of the inertia group $I_{e}$. The total degree for every edge will then just be the order of the Galois group, by Lemma \ref{Orbitstabilizer} and the fact that the residue field is algebraically closed. 
\end{rem}

\begin{mydef}\label{HarmonicWeighted}{\bf{[Harmonic morphisms weighted metric graphs]}}
A \emph{harmonic} morphism of weighted metric graphs is a morphism of weighted metric graphs $\psi:\Sigma'\rightarrow{\Sigma}$ such that the total degree $\text{deg}_{e}(\psi)$ is the same for every edge $e\in\Sigma$. 
\end{mydef}

We are now ready to define {\it{harmonic morphisms of metrized complexes}}. 
\begin{mydef}
Let  $\Sigma$ and $\Sigma'$ be two metrized complexes of $k$-curves. A harmonic morphism from $\Sigma'$ to $\Sigma$ is then given by the following data:
\begin{itemize}
\item A harmonic morphism $\Sigma'\rightarrow{\Sigma}$ of weighted metric graphs.
\item A finite separable morphism of algebraic curves $\phi_{v'}:C_{v'}\rightarrow{C_{\phi_{v'}}}$ for every $v'\in{\Sigma'}$.
\end{itemize}
They are required to satisfy the following compatibility data:
\begin{enumerate}
\item Let $v'$ be any vertex in $\Sigma'$. For every $e'\in{T_{v}}$, we have $\text{red}_{v}(\phi(e'))=\phi_{v}(\text{red}_{\phi(v)}(e'))$.
\item Let $v'$ be any vertex in $\Sigma'$ and $e'\in{T_{v'}}$. Then the ramification index of $\text{red}(e')$ for $\phi_{v}$ is equal to $l(e')/l(e)$.
\item Let $v$ be any vertex in $\Sigma'$ and $e\in{T_{\phi(v)}}$. Then for every $x'\in\phi^{-1}_{v}(\text{red}_{\phi(v)}(e))$, there exists an $e'$ in $T_{v}$ such that $\text{red}_{v}(e')=x'$. 
\item For every $v'\in\Sigma'$, we have that $\text{deg}(\phi_{v})$ is equal to the degree on every edge.
\end{enumerate}
\end{mydef}
As noted in \cite{ABBR1}, the last property actually follows from the other properties, since the sum of the ramification indices is equal to the degree of $\phi_{v}$ (see Proposition \ref{SerLocFields} or Lemma \ref{Orbitstabilizer}).

\begin{pro}
A disjointly branched morphism $\mathcal{C}\rightarrow{\mathcal{D}}$ with Galois group $G$ gives rise to a morphism of metrized complexes $\Sigma(\mathcal{C})\rightarrow{\Sigma(\mathcal{D})}$.  
\end{pro}
\begin{proof}
By Proposition \ref{MainQuotientTheorem1}, we already have a morphism of graphs $\phi_{\mathcal{C}}:\Sigma(\mathcal{C})\rightarrow{\Sigma(\mathcal{D})}$. We have to check that this is a harmonic morphism. This follows from the following properties:
\begin{enumerate}
\item $I_{e'}=d_{e'}(\phi_{\mathcal{C}})$, see Proposition \ref{InertiagroupIntersectionPoint1}.
\item $\sum_{e'\in\phi^{-1}_{\mathcal{C}}(e)}d_{e'}(\phi_{\mathcal{C}})=\sum_{e'}|I_{e'}|=|G|$, see Lemma \ref{Orbitstabilizer}. 
\end{enumerate}

The morphisms $C_{v'}\rightarrow{C_{v}}$ are now given by considering the image of $\Gamma_{v'}$ in $\mathcal{D}$ for the composite $\Gamma_{v'}\rightarrow\mathcal{C}\rightarrow{\mathcal{D}}$. 
 The only nontrivial part in the compatibility data that we have to check is the persistence of the ramification indices, which is criterium number two. We will defer this to Chapter \ref{Inertiagroups}, where it will be proved in Proposition \ref{ramind2}. 

\end{proof}

Let us now define what a Galois quotient is for a metrized complex of $k$-curves. 
\begin{mydef}
A {\bf{Galois action}} of a finite group $G$ on a metrized complex of $k$-curves $\Sigma$ consists of the following data.
\begin{itemize}
\item A homomorphism $G\rightarrow{\text{Aut}(\Sigma)}$.
\item A homomorphism $D_{v}\rightarrow{\text{Aut}(K(C_{v})/K(C_{\phi(v)}))}$ for every $v$.
\end{itemize}
We then say that a morphism $\phi$ of metrized complexes $\Sigma'\rightarrow{\Sigma}$ is a {\bf{Galois quotient}} if
\begin{itemize}
\item $\Sigma'/G=\Sigma$. 
\item $C_{v}/D_{v}=C_{\phi(v)}$.
\end{itemize}
In this case, we write $\Sigma'/G=\Sigma$. 
\end{mydef}

Note that for a disjointly branched morphism $\mathcal{C}\rightarrow{\mathcal{D}}$, we have a homomorphism $G\rightarrow{\text{Aut}(\Sigma(\mathcal{C}))}$ by Lemma \ref{LemmaGaloisAction}. The homomorphism of the decomposition groups is then given by Proposition \ref{QuotientLemma10} (part 3) or Proposition \ref{SerExtGal1} (note that the components define codimension one points in $\mathcal{C}$ and $\mathcal{D}$). We thus see that we have a Galois action of metrized complexes on $\Sigma(\mathcal{C})$. 

\begin{pro}
Let $\mathcal{C}\rightarrow{\mathcal{D}}$ be a disjointly branched morphism. Then $\Sigma(\mathcal{C})/G=\Sigma(\mathcal{D})$ as a quotient of metrized complexes. 
\end{pro}
\begin{proof}
This follows from Proposition \ref{MainQuotientTheorem1} and Corollary \ref{QuotientSpecialFiber}. 
\end{proof}

\chapter{Decomposition and inertia groups for disjointly branched morphisms}\label{Inertiagroups}


In this chapter, we will prove several theorems about decomposition and inertia groups that are specific to disjointly branched morphisms. We will start with a well-known theorem on the thickness of an ordinary double points under a quotient by a finite group. After that, we will see how decomposition and inertia groups on a global scheme $\mathcal{C}$ for a disjointly branched morphism $\mathcal{C}\rightarrow{\mathcal{D}}$ are related to those on the special fiber, see Proposition \ref{ramind2}. This will give us a formula for the order of the decomposition group of a vertex in the case where the underlying vertex has genus zero.

We then study the following procedure. For a disjointly branched morphism $\mathcal{C}\rightarrow{\mathcal{D}}$, we can subdivide an edge $e$ in $\mathcal{D}$ and obtain a regular desingularization $\mathcal{D}_{0}$ above $e$. If we then take the normalization of $\mathcal{D}_{0}$ in $\mathcal{D}$, some of the new vertical components might be ramified. We would like to know exactly how they are ramified and we will in fact give a formula for the ramification indices. These ramification indices are then related to the inertia groups of the original edge $e$, which shows that they can be used to calculate the decomposition group of an edge.




\section{The inertia group of an intersection point}

Let $\phi_{\mathcal{C}}:\mathcal{C}\rightarrow{\mathcal{D}}$ be a disjointly branched morphism and let $x\in\mathcal{C}$ be an intersection point. 
The completed ring $\hat{\mathcal{O}}_{\mathcal{C},x}$ is then isomorphic to $R[[x,y]]/(xy-\pi^{n})$ for some $n\in\mathbb{N}$ and the length of the corresponding point is then by definition $n$. By Corollary \ref{CorSmooth2}, we then find that $y:=\phi_{\mathcal{C}}(x)$ is also an ordinary double point. Let us denote their lengths by $l(x)$ and $l(y)$. We then have:
\begin{pro}\label{InertiagroupIntersectionPoint1}
\begin{equation}
|I_{x}|=l(y)/l(x).
\end{equation}
\end{pro}
\begin{proof}
We use \cite[Chapter 10, Proposition 3.48]{liu2}, which says that the length of $x$ is multiplied by the order of the image of the inertia group in $\text{Aut}_{R}(\mathcal{O}_{\mathcal{C},x})$. But by Lemma \ref{InjectiveGaloisAction}, we find that $G\hookrightarrow{\text{Aut}_{R}(\mathcal{C})}$, so this order is just the order of the inertia group $I_{x}$. This finishes the proof. 
\end{proof}
\section{Specialization of decomposition groups and inertia groups}
Let $\mathcal{C}\rightarrow{\mathcal{D}}$ be a disjointly branched morphism. Let $y$ be the generic point of an irreducible component $\Gamma'\subset{\mathcal{C}_{s}}$ and let $x$ be an intersection point lying on $\Gamma'$. 
 We can, in general, not find an injective morphism
\begin{equation*}
D_{{x}}\longrightarrow{D_{{y}}}
\end{equation*}
for general coverings of semistable models.
Indeed, we saw this in Remark \ref{NonQuotGraph1}. We will now show that we do obtain such an injection for disjointly branched morphisms. 
\begin{pro}\label{InjectionDecomposition}
Let $\phi$ be a disjointly branched morphism with $y$ a generic point of an irreducible component $\Gamma\subset{\mathcal{C}_{s}}$ and $x$ an intersection point lying on $\Gamma$. 
There is then a canonical injective morphism $D_{x}\longrightarrow{D_{y}}$.
\end{pro}
\begin{proof}
Let us write down the condition that $x$ is an intersection point on affines. Let $A$ be an affine neighborhood of $x$. Then $A$ also contains $y$.  Let $\mathfrak{m}$ be the maximal ideal corresponding to $x$ and $\mathfrak{p}$ the prime ideal corresponding to $y$. We then have $\mathfrak{m}\supset{\mathfrak{p}}$.  We will show the following: if $\sigma$ fixes $\mathfrak{m}$, then it also fixes $\mathfrak{p}$. Suppose for a contradiction that it doesn't fix $\mathfrak{p}$. Then $\sigma(\mathfrak{p})$ corresponds to a different component. We have $\sigma(\mathfrak{m})=\mathfrak{m}$, so we find
\begin{equation*}
\mathfrak{m}\supset{\sigma(\mathfrak{p})}.
\end{equation*}
This just means (together with $\mathfrak{m}\supset{\mathfrak{p}}$) that $\mathfrak{m}$ is an intersection point of the components $\Gamma$ and $\sigma(\Gamma)$, which correspond to $\mathfrak{p}$ and $\sigma(\mathfrak{p})$.
We will now find a contradiction using
\begin{lemma}\label{QuotientNons1}
Let $G$ be a finite group acting on a semistable model $\mathcal{C}$. Let $x$ be an ordinary double point of $\mathcal{C}$, connecting two components $\Gamma$ and $\Gamma'$. Let $I$ be the inertia subgroup of $x$ and let
\begin{equation*}
\pi:\mathcal{C}\longrightarrow{\mathcal{C}/I}
\end{equation*}
be the corresponding quotient map. Then $\pi(x)$ is smooth in $\mathcal{C}/I$ if and only if there exists an element $\sigma\in{I}$ such that
\begin{eqnarray*}
\sigma(\Gamma)&=&\Gamma'.
\end{eqnarray*}
\end{lemma}
\begin{proof}
Let 
\begin{equation*}
I_{0}=\{\sigma\in{I}:\sigma(\Gamma)=\Gamma\}.
\end{equation*} 
By tracing through the proof of \cite[Page 527, Proposition 3.48]{liu2}, one finds that the case with $I_{0}\subsetneq{I}$ corresponds to $\pi(x)$ being smooth and the case $I_{0}=I$ to $\pi(x)$ being an ordinary double point. The Lemma then quickly follows. 
\end{proof}
The inclusion $\mathfrak{m}\supset{\sigma(\mathfrak{p})}$ will now give us the desired contradiction, which will conclude the proof of Proposition \ref{InjectionDecomposition}. Indeed, we see that $\sigma(\Gamma)=\Gamma'$ and $\sigma$ is an element of the inertia subgroup of $x$ (here we use that our residue field $k$ is algebraically closed). But then $\pi(x)$ is smooth by Lemma \ref{QuotientNons1}. This contradicts Corollary \ref{CorSmooth1}, as desired.   
\end{proof}

We note that for smooth points $x\in\mathcal{C}$ we also have an injection 
\begin{equation*}
D_{x}\longrightarrow{D_{y}}. 
\end{equation*}
This is much easier to prove however, since there is only one component that contains $x$.

\section{Reduction of inertia groups}

In this section, we will give a quick review of several results presented in \cite{supertrop}. Let $\mathcal{C}\rightarrow{\mathcal{D}}$ be a disjointly branched morphism and fix a component $\Gamma'\subset{\mathcal{C}_{s}}$ with image $\Gamma\subset{\mathcal{D}_{s}}$. The decomposition group $D_{\Gamma'}$ is then the Galois group of the covering $\Gamma'\rightarrow{\Gamma}$. 

Now let $x$ be an intersection point of $\Gamma'$ or a ramification point of the morphism $C\rightarrow{D}$ reducing to $\Gamma'$ under the natural reduction map $r_{\mathcal{C}}$ introduced in Definition \ref{ReductionMap11}. 
We let $r_{\mathcal{C}}(x)$ be the corresponding point in $\Gamma$. We now have two inertia groups: inertia groups for the covering $\mathcal{C}\rightarrow{\mathcal{D}}$ and inertia groups for the covering $\Gamma'\rightarrow{\Gamma}$. We would like to relate these two kinds of inertia groups. This is done by the following  

\begin{pro}\label{ramind2}
Let $x\in\mathcal{C}$ be a generic ramification point or an intersection point of a disjointly branched morphism $\phi_{\mathcal{C}}:\mathcal{C}\rightarrow{\mathcal{D}}$. Let $\Gamma$ be any component in the special fiber $\mathcal{C}_{s}$ containing $r_{\mathcal{C}}(x)$.
Then
$$
I_{x,\mathcal{C}}=I_{r_{\mathcal{C}}(x),\Gamma'}
$$
where the second inertia group is an inertia group of the Galois covering $\Gamma'\rightarrow{\Gamma}$ on the special fiber.
\end{pro}

\begin{proof}
We give a sketch of the proof and refer the reader to \cite[Proposition 3.9]{supertrop} for the details. For any closed point $x\in\mathcal{C}_{s}$, we have a natural injection $D_{x}\rightarrow{D_{y}}$, where $y$ is the generic point of an irreducible component $\Gamma'$. See Proposition \ref{InjectionDecomposition}. Under this morphism, we easily obtain the identification $I_{x,\mathcal{C}}=I_{r_{\mathcal{C}(x)},\Gamma'}$. We are thus left with the case where $x$ is a generic ramification point. One then considers the image $y$ of $x$ under $\phi$. Since $\phi_{\mathcal{C}}$ is disjointly branched, we have that $y$ is in the regular locus. We consider the morphism $\mathcal{C}/I_{x}\rightarrow{\mathcal{D}}$ and let $z$ be the image of $r_{\mathcal{C}}(x)$ in $\mathcal{C}/I_{x}$. Supposing that $\mathcal{C}/I_{x}\rightarrow{\mathcal{D}}$ is ramified at $z$, one then obtains a contradiction as follows. Let $z'$ be the image of $z$ in $\mathcal{D}$. It is  in the branch locus and it is not an ordinary double point. By Lemma \ref{LemmaSmooth2}, we find that it is smooth. By purity of the branch locus, there exists a codimension one point above which $\mathcal{C}/I_{x}\rightarrow{\mathcal{D}}$ is ramified. By our assumption on disjointly branched morphisms, this point must be $y$. But this contradicts the fact that $\mathcal{C}/I_{x}\rightarrow{\mathcal{D}}$ is unramified above $y$ (see Lemma \ref{RamificationInertiaLemma}), a contradiction.

\end{proof}

\section{The decomposition group of a vertex}

Let $\phi:\mathcal{C}\longrightarrow{\mathcal{D}}$ be a disjointly branched Galois morphism, with Galois group $G$. 
Let $\Gamma'$ be any irreducible component in the special fiber of $\mathcal{C}$ and let $\Gamma$ be its image in $\mathcal{D}$. 
\begin{theorem}\label{DecompVert}
Suppose that the genus of $\Gamma$ is zero. Then
\begin{equation}
D_{\Gamma'}=\prod_{P\in{\Gamma'(k)}}I_{P}.
\end{equation}
\end{theorem}

\begin{proof}
Let $y'$ be the generic point of $\Gamma'$. We factorize the morphism $\phi:\mathcal{C}\longrightarrow{\mathcal{D}}$ into $\mathcal{C}\longrightarrow{\mathcal{C}/D_{\Gamma'}}\longrightarrow{\mathcal{D}}$. Note that last morphism is {\it{"Nisnevich"}} at the image of $y$. That is, it is \'{e}tale and the induced map of residue fields is an isomorphism. In fact, $K(\mathcal{C}/D_{\Gamma'})$ is the largest among all such fields. Since the map on the residue fields is an isomorphism, we don't have any ramification and as such we find that $D_{\Gamma'}\supset{\prod_{P\in{\Gamma'(k)}}I_{P}}$. Note that this didn't use the condition on the genus of $\Gamma$. 

The induced morphism $\mathcal{C}/\prod_{P\in{\Gamma'(k)}}I_{P}\longrightarrow{\mathcal{C}/D_{\Gamma'}}$ is unramified above every point in the image of $\Gamma'$ in $\mathcal{C}/D_{\Gamma'}$. But this component has the same function field as $\Gamma$, which has genus zero. Since genus zero curves have no unramified coverings (by the Riemann-Hurwitz formula for instance), we find that $D_{\Gamma'}=\prod_{P\in{\Gamma'(k)}}I_{P}$, as desired. 
\end{proof}

\begin{rem}
Note that the condition on the genus of $\Gamma$ is indeed necessary. Take an elliptic curve $E$ with good reduction over $K$ and a corresponding model $\mathcal{E}$ with good reduction over $R$. Now take any unramified Galois covering of $E$ (which is in fact abelian, but we won't be needing this) with Galois group $G$ . Then the corresponding curve $E'$ again has genus $1$ by the Riemann-Hurwitz formula and the corresponding intersection graph consists of only one vertex with weight $1$. We therefore see that $D_{\Gamma'}=G$, even though $I_{P}=(1)$ for every $P$. 
\end{rem}

\section{Subdivisions and inertia groups for edges}

In this section, we prove a continuity result for inertia groups of a disjointly branched morphism, as defined in Section \ref{DisjointBran11}. 
More precisely, for a regular subdivision $\mathcal{D}_{0}$ of $\mathcal{D}$, we will give a formula for the inertia groups of the new components in $\mathcal{D}_{0}$ in terms of the inertia group of the corresponding edge in $\mathcal{D}$. This will allow us to determine the inertia group of an edge in terms of codimension one phenomena, namely the inertia groups of the generic points of these new components. 
We will also give a formula for the decomposition group of a vertex $v'\in\Sigma(\mathcal{C})$ lying above a vertex $v\in\Sigma(\mathcal{D})$, where the corresponding component $\Gamma_{v}$ has genus zero. 


Consider a disjointly branched Galois morphism $\phi:\mathcal{C}\rightarrow{\mathcal{D}}$ with $x\in\mathcal{C}$ an intersection point with length $n_{x}$ and $y$ its image in $\mathcal{D}$ with length $n_{y}$. We will denote the Galois group by $G$. From Proposition \ref{InertiagroupIntersectionPoint1}, 
 we then have the formula
\begin{equation}\label{InertiaFormula}
n_{y}=|I_{x/y}|\cdot{n_{x}}.
\end{equation}

Let $y$ be an intersection point in $\mathcal{D}$, with corresponding components $\Gamma_{0}$ and $\Gamma_{n}$. Here $n$ is the length of $y$.
We now take a regular subdivision $\mathcal{D}_{0}$ of $\mathcal{D}$ in $y$. That is, we have a model $\mathcal{D}_{0}$ with a morphism $\psi: \mathcal{D}_{0}\rightarrow{\mathcal{D}}$ that is an isomorphism outside $y$ and the pre-image $\psi^{-1}\{y\}$ of $y$ consists of $n-1$ projective lines $\Gamma_{i}$. Here, the projective lines are labeled such that $\Gamma_{i}$ intersects $\Gamma_{i+1}$ in one point: $y_{i,i+1}$. Furthermore, we have that $\Gamma_{1}$ intersects an isomorphic copy of the original component $\Gamma_{0}$ in $y_{0,1}$ and likewise $\Gamma_{n-1}$ intersects an isomorphic copy of the original component $\Gamma_{n}$ in $y_{n-1,n}$, see \cite[Chapter 8, Example 3.53. and Chapter 9, Lemma 3.21]{liu2} 
for the details. 

 We now take the normalization $\mathcal{C}_{0}$ of $\mathcal{D}_{0}$ in $K(\mathcal{C})$. By virtue of the universal property for normalizations, we have a natural morphism
\begin{equation}
\mathcal{C}_{0}\rightarrow{\mathcal{C}}
\end{equation}
that is an isomorphism outside $\phi^{-1}(y)$.

Taking the tamely ramified extension $K\subset{K'}$ of order $\text{lcm }(|I_{\Gamma_{i}}|)$, we obtain a new model $\mathcal{D}'_{0}=\mathcal{D}_{0}\times_{\text{Spec}(R)}{\text{Spec}(R')}$ over $R'$, which is the normalization of $\mathcal{D}_{0}$ in $K'(\mathcal{D})$.
 Taking the normalization $\mathcal{C}'_{0}$ of this model inside $K'(\mathcal{C})$, we then naturally obtain morphisms
\begin{equation}
\mathcal{C}'_{0}\rightarrow{\mathcal{C}_{0}}\rightarrow{\mathcal{C}}.
\end{equation}
Here the first morphism is finite and the second one is birational. Note that by \cite[Chapter 10, Proposition 4.30]{liu2}, we have that $\mathcal{C}'_{0}$ is again semistable and that $G$ naturally acts on $\mathcal{C}_{0}$ and  $\mathcal{C}'_{0}$ such that $\mathcal{C}_{0}/G=\mathcal{D}_{0}$ and $\mathcal{C}'_{0}/G=\mathcal{D}'_{0}$ (which follows from the fact that $G$ acts naturally on any normalization, see Proposition \ref{GalExtNormInt}). 

We now wish to study the inertia groups of the various points in $\mathcal{C}'_{0}$, $\mathcal{C}_{0}$ and $\mathcal{C}$. To do that, we will introduce the notion of a "\emph{chain}".

\begin{mydef}
Let $y_{i,i+1}$ and $y'_{i,i+1}$ be the intersection points in $\mathcal{D}_{0}$ and $\mathcal{D}'_{0}$ respectively that map to $y\in\mathcal{D}$ under the natural morphism. Similarly, let $y_{i}$ and $y'_{i}$ be the generic points of the components in $\mathcal{D}_{0}$ and $\mathcal{D}'_{0}$ that map to $y$. Here the generic points are labeled such that $y_{i,i+1}$ is a specialization of both $y_{i}$ and $y_{i+1}$. 
A {\bf{chain}} lying above these points is a collection of generic points $x_{i}$ in the special fiber of $\mathcal{C}_{0}$ or $\mathcal{C}'$ and closed points $x_{i,i+1}$ in $\mathcal{C}_{0}$ or $\mathcal{C}'_{0}$ such that:
\begin{enumerate}
\item $x_{i,i+1}$ is a specialization of both $x_{i}$ and $x_{i+1}$,
\item The $x_{i,i+1}$ map to $y_{i,i+1}$,
\item The $x_{i}$ map to $y_{i}$.
\end{enumerate}
For the remainder of this section, we will refer to these simply as a "{\it{chain}}".
\end{mydef}  

\begin{lemma}\label{LiftChains}
Let $\{x_{i,i+1}\}\cup{\{x_{i}\}}$ be a chain in $\mathcal{C}_{0}$. Then there exists a chain $\{x'_{i,i+1}\}\cup{\{x'_{i}\}}$ in $\mathcal{C}'_{0}$ mapping to $\{x_{i,i+1}\}\cup{\{x_{i}\}}$.
\end{lemma}
\begin{proof}
The idea of the proof is to apply the going-up and going-down theorems for integral extensions several times as follows.  Since $\mathcal{C}'_{0}\rightarrow{\mathcal{C}_{0}}$ is finite, the base change to the special fiber of $\mathcal{C}_{0}$ is also finite. This ensures that any lifts we obtain will be either closed points or generic points of components. We start with $x_{0}$ and $x_{0,1}$ and pick lifts $x'_{0}$ and $x'_{0,1}$ (which exist by the going-up theorem). 
Using the going-down theorem for $x_{0,1}$, $x_{1}$ and $x'_{0,1}$, we obtain a point $x'_{1}$ lying above $x_{1}$. 
Continuing in this fashion yields the lemma.  
\end{proof}

\begin{lemma}\label{ChainLemma}
For every intersection point $x\in\mathcal{C}$, there is only one chain in $\mathcal{C}_{0}$ and in $\mathcal{C}'_{0}$ lying above it.
\end{lemma}
\begin{proof}
Let us prove this for $\mathcal{C}'_{0}$ first. Since $\mathcal{C}'_{0}$ is semistable, we know that the morphism $\mathcal{C}'_{0}\rightarrow{\mathcal{C}}$ is just a blow-up over $R'$ in the sense that the edge $x$ is subdivided into a chain of projective lines. This gives a one-to-one correspondence between chains in $\mathcal{C}'_{0}$ and edges $x\in\mathcal{C}$ lying above $y$. This then also gives the result for $\mathcal{C}_{0}$ as follows. Since every chain in $\mathcal{C}_{0}$ is liftable to a chain in $\mathcal{C}$ (by Lemma \ref{LiftChains}), it has to be unique. Indeed, if there exist two different chains in $\mathcal{C}_{0}$ mapping to $x$, then there would be two different chains in $\mathcal{C}'_{0}$ mapping to $x$, a contradiction. 
\end{proof}

\begin{lemma}
Let $x\in\mathcal{C}$ be an intersection point lying over $y$ 
 and let $x'_{i,i+1}$ and $x_{i,i+1}$ be closed points in $\mathcal{C}'_{0}$ and $\mathcal{C}_{0}$ respectively that map to $x$. 
Then
\begin{equation} I_{x}=I_{x'_{i,i+1}}=I_{x_{i,i+1}}.
\end{equation} 
\end{lemma}
\begin{proof}
We will prove that $D_{e'}=D_{x'_{i,i+1}}=D_{x_{i,i+1}}$. Since the residue field $k$ is algebraically closed by assumption, we have that they are equal to their inertia groups. For any chain $\{x_{i,i+1}\}\cup{\{x_{i}\}}$ (or $\{x'_{i,i+1}\}\cup{\{x'_{i}\}}$ for $\mathcal{C}'_{0}$) mapping to $x$ and $\sigma\in{G}$, we have the induced chain $\{\sigma(x_{i,i+1})\}\cup{\{\sigma(x_{i})\}}$, which maps down to $\sigma(x)$. Using this and Lemma \ref{ChainLemma}, we immediately obtain the desired result. 






\end{proof}

\begin{pro}
Let $x_{i}\in\mathcal{C}_{0}$ be as above. Then
\begin{equation}
I_{e}=I_{x_{i,i+1}}\subset{\prod_{i=0}^{n}I_{x_{i}}}.
\end{equation}  
\end{pro}
\begin{proof}
Consider the morphism $\mathcal{C}_{0}/{\prod_{i=0}^{n}I_{x_{i}}}\rightarrow{\mathcal{D}_{0}}$ and  suppose that it is ramified at the image of some $x_{i,i+1}$. Then it has to ramify in codimension one by purity of the branch locus. But the only possible candidates for this are the images of the $x_{i}$, a contradiction. This gives the desired result.
\end{proof}

Let us quickly try the same argument to prove the other inclusion. Consider the morphism $\mathcal{C}_{0}/I_{e}\rightarrow{\mathcal{D}_{0}}$ and suppose that it is ramified at a vertical component $y_{i}$. From this point on, it is not directly evident how to predict the behavior of the corresponding connected edges. We will illustrate this in an example.

\begin{exa}
Let $A:=R[x,y]/(xy-\pi)$ and consider the covering given by the function field extension
\begin{equation}
K(x)\subset{K(x)[z]}/(z^2-\pi(x+1))=:L.
\end{equation}
The normalization of $A$ in $L$ is then vertically ramified at both components $\Gamma_{1}=Z(x)$ and $\Gamma_{2}=Z(y)$. Taking the tamely ramified extension $K\subset{K(\pi^{1/2})}$, we see that the normalization of $R'[x,y]/(xy-\pi)$ inside $L$ is now \'{e}tale above $(x,y,\pi^{1/2})$. We thus see that the inertia group can be unrelated to the inertia group of the edge after the extension.
\end{exa}

We will now prove that $I_{e}\supset{\prod_{i=0}^{n}I_{x_{i}}}$. To do this, we will use {\it{Abhyankar's Lemma}}. 

\begin{lemma}\label{ramstruct2}{\bf{[Abhyankar's Lemma]}}
Let $X$ be a strictly Henselian local regular scheme of residue characteristic $p$, $D=\sum_{i=1}^{r}\text{div}(f_{i})$ a divisor with normal crossings on $X$ and $U=X-D$. Then every connected finite \'{e}tale covering of $U$ which is tamely ramified along $D$ is a quotient of a (tamely ramified) covering of the form
\begin{equation}
U'=U[T_{1},...,T_{r}]/(T_{1}^{n_{1}}-f_{1},...,T_{r}^{n_{r}}-f_{r}),
\end{equation}
where the $n_{i}$ are natural numbers prime to $p$.
\end{lemma}
\begin{proof}
See \cite[Theorem 1.2]{tamearithmetic} for the current formulation and \cite[Exp. XIII, 5.3., Page 316]{SGA1} for the proof.
\end{proof}

Let us consider this lemma for $y_{i,i+1}$ an intersection point in $\mathcal{D}_{0}$. Note that we have a natural morphism
\begin{equation}
\mathcal{O}_{\mathcal{D}_{0},y_{i,i+1}}\rightarrow{\mathcal{O}_{\mathcal{C}_{0},x_{i,i+1}}},
\end{equation}
giving rise to a morphism of completed rings
\begin{equation}
A:=\hat{\mathcal{O}}_{\mathcal{D}_{0},y_{i,i+1}}\rightarrow{\hat{\mathcal{O}}_{\mathcal{C}_{0},x_{i,i+1}}}.
\end{equation}
The ring $A$ is strictly Henselian, so we can apply Lemma \ref{ramstruct2}. We have
\begin{equation}
A\simeq{R[[u,v]]/(uv-\pi)}
\end{equation}
by assumption, and we thus obtain that $\hat{\mathcal{O}}_{\mathcal{D}_{0},y_{i,i+1}}\rightarrow{\hat{\mathcal{O}}_{\mathcal{C}_{0},x_{i,i+1}}}$ is a quotient of a Kummer covering of the form
\begin{equation}\label{Kummercov1}
A\rightarrow{}A[T_{1},T_{2}]/(T_{1}^{n_{1}}-u,T_{2}^{n_{2}}-v)
\end{equation}
for $n_{i}$ coprime to $p$.  

\begin{pro}\label{Inertiagroup1}
\begin{equation}
I_{e}={\prod_{i=0}^{n}I_{x_{i}}}.
\end{equation}
\end{pro}
\begin{proof}
We already proved that $I_{e}\subset{\prod_{i=0}^{n}I_{x_{i}}}$, so we will now prove the other inclusion. 
Let us first show that $I_{e}=I_{x_{1}}$. We first note that the natural morphism
\begin{equation}
\mathcal{C}_{0}/I_{x_{1}}\rightarrow{\mathcal{D}_{0}}
\end{equation}    
is \'{e}tale at the image of $x_{0,1}$. Indeed, if it were ramified, then it would be ramified in codimension one by purity of the branch locus. But it is already unramified at the image of both $x_{0}$ and $x_{1}$ (the first by the assumption on disjointly branched morphisms and the second by Proposition \ref{RamificationInertiaLemma}, part 4).  
We conclude that this is impossible.

We would now like to show that $\mathcal{C}_{0}/I_{x_{0,1}}\rightarrow{\mathcal{D}_{0}}$ is unramified at the image of $x_{1}$. Suppose that it is ramified. Since $x_{0}$ is unramified, we see that the associated morphism of completions from Equation \ref{Kummercov1} is of the form
\begin{equation}
A[T]/(T^{n}-v),
\end{equation}
where $v$ is a uniformizer for the local ring at $y_{1}$. But then a simple calculation shows that the length of the corresponding ring in $\mathcal{C}'_{0}/I_{x_{0,1}}$ would be strictly smaller. This contradicts the fact that $\mathcal{C}'_{0}/I_{x_{0,1}}=\mathcal{C}'_{0}/I_{x'_{0,1}}$ is \'{e}tale at the image of $x'_{0,1}$. We thus conclude that $I_{e}=I_{x_{1}}$.

We will now prove by rising induction that $I_{e}={\prod_{i=1}^{j}I_{x_{i}}}$ for every $j\leq{n}$. The case with $j=1$ was just treated. So assume that $I_{e}={\prod_{i=1}^{j-1}I_{x_{i}}}$. Consider the morphism
\begin{equation}
\mathcal{C}_{0}/{\prod_{i=1}^{j-1}I_{x_{i}}}\rightarrow{\mathcal{C}_{0}/{\prod_{i=1}^{j}I_{x_{i}}}}.
\end{equation}
Using the same reasoning as before, we see that $\mathcal{C}_{0}/{\prod_{i=1}^{j-1}I_{x_{i}}}\rightarrow{\mathcal{D}_{0}}$ is \'{e}tale at the image of $x_{j-1,j}$. The corresponding completed local ring 
in $\mathcal{C}_{0}/{\prod_{i=1}^{j-1}I_{x_{i}}}$ is thus regular. Using Equation \ref{Kummercov1}, we see that the corresponding covering again must be of the form
\begin{equation}
A[T]/(T^{n}-v).
\end{equation}
Indeed, $\mathcal{C}_{0}/{\prod_{i=1}^{j-1}I_{x_{i}}}\rightarrow{\mathcal{C}_{0}/{\prod_{i=1}^{j}I_{x_{i}}}}$ is unramified at the image of $x_{j-1}$, so there is no other option. But then the corresponding length again decreases and we obtain another contradiction as in the $j=1$ case. By induction, we then conclude that $I_{e}={\prod_{i=0}^{n}I_{x_{i}}}$. 

\end{proof}

We now set out to prove a formula for the inertia group $I_{x}$ in terms of the $I_{x_{i}}$. In the proof of Proposition \ref{Inertiagroup1}, we already saw that $I_{x}=I_{x_{1}}$. In general, the other inertia groups will be smaller. We first have the following
\begin{lemma}
Let $x$ be an intersection point in $\mathcal{C}$, $y$ its image in $\mathcal{D}$ and let $I_{x}$ be the corresponding inertia group. Then $I_{x}$ is cyclic.
\end{lemma}
\begin{proof}
This follows from $I_{x}=I_{x_{1}}$ and the fact that $I_{x_{1}}$ is cyclic (which is a result on tame Galois coverings of discrete valuation rings). For another proof, we note that
\begin{equation}
I_{x}=I_{\tilde{x}},
\end{equation}
where $\tilde{x}$ is the intersection point, considered as an element of a component $\Gamma'\subset{\mathcal{C}_{s}}$ and the inertia group is an inertia group for the induced Galois covering $\Gamma'\rightarrow{\Gamma}$. This equality follows from Proposition \ref{ramind2}. 
Since the local ring for $\tilde{x}$ in $\Gamma'$ is a discrete valuation ring, we again obtain that the inertia group is cyclic.  
\end{proof}

We now consider the cyclic abelian extension
\begin{equation}
\mathcal{C}_{0}\rightarrow{\mathcal{C}_{0}/I_{e}}.
\end{equation}

We note that $\mathcal{C}_{0}/I_{e}$ is again regular at the image of the chain induced by $e$. We now have

\begin{theorem}\label{InertProp2}
Let $I_{x_{i}}$ be as above. Then 
\begin{equation}
|I_{x_{i}}|=\dfrac{|I_{e}|}{\gcd(i,|I_{e}|)}.
\end{equation}
In particular, for $i$ such that $\gcd(i,|I_{e}|)=1$, we have that 
\begin{equation}
I_{x_{i}}=I_{e}.
\end{equation}
\end{theorem}
 
\begin{proof}
The corresponding extension of function fields for $\mathcal{C}_{0}\rightarrow{\mathcal{C}_{0}/I_{e}}$ is cyclic abelian, so we find by Kummer theory that it is given by an extension of the form
\begin{equation}
z^n=f
\end{equation}
for some $f\in{}K(\mathcal{C}_{0}/I_{e})$. Since $\mathcal{C}_{0}\rightarrow{\mathcal{C}_{0}/I_{e}}$ is unramified at $x_{0}$, we can assume that $v_{\tilde{x}_{0}}(f)=0$. Here $\tilde{x}_{0}$ is the image of $x_{0}$ in $\mathcal{C}_{0}/I_{e}$. Let $\tilde{x}$ be the image of $x$ in $\mathcal{C}/I_{e}$.  By Proposition \ref{ValCor1}, 
we now find that 
\begin{equation}
v_{x_{i}}(f)=\delta_{\tilde{x}}(\psi))\cdot{i},
\end{equation}
where $\psi$ is the Laplacian of $f$ and $\delta_{\tilde{x}}(\psi)$ is the slope of $\psi$ along $\tilde{x}$ in $\mathcal{C}/I_{e}$.  Since $\tilde{x}$ is completely ramified in this extension, we find that $\text{gcd}( \delta_{\tilde{x}}(\psi),n)=1$. Using Proposition \ref{UnrAbelExt1}, we see that the order of the inertia group is as stated in the theorem. 
\end{proof}

\chapter{Tropical separating trees}\label{Appendix2}

In this section we will associate a tree to a set of elements $S$ of $\tilde{K}:=K\cup{\{\infty\}}$. 
This tree canonically gives a semistable model $\mathcal{D}_{S}$ of $\mathbb{P}^{1}$ such that the closure of these elements consists of disjoint, smooth sections. We will first give the construction in terms of blow-ups and then give the construction in terms of $\pi$-adic expansions.

This semistable model gives us our main way of obtaining disjointly branched models. For a Galois covering of curves $C\rightarrow{\mathbb{P}^{1}}$, we first determine the set of branch points: $S$. We then consider the associated separating model $\mathcal{D}_{S}$ constructed in this chapter and take a finite extension to eliminate the vertical ramification, giving a new model $\mathcal{D}_{S}\times{R'}$. The normalization $\mathcal{C}$ of this model in $K(C)$ is then semistable.

 \section{Construction in terms of blow-ups}\label{ConstructionBlowup}
Let 
\begin{equation*}
S:=\{\alpha_{1},...,\alpha_{r}\}\subset{\mathbb{P}^{1}(K)}. 
\end{equation*}
We will assume that no $\alpha_{i}$ is equal to another $\alpha_{j}$. 
Consider the standard model $\mathbb{P}^{1}_{R}$ given by glueing the rings $R[x]$ and $R[1/x]$. Let $\tilde{x}$ be any closed point of the special fiber $\mathbb{P}^{1}_{k}$.
 Let
\begin{equation*}
S_{\tilde{x}}:=\{\alpha\in{S}:r(\alpha)=\tilde{x}\},
\end{equation*}
where $r(\cdot)$ is the reduction map associated to $\mathbb{P}^{1}_{R}$.
This partitions the original set $S$ (because points have a unique reduction point by the fact that our ring $R$ is Henselian).\\
We label the points in $\mathbb{P}^{1}_{k}$ that the set $S$ reduces to by
\begin{equation*}
x_{1},x_{2},...,x_{l}.
\end{equation*}
We will then write $S_{i}$ for $S_{x_{i}}$. Let us consider the blow-up of $\mathbb{P}^{1}_{R}$ at all the $x_{i}$. We will later give an interpretation using $\pi$-adic expansions. We denote the blow-up by $P_{1}$. Let $z$ be any closed point in the exceptional divisor of $P_{1}$. Consider the set
\begin{equation*}
r_{P_{1}}^{-1}(z)=\{\alpha\in{S_{i}}:r_{P_{1}}(\alpha)=z\}.
\end{equation*}
By varying $z$ over the closed points of the exceptional divisor of $P_{1}$, we obtain a partition of $S_{i}$. We label the points in the exceptional divisor that the set $S_{i}$ reduces to by
\begin{equation*}
x_{i,1},x_{i,2},...,x_{i,l_{i}}.
\end{equation*}
As before, we then define $S_{i_{1},i_{2}}=r_{P_{1}}^{-1}(x_{i_{1},i_{2}})$. Continuing this process, we then obtain sequences of sets
\begin{equation*}
S_{i_{1},i_{2},...,i_{t}}.
\end{equation*}
and models $P_{i}$ of $\mathbb{P}^{1}$. 

\begin{lemma}\label{LemmaLargeT}
For $t$ large, we have $|S_{i_{1},i_{2},...,i_{t}}|=1$.
\end{lemma}
\begin{proof}
The easiest way to see this is using $\pi$-adic expansions. We defer this to the next section. The basic idea is that the semistable model corresponding to $S_{i_{1},i_{2},...,i_{t}}$ separates certain $\pi$-adic expansions up to a certain height.
\end{proof}

We can now define the {\bf{tropical separating tree}} of $S$.
\begin{mydef}\label{TropSep}
Consider the set of all $S_{i_{1},i_{2},...,i_{t}}$ as constructed above. We consider the following inclusions:
\begin{equation*}
S_{i_{1},i_{2},..,i_{t-1},i_{t}}\subset{S_{i_{1},i_{2},...,i_{t-1}}}.
\end{equation*}
Consider the graph $\Sigma_{S}$ consisting of all  $S_{i_{1},i_{2},...,i_{t}}$ as vertices. The edge set consists of all pairs of vertices such that
\begin{equation*}
S_{i_{1},i_{2},..,i_{t-1},i_{t}}\subset{S_{i_{1},i_{2},...,i_{t-1}}}.
\end{equation*}
Furthermore, we create one vertex $S_{\emptyset}$ that is connected to all $S_{i}$.
The {\bf{tropical separating tree}} for $S$ is the finite complete subgraph consisting of the following vertex set:
\begin{itemize}
\item All vertices $S_{i_{1},i_{2},...,i_{t}}$ such that $|S_{i_{1},i_{2},..,i_{t}}|>1$. 
\item The vertex $S_{\emptyset}$.
\end{itemize}
\end{mydef}

\begin{rem}
From Lemma \ref{LemmaLargeT}, we see that the tropical separating tree is indeed finite. 
\end{rem}
\begin{mydef}
The semistable model $\mathcal{D}_{S}$ for $\mathbb{P}^{1}$ constructed before Definition \ref{TropSep} will be referred to as the {\it{separating semistable model}} for $S$. Its intersection graph is the same as $\Sigma_{S}$.
\end{mydef}
We will give some explicit equations for parts of this semistable model $\mathcal{D}_{S}$ in Section \ref{Appendix3}. 
\begin{exa}\label{Exa111}
Let us consider the 4 elements
\begin{equation*}
S=\{0,\pi,\pi+\pi^2, \pi+2\pi^2\}.
\end{equation*}
We see that they all reduce to the point $0$. We therefore consider the blow-up, given by
\begin{equation*}
t'\pi=x.
\end{equation*}
The corresponding $t'$-coordinates for the last three points are
\begin{equation*}
S'=\{1,{1+\pi},{1+2\pi}\}.
\end{equation*}
These reduce to the same point given by $(x,t'-1,\pi)$. We consider one patch of the blow-up in this point, given by
\begin{equation*}
R[x,t'][t'']/(t'\pi-x,t''\pi-(t'-1)).
\end{equation*} 
The last two points of $S'$ now have $t''$-coordinates given by
\begin{equation*}
S''=\{1,2\}.
\end{equation*}
We see that these points reduce to different points, meaning we have reached our endpoint.\\
The corresponding tropical separating tree is now given as follows: we have a tree consisting of three vertices. The first set is
\begin{equation*}
S_{0}=\{0,\pi,\pi+\pi^{2},\pi+2\pi^{2}\}.
\end{equation*}
On the blow-up, these reduce to two different points $x_{0,0}$ and $x_{0,1}$. We have
\begin{eqnarray*}
S_{0,0}&=&\{0\}, \\
S_{0,1}&=&\{\pi,\pi+\pi^{2},\pi+2\pi^{2}\}.
\end{eqnarray*}
The graph now has vertex set
\begin{equation*}
V_{\Sigma}=\{S_{\emptyset},S_{0},S_{0,1}\},
\end{equation*}
with edges between $S_{\emptyset}$ and $S_{0}$, and $S_{0}$ and $S_{0,1}$.
\end{exa}
\section{Construction using $\pi$-adic expansions}
Let $N\subset{R}$ be a set of representatives of $R/\mathfrak{m}$, where we assume that $0\in{N}$. 
Then every element $x$ of $R$ can be written {\it{uniquely}} as
\begin{equation*}
x=\sum_{i=0}^{\infty}a_{k}\pi^{k},
\end{equation*}
where $a_{k}\in{N}$. Suppose that we are given a finite set $S$ of elements in $R$. Write every element $\alpha_{i}$ as
\begin{equation*}
\alpha_{i}=\sum_{k=0}^{\infty}a_{i,k}\pi^{k}
\end{equation*}
as above. 

As an example, we now consider the points $\alpha_{i}$ that reduce to the point $\tilde{0}$, that is: $a_{i,0}=0$. We then have
\begin{equation*}
\alpha_{i}=\pi\cdot{f_{i,1}},
\end{equation*}
where 
\begin{equation*}
f_{i,1}={\sum_{k=0}^{\infty}a_{i,k+1}\pi^{k}.}
\end{equation*}
As indicated in the previous section, we create the separating tree for $S$ by considering blow-ups in the points where $S$ is not separated. The critical point for these $\alpha_{i}$ with $a_{i,0}=0$ is of course $P=(x,\pi)$. Let us consider a specific affine patch of the blow-up of $\text{Spec}{(R[x])}$ in $P$:
\begin{equation*}
R_{1}:=R[x,t']/(t'\pi-x).
\end{equation*}
We then easily see that the $t'$ coordinates of the elements of $S$ are given by
\begin{equation*}
t'(\alpha_{i})=f_{i,1}.
\end{equation*}
Considering the prime ideal $(x-\alpha_{i},t'-f_{i,1},\pi)$, we see that it gives a reduced coordinate $\overline{t'}=\overline{f_{i,1}}=\overline{a_{i,1}}$. \\ 
We state this observation separately:
\begin{itemize}
\item {\it{The extra coordinate $t'$ on the blow-up keeps track of the coefficient $a_{i,1}$ in the $\pi$-adic expansion}}.
\end{itemize}
That is, we have separated these coordinates up to their first ($k=1$) $\pi$-adic coefficient.
Now if, for instance, $a_{1,1}$ and $a_{2,1}$ are the same, we cannot distinguish between them on the special fiber of this blow-up. We therefore blow-up $\text{Spec}(R_{1})$ in the point $Q:=(x-\alpha_{i},t'-a_{1,1},\pi)$. This gives a new algebra $R_{2}:=R_{1}[t'']/(t''\pi-(t'-a_{1,1}))$. We can then write
\begin{eqnarray*}
f_{1,1}-a_{1,1}&=&\pi{f_{1,2}},\\
f_{2,1}-a_{1,1}&=&\pi{f_{2,2}},
\end{eqnarray*}
where
\begin{eqnarray*}
f_{1,2}&=&\sum_{k=0}^{\infty}a_{1,k+2}\pi^{k},\\
f_{2,2}&=&\sum_{k=0}^{\infty}a_{2,k+2}\pi^{k}.
\end{eqnarray*}
This means that the $t''$-coordinates of $\alpha_{1}$ and $\alpha_{2}$ are given by
\begin{eqnarray*}
t''(\alpha_{1})&=&f_{1,2},\\
t''(\alpha_{2})&=&f_{2,2}.
\end{eqnarray*}
As before, we have that their reduced coordinates are now respectively $\overline{t''}=\overline{f_{1,2}}=\overline{a_{1,2}}$ and $\overline{t''}=\overline{f_{2,2}}=\overline{a_{2,2}}$. These new coordinates can be the same of course and we can then continue the blow-up process. At some point however we must have that their $\pi$-adic coefficients are different (at least, if $\alpha_{1}\neq{\alpha_{2}}$). This happens exactly at the $k$-th $\pi$-adic coefficient, where $k=v(\alpha_{1}-\alpha_{2})$. Note that this also proves Lemma \ref{LemmaLargeT}, since in any finite set of distinct elements in $\mathbb{P}^{1}(K)$, elements agree only up to a certain finite height in their $\pi$-adic expansions. 
\begin{rem}
An important observation now is the following: every component $\Gamma$ in the semistable model $\mathcal{D}_{S}$ corresponds to a finite $\pi$-adic expansion
\begin{equation}
z_{\Gamma}=a_{0}+a_{1}\pi^{1}+...+a_{k}\pi^{k}.
\end{equation}
Points $z$ in $\mathbb{P}^{1}(K)$ that have this expansion up to height $k$ will reduce to this component $\Gamma$ if there are no further components $\Gamma'$ (with their own $\pi$-adic expansions) that agree with $z$ up to a higher power of $\pi$. 
\end{rem}
\subsection{An algorithm for separation}\label{AlgSep1}
Let us give this $\pi$-adic separation process in an algorithmic fashion. We suppose that we are given $n$ points $S:=\{\alpha_{i}\}$ that all reduce to finite points in $\mathbb{P}^{1}_{k}$. That is, $v(\alpha_{i})\geq{0}$. The infinite case is similar. 
\begin{mydef}
We say that $S$ is {\bf{separated up to height }}$k$ if the images of the $\alpha_{i}$ in the ring $R/(\pi)^{k}$ are all different.
\end{mydef}
For any finite set $S$, there exists a finite integer $k$ such that $S$ is separated up to height $k$. For instance, the roots of a single polynomial $f$ are separated up to $v(\Delta(f))$ (which is of course not always the smallest integer that has this property). At any rate, we are now ready to give the separating tree in terms of $\pi$-adic expansions. We will assume that the set $S$ is separated up to height $k$ for the next algorithm.
\begin{algo}
{\center{
{\bf{[Algorithm for separating trees]}}
\begin{flushleft}
Input: A finite subset $S\subset{\mathbb{P}^{1}(K)}$. 
\end{flushleft}
\begin{itemize}
\item Determine the separating height $k$ for the subset $S$.
\item Calculate for every $i$ the $\pi$-adic expansion $\alpha_{i}=u_{0,i}+u_{1,i}\pi+u_{2,i}\pi^2+...+u_{k,i}\pi^{k}+r_{i}$, where $r_{i}$ has valuation strictly greater than $k$.
\item First partition: partition $S$ into subsets $S_{j}$ that have the same zeroth order approximation $u_{0,i}$.
\item Second partition: partition every $S_{j_{1}}$ into subsets $S_{j_{1},j_{2}}$ that have the same first order approximation $u_{1,i}$.
\item Third partition: partition every $S_{j_{1},j_{2}}$ into subsets $S_{j_{1},j_{2},j_{3}}$ that have the same second order approximation $u_{2,i}$. 
\item (Iterate the partition process up to $k$).
\item  Construct the finite tropical separating tree $\Sigma_{S}$ according to Definition \ref{TropSep}. 
\end{itemize}}}
\begin{flushleft}
Output: The tropical separating tree $\Sigma_{S}$. 
\end{flushleft}
\end{algo}

\section{
Hyperelliptic coverings of the projective line}\label{Appendix3}

To illustrate how these separating semistable models are used, we consider the example of hyperelliptic coverings of the projective line. 
Let $\phi: D\rightarrow{\mathbb{P}^{1}}$ be a hyperelliptic covering, given generically by an equation of the form
\begin{equation}
y^2=f(x),
\end{equation}
where we assume that $f(x)$ is a squarefree polynomial. 
Let $S$ be a set in $\mathbb{P}^{1}(K)$ containing the branch locus of $\phi$. The previous section demonstrated a canonical semistable model $\mathcal{D}_{S}$ of $\mathbb{P}^{1}$ that separates $S$ in the special fiber. After a finite extension of $K$, we find that the normalization $\mathcal{D}$ of $\mathcal{D}_{S}$ in $K(D)$ gives a semistable model of $D$ over $R'$. In this section, we give an explicit representation of the residue field extension
\begin{equation}
k(\Gamma)\rightarrow{k(\Gamma')},
\end{equation} 
where $\Gamma\subset{\mathcal{D}_{S,s}}$ is an irreducible component in the special fiber. This representation is needed in the algorithm for the twisting data.

We apply the construction of Section \ref{ConstructionBlowup} to $S$ and find the separating tree $\Sigma(\mathcal{D}_{S})$. 
A component $\Gamma$ in this tropical separating tree now corresponds to a finite $\pi$-adic expansion
\begin{equation*}
z_{\Gamma}=a_{0}+a_{1}\pi^{1}+...+a_{k}\pi^{k}.
\end{equation*}
We now wish to obtain an expression of $x$ in terms of a local uniformizer (which is $\pi$) and a generator of the residue field of $\Gamma$. This is in fact not too hard: we consider the following chain of blow-ups
\begin{eqnarray*}
R_{0}&=&R[x], \\
R_{1}&=&R_{0}[t_{1}]/(t_{1}\pi-(x-a_{0})), \\
R_{2}&=&R_{1}[t_{2}]/(t_{2}\pi-(t_{1}-a_{1})), \\
\vdots&{}&\vdots\\
R_{k}&=&R_{k-1}[t_{k}]/(t_{k}\pi-(t_{k-1}-a_{k-1})).
\end{eqnarray*}
This then expresses $x$ in terms of $t_{k}$: $x=g(t_{k})$. We then find $f(x)=f(g(t_{k}))$. To obtain the normalization, we take the highest power of $\pi$ out:
\begin{equation*}
f(x)=f(g(t_{k}))=\pi^{r}h(t_{k}).
\end{equation*}
Note that this $h(t_{k})$ is indeed a polynomial in $t_{k}$, as can easily be seen from the above equations. The normalization is then given by
\begin{equation}\label{NormEq1}
y'^2=h(t_{k}),
\end{equation}
where $y'=\dfrac{y}{\pi^{r/2}}$.
Reducing the equation $\text{mod }{\pi}$ might result in some multiple factors in $\overline{h(t_{k})}$, or even worse: the equation might be reducible.
\begin{itemize}
\item If Equation \ref{NormEq1} is reducible, then the residue field extension is an isomorphism and we have $\overline{y}'=\tilde{h}$ for some ${\tilde{h}}$.
\item If Equation \ref{NormEq1} is irreducible and has multiple factors, we normalize to obtain a new equation. The residue field extension then has degree $2$.
\end{itemize} 

\begin{rem}
The unique components of $R_{i}$ and $R_{i-1}$ have a single intersection point $P$ on the semistable model $\mathcal{D}_{S}$. Note that this intersection point is not visible in these equations: it is given on the special fiber as "$t_{i}=\infty$". To illustrate this, consider the algebra $R[x,t_{1}]/(t_{1}\pi-(x-a_{0}))$. The natural missing algebra is then given by $R[x,t'_{1}]/((x-a_{0})t'_{1}-\pi)$. 
On the overlap of these affine charts, we find that $t_{1}$ and $t'_{1}$ satisfy $t_{1}\cdot{t'_{1}}=1$. 
It is clear from this equation why the intersection point with $t_{1}=0$ is not visible in the other chart.  
\end{rem}

\chapter{Covering data}\label{Coveringdata}
In this section, we will give an algorithm for finding the covering data of a disjointly branched morphism $\mathcal{C}\rightarrow{\mathcal{D}}$ with the associated morphism $\Sigma(\mathcal{C})\rightarrow{\Sigma(\mathcal{D})}$. This covering data consists of the number of edges and vertices lying above every edge and vertex of $\Sigma(\mathcal{D})$.

We will find these quantities using Lemma \ref{Orbitstabilizer}, which gives an expression in terms of the decomposition group of the edge or vertex in question.  
Since the decomposition group of an edge is equal to its inertia group, it suffices to find the inertia group. We will do this using Theorem \ref{InertProp2}. For a vertex, we use Theorem \ref{DecompVert}, which expresses the decomposition group of a vertex in terms of inertia groups. After that, we will explicitly give the covering data of an abelian covering $C\rightarrow{D}$ using Kummer extensions. 



\section{Algorithm for the covering data}

We now give an algorithm for finding the covering data of a Galois covering $C\rightarrow{\mathbb{P}^{1}}$ with Galois group $G$ and function field extension $K(x)\rightarrow{K(C)}$. The idea is as follows. Let $A$ be a discrete valuation ring in $K(x)$ with maximal ideal $\mathfrak{p}$. We consider the normalization $A'$ of $A$ in $K(C)$. This is quite easy to calculate, see Appendix \ref{Appendix1} for the case of $S_{3}$ coverings. Let $\mathfrak{q}$ be any prime of $A'$ lying above $\mathfrak{p}$. We then calculate the order of the inertia group of $\mathfrak{q}$: $|I_{\mathfrak{q}}|$. If we know these for sufficiently many valuations, then we know the covering data by Theorems \ref{DecompVert} and \ref{InertProp2}.

\begin{algo}\label{AlgorithmCoveringDataDisjointlyBranched}
\begin{center}
{\bf{[Algorithm for the covering data of a disjointly branched morphism]}}
\end{center}
\begin{flushleft}
Input: The Galois covering $C\rightarrow{\mathbb{P}^{1}}$ as a function field extension $K(x)\subset{K(C)}$.
\end{flushleft}
\begin{enumerate}
\item Determine the branch locus $S$ of $C\rightarrow{\mathbb{P}^{1}}$.
\item Construct the corresponding separating semistable model $\mathcal{D}_{S}$ using Chapter \ref{Appendix2}. 
\item For every irreducible component ${\Gamma}$ in the special fiber of $\mathcal{D}_{S}$ corresponding to a valuation $v_{\Gamma}$, determine the inertia group $I_{\Gamma}$.
\item Take a finite extension $K\subset{K'}$ of order $\text{lcm }(|I_{\Gamma}|)$.
\item Consider an edge $e\in\Sigma(\mathcal{D}_{S})$ and let $\Gamma_{1}$ be a component in a regular subdivision of $e$, with valuation $v_{\Gamma_{1}}$. Determine the order of an inertia group of $\Gamma_{1}$. 
This gives $|I_{e}|$ by Theorem \ref{InertProp2}.
\item The order of the decomposition group of a vertex is then determined by Theorem \ref{DecompVert}. 
\end{enumerate}
Output: The covering data $|D_{v}|$ and $|D_{e}|$ for every vertex and edge in $\Sigma(\mathcal{D}_{S})$. 
\end{algo}
\begin{proof}
The orders of $D_{v}$ and $D_{e}$ are correct by Theorems \ref{DecompVert} and \ref{InertProp2}. The algorithm terminates because there are only finitely many normalizations that have to be calculated: one for every branch point, one for every vertex of $\mathcal{D}_{S}$ and one for a component in a regular subdivision of an edge $e$.  
\end{proof}
\begin{rem}
The way we use this algorithm in practice is as follows. Instead of focusing on one covering, we will be slightly more ambitious and consider a \emph{family} of coverings that have the same Galois group and then specialize. For instance, in Chapter \ref{Abelian} we will consider abelian coverings of the projective line and in Chapter \ref{Solvable} we will consider $S_{3}$-coverings of the projective line. Calculating {\it{enough}} normalizations in these cases then gives us global statements on the covering data. The calculations for abelian coverings amount to normalizations in Kummer extensions, which we will give in the next section. The normalizations for $S_{3}$-coverings are a bit more work and are given in Appendix \ref{Appendix1}. The results of these calculations for $S_{3}$-coverings can be found in Proposition \ref{InertS3}.  

\end{rem}

\begin{rem}
Let us make a remark about the genus of any vertex $v'$ lying above a vertex $v$ in the tropical separating tree ${T_{S}}$. We automatically have a morphism $\Gamma_{v'}\rightarrow{\Gamma_{v}}$ and we know the ramification points of this morphism: they are the reductions of generic ramification points or the edges. Their decomposition groups were calculated in Algorithm \ref{AlgorithmCoveringDataDisjointlyBranched}. Using the {\emph{Riemann-Hurwitz formula}} given in Theorem \ref{RiemannHurwitz}, we then immediately know the genus of $\Gamma_{v'}$.
\end{rem}

\begin{rem}
Once the decomposition groups of all the edges are known, the lengths of these edges are also known. They are given by Proposition \ref{InertiagroupIntersectionPoint1}. 
\end{rem}


\section{Covering data for abelian coverings}

We will now consider the special case of abelian coverings. Let $\phi: C\rightarrow{D}$ be a cyclic abelian covering, with Galois group $\mathbb{Z}/n\mathbb{Z}$ for some $n\in\mathbb{N}_{>0}$. This induces a morphism of function fields
\begin{equation}
K(D)\rightarrow{K(C)},
\end{equation}
which we can explicitly describe using Kummer theory. Indeed, Proposition \ref{AbelExt1} tells us that 
\begin{equation}
K(C)\simeq{K(D)[z]/(z^{n}-f)}
\end{equation} 
for some $f\in{K(D)}$. We will give the covering data in terms of $f$ and its Laplacian $\phi_{f}$. To that end, let $\mathcal{C}\rightarrow{\mathcal{D}}$ be a disjointly branched morphism for $\phi$ and consider an edge $e\in\Sigma(\mathcal{D})$. Let $\delta_{e}(\phi_{f})$ be the slope of $\phi_{f}$ along $e$.

\begin{pro}\label{PropositionCoveringData}
Let $e'$ be any edge lying above $e$. Then
\begin{equation}
|I_{e'}|=\dfrac{n}{\text{gcd}(n,\delta_{e}(\phi_{f}))}.
\end{equation} 
\end{pro}
\begin{proof}
We will give two proofs of this fact. The first one uses Proposition \ref{ramind2} and the second one uses Theorem \ref{InertS3}.

Let $f^{\Gamma}$ be the $\Gamma$-modified form of $f$. We then obtain an \'{e}tale algebra $k(\Gamma)[z]/(z^{n}-f^{\Gamma})$ that describes the function field extensions above $\Gamma$. We consider this function field extension above the intersection point $\tilde{x}_{e}$ corresponding to $e$. The valuation of $\overline{f^{\Gamma}}$ at this point is equal to the slope $\delta_{e}(\phi_{f})$ by the Poincar\'{e}-Lelong formula, see Theorem \ref{ValCor1}. If we consider the Newton polygon of $z^{n}-f^{\Gamma}$ at this point (that is, the valuation corresponding to $\tilde{x}_{e}$), then it consists of a single line segment with slope $-\delta_{e}(\phi_{f})/n$. Clearing the denominator and the numerator, we obtain $n/\text{gcd}(n,\delta_{e}(\phi_{f}))$ in the denominator. This denominator is the ramification index, which is then equal to the order $|I_{e'}|$ by Proposition \ref{ramind2}. This finishes the proof.

For the second proof, we consider a regular subdivision $\mathcal{D}_{0}$ of $e$. That is, the pre-image of $e$ consists of $l(e)-1$ projective lines. Consider the new component $\Gamma_{1}$ in $\mathcal{D}_{0,s}$ that intersects $\Gamma$. The valuation of $f^{\Gamma}$ at this component $\Gamma_{1}$ is then given by Theorem \ref{MainThmVert}: it is the slope $\delta_{e}(\phi_{f})$. The extension of discrete valuation rings $\mathcal{O}_{\mathcal{D}_{0},y_{1}}\rightarrow{\mathcal{O}_{\mathcal{D}_{0},y_{1}}[z]/(z^{n}-f^{\Gamma})}$ is then ramified of order $n/\text{gcd}(\delta_{e}(\phi_{f}))$ by a Newton polygon computation as in the first proof. Here $y_{1}$ is the generic point of $\Gamma_{1}$ in $\mathcal{D}_{0}$.  By Theorem \ref{InertProp2}, we find that this ramification index is equal to the order $|I_{e'}|$, as desired. 

\end{proof}
\begin{rem}
We note that the first proof of Proposition \ref{PropositionCoveringData} is almost the same as that of \cite[Proposition 4.1]{supertrop}. The main difference is that the coverings in \cite{supertrop} are all superelliptic, i.e. of the form $C\rightarrow{\mathbb{P}^{1}}$, whereas the coverings in this section are general cyclic abelian coverings $C\rightarrow{D}$. We will give plenty of examples of these general coverings soon. 

\end{rem}

We now give the decomposition group of a vertex for a cyclic abelian cover $C\rightarrow{D}$ with disjointly branched morphism $\mathcal{C}\rightarrow{\mathcal{D}}$. Let $v$ be a vertex in $\Sigma(\mathcal{D})$ and let $v'$ be any vertex in $\Sigma(\mathcal{C})$ lying above it. We will denote their corresponding components by $\Gamma_{v}$ and $\Gamma_{v'}$. There are two cases to consider: the case where $g(\Gamma_{v})=0$ and the case where $g(\Gamma_{v'})>0$. The first case is dealt with by Theorem \ref{DecompVert}, which we repeat for the reader's convenience.

\begin{reptheorem}{DecompVert}
Let $\mathcal{C}\rightarrow{\mathcal{D}}$ be a disjointly branched morphism with Galois group $G$ and let $v\in\Sigma(\mathcal{D})$ be a vertex with $g(\Gamma_{v})=0$. Let $v'$ be any vertex lying above $v$. Then
\begin{equation}
D_{v'}=\prod_{P\in\Gamma_{v'}}{I_{P}}.
\end{equation}
\end{reptheorem}

For $g(\Gamma_{v})>0$, the situation is different since we can obtain nontrivial unramified abelian coverings of $\Gamma_{v}$. We will study these in Chapter \ref{Twistingdata}. For now, we state the following

\begin{pro}\label{DecompositionVertexProposition}
Let $\mathcal{C}\rightarrow{\mathcal{D}}$ be a disjointly branched morphism with Galois group $\mathbb{Z}/n\mathbb{Z}$ and let $v\in\Sigma(\mathcal{D})$ be a vertex with $g(\Gamma_{v})>0$. Let the extension of function fields be given by $z^{n}=f$. Then $|D_{v'}|=r$, where $r$ is the smallest divisor of $n$ such that $\overline{f^{\Gamma}}=h^{n/r}$ for some $h\in{k(\Gamma)}$.  
\end{pro}
\begin{proof}
Let $y$ be the generic point of $\Gamma_{v}$. The extension $\mathcal{O}_{\mathcal{D},y}\rightarrow{\mathcal{O}_{\mathcal{D},y}[z]/(z^{n}-f^{\Gamma})}$ is \'etale, so the splitting behaviour is determined by the factorization of $f^\Gamma$ in the residue field $k(\Gamma_{v})$. Let $r$ be the smallest divisor of $n$ such that $\overline{f^{\Gamma}}=h^{n/r}$. We can now write 
\begin{equation}
z^{n}-\overline{f^{\Gamma}}=\prod_{i=0}^{n/r-1}(z^{r}-\zeta_{n/r}^{i}h),
\end{equation}
where $\zeta_{n/r}$ is a primitive $n/r$-th root of unity. This now gives us the $n/r$ solutions $z^{r}=\zeta_{n/r}^{i}h$. Note that these do not factorize further by assumption on $r$. The order of the decomposition group is then $\dfrac{|G|}{n/r}=\dfrac{n}{(n/r)}=r$, as desired.
\end{proof}

\begin{exa}\label{Exa3Tors1}
Suppose that we are given an elliptic curve $E$ over $K$ with multiplicative reduction, with a $3$-torsion point $P$ reducing to the singular point. An explicit family of these curves can be found in \cite{Paul1}. 
This point $P$ then gives a point of order three in the component group of the N\'{e}ron model of $E$. By definition, there exists a function $f$ such that
\begin{equation*}
(f)=3(P)-3(\infty).
\end{equation*}
\begin{figure}[h!]
\centering
\includegraphics[scale=0.6]{{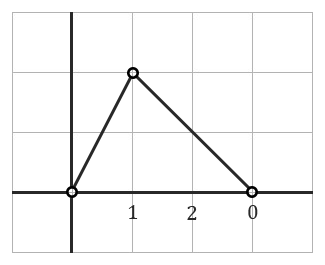}}
\caption{\label{333eplaatje} {\it{The Laplacian of $f$.}}} 
\end{figure}
Subdividing the reduction graph $\Sigma(E)$ into three equidistant parts, we see that $P$ must reduce to one third of the length of $\Sigma(E)$, whose component we denote by $\Gamma_{1}$. We thus have the Laplacian
\begin{equation*}
\rho(\text{div}_{\eta}(f))=3(\Gamma_{1})-3(\Gamma_{0}).
\end{equation*}
We take the solution $\phi$ with
\begin{eqnarray*}
\phi(0)=0,\\
\phi(1)=2,\\
\phi(2)=1,\\
\end{eqnarray*}
which has slope $2$ between $\Gamma_{0}$ and $\Gamma_{1}$ on the left side and slope $1$ between $\Gamma_{1}$ and $\Gamma_{0}$ on the right side, as in Figure \ref{333eplaatje}.

 If we consider the extension
\begin{equation*}
z^3=f,
\end{equation*}
then this gives a morphism $E'\longrightarrow{E}$, which ramifies twice at every component (namely at the intersection points), since the slope is not divisible by $3$. The reduction graph is thus the same and $E'$ is an elliptic curve with multiplicative reduction. This was to be expected from an isogeny of two elliptic curves where one has bad reduction, see \cite[Chapter VII, Corollary 7.2]{Silv1}. The covering of graphs can be found in Figure \ref{23eplaatje}. We note that the lengths of the edges in $\Sigma(E')$ are multiplied by three, i.e. $3\cdot{}l(e')=l(e)$ by 
Proposition \ref{InertiagroupIntersectionPoint1}.
\begin{figure}[h!]
\centering
\includegraphics[scale=0.7]{{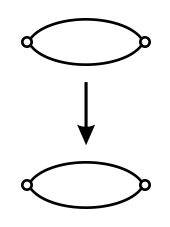}}
\caption{\label{23eplaatje} {\it{The covering of graphs in Example \ref{Exa3Tors1}.}}} 
\end{figure}
\end{exa}

\begin{exa}\label{Exa3torsgen2}
Suppose we take Example \ref{ExampleChapter1}.\ref{ExaS2} again, with the banana graph of genus 2, as in Figure \ref{52eplaatje}. The corresponding equation is
\begin{equation*}
y^2=x(x-\pi)(x+1)(x+1-\pi)(x+2)(x+2-\pi).
\end{equation*}
\begin{figure}[h!]
\centering
\includegraphics[scale=0.8]{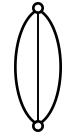}
\caption{\label{52eplaatje} {\it{The intersection graph of the genus 2 curve in Example \ref{Exa3torsgen2}.}}} 
\end{figure}We label the components by $\Gamma$ and $\Gamma'$. There is a natural $3$-torsion point $D'$ on this graph, namely
\begin{equation*}
D'=(\Gamma)-(\Gamma').
\end{equation*}
Suppose we have a divisor $D$ of order $3$ in the Jacobian of $C$ such that $\rho(D)=D'$. For some function $f$, we have that
\begin{equation*}
3D=\text{div}_{\eta}(f).
\end{equation*}
Then $\rho(\text{div}_{\eta}(f))=3D'$ and the corresponding Laplacian function up to scaling is just the indicator function of $\Gamma$, as in Figure \ref{4eplaatje}.
\begin{figure}[h!]
\centering
\includegraphics[scale=0.7]{{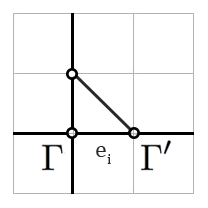}}
\caption{\label{4eplaatje} {\it{The Laplacian function $\phi$ of $f$, as in Example \ref{Exa3torsgen2}. The $e_{i}$ denote the three edges between the two vertices.}}} 
\end{figure}
We thus see that the morphism on the components is ramified at every vertex, with the vertices $\Gamma_{0}$ and $\Gamma_{1}$ having three ramification points and the ones elsewhere having only $2$. Using the Riemann-Hurwitz formula, we see that the primes dividing $\Gamma_{0}$ and $\Gamma_{1}$ have genus 1. 
The graph thus consists of two vertices with weights $1$ and three edges connecting them. This gives a genus 4 graph, as expected. The covering of graphs can be found in Figure \ref{25eplaatje}. Note that the lengths of the edges are again multiplied by three by Proposition \ref{InertiagroupIntersectionPoint1}.
\begin{figure}[h!]
\centering
\includegraphics[scale=0.65]{{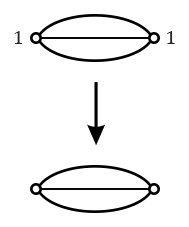}}
\caption{\label{25eplaatje} {\it{The covering of graphs in Example \ref{Exa3torsgen2}.}}} 
\end{figure}
\end{exa}

\chapter{Unramified abelian coverings and twisting data}\label{Twistingdata}

In this chapter, our main goal is to give an algorithm that reconstructs the Berkovich skeleton of a curve $C$ that admits an abelian covering $C\rightarrow{D}$ to a curve $D$ whose Berkovich skeleton we already know. In general, the covering data given in Chapter \ref{Coveringdata} are not sufficient to reconstruct the Berkovich skeleton of $C$ and we will quickly see why. The main reason for this is that the covering data only {\it{locally}} fix the covering graph and thus do not take into account any global twisting. This phenomenon also appears in the algebraic case in the form of {\it{unramified coverings}}. We will first study the algebraic case, where the Jacobian of the curve classifies all unramified abelian coverings. The decomposition of the torsion subgroups of the Jacobian of a curve induced by its N\'{e}ron model then gives rise to different coverings of graphs. 
We will be particularly interested in the coverings coming from the toric part of the identity component, which give rise to {\it{completely decomposable}} covering graphs (which are just covering spaces in terms of algebraic topology). 

After this, we will define the {\it{twisting data}} of a general abelian covering $C\rightarrow{D}$. It will be presented as a $2$-cocycle on the intersection graph $\Sigma(\mathcal{D})$ for a disjointly branched morphism $\mathcal{C}\rightarrow{\mathcal{D}}$. This $2$-cocycle, together with the covering data then give us an algorithm for finding the Berkovich skeleton of $C$. 




\section{Unramified abelian coverings and Jacobians}\label{Unramified}

In this section, we will recall some of the algebraic notions of unramified coverings of a curve. We will take a Galois theory point of view, which can be found in \cite[Chapter 6]{Lenstra} or in \cite[Page 113]{Milne1}. We will give the correspondence between finite abelian unramified coverings of a curve $D$ and torsion subgroups in the Jacobian $J(D)$. This correspondence is quite explicit and we will use it in various examples in this chapter.  

Let $D$ be a smooth, geometrically irreducible projective curve over $K$. Taking the base change to $\overline{K}$, we obtain the curve $D_{\overline{K}}$ with function field $\overline{K}(D)$. We will set $K=\overline{K}$ for this section. The finite, smooth coverings $C\rightarrow{D}$ are then classified by the Galois group of the algebraic closure $({K}(D))^{ac}\supset{{K}(D)}$. 
That is, there is a bijection between subgroups of finite index in $\text{Gal}(({K}(D))^{ac}/{K}(D))$ and finite smooth coverings $C\rightarrow{D}$. We now consider the field extensions ${K}(D)\rightarrow{{K}(C)}$ where $C\rightarrow{D}$ is {\it{unramified}} and we take the composite of these fields: $({K}(D))^{un}$. This field is then easily seen to be Galois over ${K}(D)$ 
and we can thus define the following: 
\begin{mydef}
The {\it{geometric fundamental group}} $\pi_{1}(D)$ of $D$ is the Galois group of the field extension $({K}(D))^{un}\supset{{K}(D)}$.
\end{mydef}
\begin{rem}
For a curve $D$ over $\mathbb{C}$, The group $\pi_{1}(D_{\mathbb{C}})$ is not the usual fundamental group defined in the analytic category, but rather the profinite completion of the analytic fundamental group, see \cite[Theorem 6.3.1]{serre2008topics}. 
\end{rem}
 We are now particularly interested in the abelian unramified coverings. That is, we consider the subfield $({K}(D))^{un,ab}$ of ${K}(D)^{un}$, which is the composite of all field extensions $K(C)\supset{{K}(D)}$ such that $C\rightarrow{D}$ is finite, Galois with abelian Galois group. This subfield is again Galois, since the composite of two abelian field extensions is abelian. The Galois group of this field extension is then the {\it{abelianization}} $\pi_{1,ab}(D)$ of $\pi_{1}(D)$.\footnote{Recall that the abelianization of a group $G$ is the quotient $G/[G,G]$, where $[G,G]$ is the group generated by all elements of the form $[g,h]:=g^{-1}h^{-1}gh$.}

We have the following theorem regarding the unramified abelian coverings of the curve $D$.

\begin{theorem}\label{UnramifiedAbelianCoverings}
Let $D$ be a smooth irreducible curve over an algebraically closed field $K$ as above. Then
\begin{equation}
\text{Hom}_{\text{cont}}(\pi_{1}(D),\mathbb{Z}/n\mathbb{Z})=\text{Hom}_{cont}(\pi_{1,ab}(D),\mathbb{Z}/n\mathbb{Z})\simeq{J(D)[n]}.
\end{equation}
Here the homomorphisms are continuous with respect to the profinite (Krull) topology on $\pi_{1}(D)$ and the discrete topology on $\mathbb{Z}/n\mathbb{Z}$. 
\end{theorem}

\begin{proof}

We give a sketch of the proof. 
For every torsion point $D$ of order $n$ on the Jacobian, we obtain by definition the relation 
\begin{equation}
nD=\text{div}(f)
\end{equation}
for some $f\in{K}(D)$. We then take the normalization $C$ of $D$ in ${K}(D)[z]/(z^n-f)$. The induced morphism $C\rightarrow{D}$ is then a finite unramified abelian covering (since the valuation of $f$ at every point is divisible by $n$, see Corollary \ref{AbelExt2}). Conversely, given an unramified abelian cover $C\rightarrow{D}$, we obtain from Kummer theory (see Proposition \ref{AbelExt1}) an isomorphism ${K}(C)=K(D)[z]/(z^n-f)$ for some $f$ in $K(D)$. The divisor of $f$ is then divisible by $n$ for every point in its support, giving an $n$-torsion divisor. 
\end{proof}


\section{\'{E}tale abelian coverings of graphs}

In this section, we will define the corresponding notions of {\it{unramified coverings}} for metrized complexes of $k$-curves, as defined in Section \ref{Metrizedcomplex}. We note that an unramified covering might be ramified on the intersection graphs, see Example \ref{Exa3Tors1}. 

Suppose we have a disjointly branched covering $\mathcal{C}\longrightarrow{\mathcal{D}}$ with Galois group $\mathbb{Z}/n\mathbb{Z}$. Suppose that for every edge $e'$ of $\Sigma(\mathcal{C})$ dividing an edge $e$ in $\Sigma(\mathcal{D})$, we have that $D_{e'}=(1)$.   
 In other words, we know for every edge on the intersection graph $\Sigma(\mathcal{D})$ that there are exactly $n$ edges lying above them. If in addition the morphisms on components
\begin{equation*}
\Gamma_{v}\longrightarrow{\Gamma_{w}}
\end{equation*}
are all unramified, we will say that the induced morphism of graphs
\begin{equation*}
\Sigma({\mathcal{C}})\longrightarrow{\Sigma({\mathcal{D}})}
\end{equation*}
is "\'{e}tale". \begin{mydef}
Suppose we have a disjointly branched, abelian Galois morphism $\phi:\mathcal{C}\longrightarrow{\mathcal{D}}$ of degree $n$ such that for every edge $x$ in $\Sigma(\mathcal{D})$ there exist exactly $n$ edges dividing $x$. Suppose in addition that the morphisms
\begin{equation}
\Gamma_{v}\longrightarrow{\Gamma_{w}}
\end{equation}
on components are all unramified. Then this morphism $\phi$ with the corresponding morphism $\phi_{\Sigma}$ 
is then referred to as an {\bf{\'{e}tale morphism of graphs}}. 
\end{mydef}
Let us now see why 
this terminology of "\'{e}tale" abelian morphisms of graphs is appropriate.
\begin{lemma}
Let $\phi:\mathcal{C}\longrightarrow{\mathcal{D}}$ be disjointly branched of degree $n$. Then $\phi_{\Sigma}$ is an \'{e}tale morphism of graphs if and only if $\phi$ is \'{e}tale on the points corresponding to elements of the intersection graph $\Sigma(\mathcal{C})$ and the morphisms $\Gamma_{v}\longrightarrow{\Gamma_{w}}$ are \'{e}tale. 
\end{lemma}
\begin{proof}
First note that $\phi$ is always \'{e}tale at primes corresponding to components. Furthermore, we see that $\phi$ is \'{e}tale at an intersection point if and only if there are $n$ pre-images. These two conditions quickly give the lemma.
\end{proof} 
\begin{lemma}
Let $\phi:\mathcal{C}\longrightarrow{\mathcal{D}}$ be a disjointly branched morphism of degree $n$. Suppose that $\phi_{\Sigma}$ is \'{e}tale. 
Then $\phi_{\eta}$ is unramified.
\end{lemma}
\begin{proof}

Suppose that $\phi_{\eta}$ is ramified. Then there exists a branch point $Q\in{D}$. Let $P\in{C}$ be any point lying above $Q$. By Proposition \ref{ramind2}, we see that the morphism of components corresponding to $Q$ and $P$ is ramified, a contradiction.  
\end{proof}
\begin{rem}
Note that the converse is definitely not true, since we can have an unramified morphism $\phi:{C}\longrightarrow{{D}}$ with edges having $D_{e}=\mathbb{Z}/n\mathbb{Z}$. This happens for instance if we take an elliptic curve $E$ with multiplicative reduction with a $n$-torsion point that reduces to the singular point. The corresponding extension is unramified and yields the same reduction type as $E$. See Example \ref{Exa3Tors1} 
for instance. 
\end{rem}

 Let us now found out what unramified abelian extensions correspond to \'{e}tale morphisms of graphs.
 \begin{pro}\label{UnrGraph1}
Let $P\in{J(D)[q]}$ be a $n$-torsion point giving rise to an unramfied abelian morphism of degree $n$:
\begin{equation*}
\phi: C\longrightarrow{D}.
\end{equation*}
Then the induced morphism 
\begin{equation*}
\phi_{\Sigma}:\Sigma(\mathcal{C})\longrightarrow{\Sigma(\mathcal{D})}
\end{equation*}
is an \'{e}tale morphism of graphs if and only if $P\in\mathcal{J}^{0}(R)[n]$.
\end{pro}
\begin{proof}
If $P\in\mathcal{J}^{0}(R)[n]$, then we see that the Laplacian corresponding to $nP$ is zero everywhere. For every edge in $\Sigma(\mathcal{D})$, we then have $n$ primes lying above it by Proposition \ref{PropositionCoveringData}, so we have an unramified morphism of graphs.

Conversely, suppose that we have an unramified morphism of graphs. Then for every edge we have that the Laplacian has slope divisible by $n$. We can then quite easily find a new function $\phi'$ such that $n\Delta(\phi{'})=\Delta(\phi)$. But then $\rho(P)=\Delta(\phi{'})$ and so the class of $P$ is in the identity component of $\mathcal{J}(D)$, as desired.  
\end{proof}

\begin{exa}\label{RedGraphGen2Ab1}
Suppose we take a genus 2 curve $D$ with reduction graph consisting of two vertices and two edges. We label the 2 corresponding components by $\Gamma_{0}$ and $\Gamma_{1}$. One of them must have genus one, so let that component be $\Gamma_{0}$. 
\begin{figure}[h!]
\centering
\includegraphics[scale=0.7]{{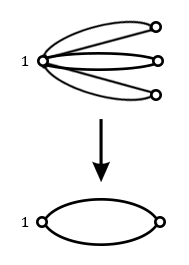}}
\caption{\label{555eplaatje} {\it{The covering in Example \ref{RedGraphGen2Ab1}.}}}
\end{figure}
 We now take a 3-torsion point in the Jacobian of $D$ that reduces entirely to $\Gamma_{0}$. That is, we take a 3-torsion point of the corresponding genus 1 curve. If we consider the extension defined by that 3-torsion point, we obtain the intersection graph consisting of 4 vertices, 3 lying above $\Gamma_{1}$ and $1$ above $\Gamma_{0}$ with 2 edges between each component $\Gamma_{1}'$ and $\Gamma_{0}'$ (so 6 in total). The resulting covering of intersection graphs is in Figure \ref{555eplaatje}. Note that the component $\Gamma_{0}'$ again has genus 1, since it is given as an unramified covering of a genus 1 curve. The Betti number of the graph is 3 and the total genus is $3+1=4$, as expected.\\
 Note that the covering of graphs in this case is {\it{unramified}}: for every edge there are exactly three pre-images. The Galois action then permutes these edges accordingly.
\end{exa}

\section{Completely decomposable coverings}\label{CompletelyDecomposable}

We will now continue our study of {\it{\'{e}tale morphisms of graphs}}. As we saw in the last section in Proposition \ref{UnrGraph1}, they arise from $n$-torsion points in the identity component $\mathcal{J}^{0}$ of the Jacobian $J(D)$. For these morphisms we know the decomposition groups of the edges: $D_{e}=(1)$ for every edge $e$ in $\Sigma(\mathcal{C})$. This does not fix the order of the decomposition groups of the vertices however. 
In this section we will study the completely reducible case, where $D_{v}=(1)$ for every vertex. 
\begin{mydef}
Suppose that we are given an unramified Galois cover of metrized complexes 
$\phi_{\Sigma}:\Sigma_{1}\longrightarrow{\Sigma_{2}}$ with Galois group $\mathbb{Z}/n\mathbb{Z}$. Suppose that $D_{v}=(1)$ for every vertex. Then $\phi_{\Sigma}$ is referred to as a {\bf{completely decomposable morphism of metrized complexes}}. 
\end{mydef}


\begin{exa}\label{Trampoline1}
Suppose we take an elliptic curve $E$ with multiplicative reduction and 
reduction graph consisting of two vertices with two edges between them. We now take a 2-torsion point $P$ reducing to a {\it{nonsingular}} point, in terms of reductions defined in \cite[Chapter VII]{Silv1}. 
The corresponding degree two morphism $E'\rightarrow{E}$ is completely reducible everywhere (we will in fact write down the equations explicitly soon, where it will be clear why this is true). We thus obtain four vertices with four edges between them. The graph has to be connected, so there is only option. Note that for any component $\Gamma$ of $\mathcal{E}$, the two primes lying above $\Gamma$ do not intersect. 
Since $\Sigma(\mathcal{E}')$ has Betti number 1, we find that $E'$ has multiplicative reduction. The covering of graphs can be found in Figure \ref{66eplaatje}.
\begin{figure}[h!]
\centering
\includegraphics[scale=0.6]{{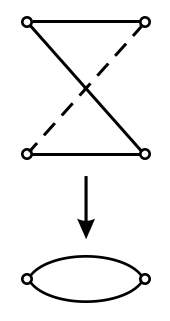}}
\caption{\label{66eplaatje} {\it{The covering in Example \ref{Trampoline1}.}}}
\end{figure}

\end{exa}

\begin{pro}\label{CritEtalCover1}
Let $P\in{J(D)[n]}$ be an $n$-torsion point giving rise to an unramified abelian morphism of degree $n$: 
\begin{equation*}
\phi: C\longrightarrow{D}.
\end{equation*}
Then $\phi_{\Sigma}:\Sigma(\mathcal{C})\longrightarrow\Sigma(\mathcal{D})$ is a completely decomposable morphism of graphs if and only if
\begin{equation*}
P\in{\mathcal{J}^{0}_{T}[n]}.
\end{equation*}
\end{pro}
\begin{proof}
Suppose that 
\begin{equation*}
\Sigma(\mathcal{C})\longrightarrow{\Sigma(\mathcal{D})}
\end{equation*}
is completely decomposable. Then for every vertex $v$ in $\Sigma(\mathcal{D})$, we have that $g_{v}=n$. 
Using Proposition \ref{DecompositionVertexProposition}, we then see that the reduced divisor $(f)|_{v}$ is trivial in every
\begin{equation*}
\text{Pic}^{0}(\Gamma_{i}),
\end{equation*} so that $P$ is in fact a toric divisor by Theorem \ref{ToricPic}.

 Now suppose that $P$ is a toric divisor. We then see that the pulled back divisors are trivial for every component. Using Proposition \ref{DecompositionVertexProposition} 
 again, we see that the corresponding extensions have $g_{v}=n$, as desired. 
\end{proof}
\subsection{Explicit computations for completely decomposable morphisms}\label{IdentCompTors1}
We saw in the previous section that one can distinguish between torsion points in 
$\mathcal{J}$, $\mathcal{J}^{0}$ and $\mathcal{J}^{0}_{T}$ by considering their corresponding coverings of graphs. We would now like to distinguish between the different toric divisors and their corresponding extensions. To do this, we will take a strong hint from graph cohomology, as introduced in Section \ref{TorExtJac1} and apply it to our setting. This will use a result by Bosch and L\"{u}tkebohmert on the reduction of the divisors corresponding to these torsion points. 

So suppose we have a {\bf{toric}} $n$-torsion point $[P]$ in $\mathcal{J}^{0}(K)={\text{Div}^{(0)}(D)/\text{Prin}^{(0)}(D)}$ with representative $P\in{\text{Div}^{(0)}(D)}$.  If we restrict $P$ to a component $\Gamma$, then we can write
\begin{equation*}
P|_{\Gamma}=\text{div}(\overline{h_{\Gamma}})
\end{equation*}
for some $\overline{h_{\Gamma}}$ in the function field $k(\Gamma)$ of the component $(\Gamma)$, because $P$ is trivial. A local lift of $\overline{h_{\Gamma}}$ to $\mathcal{D}$ will be denoted by $h_{\Gamma}$. 
Since $P$ is an $n$-torsion point, we have that  
\begin{equation*}
n\cdot{P}=\text{div}_{\eta}(f)
\end{equation*}
for some $f\in\text{Prin}^{(0)}(D)$. If we restrict this equality, then locally we have that
\begin{equation*}
(\overline{f})_{\Gamma}=(\overline{h}^{n}_{\Gamma}),
\end{equation*}
meaning that $f$ is locally a $q$-th power. We {\it{assume}} here that we have scaled $h_{\Gamma}$ such that
\begin{equation*}
\overline{f}(x)=(\overline{h}_{\Gamma}(x))^{n}.
\end{equation*}
 There are exactly $n$ functions that satisfy $\overline{f}_{\Gamma}=\overline{h}^{n}_{\Gamma}$, so we will introduce notation for them. Let $\zeta:=\zeta_{n}$, a primitive $n$-th root of unity. We define
\begin{equation*}
h_{\Gamma,i}:=\zeta^{i}\cdot{h_{\Gamma}}.
\end{equation*}
for every component $\Gamma$ in $\mathcal{D}_{s}$. As we will see, these functions correspond to the components that lie above $\Gamma$.

Now let $x$ be an intersection point of two components $\Gamma$ and $\Gamma'$ in $\mathcal{D}_{s}$. We then have:
\begin{lemma}\label{BoschLemma}
\begin{equation}
\overline{f(x)}\in{k^{*}}.
\end{equation}
\end{lemma}
\begin{proof}
See \cite[Page 274]{Bosch1984}. 
\end{proof}
We now define for any $x$ an intersection point the following set:
\begin{equation*}
\mathcal{S}_{x}=\{\alpha\in{k^{*}}:\alpha^{n}=\overline{f(x)}\}.
\end{equation*}
We will see that these elements correspond exactly to the $n$ intersection points lying above $x$.

Now let us review the situation we are in. Let us consider $f$ as an element of the function field of $K(\mathcal{D})$. For every $x\in\mathcal{D}$ we have a natural injection
\begin{equation*}
\mathcal{O}_{\mathcal{D},x}\longrightarrow{K(\mathcal{D})}.
\end{equation*} We can then consider the following set:
\begin{equation*}
\mathcal{D}_{f}=\{x\in{\mathcal{D}}:f_{x}\in\mathcal{O}_{\mathcal{D},x}\},
\end{equation*}
where $\mathcal{O}_{\mathcal{D},x}$ is identified with its image in ${K(\mathcal{D})}$. 
\begin{lemma}
$\mathcal{D}_{f}$ is open.
\end{lemma}
\begin{proof}
Locally for every point $x\in\mathcal{D}_{f}$, we can write
\begin{equation*}
f|_{U}=g/h
\end{equation*}
for some open affine $U$. Here $h$ is not contained in the prime corresponding to $x$, by assumption on $f$. Let us consider the open subset $D(h)$ in $U$. Then for every $y\in{D(h)}$, we see that $f$ is contained in $\mathcal{O}_{\mathcal{D},y}$, as desired.
\end{proof}

By Lemma \ref{BoschLemma}, we see that any intersection point in the special fiber of $\mathcal{D}$ lies in $\mathcal{D}_{f}$ ($f$ is in fact invertible in $\mathcal{O}_{\mathcal{D},x}$ for an intersection point $x$). We then also see that any generic point $y$ lying under $x$ must also be an element of $\mathcal{D}_{f}$.\\
Since $\mathcal{D}_{f}$ is open, for every $x\in\mathcal{D}_{f}$ we can find an open affine $U=\text{Spec}(A)$ such that for every $\mathfrak{p}\in\text{Spec}(A)$ we have that $f\in{A}_{\mathfrak{p}}$. This then also means that $f\in{A}$.
The ring
\begin{equation*}
B=A[z]/(z^n-f)
\end{equation*}
is thus integral over $A$ and is thus contained in $A'$.
\begin{lemma}
The algebra $C=A[z][1/z]/(z^n-f)$ is standard \'{e}tale over $A$. 
\end{lemma}
\begin{proof}
We only have to check that the derivative of $z^n-f$ with respect to $z$, $nz^{n-1}$, is invertible in $C$. 
Since $n$ is invertible by assumption on our rings and $z$ is invertible by the localization we applied, we see that $B$ is standard \'{e}tale.
\end{proof}

Let us recall that for a morphism of schemes $f:X\longrightarrow{Y}$ of finite type with $Y$ locally Noetherian, we have the following equivalent statements:
\begin{enumerate}
\item $f$ is \'{e}tale at $x\in{X}$.
\item $\hat{\mathcal{O}}_{X,x}$ is a free $\hat{\mathcal{O}}_{Y,y}$-module and $\hat{\mathcal{O}}_{X,x}/\mathfrak{m}_{y}\hat{\mathcal{O}}_{X,x}$ is a finite separable field extension of $k(y)$. Here $y=f(x)$.
\end{enumerate}

We furthermore have that if the induced map $k(y)\longrightarrow{k(x)}$ is an isomorphism, then we have an isomorphism $\hat{\mathcal{O}}_{X,x}=\hat{\mathcal{O}}_{Y,y}$. All of this is contained in \cite[Proposition 17.6.3]{EGA4}. 
Applying this to our situation, we have the following
\begin{lemma}\label{ClosedPointEtale1}
Let $\mathfrak{m}$ be a maximal ideal of $A$ in $D(f)$. Then the induced morphism
\begin{equation*}
\hat{A}_{\mathfrak{m}}\longrightarrow{\hat{B}_{\mathfrak{m}'}}
\end{equation*}
is an isomorphism. Here $\mathfrak{m}'$ is a maximal ideal of the algebra $B=A[z]/(z^n-f)$ lying above $\mathfrak{m}$.
\end{lemma}
\begin{proof}
We note that since $f\notin{\mathfrak{m}'}$ for any $\mathfrak{m}'$ lying above it, we have that $z\notin{\mathfrak{m}'}$ and that we thus have that the corresponding morphism of rings factors through the above standard \'{e}tale algebra $A[z][1/z]/(z^n-f)$. We thus see that $f$ is \'{e}tale at $\mathfrak{m}'$. Furthermore, the residue fields of all these points are assumed to be algebraically closed, so we obtain an isomorphism of completions by the above considerations.
\end{proof}

We will mainly use this lemma for the intersection points in the intersection graph. Let us now explicitly compute some prime ideals. Let $\mathfrak{p}\in\text{Spec}(A)$ correspond to a component in the special fiber and let $\mathfrak{m}\in\text{Spec}(A)$ be a closed point in that component (i.e. $\mathfrak{m}\supset{\mathfrak{p}}$).  

\begin{lemma}\label{Easylemma2}
The primes above $\mathfrak{p}$ are given by
\begin{equation*}
\mathfrak{p}_{i}=\mathfrak{p}+<z-h_{\Gamma,i}>.
\end{equation*}
The maximal ideals above $\mathfrak{m}$ are given by
\begin{equation*}
\mathfrak{m}'=\mathfrak{m}+<z-\alpha>,
\end{equation*}
where $\alpha\in\mathcal{S}_{x}=\{\alpha\in{k^{*}}:\alpha^{n}=\overline{f(x)}\}$.
\end{lemma}
\begin{proof}
The primes above $\mathfrak{p}$ correspond to primes of the ring
\begin{equation*}
k(\mathfrak{p})[z]/(z^n-f),
\end{equation*}
where $k(\mathfrak{p})$ is the residue field of $\mathfrak{p}$. 

Writing out this correspondence for 
$\mathfrak{p}$ yields the desired form as stated in the lemma. One similarly proceeds for $\mathfrak{m}$. 
\end{proof}

\begin{lemma}\label{LemmaDoublePoint2}
Let $\mathfrak{m}$ be an intersection point in $\mathcal{D}$ such that $\hat{\mathcal{O}}_{\mathcal{D},x}\simeq{R[[x,y]]/(xy-\pi^{m})}$. 
The completion of ${{(A[z]/(z^{n}-f))_{\mathfrak{m}'}}}$ with respect to $\mathfrak{m}'$ is then also isomorphic to $R[[x,y]]/(xy-\pi^{m})$. 
\end{lemma}
\begin{proof}
This follows from Lemma \ref{ClosedPointEtale1}.
\end{proof}

\begin{cor}
${{(A[z]/(z^{n}-f))_{\mathfrak{m}'}}}$ is normal.
\end{cor}
\begin{proof}
First of all, $A_{\mathfrak{m}}$ is an {\it{excellent}} ring. Furthermore, any finitely generated algebra over an excellent ring is again excellent and any localization of an excellent algebra is also excellent. We thus see that ${{(A[z]/(z^{n}-f))_{\mathfrak{m}'}}}$ is excellent. We then use the following:
\begin{lemma}
[{\it{Normality of excellent rings}}] Let $A$ be an excellent Noetherian local ring. Let $\hat{A}$ be its formal completion. Then $A$ is normal if and only if $\hat{A}$ is normal.
\end{lemma}
\begin{proof}
See \cite[Page 344, Proposition 2.41]{liu2}. 
\end{proof}
We thus see that ${{(A[z]/(z^{n}-f))_{\mathfrak{m}'}}}$ is normal if and only if its completion is normal. Its completion is isomorphic to $R[[x,y]]/(xy-\pi^{m})$ for some $m$ by Lemma \ref{LemmaDoublePoint2}, which is a normal ring. This gives the Corollary.
\end{proof}
We have thus identified the local rings of intersection points lying above an edge with maximal ideal $\mathfrak{m}$. We now link these maximal ideals to the components.
\begin{lemma}\label{UniqComponentIntersect2}
Let $\mathfrak{m}'$ be an intersection point lying above $\mathfrak{m}$. 
\begin{enumerate}
\item $\mathfrak{m}'$ uniquely corresponds to a solution $\alpha\in{k}^{*}$ of the equation $\alpha^{n}=\overline{f}(x)$.
\item There exists a unique component $\Gamma_{i}$ lying above $\Gamma$ such that $\mathfrak{m}'$ belongs to $\Gamma_{i}$.
\item For this component $\Gamma_{i}$, we have 
\begin{equation*}
\overline{h_{\Gamma,i}}(x)=\alpha.
\end{equation*}
\end{enumerate}
\end{lemma}
\begin{proof}
Part 1 is just Lemma \ref{Easylemma2}. 
We have that
\begin{equation*}
\overline{h_{\Gamma}}(x)^n=\overline{f}(x)=\alpha^{n},
\end{equation*}
so there exists a unique $i$ such that
\begin{equation*}
\overline{h_{\Gamma,i}}(x)=\alpha.
\end{equation*}
We then easily see that $\mathfrak{m}'$ belongs to $\Gamma_{i}$ and we automatically obtain part (3) of the Lemma. 
\end{proof}

For every intersection point $\mathfrak{m}'$ lying above $\mathfrak{m}$, we can now find a unique $i$ and a unique $j$ such that $\mathfrak{m}'$ lies in both $\Gamma_{i}$ and $\Gamma'_{j}$. We will therefore denote the maximal ideal $\mathfrak{m}'$ lying above $\mathfrak{m}$ by
\begin{equation*}
\mathfrak{m}_{i,j}:=\mathfrak{m}'.
\end{equation*}

\begin{cor}
$\Gamma_{i}$ and $\Gamma'_{j}$ intersect each other in $\mathfrak{m}_{i,j}$ with $\Gamma_{i}\cdot{\Gamma_{j}}=1$.
\end{cor}
\begin{proof}
This follows from Lemmas \ref{LemmaDoublePoint2} and \ref{UniqComponentIntersect2}. 
\end{proof}

Let us now fix a single $i$ and $j$ with components $\Gamma_{i}$ and $\Gamma'_{j}$. We would now like to give a criterion for their intersection points. 

\begin{lemma}\label{Easylemma3}
$\Gamma_{i}$ and $\Gamma'_{j}$ intersect each other if and only if there exists an intersection point $x$ of $\Gamma$ and $\Gamma'$ such that 
\begin{equation*}
\overline{h_{\Gamma,i}}(x)=\overline{h_{\Gamma',j}}(x).
\end{equation*}
\end{lemma}
\begin{proof}
The only intersections between $\Gamma_{i}$ and $\Gamma_{j}$ are those that are in the pre-image of an edge of $\Sigma(\mathcal{D})$. We thus see that all intersections must arise from maximal ideals of the algebras
\begin{equation*}
A_{\mathfrak{m}}[x]/(z^n-f),
\end{equation*} 
where $\mathfrak{m}$ corresponds to an intersection point $x$ of $\mathcal{D}_{s}$. By Lemma \ref{UniqComponentIntersect2}, we now see that these maximal ideals are exactly given by equations of the form
\begin{equation*}
\overline{h_{\Gamma,i}}(x)=\alpha=\overline{h_{\Gamma',j}}(x).
\end{equation*} 
This then yields the Lemma.   
\end{proof}
We thus see that we have a complete description of the intersection points lying above the intersection points of $\Gamma$ and $\Gamma'$. In a concrete example, one has to do the following:
\begin{algo}
\begin{center}
{\bf{[Algorithm for completely decomposable 
morphisms of graphs]}}
\end{center}
\begin{enumerate}
\item Determine local functions $h_{\Gamma}$ and $h_{\Gamma'}$ such that 
\begin{eqnarray*}
(\overline{h_{\Gamma}})^{n}&=&\overline{f},\\
(\overline{h_{\Gamma'}})^{n}&=&\overline{f}.
\end{eqnarray*}
\item Determine the values
\begin{equation*}
\overline{h_{\Gamma}}(x)
\end{equation*}
for all intersection points $x$ of $\Gamma$ and $\Gamma'$. One then pairs these values as in Lemma \ref{Easylemma3}.
\end{enumerate}
\end{algo}

\begin{exa}\label{EllCurveTramp1}
Suppose we take the elliptic curve $E$ defined by the equation
\begin{equation*}
y^2=x(x-\pi)(x+1).
\end{equation*}
\begin{figure}[h!]
\centering
\includegraphics[scale=0.5]{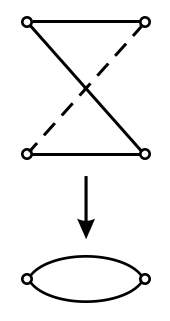}
\caption{\label{6eplaatje2} {\it{The trampoline covering in Example \ref{EllCurveTramp1}.}}}
\end{figure}
This has a 2-torsion point $P=(-1,0)$. Indeed, one can easily check that
\begin{equation*}
\text{div}_{\eta}(x+1)=2(P)-2(\infty).
\end{equation*}
The elliptic curve has the semistable model obtained by
\begin{equation*}
xt=\pi
\end{equation*}
and normalizing. The resulting equation is
\begin{equation*}
(\dfrac{y}{x})^2=(1-t)(x+1).
\end{equation*}
which is easily seen to have two vertices and two edges. We take $\Gamma=Z(x)$ and $\Gamma'=Z(t)$. We now take the normalization of this local model in the extension defined by
\begin{equation*}
z^2=x+1.
\end{equation*}
One easily finds that the divisor $x+1$ is locally a square in $E$. Indeed, for $t=0$ we find that 
\begin{equation*}
\overline{y}^2=\overline{x}+1
\end{equation*} 
and for $x=0$ we find that $\overline{x}+1=1$. In our adopted notation we now have
\begin{eqnarray*}
h_{\Gamma,0}&=&1,\\
h_{\Gamma,1}&=&-1,\\
h_{\Gamma',0}&=&\overline{y},\\
h_{\Gamma',1}&=&-\overline{y}.
\end{eqnarray*}
Let us consider the intersection point $\tilde{x}$ defined by $x=0=t$ and $\overline{y}=1$. We have
\begin{eqnarray*}
h_{\Gamma,0}(\tilde{x})&=&1,\\
h_{\Gamma,1}(\tilde{x})&=&-1,\\
h_{\Gamma',0}(\tilde{x})&=&1,\\
h_{\Gamma',1}(\tilde{x})&=&-1\\
\end{eqnarray*}
We thus see that we have two intersection points lying above $\tilde{x}$: for the value $1$ $\Gamma_{0}$ and $\Gamma'_{0}$ intersect, for the value $-1$ the components $\Gamma_{1}$ and $\Gamma'_{1}$ intersect. A similar computation gives the intersections lying above the other intersection point. We see that we obtain a {\it{trampoline}}-figure with 4 vertices, as obtained earlier. The covering can also be found in Figure \ref{6eplaatje2}. Since the obtained curve is again an elliptic curve, the reduction can also be obtained directly by calculating the reduction type.
\end{exa}

\begin{exa}\label{TripleTrampoline1}
Suppose we take the same elliptic curve $E$ with multiplicative reduction defined by
\begin{equation*}
y^2=x(x-\pi)(x+1)
\end{equation*} 
and suppose that we take a three torsion point $P$ that does not reduce to $(0,0)$ (the singular point). 
\begin{figure}[h!]
\centering
\includegraphics[scale=0.6]{{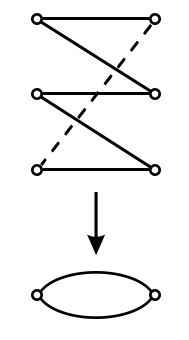}}
\caption{\label{771eplaatje} {\it{The covering in Example \ref{TripleTrampoline1}.}}}
\end{figure}
We first find a function $f$ such that
\begin{equation*}
(f)=3(P)-3(\infty).
\end{equation*}
If we label the components of $E$ as $\Gamma$ and $\Gamma'$ as before, we see that 
$f|_{\Gamma}$ is a constant and that $f|_{\Gamma'}$ is the cube of a nonconstant function $\overline{h_{\Gamma'}}$, which has a zero at $\tilde{P}$ and a pole at $\infty$.

We have the numbers
\begin{equation*}
\alpha_{h_{\Gamma',j}}(x)=\overline{h}_{\Gamma',j}(x)
\end{equation*}
for any intersection point. Note that for any $j$ and $x$ and $y$ distinct intersection points, we have that $\alpha_{h_{\Gamma',j}}(x)\neq{\alpha_{h_{\Gamma',j}}(y)}$ by the fact that $P$ is a nontrivial torsion point in the identity component.\\
We see that if we take the extension 
\begin{equation*}
z^3=f,
\end{equation*}
we obtain a reduction graph with 3 vertices above $\Gamma$ and 3 above $\Gamma'$. By earlier considerations, we see that if we take any component $\Gamma_{0}$ lying above $\Gamma$, it intersects {\it{two distinct}} other components lying above $\Gamma'$. The Galois group $\mathbb{Z}/3\mathbb{Z}$ then cycles these intersections naturally to give 6 edges. By the calculation
\begin{equation*}
e(E')-v(E')+1=6-6+1=1,
\end{equation*}
we find that this graph has Betti number one, as expected. The covering of graphs can be found in Figure \ref{771eplaatje}. We will an explicit covering in Example \ref{3Tors}.
\end{exa}

\begin{exa}\label{MixedTrampoline1}
Let us take the genus 2 curve $C$ defined by
\begin{equation*}
y^2=x(x-\pi)(x+1)(x+1-\pi)(x+2)(x+2-\pi),
\end{equation*}
which has the usual reduction graph consisting of two vertices with three edges between them. \\
We will however not take this model. As before, let 
\begin{equation*}
xt=\pi.
\end{equation*} 
Then the normalization of this model in $C$ is given by
\begin{equation*}
(y/x)^2=(1-t)(x+1)(x+1-\pi)(x+2)(x+2-\pi).
\end{equation*}
For $x=0$ (with corresponding component $\Gamma$), we obtain a single component, which we will call $\Gamma_{0}$. For $t=0$ (corresponding to $\Gamma'$), we obtain two components $\Gamma'_{0}$ and $\Gamma'_{1}$ intersecting each other in two points. We also have that $\Gamma_{0}$ intersects both $\Gamma'_{0}$ and $\Gamma'_{1}$ exactly once.  We see that this yields a {\it{subdivision}} of the original intersection graph given by just taking the special fiber. This can also be seen in Figure \ref{8eplaatje}.
\begin{figure}[h!]
\centering
\includegraphics[scale=0.7]{{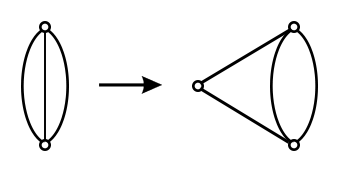}}
\caption{\label{8eplaatje} {\it{The subdivision of the original graph in Example \ref{MixedTrampoline1}.}}}
\end{figure}
  We now take the following divisor: $D=(\pi,0)-(0,0)$. This divisor cannot be principal because otherwise the curve would have genus 0. We have that
\begin{equation*}
2D=\text{div}_{\eta}(f),
\end{equation*}
where
\begin{equation*}
f=\dfrac{x-\pi}{x}=1-t.
\end{equation*}
It thus gives an element of $J(C)$ that is 2-torsion.\\
For $t=0$, we have that $f$ is constant, whereas for $x=0$, we have that $f$ is the square of a nonconstant function. Namely, we have that
\begin{equation*}
f|_{\Gamma_{0}}=(\overline{y/x})^2.
\end{equation*}
\begin{figure}[h!]
\centering
\includegraphics[scale=0.5]{{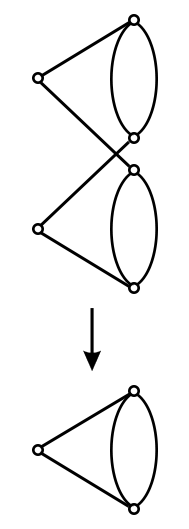}}
\caption{\label{9eplaatje} {\it{The covering in Example \ref{MixedTrampoline1}.}}}
\end{figure}
We are thus in a similar situation as in the previous 2-torsion example. If we consider the extension 
\begin{equation*}
z^2=f,
\end{equation*} 
we see that we obtain 2 vertices lying above $\Gamma_{0}$: $\Gamma_{0,0}$ and $\Gamma_{0,1}$. Similarly, we have 2 vertices lying above $\Gamma'_{0}$ ($\Gamma'_{0,0}$ and $\Gamma'_{0,1}$) and $\Gamma'_{1}$ ($\Gamma'_{1,0}$ and $\Gamma'_{1,1}$).\\
The intersection graph can now be found just as in the previous 2-torsion example. The only possible option up to relabeling components is: $\Gamma_{0,0}$ intersects $\Gamma'_{0,0}$ and $\Gamma'_{1,1}$, $\Gamma_{0,1}$ intersects $\Gamma'_{0,1}$ and $\Gamma'_{1,0}$. Furthermore, $\Gamma'_{0,0}$ intersects $\Gamma'_{0,1}$ twice and $\Gamma'_{1,0}$ intersects $\Gamma'_{1,1}$ twice. This covering of graphs can be found in Figure \ref{9eplaatje}.  This gives a graph with Betti number 
\begin{equation*}
e(\mathcal{C})-v(\mathcal{C})+1=6-4+1=3,
\end{equation*}
as was to be expected from an unramified degree 2 covering of a genus 2 curve.
\end{exa}

\begin{exa}\label{3Tors}
{\bf{[3-torsion on an elliptic curve]}}
Let us take $p=\alpha\cdot{}x$ and $q=\dfrac{1}{\sqrt{27}}(ax+b)$, where 
\begin{eqnarray*}
a&=&\dfrac{\pi-3}{2},\\
b&=&\dfrac{\pi-1}{2},\\
\alpha^3&=&1/4.
\end{eqnarray*}
Consider the curve $C$ defined by
\begin{equation*}
z^3+pz+q=0,
\end{equation*}
with discriminant $\Delta=4p^3+27q^2=x^3+(ax+b)^2.$
We now take the Galois closure of the morphism $K(x)\rightarrow{K(C)}$.  By Corollary \ref{GalClos1}, it is given by the chain
\begin{equation}
K\subset{K(y)}\subset{K(w)},
\end{equation}
where
\begin{equation}
w^3=y-\sqrt{27}q
\end{equation}
and
\begin{equation}
y^2=\Delta.
\end{equation}
We will study these $S_{3}$-coverings more closely in Chapter \ref{Solvable}. 

The intermediate curve given by
\begin{equation*}
D:y^2=\Delta
\end{equation*}
has genus 1, with a $3$-torsion point $P=(0,b)$ in its Jacobian.\footnote{The way we created this example is as follows. We took the family with $p(x)=x$ and $q(x)=ax+b$ linear and we imposed two conditions: that $x=-1$ be a zero of $\Delta$ and that $\Delta'(-1)=0$ (the derivative with respect to $x$). This then implies that the singular point is different from the $3$-torsion point. This can also be used to create examples of higher genus.}
That is, we have
\begin{equation*}
\text{div}(y-(ax+b))=3P-3(\infty).
\end{equation*}
We then easily see that $D$ has split multiplicative reduction and that $P$ doesn't reduce to the singular point. In other words, $P$ defines a $3$-torsion point in the \emph{toric} part of the Jacobian of $D$. 

\begin{figure}[h!]
\centering
\includegraphics[scale=0.27]{{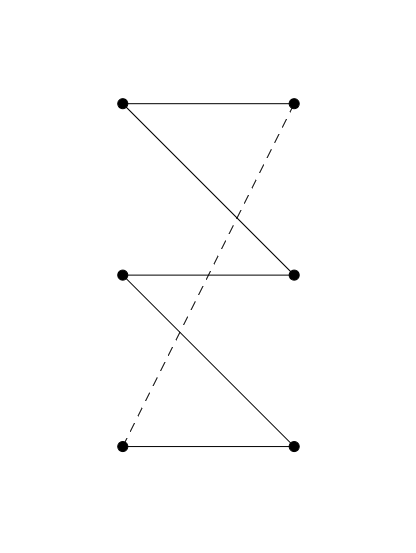}}
\caption{\label{7eplaatje} The intersection graph of the Galois closure in Example \ref{3Tors}.}
\end{figure}

\begin{figure}[h!]
\centering
\includegraphics[scale=0.18]{{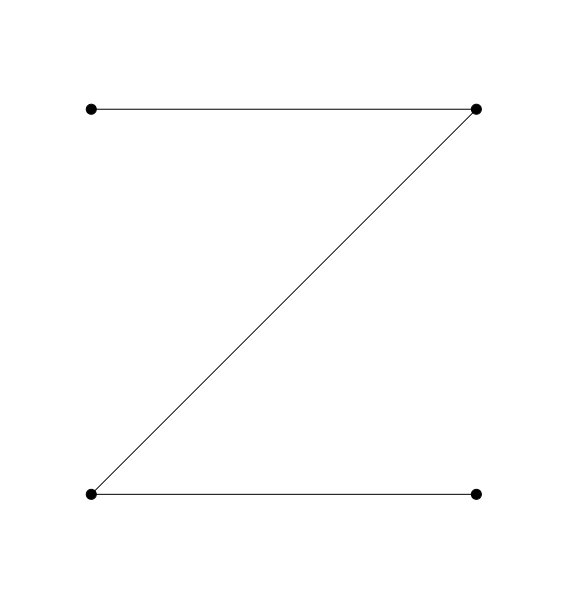}}
\caption{\label{8eplaatje} The intersection graph of the quotient of the Galois closure under a subgroup of order two in Example \ref{3Tors}.}
\end{figure}

At any rate, after a transformation $x\longmapsto{x+1}$ we obtain the equation
\begin{equation*}
y^2=(x-1)^3+(a(x-1)+b)^2=x^3+x^2(a^2-3)+\pi{x}.
\end{equation*}
Let $y'=\dfrac{y}{x}$. 
Taking the model defined by
\begin{equation*}
xt=\pi,
\end{equation*}
we obtain the equation
\begin{equation}\label{EquationExample5}
y'^2=x+a^2-3+t.
\end{equation}
We thus see that we have an intersection graph with two vertices and two edges, giving the multiplicative reduction.  
For $t=0$, we know that there exists a function $g$ such that $\overline{g^3}=\overline{y-\sqrt{27}q}$. We will find this function now.\\
Plugging in $t=0$ in Equation \ref{EquationExample5}, we obtain
\begin{equation*}
y'^2=x-3/4.
\end{equation*}
We thus see that $y'$ parametrizes the corresponding projective line. We write (without the reduction bar for $t=0$):
\begin{equation*}
y-\sqrt{27}q=xy'+3x/2-1=(y'^2+3/4)y'+3/2(y'^2+3/4)-1=(y'+1/2)^3.
\end{equation*}

Thus $g=y'+1/2$ solves the corresponding problem. We have 
\begin{eqnarray*}
g(x_{0})&=&-\zeta,\\
g(x_{1})&=&-\zeta^2,
\end{eqnarray*}
where $\zeta$ is a primitive third root of unity and $x_{0}$ and $x_{1}$ are the intersection points. The covering graph is now given by Figure \ref{7eplaatje} and the quotient graph under a subgroup of order two by Figure \ref{8eplaatje}. Note that all components in these figures have genus zero.

\end{exa}

\section{Twisting data for abelian coverings}\label{TwistingDataFinal}

In the last three sections, we saw that knowing the covering data for a disjointly branched morphism with morphism $\Sigma(\mathcal{C})\rightarrow\Sigma(\mathcal{D})$ is not enough, but we can recover $\Sigma(\mathcal{C})$ for unramified abelian coverings by adding the additional data of a $2$-cocycle on $\Sigma(\mathcal{D})$. In this section, we will continue this train of thought and define a $2$-cocycle for general abelian coverings $C\rightarrow{D}$. This will be a $2$-cocycle on the "unramified part" of $\Sigma(\mathcal{D})$. For simplicity, we will assume throughout this section that the covering is of prime degree $q$ with Galois group $\mathbb{Z}/q\mathbb{Z}$. The covering is then given by $z^{q}=f$ for some $f\in{K(D)}$. 




\subsection{Reconstructing $\Sigma(\mathcal{C})$}\label{ReconstructingTwisting}


Let $\mathcal{C}\rightarrow{\mathcal{D}}$ 
be disjointly branched morphisms associated to the morphism $C\rightarrow{D}$. 
Consider a vertex $v$ in $\Sigma(\mathcal{D})$ with corresponding component $\Gamma\subset{\mathcal{D}_{s}}$. Let $v'$ be any vertex in $\overline{\mathcal{C}}$ lying above $v$ and let $\Gamma'$ be its corresponding component. If $D_{v'/v}=\mathbb{Z}/q\mathbb{Z}$, then there are no options for the edges lying above $v$: they are connected to $v'$. 
We now remove these "ramified parts" of $\Sigma(\mathcal{D})$ and then consider the local \'{e}tale equations.
\begin{mydef}\label{UnramifiedPartGraph}
Let $\Sigma(\mathcal{C})\rightarrow{\Sigma(\mathcal{D})}$ be the degree $q$ abelian morphism of intersection graphs coming from a disjointly branched morphism $\phi_{\mathcal{C}}$. Let $U(\Sigma(\mathcal{D}))\subset{\Sigma(\mathcal{D})}$ be the (possibly disconnected and incomplete) subgraph of $\Sigma(\mathcal{D})$, consisting of all edges and vertices that have trivial decomposition groups. We call this graph the \emph{unramified} part of $\Sigma(\mathcal{D})$. 
\end{mydef}

Let $v$ be a vertex in $U(\Sigma(\mathcal{D}))$ with corresponding component $\Gamma=\Gamma_{1}$. 
The abelian covering is given on the level of function fields as
\begin{equation}
K(D)\rightarrow{K(D)[z]/(z^q-f)}.
\end{equation}
 We now consider the $\Gamma$-modified form of $f$. That is, we set $k=v_{\Gamma}(f)$ and consider
\begin{equation}
f^{\Gamma}=\dfrac{f}{\pi^{k}},
\end{equation}
so that $v_{\Gamma}(f^{\Gamma})=0$. This means that we can now safely consider the image of $f^{\Gamma}$ in the function field $k(\Gamma)$. We then have
\begin{lemma}\label{FactorizationVertexSplit}
\begin{equation}
\overline{f^{\Gamma}}=h_{v}^{q}
\end{equation}
for some $h_{v}\in{k(\Gamma)}$. In particular, the components lying above $\Gamma$ are given by the prime ideals
\begin{equation}
\mathfrak{q}_{i}=\mathfrak{p}+(w-\zeta^{i}h_{v}),
\end{equation}
where $\zeta$ is a primitive $q$-th root of unity and $i\in\{0,1,2...,q-1\}$.
\end{lemma}  
\begin{proof}
Since there are $q$ vertices lying above $v$, we know that 
\begin{equation}
\text{div}(\overline{f^{\Gamma}})=qD
\end{equation}
for some divisor $D$. Indeed, otherwise the extension $z^q=\overline{f^{\Gamma}}$ would be ramified at some point and thus there would only be one component lying above it. Suppose now that $D$ is not a principal divisor. Then $D$ is a $q$-torsion point in $\text{Pic}^{0}(\Gamma)$. The corresponding extension would then give a {\it{connected}} unramified covering of $\Gamma$, which contradicts the fact that there are $q$ vertices lying above $v$. This finishes the proof.
\end{proof}

We now \emph{desingularize} $\mathcal{D}$ to obtain a morphism $\mathcal{D}_{0}\rightarrow{\mathcal{D}}$ that is regular in the pre-image of every $e\in{U(\Sigma(\mathcal{D}))}$. We then have
\begin{lemma}
The normalization $N(\mathcal{D}_{0},K(\mathcal{C}))$ of $\mathcal{D}_{0}$ in $K(\mathcal{C})$ is semistable. 
\end{lemma}
\begin{proof}
We will show that there exists no further vertical ramification above the new components created in $\mathcal{D}_{0}$. This then implies that $N(\mathcal{D}_{0},K(\mathcal{C}))$ is semistable by Theorem \ref{MaintheoremSemSta}. 

Let $\Gamma$ be any component in $\mathcal{D}$ and consider its image in $\mathcal{D}_{0}$ (which we will still denote by $\Gamma$). For every edge $e\in{U(\Sigma(\mathcal{D}))}$ with length $l(e)>1$, there are new components in $\mathcal{D}_{0}$. The valuation of $f^{\Gamma}$ at these new components is divisible by $q$ by Theorem \ref{ValCor1}. In other words, the normalization is \'etale above these components. We thus see that there is no further vertical ramification and the Lemma follows. 
\end{proof}

We will now continue with this model $\mathcal{D}_{0}$ and its normalization $\mathcal{C}_{0}$ in $K(C)$. We will write $\mathcal{D}=\mathcal{D}_{0}$ from now on. 

To construct the twisting data, we first gather some standard facts about ordinary double points of length one that are probably familiar to the reader. Let $P\in\mathcal{D}$ be an intersection point and let $A:=\mathcal{O}_{\mathcal{D},P}$. We write $\mathfrak{p}_{1}$ and $\mathfrak{p}_{2}$ for the generic points of the components passing through $P$. Since $P$ is an ordinary double point with length one, we find that
\begin{equation}
\hat{A}=\hat{\mathcal{O}}_{\mathcal{D},P}\simeq{R[[x,y]]/(xy-\pi)}.
\end{equation}
We have the following 
\begin{lemma}\label{LabelOrdDouble}
Let $A$ be as above. Then
\begin{enumerate}
\item $A$ is a unique factorization domain.
\item There exist $x_{1}$ and $y_{1}$ in $A$ such that $v_{\mathfrak{p}_{1}}(x_{1})=0$, $v_{\mathfrak{p}_{2}}(x_{1})=1$, $v_{\mathfrak{p}_{1}}(y_{1})=1$ and $v_{\mathfrak{p}_{2}}(y_{1})=0$ and $x_{1}y_{1}=\pi$.
\item Every element $f\in{K(\mathcal{D})}$ can be written uniquely as
\begin{equation}
f=x^{i}_{1}\cdot{y^{j}_{1}}\cdot{u},
\end{equation}
where $(i,j)\in\mathbb{Z}^{2}$ and $u\in{A}^{*}$. 
\end{enumerate}
\end{lemma}
\begin{proof}
We have that $\hat{A}$ is a unique factorization domain, so by \cite[Lemma 1.2]{samdomains} we find that $A$ is a unique factorization domain\footnote{One could also reason as follows. The ring $\hat{A}$ is regular, so $A$ is regular. A famous result by Auslander and Buchbaum then says that any regular local ring is a unique factorization domain. The lemma cited above is far easier to prove however, only needing Nakayama's Lemma and the Mittag-Leffler condition.}.

By the approximation theorem for valuations (\cite[Chapter 9, Lemma 1.9]{liu2}), we can find an element $x_{1}\in{K(\mathcal{D})}$ such that $v_{\mathfrak{p}_{1}}(x_{1})=0$ and $v_{\mathfrak{p}_{2}}(x_{1})=1$. Since a normal domain is the intersection of its localizations, we find that $x_{1}\in{A}$. The special fiber of $\mathcal{D}$ is reduced, so we find that $v_{\mathfrak{p}_{i}}(\pi)=1$ for both $i$. Now consider the element $y_{1}:=\dfrac{\pi}{x_{1}}$. We find that $v_{\mathfrak{p}_{1}}(y_{1})=1$ and $v_{\mathfrak{p}_{2}}(y_{1})=0$ and thus $y_{1}\in{A}$. The unique factorization as stated in the Lemma now directly follows. 
\end{proof}

We now return to the normalized form $f^{\Gamma}$, where we we take $\Gamma=\overline{\{\mathfrak{p}_{1}\}}$. As before, we focus on an intersection point $P$ that corresponds to an edge $e\in{}U(\Sigma(\mathcal{D}))$. We have that $v_{\mathfrak{p}_{1}}(f^{\Gamma})=0$, so that we can write
$f^{\Gamma}=x_{1}^{i}{f'}$ for some $f'$. In fact, by the Poincar\'{e}-Lelong formula, Theorem \ref{ValCor1}, 
we must have $i=qn$. We now consider the element $f'$. We then have $f'\in{A}$. In fact, we see that $v_{\mathfrak{p}_{i}}(f')=0$ for both $i$, so by Lemma \ref{LabelOrdDouble}, we find that $f'\in{A^{*}}$.

We now consider the equation $z^{q}=f$, which gives the extension $K(D)\rightarrow{K(C)}$ on the generic fiber.  Since $\mathcal{C}\rightarrow{\mathcal{D}}$ is disjointly branched, we have that the extension is unramified above $\Gamma$. In other words, $v_{\mathfrak{p}_{1}}(f)$ is divisible by $q$. We can then normalize this equation to obtain
\begin{equation}
z'^q=f^{\Gamma}.
\end{equation} 

We now consider the element $z''=\dfrac{z'}{x^{n}_{1}}$ in the function field of $K(\mathcal{C})$.

\begin{cor}
The algebra $A[z'']$ is finite \'{e}tale over $A=\mathcal{O}_{\mathcal{D},P}$.
\end{cor}
\begin{proof}
We have $z''^q=f'$, 
so it is finite. It is standard \'{e}tale by $f'\in{A^{*}}$, and thus also \'{e}tale.
\end{proof}

By Lemma \ref{FactorizationVertexSplit}, we can now factorize $f^{\Gamma}$ and even $f'$ as a $q$-th power in $k(\Gamma_{1})$ and $k(\Gamma_{2})$. To avoid confusion, we will write $\text{red}(f',\mathfrak{p}_{i})$ for the image of $f'$ in $\text{Frac}(A/\mathfrak{p}_{i})=k(\mathfrak{p}_{i})$. We then have 
\begin{align*}
\text{red}(f',\mathfrak{p}_{1})&=\overline{g_{1}}^q,\\
\text{red}(f',\mathfrak{p}_{2})&=\overline{g_{2}}^q.
\end{align*}
Note that we can evaluate both $g_{1}$ and $g_{2}$ at the point $P$. 
The components lying above $\Gamma_{1}$ and $\Gamma_{2}$ are now given by the prime ideals
\begin{align*}
\mathfrak{q}_{1,i}&=\mathfrak{p}_{1}+(z''-\zeta^{i}g_{1}),\\
\mathfrak{q}_{2,i}&=\mathfrak{p}_{2}+(z''-\zeta^{i}g_{2}).
\end{align*}
We denote the corresponding components by $\Gamma_{1,i}$ and $\Gamma_{2,i}$ for $i\in\{0,1,2,...,q-1\}$.
Now consider the component $\Gamma_{1,0}$ labeled by $z''=\overline{g_{1}}$. Evaluating $f'$ at $P$, we obtain
\begin{equation}
g_{1}(P)^q=g_{2}(P)^q.
\end{equation}
In other words, there exists a $j\in\{0,1,2,...,q-1\}$ such that $g_{1}(P)=\zeta^{j}g_{2}(P)$. This implies that $\Gamma_{1,0}$ is connected to $\Gamma_{2,i}$ and by the cyclic $\mathbb{Z}/q\mathbb{Z}$-action, this also determines the rest of the edges. We now give a summary of the procedure: 

\begin{algo}\label{Linkingcomponents}
\begin{center}
{\bf{[Algorithm for linking components]}}
\end{center}
\begin{enumerate}
\item Find $g_{1}(P)$ and $g_{2}(P)$ using an explicit representation of the function field of $\Gamma_{1}$.
\item There exists an $j$ such that 
\begin{equation}
g_{1}(P)=\zeta^{j}\cdot{}g_{2}(P).
\end{equation} 
\item Connect the vertex labeled by $z''=g_{{1}}$ to the vertex labeled by $z''=\zeta^{j}g_{2}$.
\item The other vertices and connecting edges lying above $\Gamma_{1}$ are now completely determined by the cyclic $\mathbb{Z}/q\mathbb{Z}$-action. 
\end{enumerate}
\end{algo}
 
 By considering these functions $g_{1,P}$ for various intersection points $P$, we now obtain a $2$-cocycle on the intersection graph as follows: we define
 \begin{equation}
 \alpha(e_{1},e_{2})=\dfrac{g_{1,P_{1}}(P_{1})}{g_{1,P_{2}}(P_{2})},
 \end{equation}
 where $P_{i}$ is the intersection point corresponding to the edge $e_{i}$. This can be seen as a generalization of the usual $2$-cocycle one obtains when studying the Picard group of a reduced (possibly reducible) curve with ordinary singularities over a field $k$, see Section \ref{TorExtJac1}.

\section{An algorithm for abelian coverings}\label{Algorithmabeliancoverings}

In this section, we use the twisting data obtained in the previous sections to obtain an algorithm that reconstructs the Berkovich skeleton of a curve $C$ that admits a cyclic abelian covering $C\rightarrow{D}$, where the Berkovich skeleton of $D$ is known.

\begin{algo}\label{AlgorithmAbelian}

\begin{center}
{\bf{[Reconstructing Berkovich skeleta using abelian coverings]}}
\end{center}
\begin{flushleft}
Input: A cyclic abelian covering $C\rightarrow{D}$, a semistable model $\mathcal{D}$ with intersection graph $\Sigma(\mathcal{D})$.
\end{flushleft}
\begin{enumerate}
\item Determine the covering data of $\Sigma(\mathcal{C})\rightarrow{\Sigma(\mathcal{D})}$ using Propositions \ref{PropositionCoveringData}, \ref{DecompositionVertexProposition} and Theorem \ref{DecompVert}. 
\item Determine the unramified part $U(\Sigma(\mathcal{D}))$ of $\Sigma(\mathcal{D})$, as in Definition \ref{UnramifiedPartGraph}.
\item For every $v\in{U(\Sigma(\mathcal{D}))}$ and every adjacent edge $e$, determine $g_{1}(e)$ and $g_{2}(e)$, as in Section \ref{ReconstructingTwisting}. 
\item Use the procedure in Algorithm \ref{Linkingcomponents} to link together the components. 
\end{enumerate}
Output: The Berkovich skeleton of $C$.
\end{algo}
\begin{proof}
({\it{Correctness of the algorithm}})
The covering data and the twisting data are correct by Chapter \ref{Coveringdata} and Section \ref{ReconstructingTwisting}. The algorithm terminates because there are only finitely many vertices in $\Sigma(\mathcal{D})$ and finitely many options for linking together the covering vertices and edges.
\end{proof}


\subsection{An algorithm for solvable Galois coverings}


In this section we iterate the algorithm obtained in Section \ref{Algorithmabeliancoverings} for solvable Galois coverings. 
We first describe several generalities regarding solvable groups and then give an algorithm that gives the Berkovich skeleton of $C$, where $C\rightarrow{D}$ is a solvable Galois covering to a curve $D$ with a known Berkovich skeleton. 

Let $G$ be a finite group. Then $G$ is said to be {\it{solvable}} if there exists a finite chain of subgroups of $G$:
\begin{equation*}
G_{0}=(1)\subseteq{G_{1}}\subseteq{...}\subseteq{G_{n}=G},
\end{equation*}
where $G_{i}$ is normal in $G_{i+1}$ and $G_{i+1}/G_{i}$ is abelian. An equivalent definition is that $G$ admits a {\it{composition series}} such that every factor is cyclic of prime order. Recall that a composition series is a series
\begin{equation*}
(1)=H_{0}\vartriangleleft{H_{1}}\vartriangleleft{...}\vartriangleleft{H_{r}=G}
\end{equation*}   
such that each $H_{i}$ is normal in $H_{i+1}$ and $H_{i+1}/H_{i}$ is simple. \\
We will now return to our usual setting of coverings. Suppose that $C\longrightarrow{D}$ is Galois with solvable Galois group. Then there exists a composition series
\begin{equation*}
(1)=H_{0}\vartriangleleft{H_{1}}\vartriangleleft{...}\vartriangleleft{H_{r}=G}
\end{equation*}
with corresponding inclusions of function fields
\begin{equation*}
K(C)\longleftarrow{(K(C))^{H_{1}}}\longleftarrow{...}\longleftarrow{(K(C))^{H_{r}}=K(D)}.
\end{equation*}
We define $K(C_{i}):=(K(C))^{H_{i}}$. Then for every inclusion $H_{i}\vartriangleleft{H_{i+1}}$ we have a Galois extension
\begin{equation*}
K(C_{i+1})\longrightarrow{K(C_{i})}
\end{equation*} 
that is cyclic of degree $n_{i}:=[H_{i+1}:H_{i}]$. If our field $K$ contains $\zeta_{n}$ for every $n$ dividing $|G|$ (so it contains $\zeta_{n_{i}}$ in particular), 
then this field extension can be described by an extension of the form
\begin{equation*}
K(C_{i+1})\longrightarrow{K(C_{i+1})[z]/(z^{q}-f_{i}})=K(C_{i})
\end{equation*}
by Kummer theory.

We can now state the algorithm for solvable coverings $C\rightarrow{D}$. 
\begin{algo}\label{AlgorithmSolvable}
\begin{center}
{\bf{[Reconstructing Berkovich skeleta using solvable coverings]}}
\end{center}
\begin{flushleft}
Input: A solvable Galois covering $C\rightarrow{D}$ with Galois group $G$, a semistable model $\mathcal{D}$ with intersection graph $\Sigma(\mathcal{D})$.
\end{flushleft}
\begin{enumerate}
\item Find the subextensions $K(C_{i})$ of $K(C)$ corresponding to a composition series of $G$. 
\item Use Algorithm \ref{AlgorithmAbelian} on the morphisms $C_{i}\longrightarrow{C_{i+1}}$ to calculate the intersection graph of $C_{i}$, starting with $C_{r}=D$. 
\end{enumerate}
Output: The Berkovich skeleton of $C$. 
\end{algo}
\begin{proof}
({\it{Correctness of the algorithm}}) The morphisms $C_{i}\rightarrow{C_{i+1}}$ are cyclic abelian by assumption, so Algorithm \ref{AlgorithmAbelian} is applicable. The group $G$ is finite, so there are only finitely many morphisms $C_{i}\rightarrow{C_{i+1}}$ for which Algorithm \ref{AlgorithmAbelian} has to be run.  Since that algorithm terminates in finite time for each morphism $C_{i}\rightarrow{C_{i+1}}$, we find that this algorithm also terminates. 
\end{proof}
\chapter{Cyclic abelian coverings of the projective line}\label{Abelian}

In this chapter, we will use the techniques obtained in the previous chapters to study cyclic abelian coverings of the projective line. In this case, the twisting data developed in Chapter \ref{Twistingdata} are not required, since these lead to a $2$-cocycle on a tree, which is trivial. We start with hyperelliptic coverings, which has been a known case for quite some time, see \cite{trophyp}. After that we will give some examples of abelian coverings of higher degree.
We will also present several results obtained in \cite{supertrop}, where a general algorithm for these coverings was given. In this paper, it was shown that any cyclic abelian covering of metrized complexes (in our language) $\Sigma\rightarrow{T}$, where $T$ is a tree, arises from some algebraic covering $C\rightarrow{\mathbb{P}^{1}}$. The procedure given there is explicit, in the sense that a polynomial $f(x)$ is given such that the curve $z^n=f(x)$ with covering $(x,z)\mapsto{x}$ has the desired properties.


\section{Hyperelliptic curves}


Let $C\longrightarrow{\mathbb{P}^{1}}$ be a hyperelliptic covering. Since $\text{char}(K)=0\neq{0}$, we can use Kummer Theory to see that the function field extension is given by an equation
\begin{equation*}
y^2=f(x),
\end{equation*}
where $f(x)$ is a polynomial of a certain degree over the field $K$. Over a finite extension of $K$, we can now write $f(x)$ as
\begin{equation*}
f(x)=\prod_{i=1}^{r}(x-\alpha_{i})
\end{equation*}
for certain elements $\alpha_{i}\in\overline{K}$. We assume that we have made the finite extension already and that $\alpha_{i}\in{{K}}$. For simplicity, we will now assume that $v(\alpha_{i})\geq{0}$ for every $i$. We will also assume that $f(x)$ is squarefree.

\begin{exa}\label{ExaHyp1}
We take the curve $C$ defined by
\begin{equation}\label{EqnExaHyp1}
y^2=x(x-\pi)g(x),
\end{equation}
where $\pi$ is a uniformizer and $g(x)$ is a polynomial of odd degree $c=2k+1$ (the case of a polynomial with even degree is similar but with two points at infinity). We assume that the roots of $g$ reduce to distinct points not equal to $0$. We have
\begin{equation*}
g(C)=k+1.
\end{equation*}
Since the points $(0)$ and $(\pi)$ are not disjoint in the special fiber, we will want to create a semistable model for $\mathbb{P}^{1}$ that makes them disjoint. We take:
\begin{equation*}
\text{Proj}R[X,T,W]/(XT-\pi{W}^2)
\end{equation*}
with affine model
\begin{equation*}
\text{Spec}R[x,t]/(xt-\pi).
\end{equation*}
We have that the point $(0)$ on the generic fiber is now transferred to the affine part 
\begin{equation*}
\text{Spec}R[x',w]/(x'-\pi{w^2}),
\end{equation*}
where $x'=\dfrac{X}{T}$ and $w=\dfrac{W}{T}$. Indeed, the corresponding prime ideal is  $(x',w)$. The point $(\pi)$ now corresponds to the prime ideal $(x-\pi,t-1)$ lying on the generic fiber. We see that the reductions of $(0)$ and $(\pi)$ now lie on the same component $(x)$, but they have distinct $t$-coordinates:  one has $t=1$ and the other has "$t=\infty$".\\
We can thus use Theorem \ref{MaintheoremSemSta} and calculate the normalization of this scheme in the finite extension defined by Equation \ref{EqnExaHyp1}. We'll first take a different route however, using only our knowledge of the divisors involved. Consider the divisor
\begin{equation*}
\text{div}_{\eta}(f)=(0)+(\pi)+Z(g)-(2+c)\cdot{\infty}.
\end{equation*}
We calculate
\begin{equation*}
\rho(\text{div}_{\eta}(f))=2\cdot(\Gamma_{1})-2\cdot({\Gamma_{2}}).
\end{equation*}

This means that the corresponding Laplacian function has slope $\pm{2}$ between $\Gamma_{1}$ and $\Gamma_{2}$. The Laplacian can also be found in Figure \ref{1eplaatjeExtra}.
\begin{figure}[h!]
\centering
\includegraphics[scale=0.45]{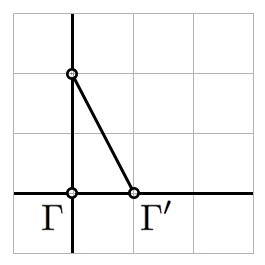}
\caption{\label{1eplaatjeExtra} {\it{The Laplacian of $f$ in Example \ref{ExaHyp1}.}}}
\end{figure} For the edge corresponding to the intersection $\mathfrak{m}=(x,t,\pi)$, we thus obtain two edges in the pre-image.\\
We now calculate $f^{\Gamma_{1}}$. We obtain 
\begin{equation*}
(f^{\Gamma_{1}})=(x',w,\pi)+(x,t-1,\pi)-2\cdot({\Gamma_{1}\cap{\Gamma_{2}}}).
\end{equation*}
If we thus consider the local equation
\begin{equation*}
y^2=f^{\Gamma_{1}},
\end{equation*}
it will ramify at 2 points. Thus the genus of the corresponding component above is $0$. For $\Gamma_{2}$ we have
\begin{equation*}
(f^{\Gamma_{2}})=Z(g)+2\cdot({\Gamma_{1}\cap\Gamma_{2}})-(c+2)(\infty).
\end{equation*}
Thus the equation $y^2=f^{\Gamma_{2}}$ ramifies in the points defined by $Z(g)$ and $\infty$. There are $c+1$ of these, thus we can use the Riemann-Hurwitz formula to obtain
\begin{equation*}
2g_{\Gamma'_{2}}-2=2(-2)+c+1
\end{equation*}
and thus
\begin{equation*}
g_{\Gamma'_{2}}=k.
\end{equation*}
\begin{figure}[h!]
\centering
\includegraphics[scale=0.6]{{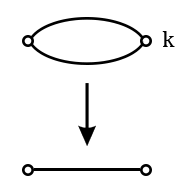}}
\caption{\label{10eplaatje} {\it{The covering in Example \ref{ExaHyp1}.}}}
\end{figure}
Thus the reduction graph consists of two vertices with two edges meeting them. The first component has genus $0$ and the second component has genus $k$. The covering of graphs can be found in Figure \ref{10eplaatje}.

We could have also calculated the normalization directly:
\begin{eqnarray*}
z^2=(1-t)g(x),
\end{eqnarray*}
where $z=y/x$. Plugging in $t=0$ and $x=0$ then yields the same reduction graph.
\end{exa}

\begin{exa}\label{ExaHyp2}
Let us take a slightly more involved example. We take
\begin{equation*}
z^2=x(x-\pi)(x-\pi^2)g(x),
\end{equation*}
where $g(x)$ is a polynomial of even degree $c=2k$. Then $g(C)=k+1$. If we now consider the open affine defined by
\begin{equation*}
R[x,y]/(xy-\pi^2),
\end{equation*}
then $\pi$ does not reduce to a regular point. When we blow this point up, we obtain a new component where $\pi$ does reduce to a regular point.
\begin{figure}[h!]
\centering
\includegraphics[scale=0.5]{{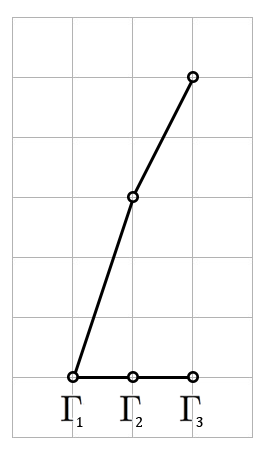}}
\caption{\label{11eplaatje} {\it{The Laplacian $\phi$ of $f$ in Example \ref{ExaHyp2}.}}}
\end{figure} The blow-up is given by the local charts
\begin{eqnarray*}
U_{1}=\text{Spec}(R[x,t_{2}]),\\
U_{2}=\text{Spec}(R[y,t_{3}]),
\end{eqnarray*}
with relations
\begin{eqnarray}
xt_{2}&=&\pi,\\
yt_{3}&=&\pi,
\end{eqnarray}
and the "obvious" local isomorphisms. We label the components $Z(t_{2})=\Gamma_{1}$, $Z(x,y)=\Gamma_{2}$ and $Z(t_{3})=\Gamma_{3}$. Here $\Gamma_{i}$ intersects $\Gamma_{i+1}$.\\

We have
\begin{equation*}
\text{div}_{\eta}(f(x))=(0)+(\pi)+(\pi^2)+Z(g)-(c+3)(\infty),
\end{equation*}
where $(0)$ and $(\pi^2)$ reduce to $\Gamma_{3}$ (the component with "$v(x)\geq{2}$"), $(\pi)$ reduces to $\Gamma_{2}$ and $Z(g)$ and $\infty$ reduce to $\Gamma_{1}$. Furthermore
\begin{equation*}
\rho(\text{div}_{\eta}(f))=2(\Gamma_{3})+(\Gamma_{2})-3(\Gamma_{1}),
\end{equation*}
whose Laplacian is depicted by a slope of $2$ between $\Gamma_{3}$ and $\Gamma_{2}$ and a slope of $3$ between $\Gamma_{2}$ and $\Gamma_{1}$, as in Figure \ref{11eplaatje}.
\begin{figure}[h!]
\centering
\includegraphics[scale=0.6]{{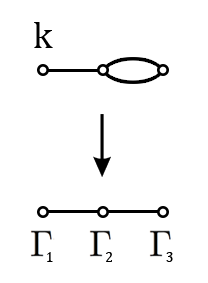}}
\caption{\label{12eplaatje} {\it{The covering of graphs in Example \ref{ExaHyp2}.}}}
\end{figure}
 Correspondingly, the edge $e_{2,3}$ has two pre-images in $\mathcal{C}$ and the edge $e_{1,2}$ has one pre-image in $\mathcal{C}$. One finds that $f^{\Gamma_{1}}$ has $c+1$ ramification points and thus $\Gamma'_{1}$ has genus $k$ (Check this with the Riemann-Hurwitz formula). Similarly, $f^{\Gamma_{2}}$ and $f^{\Gamma_{3}}$ have two ramification points and as such they have genus $0$. Thus the reduction graph consists of three vertices $v_{1},v_{2},v_{3}$ where $v_{1}$ and $v_{2}$ intersect once and $v_{2}$ and $v_{3}$ intersect twice. The covering of graphs can be found in Figure \ref{12eplaatje}.\\
We also give the normalizations for completeness. For the first chart $U_{1}$ they are given by
\begin{eqnarray*}
z_{1}^2&=&x(1-t_{2})(1-t_{2}\pi)g(x),\\
z_{2}^2&=&t_{2}(1-t_{2})(1-t_{2}\pi)g(x),\\
z_{1}\cdot{z_{2}}&=&\pi^{1/2}(1-t_{2})(1-t_{2}\pi)g(x).
\end{eqnarray*}
where $z_{1}=\dfrac{z}{x}$ and $z_{2}=\dfrac{t_{2}z}{\pi^{1/2}x}$.
For the second chart $U_{2}$ we have a single algebra given by
\begin{eqnarray*}
z_{3}^2=(t_{3}-1)(1-y)g(t_{3}\cdot{\pi}),
\end{eqnarray*}
where $z_{3}=\dfrac{z}{\pi^{1/2}t_{3}}$. Note that in both charts we clearly see the need for the ramified extension of degree $2$ given by $K\subseteq{K(\pi^{1/2})}$. 
\end{exa}

\section{Cyclic abelian coverings of $\mathbb{P}^{1}$}

In the previous section, we saw that we can quite easily determine the reduction graph of a lot of hyperelliptic curves quite easily using divisors and their reductions. We will now state the process more generally for abelian covers of $\mathbb{P}^{1}$. The key step in this process will be the determination of the Laplacian of a certain defining function $f$. This Laplacian will determine most of the overlying reduction graph. We also refer the reader to \cite[Algorithm 4.2]{supertrop}, where the algorithm first appeared.  

So suppose we are given an abelian cover $C\longrightarrow{\mathbb{P}^{1}}$ of degree $n$ over $K$. By Kummer Theory, we have that it is given by
\begin{equation*}
z^{n}=f(x),
\end{equation*}
where $f(x)$ possibly has multiple factors. 

\begin{algo}\label{AbelianAlgorithm}

\begin{center}
{\bf{[Algorithm for semistable graphs of abelian coverings of $\mathbb{P}^{1}$]}}
\end{center}
\begin{flushleft}
Input: A polynomial $f\in{K[x]}$. 
\end{flushleft}
\begin{enumerate}
\item Let $S=Z(f)\cup{\{\infty\}}$. Construct the tropical separating $T_{S}$ and its corresponding semistable model $\mathcal{D}_{S}$, as in Chapter \ref{Appendix2}.    
\item Determine $\rho(P)$ for any $P\in\text{Supp}(f)$. 
\item Determine the Laplacian function $\Delta(f)$. 
\item Determine $|I_{e}|$ for every edge using Proposition \ref{PropositionCoveringData}. This determines the edge length for any edge $e'$ lying above $e$ by \ref{InertiagroupIntersectionPoint1}.  
\item Determine the genera of the vertices $v'$ lying above every vertex $v\in{T_{S}}$ using the Riemann-Hurwitz formula (see Theorem  \ref{RiemannHurwitz}). 
\item Determine the covering graph $\Sigma(\mathcal{C})$ using the covering data obtained above.
\end{enumerate}
Output: The Berkovich skeleton of the curve $C$ defined by $z^{n}=f$.
\end{algo}
\begin{proof}
({\it{Correctness of the algorithm}})
The covering data obtained during the algorithm are correct by the cited propositions and theorems. There is only one covering graph up to the cyclic action of $\mathbb{Z}/n\mathbb{Z}$ for the given covering data, since it corresponds to a $2$-cocycle on $T_{S}$, which must be trivial. We note that a more elementary argument explicitly determining the covering graphs is also possible here. 
\end{proof}

\begin{exa}\label{ExaHyp3}
Let us do an example where our curve is given by a degree $3$ covering of $\mathbb{P}^{1}$. Let's take the curve $C$ given by the equation
\begin{equation*}
z^3=f(x):=x(x-\pi)g(x),
\end{equation*}
with $c:=\text{deg}(g(x))$ and $g(x)$ separable. We take $g(x)$ such that $\text{deg}(g)+2\neq{0}\mod{3}$. Then $C$ has exactly one point at infinity and $C$ ramifies above that point. We calculate
\begin{equation*}
2g-2=-6+2\cdot{}\text{Card}(\mathcal{R})
\end{equation*}
and 
\begin{equation*}
\text{Card}(\mathcal{R})=c+3,
\end{equation*}
so that 
\begin{equation*}
g(C)=c+1.
\end{equation*}
We'll take the usual model with chart
\begin{equation*}
R[x,t]/(xt-\pi).
\end{equation*}
As before, we have
\begin{equation*}
\text{div}_{\eta}(f)=(0)+(\pi)+Z(g)-(2+c)\cdot({\infty})
\end{equation*}
and
\begin{equation*}
\rho(\text{div}_{\eta}(f))=2(\Gamma_{1})-2(\Gamma_{2}).
\end{equation*}
We thus see that the edge $e=\Gamma_{1}\cap\Gamma_{2}$ is preserved in $\mathcal{C}_{s}$, meaning that $g_{e}=1$. We find that $f^{\Gamma_{1}}$ has 3 ramification points and $f^{\Gamma_{2}}$ has $c+2$. Using the formula
\begin{equation*}
g=-2+\#{\mathcal{R}},
\end{equation*}
we find that $\Gamma'_{1}$ has genus 1 and $\Gamma'_{2}$ genus $c$. Thus the reduction graph consists of two vertices intersecting once, with weights $1$ and $c$. The covering of graphs can be found in Figure \ref{14eplaatje}.
\begin{figure}[h!]
\centering
\includegraphics[scale=0.3]{{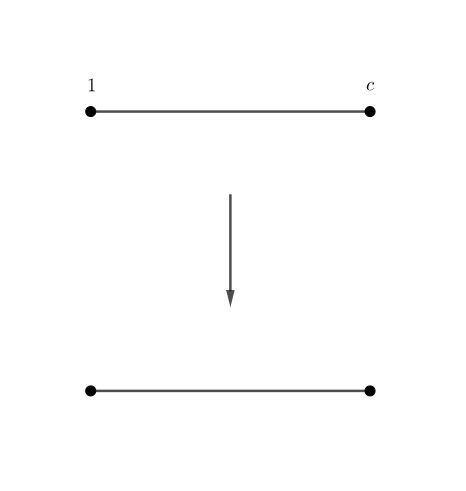}}
\caption{\label{14eplaatje} {\it{The covering of graphs in Example \ref{ExaHyp3}.}}}
\end{figure}
\end{exa}
\begin{exa}\label{TripEllExa1}
Now take
\begin{equation*}
z^3=x(x-\pi)(x-2\pi)g(x),
\end{equation*}
where $c:=\text{deg}(g(x))$ is such that $c+3\neq{0}\mod{3}$. We then have that $g(C)=c+4-2=c+2$. We take the model 
\begin{equation*}
R[x,t]/(xt-\pi).
\end{equation*}
with components $Z(x)=\Gamma_{1}$ and $Z(t)=\Gamma_{2}$. Note that $\Gamma_{1}$ corresponds to points with $v(x)\geq{1}$ and $\Gamma_{2}$ to points with $v(x)\leq{1}$. We thus find that $(0),(\pi),(2\pi)\longmapsto\Gamma_{1}$ and $Z(g),(\infty)\longmapsto{\Gamma_{2}}$. We have
\begin{equation*}
\text{div}_{\eta}(f)=(0)+(\pi)+(2\pi)+Z(g)-(c+3)(\infty)
\end{equation*}
with
\begin{equation*}
\rho(\text{div}_{\eta}(f))=3(\Gamma_{1})-3(\Gamma_{2})
\end{equation*}
and thus $g_{e}=3$ for $e=\Gamma_{1}\cap\Gamma_{2}$. The Laplacian can also be found in Figure \ref{15eplaatje}.
\begin{figure}[h!]
\centering
\includegraphics[scale=0.5]{{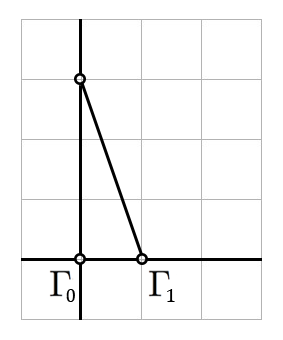}}
\caption{\label{15eplaatje} {\it{The Laplacian $\phi$ of $f$ in Example \ref{TripEllExa1}.}}}
\end{figure}
\begin{figure}[h!]
\centering
\includegraphics[scale=0.7]{{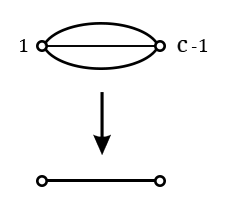}}
\caption{\label{14e2plaatje} {\it{The covering of graphs in Example \ref{TripEllExa1}.}}}
\end{figure}
We see that
\begin{equation*}
(f^{\Gamma_{1}})=(\overline{0})+(x,t-1)+(x,t-2)-3(\Gamma_{1}\cap\Gamma_{2})
\end{equation*}
and thus that there are 3 ramification points. Thus $g(\Gamma'_{1})=1$. For $\Gamma_{2}$ we have
\begin{equation*}
(f^{\Gamma_{2}})=Z(g)-(3+c)(\infty)+3(\Gamma_{1}\cap\Gamma_{2})
\end{equation*}
and thus $g(\Gamma'_{2})=c-1$. All in all, we can see that our reduction graph consists of 2 vertices with 3 edges between them. The corresponding weights on the vertices are $c-1$ and $1$. The corresponding covering of graphs can be found in Figure \ref{14e2plaatje}.
\end{exa}

Now, let us suppose that we are only interested in the reduction graph {\it{without}} the weights of the genera of the components. We'll do an example where we only address this problem.


\begin{exa}\label{TripHypExp1}
Suppose we take something like
\begin{equation*}
z^3=x(x-\pi)(x-2\pi)(x-\pi^2)(x-2\pi^2)(x-\pi^3)g(x),
\end{equation*}
where $c:=\text{deg}(g)$ is such that $6+c\neq{0}\mod{3}$.
We are now only interested in the {\it{unweighted}} graph corresponding to this curve. We create a semistable model with $R[x,t]/(xt-\pi^3)$ blown up two times. We have $4$ components $\Gamma_{0}$, $\Gamma_{1}$, $\Gamma_{2}$, $\Gamma_{3}$ where $\Gamma_{0}$ corresponds to $v(x)\leq{0}$, $\Gamma_{1}$ to $v(x)=1$, $\Gamma_{2}$ to $v(x)=2$ and $\Gamma_{3}$ to $v(x)\geq{3}$. We see that

\begin{eqnarray*}
0,(\pi^3)&\longmapsto&\Gamma_{3},\\
(\pi^2),(2\pi^2)&\longmapsto&\Gamma_{2},\\
(\pi),(2\pi)&\longmapsto&\Gamma_{1},\\
Z(g),(\infty)&\longmapsto&\Gamma_{0}.
\end{eqnarray*}
Our Laplacian is then
\begin{equation*}
\rho(\text{div}_{\eta}(f))=2(\Gamma_{3})+2(\Gamma_{2})+2(\Gamma_{1})-6(\Gamma_{0}).
\end{equation*}

The corresponding function has slope $6$ from $\Gamma_{0}$ to $\Gamma_{1}$, slope $-4$ from $\Gamma_{1}$ to $\Gamma_{2}$, slope $-2$ from $\Gamma_{2}$ to $\Gamma_{3}$, as in Figure \ref{16eplaatje}. We thus see that 
\begin{eqnarray*}
g_{\Gamma_{0}\cap\Gamma_{1}}&=&3,\\
g_{\Gamma_{1}\cap\Gamma_{2}}&=&1,\\
g_{\Gamma_{2}\cap\Gamma_{3}}&=&1,\\
\end{eqnarray*}
which determines the graph. It is a graph with 4 neighbouring vertices, two of which have 3 edges between them. The covering of graphs can be found in Figure \ref{17eplaatje}.
\begin{figure}[h!]
\centering
\includegraphics[scale=0.5]{{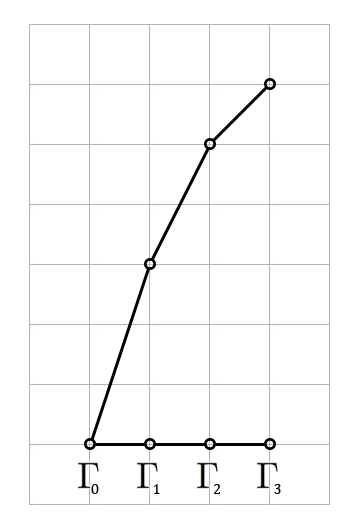}}
\caption{\label{16eplaatje} {\it{The Laplacian $\phi$ of $f$ in Example \ref{TripHypExp1}.}}}
\end{figure}
\begin{figure}[h!]
\centering
\includegraphics[scale=0.6]{{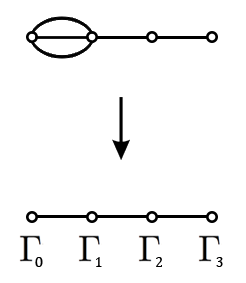}}
\caption{\label{17eplaatje} {\it{The covering of graphs in Example \ref{TripHypExp1}.}}}
\end{figure}
\end{exa}

\section{Tropical superelliptic coverings}

In this section, we give a short review of the results obtained in \cite{supertrop}. Here, a cyclic abelian covering $C\rightarrow{\mathbb{P}^{1}}$ with Galois group $G=\mathbb{Z}/n\mathbb{Z}$ is called a \emph{superelliptic covering}. 
\begin{theorem}
\label{realizabilitythm}
Let $p$ be a prime number. A covering $\phi_\Sigma:\Sigma \rightarrow T$ is a superelliptic covering of degree $p$ of weighted metric graphs if and only if there exists a superelliptic covering $\phi:C\rightarrow \mathbb{P}^1$ of degree $p$ tropicalizing to it. 
\end{theorem}
\begin{proof}
We give a sketch of the proof and direct the reader to \cite[Theorem 5.4]{supertrop}. If $C\rightarrow{\mathbb{P}^{1}}$ is superelliptic, then any superelliptic disjointly branched morphism $\mathcal{C}\rightarrow{\mathcal{D}}$ gives a superelliptic covering of weighted metric graphs. The converse is somewhat harder. The idea is to explicitly write down the covering equations coming from Proposition \ref{PropositionCoveringData} for a general $f=\prod_{i=0}^{r}(x-\alpha_{i})$ in terms of the reductions of the $\alpha_{i}$. One then shows that these almost linear equations have at least one solution. Any such solution then gives the desired $f$, providing us with the curve $C$ defined by $z^p=f$ and the superelliptic covering $(x,z)\mapsto{x}$.
\end{proof}
A similar result was proved for degree $d$ admissible coverings in \cite{admcov}: for every degree $d$ admissible covering of metric graphs $C_\Sigma \rightarrow T$, there exists an algebraic covering $C\rightarrow\mathbb{P}^{1}$ tropicalizing to $C_\Sigma \rightarrow T$. We note however that the covering obtained by this theorem is not necessarily {\it{Galois}}, whereas the obtained covering in Theorem \ref{realizabilitythm} is manifestly superelliptic. 
Unlike in \cite{admcov}, the approach given here is constructive; the proof of the realizability theorem presents a method for finding the defining equation of a curve $C$ with a superelliptic covering $C\rightarrow{\mathbb{P}^{1}}$. 

\begin{rem}
As noted in \cite{supertrop}, to obtain a certain superelliptic covering of graphs algebraically, one needs to consider $f$ where the roots $\alpha_{i}$ coincide. That is, we need to allow polynomials with multiple factors. It would be interesting to see what kind of further restrictions these graphs give in terms of embeddings into $\mathbb{P}^{2}$. 
\end{rem}

\begin{rem}
In the proof of Theorem \ref{realizabilitythm} one is confronted with the following convenient fact: there are too many solutions to the covering equations. The following problem would now be interesting to study: how many configurations for the divisor of the reduction $\rho((f))$ give rise to a certain covering? What are the asymptotics? A good place to start here would be to optimize the proof of Theorem  \ref{realizabilitythm} to give an explicit lower bound. 
\end{rem}

\chapter{$S_{3}$-coverings of the projective line}\label{Solvable}
In this chapter, we will use our methods to give an algorithm for finding the Berkovich skeleton of a curve $C$ admitting a degree three morphism to the projective line. These morphisms come in two flavors: they are either Galois or they are not. In the first case, Algorithm \ref{AbelianAlgorithm} can be used. If the morphism is not Galois, we then take the Galois closure of this morphism. The normalization $\overline{C}$ of $C$ in this field extension can be geometrically reducible and in that case we find that the curve admits a degree three abelian morphism over a quadratic extension of $K$. In this case, Algorithm \ref{AbelianAlgorithm} is again applicable.

We will be mostly interested in the case where $\overline{C}$ is geometrically irreducible. In that case, the morphism $\overline{C}\rightarrow{\mathbb{P}^{1}}$ is a Galois covering with Galois group $S_{3}$. In this chapter we will first find the Berkovich skeleton of this curve $\overline{C}$. Taking the quotient under a subgroup of order two then yields an intersection graph for the original curve, which gives the Berkovich skeleton of $C$ after deleting the leaves.

To illustrate these $S_{3}$-coverings, we will apply them to a natural degree three morphism on elliptic curves in a Weierstrass equation $y^2=x^3+Ax+B$. Here, the morphism is given by $(x,y)\mapsto{y}$. Using this morphism, we can give a new proof of the criterion:
\begin{equation}
E\text{ has potential good reduction if and only if }v(j)>0,
\end{equation}
where $j$ is the $j$-invariant of $E$.

After this, we will give an algorithm for finding the Berkovich skeleton of any genus three curve $C$. The techniques developed in this thesis in fact work for all curves up to genus seven, which we show in the last section.  

\section{Preliminaries}
We will assume in this chapter that the characteristic of the residue field is coprime to six. 

Let $C$ be a smooth, projective, geometrically irreducible curve over $K$ with a degree three covering  $\phi:C\longrightarrow{\mathbb{P}^{1}}$. This means that the injection of function fields 
\begin{equation}
K(\mathbb{P}^{1})\longrightarrow{K(C)}
\end{equation}
has degree three. We will write $K(\mathbb{P}^{1})=K(x)$ and $K(C)=L$ from now on. We can then find an element $z\in{L\backslash{K(x)}}$ 
that satisfies 
\begin{equation}
z^3+pz+q=0
\end{equation}
for $p,q\in{K(x)}$. We now assume that the corresponding {\it{Galois closure}} $\overline{L}$ has Galois group $S_{3}$ over $K(x)$. 
Let $\overline{C}$ be the normalization of $C$ in $\overline{L}$. We then have
\begin{lemma}\label{GeomIrr}
$\overline{C}$ is geometrically irreducible if and only $\Delta=4p^3+27q^2\notin{K}$.
\end{lemma} 
\begin{proof}
Note that being geometrically irreducible is equivalent to $\overline{L}\cap{\overline{K}}=K$ by \cite[Chapter 3, Corollary 2.14]{liu2}, where $\overline{K}$ is the algebraic closure of $K$. Note that $\overline{L}$ naturally contains the {\it{field}} (by assumption on the Galois group) $K(x)[y]/(y^2-\Delta)$. Suppose that $\overline{C}$ is geometrically reducible. Then there exists a $z_{0}\in\overline{L}$ such that $K(z_{0})\supset{K}$ is finite of degree $\neq{1}$. If $K(z_{0})/K$ is of degree $2$, we  reason as follows. For some $a,b\in{K}$, we have $a+bz_{0}=y$ (since there is only one subfield of degree two) and thus $y^2\in\overline{K}$. We then find $y^2\in{K(x)\cap{\overline{K}}}=K$, a contradiction. Now suppose that $K(z_{0})/K$ is of degree $3$. Then some conjugate $\sigma(z_{0})$ of $z_{0}$ belongs to $L$. But then $\sigma(z_{0})\in{L\cap{\overline{K}}}=K$, a contradiction. 

For the other direction, suppose that $\Delta\in{K}$. Then $y$ is an element of $(\overline{L}\cap{\overline{K}})\backslash{K}$. 
This contradicts our assumption on $\overline{L}$, finishing the proof.  
\end{proof}
For the remainder of the thesis, we assume that $\overline{C}$ is geometrically irreducible, which is quite an easy condition to check by Lemma \ref{GeomIrr}. 
The Galois closure $\overline{L}$ can now be described by the two equations
\begin{equation}
w^3=y-\sqrt{27}q
\end{equation}
and
\begin{equation}
y^2=\Delta.
\end{equation}
See Appendix \ref{Appendix1} for the details. These equations first arose in the famous Cardano formulas, where they are used to express $z$ in terms of the above radicals. We will not use these formulas in this paper, since the above equations are enough to derive all the information we need.

\section{Tame $S_{3}$-coverings of discrete valuation rings}

Let $R$ be a discrete valuation ring with quotient field $K$, residue field $k$, uniformizer $\pi$ and valuation $v$. Note that the residue field $k$ is \emph{not} assumed to be algebraically closed for this section, since we'll be using the results here for valuations corresponding to components in the special fiber of a semistable model. We will denote the maximal ideal $(\pi)$ in $R$ by $\mathfrak{p}$. We will assume that $\text{char}(K)=0$, $\text{char}(k)>3$ . Furthermore, we will assume that $K$ contains a primitive third root 
of unity $\zeta_{3}$ and a primitive fourth root of unity $\zeta_{4}$. The fourth root of unity is used to remove a minus sign in the formula for the discriminant (see Appendix \ref{Appendix1}), which is strictly speaking not necessary, but it makes the formulas somewhat nicer. 
The third root of unity will allow us to use Kummer theory for abelian coverings of degree $3$. 

Let $L$ be a degree $3$ extension of $K$ such that $\overline{L}/K$ has Galois group $S_{3}$. This is equivalent to the discriminant of $L/K$ not being a square in $K$.
After a translation, $L$ is given by an equation of the form
\begin{equation}\label{MainEq1}
z^3+p\cdot{z}+q=0,
\end{equation}
where $p,q\in{K}$.
 Let $B$ be the normalization of $R$ in $\overline{L}$ and let $\mathfrak{q}$ be any prime lying above $\mathfrak{p}$. 
We would now like to know the \emph{inertia group} of $\mathfrak{q}$. 
We will content ourselves with knowing $|I_{\mathfrak{q}}|$. 

Let us state the relevant results here and defer the actual proofs and computations to Appendix \ref{Normalizations}.  

\begin{pro}\label{InertS3}
\begin{enumerate}
\item Suppose that $3v(p)>2v(q)$. Then
\begin{eqnarray*}
|I_{\mathfrak{q}}|=3 &\iff & 3\nmid{v(q)},\\
|I_{\mathfrak{q}}|=1 & \iff & 3\mid{v(q)}.
\end{eqnarray*} 
\item Suppose that $3v(p)<2v(q)$. Then
\begin{eqnarray*}
|I_{\mathfrak{q}}|=2 &\iff & 2\nmid{v(p)},\\
|I_{\mathfrak{q}}|=1 & \iff & 2\mid{v(p)}.
\end{eqnarray*} 
\item Suppose that $3v(p)=2v(q)$. Then 
\begin{eqnarray*}
|I_{\mathfrak{q}}|=2 &\iff & 2\nmid{v(\Delta)},\\
|I_{\mathfrak{q}}|=1 & \iff & 2\mid{v(\Delta)}.
\end{eqnarray*} 
\end{enumerate}
\end{pro}
\begin{proof}
See Appendix \ref{Normalizations}. 
\end{proof}

\section{Covering data using continuity of inertia groups}\label{InertTechnique}

In this section, we give the covering data for the morphism of intersection graphs $\Sigma(\overline{\mathcal{C}})\rightarrow\Sigma({\mathcal{D}_{S}})$ associated to a disjointly branched morphism $\overline{\mathcal{C}}\rightarrow{\mathcal{D}_{S}}$ for the $S_{3}$-covering $\phi: \overline{C}\rightarrow{\mathbb{P}^{1}}$. Here $S$ is a subset of $\mathbb{P}^{1}(K)$ that contains the branch locus of $\overline{\phi}$, as in Chapter \ref{Appendix2}. Before we give the covering data, we will first give an \emph{explicit} $S\subset{\mathbb{P}^{1}(K)}$ that contains the branch locus, giving rise to a separating model $\mathcal{D}_{S}$ as explained in Appendix \ref{Appendix2}. After that, we will use Proposition \ref{InertS3} and Theorem \ref{InertProp2} to give the {\it{covering data}}. That is, we will give $|D_{x}|$, where $x$ corresponds to an edge or vertex in $\Sigma(\overline{\mathcal{C}})$.

For any $z\in\mathbb{P}^{1}$, let $v_{z}$ be the corresponding valuation of the function field $K(x)$. Consider the set $S:=\text{Supp}(p,q,\Delta)\subset{\mathbb{P}^{1}}$.  In terms of valuations, we then find that $z\in{S}$ if and only $v_{z}$ is nontrivial on $p,q$ or $\Delta$. 
From now on, we assume that  
$S\subset{\mathbb{P}^{1}(K)}$ (otherwise, we take a finite extension $K'$ of $K$ and set $K:=K'$). Let $B_{\overline{\phi}}$ be the branch locus of the morphism $\overline{\phi}:\overline{C}\rightarrow{\mathbb{P}^{1}}$. We then have
\begin{lemma}\label{BranchLocus1}
\begin{equation}
B_{\overline{\phi}}\subseteq{S}. 
\end{equation}
\end{lemma}
\begin{proof}
This follows from Proposition \ref{InertS3} and the characterization of $S$ in terms of valuations given before the lemma. 
\end{proof}

For $S=\text{Supp}(p,q,\Delta)$, we now take a model $\mathcal{D}_{S}$ of $\mathbb{P}^{1}$ such that the closure of $S$ in $\mathcal{D}_{S}$ consists of disjoint smooth sections over $R$. See Chapter \ref{Appendix2} for the construction of $\mathcal{D}_{S}$ and its corresponding intersection graph (which is also known as the {\it{tropical separating tree}}). Note that the morphism $\overline{\mathcal{C}}_{0}\rightarrow{\mathcal{D}_{S}}$ obtained by normalizing might not be disjointly branched, as the generic points of components in the special fiber of $\mathcal{D}_{S}$ can ramify. By Proposition \ref{InertS3}, we see that as soon as we know the valuations $v_{\Gamma}(p), v_{\Gamma}(q),v_{\Gamma}(\Delta)$ for a component $\Gamma$, we know what tamely ramified extension we have to take to make this morphism disjointly branched. 
We then obtain a disjointly branched morphism $\overline{\mathcal{C}}\rightarrow{\mathcal{D}'_{S}}$, where $\mathcal{D}'_{S}=\mathcal{D}_{S}\times{\text{Spec}(R')}$ for $R'$ a discrete valuation ring in $K'$ dominating $R$.

We now take a regular subdivision $\mathcal{D}''_{S}$ of $\mathcal{D}'_{S}$ in any edge $e$ and obtain a subdivision of the intersection graph $\Sigma(\mathcal{D}'_{S})$. As before, we note that the corresponding normalization of this model $\mathcal{D}''_{S}$ in $K(\overline{C})$ can then be vertically ramified over $\mathcal{D}''_{S}$. 
The good news here is that the inertia groups of the new components are directly related to the inertia group of the original edge $e$ by Theorem \ref{InertProp2}.   
We thus see by Proposition \ref{InertS3} that if we know the set $v_{\Gamma}(p), v_{\Gamma}(q),v_{\Gamma}(\Delta)$ for \emph{any} component in \emph{any} subdivision of our original intersection graph, then we know the inertia group of the original edge.   


We can find these valuations $v_{\Gamma}(p), v_{\Gamma}(q),v_{\Gamma}(\Delta)$ directly using the Laplacian operator and Theorem \ref{MainThmVert}. 

\begin{rem}\label{Remark1}
The valuation of $f$ at a component $\Gamma$ is exactly the coefficient in the vertical divisor corresponding to $\text{div}(f)$. We thus see that the above theorem gives the valuation, as soon as we know the valuation of $f$ at a single component. For $K(\mathcal{D})=K(x)$, this is quite easy: we take the valuation $v_{\Gamma_{0}}$ corresponding to the prime ideal $\mathfrak{p}=(\pi)\subset{R[x]}$. To be explicit, we write $f=\pi^{k}g$ (with $g\notin\mathfrak{p}$) and find $v_{\Gamma_{0}}(f)=k$. In other words, this valuation just measures the power of $\pi$ in $f$. 
\end{rem}
\begin{rem}
If we take any base change of the form $K\subset{K(\pi^{1/n})}$, then the corresponding Laplacians for $p,q$ and $\Delta$ are scaled by a factor $n$ (at least, if we normalize our valuation such that $v(\pi^{1/n})=1$).  
\end{rem}

We now summarize the above method for finding the covering data for a disjointly branched $S_{3}$-covering $\mathcal{C}\rightarrow{\mathcal{D}_{S}}$ corresponding to the $S_{3}$-covering $\overline{C}\rightarrow{\mathbb{P}^{1}}$. 
\begin{algo}
\begin{center}
{\bf{[Algorithm for the covering data using continuity of inertia groups]}}
\end{center}
\begin{flushleft}
Input: The polynomials $p,q,\Delta\in{K[x]}$. 
\end{flushleft}
\begin{enumerate}
\item Construct the tropical separating tree for the set $S=\text{Supp}(p,q,\Delta)\subset{\mathbb{P}^{1}(K)}$. 
\item For every root (and pole) $\alpha$ of $p$, $q$ and $\Delta$, determine $v_{\alpha}(p)$, $v_{\alpha}(q)$ and $v_{\alpha}(\Delta)$.
\item Find $v_{\Gamma_{0}}(p)$, $v_{\Gamma_{0}}(q)$ and $v_{\Gamma_{0}}(\Delta)$, as explained in Remark \ref{Remark1}. 
\item Determine the corresponding Laplacians of $p$, $q$ and $\Delta$. 
\item Use Theorems \ref{InertProp2}, \ref{DecompVert} and Proposition \ref{InertS3} to determine the covering data.
\end{enumerate}
Output: The covering data for the covering $\overline{C}\rightarrow{\mathbb{P}^{1}}$. 
\end{algo}

\begin{exa}\label{Examplecurve}
Let $C$ be the curve given by the equation
\begin{equation}
f(z)=z^3+p\cdot{z}+q=0
\end{equation}
for $p=x^3$ and $q=x^3+\pi^3$.
That is, we consider the field extension
\begin{equation}
K(x)\subset{K(x)[z]/(f(z))}
\end{equation}
and let $C\longrightarrow{\mathbb{P}^{1}}$ be the corresponding morphism of smooth curves. 
Let us find the divisors of $p$, $q$ and $\Delta=4p^3+27q^2=4x^9+27(x^3+\pi^3)^2$. 

\begin{figure}[h!]
\centering
\includegraphics[scale=0.3]{{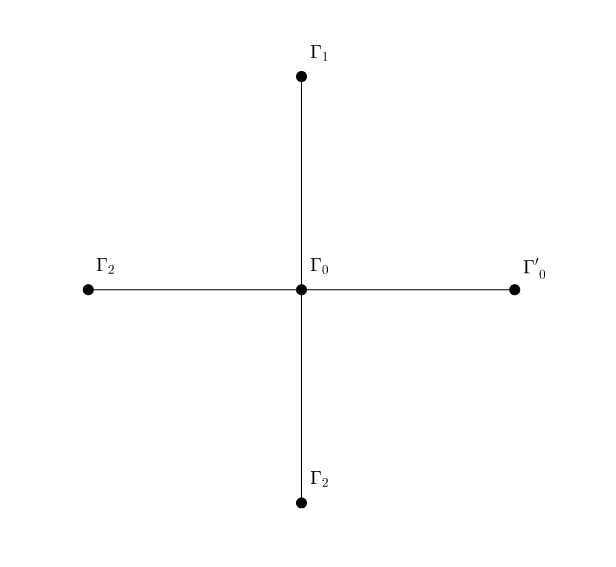}}
\caption{\label{1eplaatje} {\it{The tropical separating tree in Example \ref{Examplecurve}.}}}
\end{figure}

Let
\begin{align*}
P_{0}&=(0),\\
P_{i}&=(-\zeta^{i}_{3}\cdot{\pi}),
\end{align*}
for $i\in\{1,2,3\}$ and $\zeta_{3}$ a primitive third root of unity. Then $(p)=3P_{0}-3(\infty)$ and $(q)=P_{1}+P_{2}+P_{3}-3(\infty)$. 
We now take the tropical separating tree with five vertices, marked as in Figure \ref{1eplaatje}. 


We then have the following \emph{tropical} divisors:
\begin{align}
\rho((p))&=3\Gamma_{0}-3\Gamma'_{0},\\
\rho((q))&=\Gamma_{1}+\Gamma_{2}+\Gamma_{3}-3\Gamma'_{0}.
\end{align}
We quickly see that $p$ and $q$ contain no factors of $\pi$, so $v_{\Gamma'_{0}}(p)=v_{\Gamma'_{0}}(q)=0$. The Laplacians are then given by Figure \ref{Laplacianen1}. Note that the Laplacian $\phi_{p}$ is the same on every segment $e_{i}:=\Gamma_{i}\Gamma_{0}$ and likewise for $\phi_{q}$. We see that $\phi_{p}$ has slope zero on the $e_{i}$ and slope $3$ on $\Gamma_{0}\Gamma'_{0}$. Furthermore, we see that $\phi_{q}$ has slope $1$ on every $e_{i}$ and slope $3$ on $\Gamma_{0}\Gamma'_{0}$.  


\begin{figure}
\begin{subfigure}[b]{.45\textwidth}
  \centering
  \includegraphics[width=.5\linewidth, height=5cm]{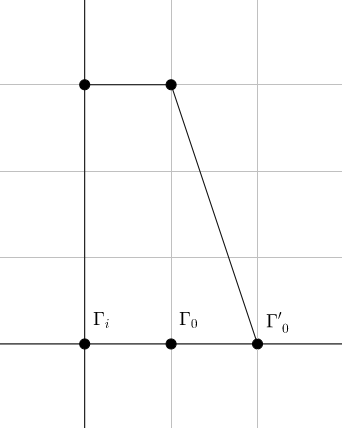}
  \caption{The Laplacian $\phi_{p}$ of $p$.}
  \label{2eplaatje}
\end{subfigure}%
\begin{subfigure}[b]{.45\textwidth}
  \centering
  \includegraphics[width=.7\linewidth, height=5cm]{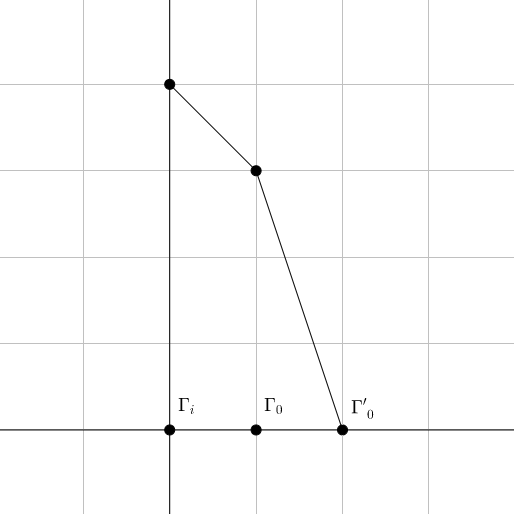}
  \caption{The Laplacian $\phi_{q}$ of $q$.}
  \label{3eplaatje}
\end{subfigure}
\caption{The Laplacians for Example \ref{Examplecurve}.}
\label{Laplacianen1}
\end{figure}

\begin{figure}[h!]
\centering
\includegraphics[scale=0.3]{{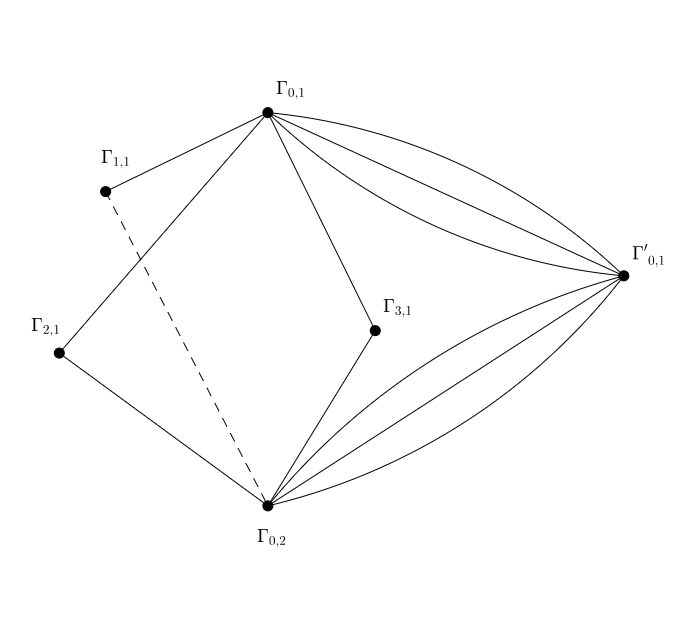}}
\caption{\label{5eplaatje} The intersection graph of the Galois closure in Example \ref{Examplecurve}.}
\end{figure}

For every vertex, we then have $3v_{\Gamma}(p)>2v_{\Gamma}(q)$, so we are in Case (I) of Theorem \ref{InertS3}. We then see that $\Gamma_{0}$ and $\Gamma'_{0}$ are unramified and every $\Gamma_{i}$ is ramified of order $3$. So we first take the tamely ramified extension of order three: $K\subset{K(\pi^{1/3})}$. If we now take a regular model after the base change, the Laplacians of both $p$ and $q$ will be scaled by a factor three (if we assume that our new valuation is normalized such that $v(\pi^{1/3})=1$). Furthermore, for this regular model we see that every edge gives rise to two new components. 
By Theorem \ref{InertProp2}, we have that the inertia groups of the new components on the subdivisions in fact give the inertia groups of the original edges.

Let us illustrate this in more detail. For instance, if we take the edge $e_{i}$, then after taking the base change we obtain four components: $\Gamma_{i}, v_{i,1},v_{i,2}$ and $\Gamma_{0}$, where $v_{i,1}$ and $v_{i,2}$ are new. We then find that the new Laplacian $\tilde{\phi}_{q}$ has 
\begin{align*}
\tilde{\phi}_{q}(\Gamma_{0})&=9, \\
\tilde{\phi}_{q}(v_{i,2})&=10, \\
\tilde{\phi}_{q}(v_{i,1})&=11,\\
\tilde{\phi}_{q}(\Gamma_{i})&=12.
\end{align*}

Again, using Theorem \ref{InertS3}, we find that $|I_{v_{i,2}}|=3$. By Theorem \ref{InertProp2}, we see that $|I_{e_{i}}|=3$. In other words, there are two edges lying above every $e_{i}$.
For $e_{0}=\Gamma_{0}\Gamma'_{0}$, using the same procedure as before, we see that there are six edges lying above $e_{0}$. Using Theorem \ref{DecompVert}, we see that there are two vertices lying above $\Gamma_{0}$. One then quickly finds that there is only one covering graph $\Sigma(\overline{\mathcal{C}})$ satisfying these conditions. It is given by Figure \ref{5eplaatje}.

We note that the genera of $\Gamma_{0,1}$, $\Gamma_{0,2}$ and $\Gamma'_{0,1}$ are one, whereas the genera of the other components are zero. This can be found using the Riemann-Hurwitz formula. 
\begin{figure}[h!]
\centering
\includegraphics[scale=0.3]{{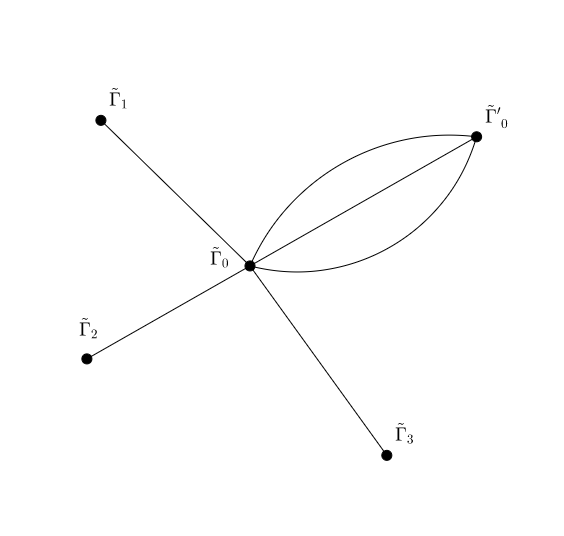}}
\caption{\label{6eplaatje} The intersection graph of the quotient under the subgroup of order two in Example \ref{Examplecurve}.}
\end{figure}

Taking the quotient under the subgroup of order two corresponding to the curve $C$, we then obtain the intersection graph of $\mathcal{C}$. It is given in Figure \ref{6eplaatje}. Note that the component labeled by $\tilde{\Gamma}_{0}$ has genus $1$, whereas $\tilde{\Gamma}'_{0}$ has genus $0$. The other three components don't contribute to the Berkovich skeleton. The entire Galois lattice, including all the intermediate intersection graphs but excluding the leaves, can now be found in Figure \ref{21eplaatje}. 

\begin{figure}[h!]
\centering
\includegraphics[scale=0.45]{{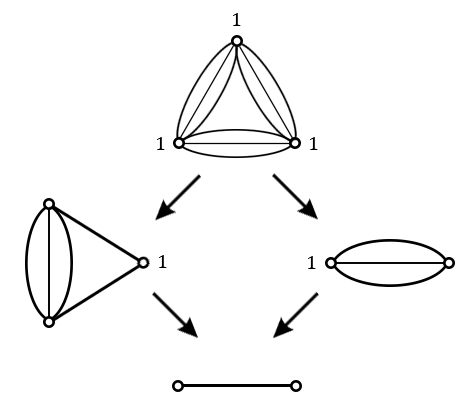}}
\caption{\label{21eplaatje} The full Galois lattice of intersection graphs in Example \ref{Examplecurve}. The leaves are omitted. }
\end{figure}

\end{exa} 

In Example \ref{Gen3NonAb1}, we will do the same example with a different technique. This technique will be presented in the next section. 

\section{Covering data using the quadratic subfield}\label{QuadraticSubfieldTechnique}


The degree three covering $\phi:C\longrightarrow{\mathbb{P}^{1}}$ can be represented on the level of function fields as
\begin{equation}
z^3+p\cdot{}z+q=0,
\end{equation}
where $p$ and $q$ are polynomials over $K$. By our initial assumption on $\phi$, we find that the Galois closure 
contains a quadratic subfield $K(D)$, corresponding to a smooth curve $D$. On the level of function fields, this is given as
\begin{equation}
K(x)\subset{}K(x)[y]/(y^2-\Delta),
\end{equation}
where $\Delta=4p^3+27q^2$.  

The corresponding degree three morphism $\overline{C}\rightarrow{D}$ can then be represented by
\begin{equation}
K(D)\subset{K(D)[w]/(w^3-(y-\sqrt{27}q))}.
\end{equation}

We thus see that the field extension $K(\overline{C})\supset{K(\mathbb{P}^{1})}$ has been subdivided into two abelian parts: $K(\overline{C})\supset{K(D)}$ of degree $3$ and $K(D)\supset{K(\mathbb{P}^{1})}$ of degree $2$. See Appendix \ref{Appendix1} for some background material regarding these equations.

Consider a model $\mathcal{D}_{S}$ of $\mathbb{P}^{1}$ such that the closure of $S:=\text{Supp}(p,q,\Delta)$ 
 is separated in $\mathcal{D}_{S}$. 
 We find by Lemma \ref{BranchLocus1} that the branch locus of $\overline{\phi}$ is contained in $S$. Over some finite extension $K'\supset{K}$, we thus obtain disjointly branched morphisms $\overline{\mathcal{C}}\rightarrow{\mathcal{D}}\rightarrow{\mathcal{D}_{S}}$  such that the base change to the generic fiber is $\overline{C}\rightarrow{D}\rightarrow{\mathbb{P}^{1}}$.   
 We won't worry about this finite extension in this section and just take $K:=K'$.   
 We first calculate the intersection graph of the intermediate model $\mathcal{D}$. 
 The covering $D\rightarrow{\mathbb{P}^{1}}$ is hyperelliptic, so we can apply Algorithm \ref{AbelianAlgorithm}. 
 We note that this step does not require any twisting data, since we are dealing with an abelian covering of a tree.

We now consider the divisor of the function $y-\sqrt{27}q\in{K(D)}$. This can be given explicitly in terms of the zero divisors of $p,q$ and $\Delta$. Since calculating divisors is a matter of normalizing, the reader will probably not be surprised that there are again three cases. The result is as follows, where we again defer the proof to Appendix \ref{Normalizations}.
\begin{pro}\label{DivisorDegree3}
Let $y-\sqrt{27}q\in{K(D)}$ be as above. Let $z\in\mathbb{P}^{1}(K)$ and denote by $v_{z}$ the corresponding valuation of $K(x)$.
\begin{enumerate}
\item Suppose that $3v_{z}(p)>2v_{z}(q)$. There are then two points $Q_{1}$ and $Q_{2}$ lying above $P$ in $D$. We then have
\begin{align*}
v_{Q_{1}}(y-\sqrt{27}q)&=3v_{z}(p)-v_{z}(q),\\
v_{Q_{2}}(y-\sqrt{27}q)&=v_{z}(q).
\end{align*}
\item Suppose that $3v_{z}(p)<2v_{z}(q)$. If $2|v_{z}(p)$, then there are two points $Q_{1}$ and $Q_{2}$ lying above $P$ in $D$. We have
\begin{equation}
v_{Q_{i}}(y-\sqrt{27}q)=3v_{z}(p)/2.
\end{equation}
If $2\nmid{v_{z}(p)}$, then there is only one point $Q$ in $D$ lying above $P$. We then have
\begin{equation}
v_{Q}(y-\sqrt{27}q)=3v_{z}(p).
\end{equation}
\item Suppose that $3v_{z}(p)=2v_{z}(q)$. If $2|v_{P}(\Delta)$, then there are two points $Q_{1}$ and $Q_{2}$ lying above $P$. We have
 \begin{equation}
 v_{Q_{i}}(y-\sqrt{27}q)=v_{z}(q).
 \end{equation}
 If $2\nmid{v_{z}(\Delta)}$, then there is only one point $Q$ lying above $P$. We have
 \begin{equation}
 v_{Q}(y-\sqrt{27}q)=2v_{z}(q).
 \end{equation}
\end{enumerate}
\end{pro}
\begin{proof}
See Appendix \ref{Normalizations}. 
\end{proof}

By calculating the reduction of every $Q_{i}$ in $\mathcal{D}$, one then obtains the tropical divisor $\rho(\text{div}(y-\sqrt{27}q))$. This is a principal divisor in $\text{Div}^{0}(\Sigma(\mathcal{D}))$, so we can write 
\begin{equation}
\Delta(\phi)=\rho(\text{div}(y-\sqrt{27}q))
\end{equation}
for some $\phi$ in $\mathcal{M}(\Sigma(\mathcal{D}))$. The covering data for an edge $e\in\Sigma(\mathcal{D})$ is then obtained as follows:
\begin{pro}\label{QuadraticSubfield}
Let $\Sigma(\mathcal{D})$ be the intersection graph of $\mathcal{D}$ and let $\phi$ be such that 
\begin{equation}
\Delta(\phi)=\rho(\text{div}(y-\sqrt{27}q)).
\end{equation} 
Let $e$ be an edge in $\Sigma(\mathcal{D})$ and let $\delta_{e}(\phi)$ be the absolute value of the slope of $\phi$ along $e$. 
Then the following hold:
\begin{enumerate}
\item There are three edges above $e$ if and only if $3|\delta_{e}(\phi)$. 
\item There is one edge above $e$ if and only if $3\nmid{\delta_{e}(\phi)}$.
\end{enumerate}
Furthermore, there are three vertices above a vertex $v$ with corresponding component $\Gamma$ if and only if the reduction of 
$(y-\sqrt{27}q)^{\Gamma}$ is a cube. 
\end{pro}
\begin{proof}
The first part is Proposition \ref{PropositionCoveringData}. The second part about the vertices is recorded in Proposition \ref{DecompositionVertexProposition}. 

\end{proof}


We now summarize the above method for finding the covering data for a tame $S_{3}$-covering $\mathcal{C}\rightarrow{\mathcal{D}_{\mathbb{P}^{1}}}$ using the quadratic subfield $K(D)$. 
\begin{algo}
\begin{center}
{\bf{[Algorithm for the covering data using the quadratic subfield]}}
\end{center}
\begin{flushleft}
Input: The polynomials $p,q,\Delta\in{K[x]}$. 
\end{flushleft}
\begin{enumerate}
\item Construct the tropical separating tree corresponding to $S=\text{Supp}(p,q,\Delta)$. 
\item Calculate the intersection graph $\Sigma(D)$ using the disjointly branched morphism $\mathcal{D}\rightarrow{\mathcal{D}_{\mathbb{P}^{1},S}}$.
\item Calculate the Laplacian of $y-\sqrt{27}q$ on $\Sigma(D)$ using Proposition \ref{DivisorDegree3}.
\item Calculate the covering data for the edges using Proposition \ref{QuadraticSubfield}.
\item If $\text{div}(\overline{f^{\Gamma}})\equiv{0}\mod{3}$, determine if it is a cube in $k(\Gamma)$. (This requires additional computations on the residue fields of the components of $\mathcal{D}_{s}$, these are given in Section \ref{Appendix3}). 
\item If $\overline{f^{\Gamma}}$ is a cube, there are three components lying above $\Gamma$. Otherwise, there is only one component lying above $\Gamma$.
\end{enumerate}
Output: The covering data for the covering $\overline{C}\rightarrow{\mathbb{P}^{1}}$. 
\end{algo}

\section{Tropicalizing degree three morphisms to the projective line: an algorithm}

In this section, we assemble the pieces from the previous sections into an algorithm for calculating the Berkovich skeleton of a curve with a degree three covering to the projective line. There are actually two algorithms for the covering data, so the reader can choose whichever method he prefers. The author is under the impression that the method presented in Section \ref{InertTechnique} (using inertia groups) is faster than the one in Section \ref{QuadraticSubfieldTechnique} (using the quadratic subfield), since it doesn't require any Laplacian computations on nontrivial graphs.  
\begin{algo}\label{Algorithm}
{}
\begin{center}
{\bf{[The Berkovich skeleton of a curve with a degree three covering to the projective line]}}
\end{center}
\begin{flushleft}
Input: $p,q\in{K[x]}$.
\end{flushleft}
\begin{enumerate}
\item Let $C$ be given by the equation $z^3+pz+q=0$. If the equation is reducible, then the covering does not have degree three. 
\item If $\Delta=4p^3+27q^2\in{K}$, then $C\rightarrow{\mathbb{P}^{1}}$ is superelliptic over a quadratic extension (namely $K(\sqrt{\Delta})$) of $K$. 
Use Algorithm \ref{AbelianAlgorithm} to determine the Berkovich skeleton. 
Otherwise, the Galois closure $\overline{C}$ is geometrically irreducible. It is described by the equation $w^6+2\sqrt{27}qw^3-4p^3=0$.
\item Construct the tropical separating tree $\Sigma(\mathcal{D}_{S})$ for the semistable model $\mathcal{D}_{S}$ as described in Chapter \ref{Appendix2}. Here $S=\text{Supp}(p,q,\Delta)$.
\item Determine the covering data for $\overline{C}\rightarrow{\mathbb{P}^{1}}$ using Section \ref{InertTechnique} or \ref{QuadraticSubfieldTechnique}.
\item Determine the twisting data using Algorithm \ref{Linkingcomponents} and use this to determine the intersection graph of $\overline{\mathcal{C}}$.
\item Calculate the genera of the vertices in $\Sigma(\overline{\mathcal{C}})$ using the Riemann-Hurwitz formula, see Proposition \ref{RiemannHurwitz}.   
\item Take the quotient of $\Sigma(\overline{\mathcal{C}})$ under the subgroup of order two corresponding to $C$ by Galois theory. The resulting graph $\Sigma(C)$ is the intersection graph of $\mathcal{C}$ by Lemma \ref{MainQuotientLemma1}.
\item Calculate the genera of the vertices in $\Sigma(\mathcal{C})$ using the Riemann-Hurwitz formula \ref{RiemannHurwitz}. 
\item Calculate the lengths of the edges in $\Sigma(\mathcal{C})$ using Proposition \ref{InertiagroupIntersectionPoint1}. 
\item Contract any "leaves" to obtain the graph $\Sigma'(\mathcal{C})$.
\end{enumerate}
\begin{flushleft}
Output: The Berkovich skeleton $\Sigma'(\mathcal{C})$.
\end{flushleft}
\end{algo}
\begin{proof}
({\it{Correctness of the algorithm}})
For the covering data and the twisting data, we refer the reader to Sections \ref{Coveringdata} and \ref{Linkingcomponents}. The fact that the quotient graph is equal to the intersection graph of the quotient is Lemma \ref{MainQuotientLemma1}. 
 Contracting any leaves then automatically yields the Berkovich skeleton.  
\end{proof}

\begin{figure}[h!]
\centering
\includegraphics[scale=0.3]{{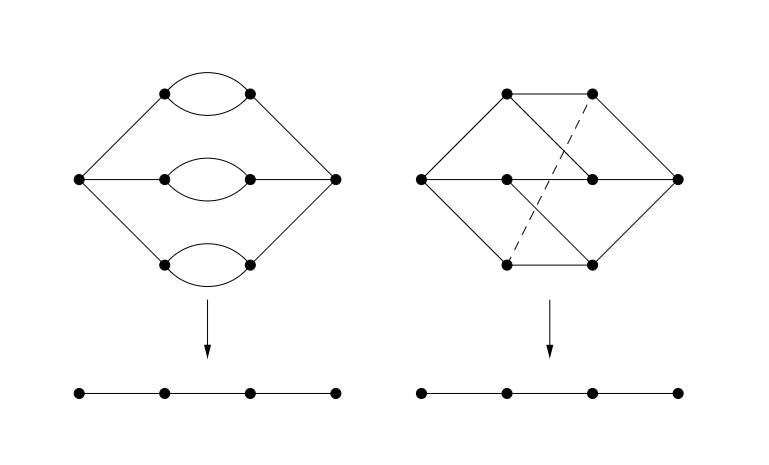}}
\caption{\label{83eplaatje} {\it{Two different $S_{3}$-coverings with the same covering data.}}}
\end{figure}

\begin{rem}
We would like to point out a difference here between abelian coverings of the projective line and nonabelian coverings of the projective line. First of all, the twisting data is not needed for abelian covers of the projective line. It is only necessary for coverings $\mathcal{C}\rightarrow{\mathcal{D}}$ where $\Sigma(\mathcal{D})$ has nonzero Betti number and the covering is completely decomposable, as in Section \ref{CompletelyDecomposable}. We note however that the covering data for an $S_{3}$-covering $\Sigma(\mathcal{C})\rightarrow{T}$ of a tree $T$ does \emph{not} fix $\Sigma(\mathcal{C})$. See Figure \ref{83eplaatje} for an example. Luckily, the nontrivial twisting can be detected on the quadratic subcover.  
\end{rem}

\section{Semistability of elliptic curves using a degree three covering}

As an application of the methods presented in the previous sections, we reprove the criterion:
\begin{equation}
"v(j)>0 \text{ if and only if }E\text{ has potential good reduction}".
\end{equation}
Here $j$ is the $j$-invariant of $E$. For the definitions of (potential) good reduction and (split) multiplicative reduction, we refer the reader to \cite[Chapter 10, Definition 2.27]{liu2}.  

Let us take an elliptic curve $E/K$. Over an extension $K'$ of $K$, one can then find an equation of the form
\begin{equation}\label{EquationElliptic}
x^3+Ax+B+y^2=0
\end{equation}
 for some $A$ and $B$ in $K'$. Just as in \cite[Chapter VII, Section 1]{Silv1}, we can assume that $v(A),v(B)\geq{0}$. In fact, we can assume that the equation has been scaled such that either $v(A)=0$ or $v(B)=0$, which again often requires a finite extension. We will assume that all these extensions have been made and the resulting field will be denoted by $K$. 
To prove semistability of the curve, one usually considers the $2:1$ covering given by
\begin{equation*}
\phi(x,y)=x
\end{equation*}
and then uses the branch points to explicitly create the semistable model. We will make life hard for us now and consider a different covering:
\begin{equation*}
\phi(x,y)=y.
\end{equation*}
This gives a degree three morphism $E\longrightarrow{\mathbb{P}^{1}}$ with corresponding extension of function fields $K(y)\subset{K(E)}$. We will use the quadratic subfield of the Galois closure and Algorithm \ref{AbelianAlgorithm} for the degree three abelian extension $\overline{E}\rightarrow{E'}$. 
The twisting data studied in Section \ref{TwistingDataFinal} 
will not be needed, as we will see that the covering data obtained here determines the covering graph uniquely. 

We note that the curve in Equation \ref{EquationElliptic} is in our normal form with $p=A$ and $q=B+y^2$. For psychological reasons, the author chose to revert the minus sign coming from the usual Weierstrass equation (given by $x^3+Ax+B-y^2=0$) to a plus sign. 

Consider the $K(A,B)[y]$-algebra $K(A,B)[y][x]/(x^3+Ax+B+y^2)$. We first calculate the discriminant of this algebra. It is given by
\begin{equation*}
\Delta=4A^3+27(B+y^2)^2.
\end{equation*}
We would like to determine whether this is a square or not. To that end, we calculate the discriminant of
\begin{equation*}
\Delta'(y_{1})=4A^3+27(B+y_{1})^2
\end{equation*}
and see that 
\begin{equation*}
\Delta(\Delta'(y_{1}))=(2\cdot{27}\cdot{B})^2-4\cdot{(27)}\cdot{(27B^2+4A^3)}=-(4\cdot{27})^2\cdot{A}^3.
\end{equation*}
Here $y_{1}=y^2$. 
We therefore see that the discriminant $\Delta$ is a square if and only if either $A=0$ or $y=0$ is a zero of $\Delta$. In the latter case we see directly that we must have $4A^3+27B^2=0$, which contradicts the assumption that $E$ is nonsingular. The case $A=0$ is a separate case, where one can easily see that $E$ has potential good reduction. 

So let us assume that $A\neq{0}$. Then the discriminant is not a square and we obtain a bonafide extension of degree two given by
\begin{equation*}
z^2=4A^3+27(B+y^2)^2.
\end{equation*}
This is again a curve of genus $1$, which we denote by $E'$. We would like to know the reduction type of this curve. We will do this in terms of the discriminant $\Delta(E)=4A^3+27B^2$. Note that the $E$'s equation has been scaled such that either $v(A)=0$ or $v(B)=0$. 

We now consider the following possible scenarios for $A,B$ and $\Delta(E)$:
\begin{enumerate}\label{ScenariosEllipticCurve1}
\item $v(A)=v(B)=0$ and $v(\Delta(E))>0$.
\item $v(A)=0$, $v(B)\geq{0}$ and $v(\Delta(E))=0$.
\item $v(A)>0$, $v(B)=0$ and $v(\Delta(E))=0$.
\end{enumerate}
\begin{lemma}\label{Onecase1}
Every elliptic curve $E/K$ belongs to exactly one of the three cases described above.
\end{lemma}
\begin{proof}
Let $E$ be given as in Equation \ref{EquationElliptic}. If $v(A)>0$, then by assumption we must have $v(B)=0$ and thus $v(\Delta)=0$. This means that we are in Case 3. Suppose that $v(A)=0$. If $v(B)>0$, then $v(\Delta)=0$ and we are in Case 2. If $v(B)=0$, then there are two possibilities: either $v(\Delta)>0$ or $v(\Delta)=0$. These are cases 1 and 2 respectively. It is now clear from the nature of these cases that they are mutually exclusive. This finishes the proof.  
\end{proof}


\begin{theorem}\label{ScenariosEllipticCurve2}
Let $\overline{E}$ be the Galois closure of the morphism $\phi$. 
For every type of $\{v(A),v(B),v(\Delta)\}$ as described above, there exists a disjointly branched morphism  $\overline{\mathcal{E}}\rightarrow{\mathcal{D}_{\mathbb{P}^{1}}}$ giving the following intersection graphs:  
\begin{enumerate}
\item Suppose that $v(A)=v(B)=0$ and $v(\Delta(E))>0$. Then $E'$ has multiplicative reduction with intersection graph $\Sigma(\mathcal{E}')$ consisting of two components intersecting in two points. $\Sigma(\overline{\mathcal{E}})$ consists of 3 copies of $\Sigma(\mathcal{E}')$ meeting in one vertex.
 The corresponding intersection graph $\Sigma(\mathcal{E})$ consists of three vertices, connected as in Figure \ref{101eplaatje}. The curve $E$ has multiplicative reduction. 
\item Suppose that $v(A)=0$, $v(B)\geq{0}$ and $v(\Delta(E))=0$. Then all curves involved are nonsingular and the corresponding models have the trivial intersection graph. 
\item Suppose that $v(A)>0$, $v(B)=0$ and $v(\Delta(E))=0$. Then $E'$ has multiplicative reduction with intersection graph $\Sigma(\mathcal{E}')$ consisting of two components intersecting in two points.  $\Sigma(\overline{\mathcal{E}})$ consists of two elliptic curves meeting twice. $E$ has good reduction, with intersection graph $\Sigma(\mathcal{E})$ as described in Figure \ref{92eplaatje}. 
\end{enumerate}
\end{theorem}
\begin{proof}
We subdivide the proof into three parts, according to the cases given in the statement of the proposition.
\begin{enumerate}
\item 
We write 
\begin{equation*}
z^2=4A^3+27B^2+2\cdot{27}By^2+27y^4=\Delta(E)+2\cdot{27}By^2+27y^4=\Delta.
\end{equation*}
We label the roots of $\Delta$ by $\alpha_{i}$ for $i\in\{1,2,3,4\}$. 
Since $v(A)=v(B)=0$ and $v(\Delta(E))>0$, we find that two of the roots of $\Delta$ coincide. Let them be $\alpha_{1}$ and $\alpha_{2}$. A Newton polygon computation then shows that  $v(\alpha_{1})=v(\alpha_{2})=v(\Delta)/2$. We therefore construct a tropical separating tree $\Sigma(\mathcal{D}_{S})$ with two vertices and one edge, which has length $v(\Delta)/2$. The reduction graph of $E'$ is then as shown in Figure \ref{101eplaatje} (which contains some spoilers regarding the final product). Indeed, the Laplacian of $\Delta(E)+2\cdot{27}By^2+27y^4$ has slope $\pm{2}$ on $e$, so we obtain two edges. Furthermore, there is only one vertex lying above each vertex of $\Sigma(\mathcal{D}_{S})$, since the roots of $\Delta$ are branch points on these components. We label these components by $\Gamma_{1}$ and $\Gamma_{2}$.  

We now consider the degree three covering $\overline{E}\rightarrow{E'}$. We'll use the formulas in Proposition \ref{DivisorDegree3}. Let $f=z-\sqrt{27}q$. We then easily see that
\begin{equation*}
\text{div}_{\eta}(f)=2\cdot({\infty_{1}})-2\cdot({\infty_{2}}),
\end{equation*}
where the $\infty_{i}$ are the two points at infinity. These points both reduce to smooth points on $\Gamma_{1}$. The corresponding Laplacian is thus trivial on $\Sigma(E')$. Using Proposition \ref{PropositionCoveringData}, 
we see that there are three edges lying above each of the two in $\Sigma(E')$.  

We now turn to the vertices. The reduced divisor of $f$ on $\Gamma_{2}$ is trivial, so there are three vertices lying above it. For $\Gamma_{1}$, the covering is ramified and thus there is only vertex lying above it. By the Riemann-Hurwitz formula, the covering vertex has genus zero. We thus directly see that the intersection graph must be as in Figure \ref{101eplaatje}. The corresponding quotient and the rest of the lattice is depicted there as well. We see that the Betti number of the quotient is one, implying that $E$ has multiplicative reduction. 

Note that the length of the cycle in $\Sigma(\mathcal{E})$ is the same as the length of the cycle in $\Sigma(\mathcal{E}')$ by inspecting the corresponding inertia groups. From the construction of the tropical separating tree, we then find that the cycle in $\Sigma(\mathcal{E}')$ has length $v(\Delta)/2+v(\Delta)/2=v(\Delta)=-v(j)$ and thus the cycle in $\Sigma(\mathcal{E})$ has the same length. This is another well-known feature of elliptic curves with split multiplicative reduction.  

\begin{figure}[h!]
\centering
\includegraphics[scale=0.3]{{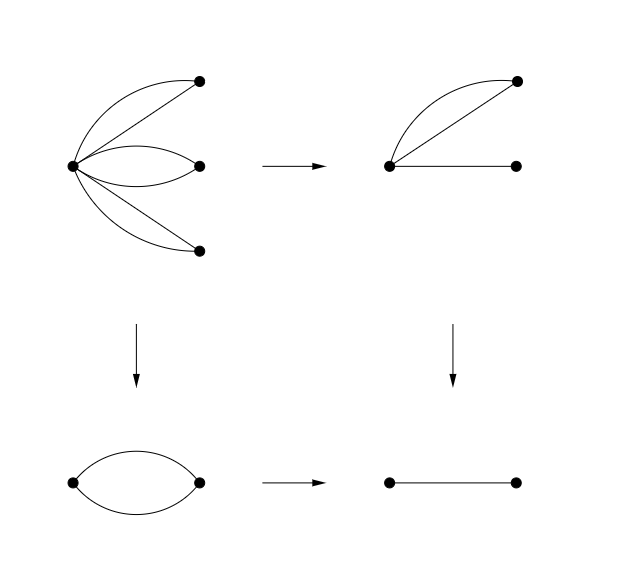}}
\caption{\label{101eplaatje} The Galois closure of graphs in Case I.}
\end{figure}

\item A quick calculation shows that all curves in sight are nonsingular. We thus obtain trivial graphs with weights $1,3,1$. Hence $E$ has good reduction. 

\item Suppose that $v(A)>0$, $v(B)=0$ and $v(\Delta(E))=0$. ({\footnote{The author has to confess that this feels like we're using too much machinery, because we already know that $E$ has good reduction from the fact that the reduced discriminant is nonzero. Nonetheless, calculating the entire Galois closure shows some interesting features.}})

\begin{figure}[h!]
\centering
\includegraphics[scale=0.3]{{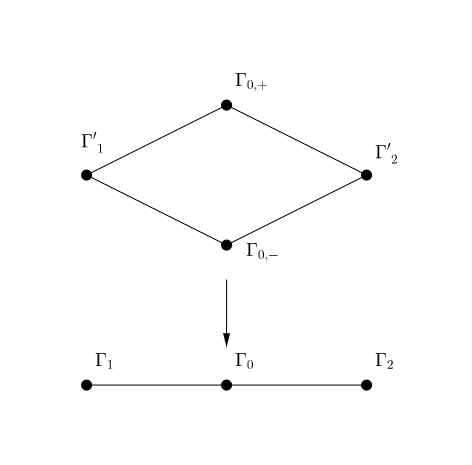}}
\caption{\label{91eplaatje} The hyperelliptic covering in Case III.}
\end{figure}

We label the roots of $\Delta$ by $\alpha_{i}$ for $i\in\{1,2,3,4\}$. We quickly find that the roots reduce to roots of the equation $y^2+B=0$. We have two roots (say $\alpha_{1}$ and $\alpha_{2}$) reducing to $y=\sqrt{-B}$ and two other roots ($\alpha_{3}$ and $\alpha_{4}$) reducing to $y=-\sqrt{-B}$. A Newton polygon calculation then shows that $v(\alpha_{1}-\alpha_{2})=3v(A)/2$ and $v(\alpha_{3}-\alpha_{4})=3v(A)/2$. We therefore construct a tropical separating tree with three vertices $\Gamma_{0}$, $\Gamma_{1}$ and $\Gamma_{2}$ as in Figure \ref{91eplaatje}. This figure also contains the corresponding degree two covering of intersection graphs. Note that the edges $\Gamma_{1}\Gamma_{0}$ and $\Gamma_{0}\Gamma_{2}$ both have length $3v(A)/2$. Since the morphism of intersection graphs is \'{e}tale above these edges, we find that the edges lying above them also have length $3v(A)/2$.

\begin{figure}[h!]
\centering
\includegraphics[scale=0.3]{{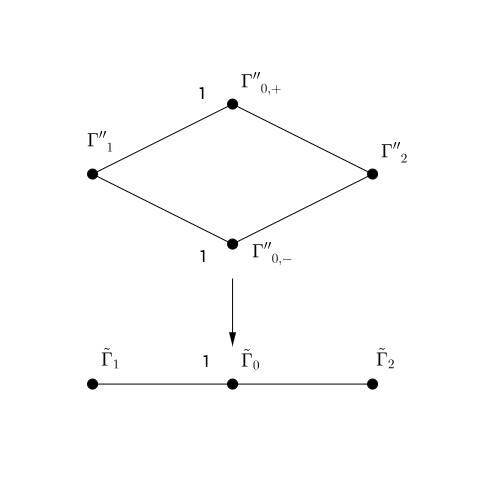}}
\caption{\label{92eplaatje} The Galois closure and its quotient in Case III. }
\end{figure}


Since $\Sigma(\mathcal{D})$ has Betti number one, we see that $E'$ has multiplicative reduction. Note that there are two points at infinity $\infty_{i}$ and that they reduce to different components: $\infty_{1}\mapsto{\Gamma_{0,+}}$ and $\infty_{2}\mapsto{\Gamma_{0,-}}$. Let $f=y-\sqrt{27}q$. We now find that $\text{div}_{\eta}(f)=2(\infty_{1})-2(\infty_{2})$. The reduction of the divisor of $f$ is then given by 
\begin{equation*}
\rho(\text{div}_{\eta}(f))=2(\Gamma_{1})-2(\Gamma_{2}).
\end{equation*}
We then find that the Laplacian corresponding to $f$ has slope $\pm{1}$ on every edge. It is not divisible by three, so there is only edge lying above every edge, again by Proposition \ref{PropositionCoveringData}. 
The Galois closure can now be found in Figure \ref{92eplaatje}. The components $\Gamma''_{0,i}$ both have genus one. Indeed, the morphisms $\Gamma''_{0,i}\rightarrow{\Gamma'_{0,i}}$ are ramified above the edges and $\infty_{i}$, giving a total of three ramification points per component. The Riemann-Hurwitz formula then gives the desired genus. We now see that the quotient consists of three vertices, where the middle vertex has genus $1$. In other words, $E$ has good reduction. 

\end{enumerate}
\end{proof}

\begin{cor}
An elliptic curve $E$ has potential good reduction if and only if $v(j)>0$. 
\end{cor}
\begin{proof}
Every elliptic curve can be put in exactly one of the three scenarios considered in Proposition \ref{ScenariosEllipticCurve2} by Lemma \ref{Onecase1}. 
Case 1 corresponds exactly to $v(j)<0$ by the calculation $v(j)=v(1728\cdot\dfrac{4A^3}{\Delta})=-v(\Delta)<0$. Cases 2 and 3 then naturally correspond to $v(j)\geq{0}$, giving the Corollary.  
\end{proof}

\section{Genus three curves}

We now turn to genus 3 curves. For genus 3 curves, we have that the moduli space of isomorphism classes has dimension $6$ (in general, $\mathcal{M}_{g}$ is irreducible of genus $3g-3$).  If we look at the subspace of all hyperelliptic curves of genus $3$, one quickly finds that this space has dimension $5$. The idea is that for $\text{char}{K}\neq{2}$, one can locally write such a curve as
\begin{equation*}
y^2=x(x-1)(x-\alpha_{1})(x-\alpha_{2})(x-\alpha_{3})(x-\alpha_{4})(x-\alpha_{5})
\end{equation*}
by putting three of the ramification points of the hyperelliptic involution at $\{0,1,\infty\}$. Since the dimension of $\mathcal{M}_{3}$ is strictly bigger than the dimension of the hyperelliptic locus, we find that not all curves of genus 3 have a hyperelliptic involution. So a different strategy is needed here. We will soon find out that one can in fact find a morphism of degree $3$ to $\mathbb{P}^{1}$ for curves that are not hyperelliptic. Such a morphism need not be Galois however, so we take the Galois closure of this morphism. 
We will see quite quickly that the Galois subextension of degree $3$ is often {\it{unramified}}. This means that it comes from a $3$-torsion point in the Jacobian. 

\subsection{From quartics to degree three morphisms}

Suppose we take a nonhyperelliptic curve of genus $3$. By \cite[Chapter IV, Proposition 5.2]{Hart1}, 
we find that the canonical divisor on $C$ defines a closed embedding
\begin{equation*}
C\longrightarrow{\mathbb{P}^{2}},
\end{equation*}
which has degree $4$, meaning that it is a nonsingular quartic. Conversely, every nonsingular quartic defines a nonhyperelliptic curve of genus $3$. We now take a point $P$ on $C$ (which might need a finite extension of $K$). Consider the space of all lines intersecting $P$. This is isomorphic to $\mathbb{P}^{1}$. If we now take any other point $Q\in{C}(\overline{K})$, we have that there is a line intersecting $Q$ and $P$. We define
\begin{equation*}
\phi(Q)=L_{P,Q},
\end{equation*} 
where $L_{P,Q}$ is the line connecting the two points. This defines a morphism $\phi:C\longrightarrow{\mathbb{P}^{1}}$ of degree $3$, since any hyperplane section intersects $C$ in four points and we already have $P$ as an intersection point.
\begin{exa}
As an example, take the plane curve defined by
\begin{equation*}
x^4+y^4-1=0.
\end{equation*}
It has the rational point $P=(1,0)$. Consider all lines of the form
\begin{equation*}
y=t(x-1).
\end{equation*}
By plugging this in, we obtain
\begin{equation*}
x^4-1+t^4(x-1)^4=0
\end{equation*}
and thus
\begin{equation*}
(x-1)(x^3+x^2+x+1+t^4(x-1)^3)=0.
\end{equation*}
We cancel out $x-1$ (the obvious intersection point), to obtain
\begin{equation*}
x^3+x^2+x+1+t^4(x-1)^3=0.
\end{equation*}
This curve has an obvious degree $3$ morphism to $\mathbb{P}^{1}$, given locally by
\begin{equation*}
(x,t)\longmapsto{t}.\\
\end{equation*}
\end{exa}

Let us now return to the general case. We have a morphism of degree $3$ 
\begin{equation*}
\phi:C\longrightarrow{\mathbb{P}^{1}}.
\end{equation*} 

We now wish to arrive at some kind of "normal form". We take any quartic
\begin{equation*}
f(x,y)=\sum_{i,j}c_{i,j}x^{i}y^{j}
\end{equation*}
and assume by translating that $P=(0,0)$ lies on $C$. Thus $c_{0,0}=0$. Write
\begin{equation*}
y=tx.
\end{equation*}
We then obtain
\begin{equation*}
f(x,tx)=\sum_{i,j}c_{i,j}x^{i}(tx)^{j}.
\end{equation*}
Canceling $x$, we thus obtain an equation of the form
\begin{equation*}
f'(x,t)=\sum_{j=0}^{j=3}a_{j}(t)x^{j}=0,
\end{equation*}
where $\text{deg}(a_{j}(t))\leq{j+1}$. 
We find that
\begin{equation*}
f'(x-\dfrac{a_{2}(t)}{3a_{3}(t)},t)=a_{3}x^3+((-3a_{2}^2 + 9a_{1}a_{3})/(9a_{3}))x + (2/3\cdot{}a_{2}^3 - 3a_{1}a_{2}a_{3} + 9a_{0}a_{3}^2)/(9a_{3}^2).
\end{equation*}  
Multiplying by $a_{3}^2$ and taking $x'=a_{3}\cdot{x}$, we obtain
\begin{equation*}
f''(x',t)=x^3+((-3a_{2}^2 + 9a_{1}a_{3})/(9))x + (2/3\cdot{}a_{2}^3 - 3a_{1}a_{2}a_{3} + 9a_{0}a_{3}^2)/(9).
\end{equation*}
We define
\begin{eqnarray*}
p(t)&=&((-3a_{2}^2 + 9a_{1}a_{3})/(9)),\\
q(t)&=&(2/3\cdot{}a_{2}^3 - 3a_{1}a_{2}a_{3} + 9a_{0}a_{3}^2)/(9),
\end{eqnarray*}
and see that
\begin{equation*}
\Delta:=(4p^3+27q^2)/(a_{3})^2=-a_{1}^2a_{2}^2 + 4a_{2}^3a_{0} + 4a_{1}^3a_{3}^1 - 18a_{1}a_{2}a_{0}a_{3}^1 + 27a_{0}^2a_{3}^2.
\end{equation*}
\begin{lemma}\label{DegDiscrim1}
For each monomial $m$ in $\Delta$, we have that $\text{deg}(m)\leq{10}$.
\end{lemma}
\begin{proof}
This follows quite easily from $\text{deg}(a_{j})\leq{j+1}$ and some easy calculations on $\Delta$. 
\end{proof}


We can now explicitly describe the curves in the Galois closure. For the quadratic subfield, we have that it is given by
\begin{equation*}
y^2=\Delta.
\end{equation*}
The cubic extension of this field is then given by
\begin{equation*}
w^3=y-\sqrt{27}q.
\end{equation*}
See Appendix \ref{Appendix1} for the details. 
From now on, we {assume} that $p$ and $q$ have {\it{no common factors}}. 
Otherwise, the formulas have to be adjusted.

\begin{lemma}
Suppose that $\text{gcd}(p,q)=1$. We then have that
\begin{eqnarray*}
g(\overline{C})\leq{12},\\
g(D)\leq{4}.
\end{eqnarray*}
\end{lemma}
\begin{proof}
From Lemma \ref{DegDiscrim1} we see that the degree of $\Delta$ is at most 10, and as such we see that the genus of the corresponding curve $y^2=\Delta$ can be at most 4.
For $\overline{C}$, we use Riemann-Hurwitz on the covering $\phi_{3}:\overline{C}\longrightarrow{D}$. Let us first describe the ramification of the degree $3$ morphism $\phi_{3}$. Recall that it is given by
\begin{equation*}
w^3=y-\sqrt{27}q.
\end{equation*} Let $f:=y-\sqrt{27}q$. Using Proposition \ref{InertS3} and our assumption $\text{gcd}(p,q)=1$, we then see that $\overline{C}\rightarrow{D}$ is unramified above any point of $D$ lying above a point of $K[t]$.  
The only points of ramification are thus the point(s) at infinity. How many there are of these, depends on the degree of the squarefree part of $\Delta$. 
Indeed, if the squarefree part of $\Delta$ has {\it{even}} degree, then $D$ has two points at infinity and if it has {\it{odd}} degree, then it has exactly one point at infinity.

Rewriting the Riemann-Hurwitz formula for the covering $\phi_{3}$, we obtain
\begin{equation*}
g_{\overline{C}}-1=3(g_{D}-1)+\#(R).
\end{equation*}
The maximal occurring $g_{D}$ and $\#(R)$ are respectively $4$ and $2$, so we obtain that
\begin{equation*}
g_{\overline{C}}\leq{9}+2+1=12,
\end{equation*}
as desired.
\end{proof}
\begin{lemma}
Suppose that $\text{gcd}(p,q)=1$. There are then 8 options for $(g(D),\#{R}, g(\overline{C}), \#(R_{\overline{C}/C}))$. They are given by
\center
\begin{tabular}{ |c| c | c|c| }

\hline
$g(D)$ & $\#(R)$ & $g(\overline{C})$& $\#(R_{\overline{C}/C})$\\
\hline
2 & 1 & 5 & 0\\
2 & 2 & 6 & 2\\
3 & 0 & 7 & 4\\
3 & 1 & 8 & 6\\
3 & 2 & 9 & 8 \\
4 & 0 & 10 & 10\\
4 & 1 & 11 & 12 \\
4 & 2 & 12 & 14\\
\hline

\end{tabular}
\end{lemma}
\begin{proof}
One uses the Riemann Hurwitz on both the covering $\overline{C}\longrightarrow{D}$ and $\overline{C}\longrightarrow{C}$. This leads to 
\begin{eqnarray*}
g_{\overline{C}}-1&=&3(g_{D}-1)+\#(R),\\
2g_{\overline{C}}-2&=&2(2g_{C}-2)+\#(R_{{\overline{C}/C}}),
\end{eqnarray*}
where $g_{C}=3$. Plugging in the possible values for $D$ and $\#(R)$ yields the above values. 
\end{proof}

Let us now find the intersection graph of a genus 3 curve by a nonabelian morphism $C\longrightarrow\mathbb{P}^{1}$ to illustrate the above. We note that this example was also done in Example \ref{Examplecurve}, using a different technique. 
\begin{exa}\label{Gen3NonAb1}
Let us consider the genus 3 curve defined by
\begin{equation*}
z^3+x^3{z}+(x^3+\pi^3)=0.
\end{equation*}
We can find the reduction type in two ways: via an abelian cover and a nonabelian cover. The abelian cover is given by
\begin{equation*}
(x,z)\longmapsto{z}
\end{equation*}
and the nonabelian one by
\begin{equation*}
(x,z)\longmapsto{x}.
\end{equation*}
We will only consider the nonabelian cover. Note that we have $p=x^3$ and $q=x^3+\pi^3$, so that
\begin{equation*}
\Delta=4p^3+27q^2=4x^9+27\cdot{(x^3+\pi^3)^2}.
\end{equation*}
The corresponding quadratic extension is then given by
\begin{equation*}
y^2=\Delta,
\end{equation*}
which gives a hyperelliptic genus 4 curve $D$.\\
Taking the model with
\begin{equation*}
xt=\pi,
\end{equation*}
we see that by normalizing we get a local model for $D$ given by
\begin{equation*}
(y/x^3)^2=4x^3+27(1+t^3)^2,
\end{equation*}
which has three components in the special fiber. Above $x=0$, we find two components $\Gamma_{-1}$ and $\Gamma_{1}$ given by the equations
\begin{equation*}
y'=\pm{\sqrt{27}(1+t^3)}.
\end{equation*}
Above $t=0$ we find an elliptic curve (labeled by $\Gamma'$) with corresponding equation
\begin{equation*}
y'^2=4x^3+27.
\end{equation*}
We have three edges between $\Gamma_{-1}$ and $\Gamma_{1}$ (given by $t^3+1=0$), one between $\Gamma_{-1}$ and $\Gamma'$ and another one between $\Gamma_{1}$ and $\Gamma'$. This gives the reduction graph of $D$. It can be found in Figure \ref{24eplaatje}.
\begin{figure}[h!]
\centering
\includegraphics[scale=0.6]{{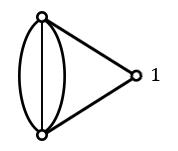}}
\caption{\label{24eplaatje} {\it{The intersection graph of the intermediate genus 4 curve in Example \ref{Gen3NonAb1}.}}} 
\end{figure}
We now examine the divisor $f=y-\sqrt{27}q$. Since $D$ has only one point at infinity, there is only 1 possible ramification point. We quite quickly see that the valuation of $f$ at infinity is divisible by $3$, so the covering $\overline{C}\longrightarrow{C}$ is unramified everywhere. It therefore comes from a 3-torsion point.\\
The support of $f$ is given by the points $P_{1}=(x,y-\sqrt{27}\pi^3)$ and $P_{2}=\infty$. We see that 
\begin{equation*}
\text{div}_{\eta}(f)=9(P_{1})-9(P_{2}).
\end{equation*}
We therefore actually have a 9-torsion point. The divisor we are interested in is $D=3P_{1}-3P_{2}$ (which gives the extension). We find
\begin{eqnarray*}
\rho(P_{1})=\Gamma_{1},\\
\rho(P_{2})=\Gamma'.
\end{eqnarray*}
We first want to clarify one thing: when writing down the reduction graph, one needs to keep in mind the lengths of the corresponding edges. For every edge between $\Gamma_{-1}$ and $\Gamma_{1}$ for instance we have that the edge has length $3$, which can be seen by the relation
\begin{equation*}
(y'-\sqrt{27}(1+t^3))(y'+\sqrt{27}(1+t^3))=4\pi^3/(t^3)
\end{equation*}
(and the fact that $t$ is invertible at these intersection points). The other two edges have length 1. We can now find a solution for the Laplacian. One of them is given by
\begin{eqnarray*}
\phi(\Gamma')=0,\\
\phi(\Gamma_{-})=3,\\
\phi(\Gamma_{+})=6.
\end{eqnarray*}
The corresponding graph of the Laplacian can be found in Figure \ref{22eplaatje}.
\begin{figure}[h!]
\centering
\includegraphics[scale=0.4]{{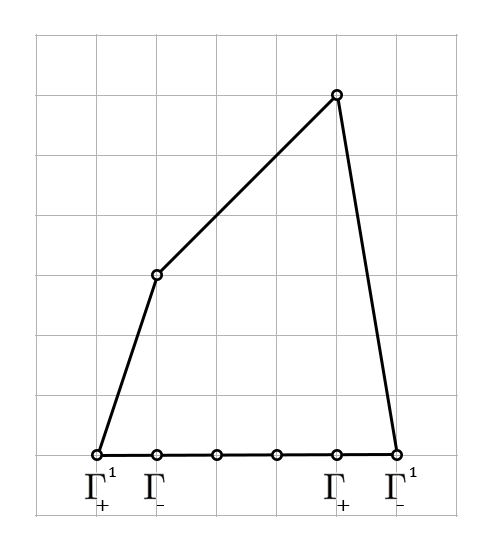}}
\caption{\label{22eplaatje} {\it{The Laplacian function $\phi$ corresponding to $f$ in Example \ref{Gen3NonAb1}. There are three edges between $\Gamma_{-}$ and $\Gamma_{+}$, but the Laplacian on each is the same. }}}
\end{figure}
Note that the increase of slope for every edge between $\Gamma_{-}$ and $\Gamma_{+}$ is taken to be $1$, so that the total increase from $\Gamma_{-}$ to $\Gamma_{+}$ is 3. \\
Let us now consider the extension
\begin{equation*}
w^3=f.
\end{equation*}
\begin{figure}[h!]
\centering
\includegraphics[scale=0.6]{{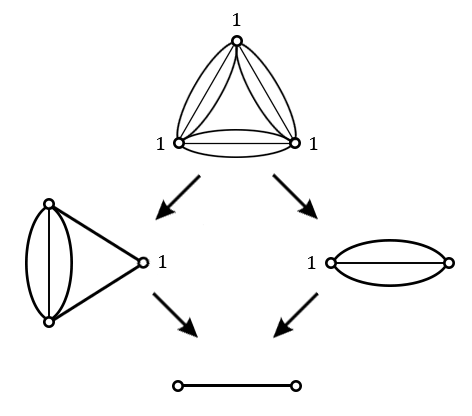}}
\caption{\label{21eplaatje} {\it{The Galois closure of graphs in Example \ref{Gen3NonAb1}.}}}
\end{figure}
The corresponding curve $\overline{C}$ has genus 10. For $f^{\Gamma_{+}}$ and $f^{\Gamma_{-}}$, we see that there are 3 ramification points, because the slope between them on the three edges is not divisible by $3$. The slope of $\phi$ between $\Gamma_{+}$ and $\Gamma'$ for instance is $-6$, so $f^{\Gamma_{+}}$ does not ramify at that intersection point. We therefore have that there are two components $\Gamma_{+,0}$ and $\Gamma_{-,0}$ above $\Gamma_{+}$ and $\Gamma_{-}$ respectively. The corresponding morphism on the special fiber is ramified at exactly 3 points and so these components are genus 1 curves. On the other 2 edges, we have that the slope of $\phi$ is divisible by 3, so there are 3 edges lying above them. On $\Gamma'$, it just defines an unramified extension of an elliptic curve, so we have one component which we call $\Gamma'_{0}$ with genus 1 again.\\
We obtain the following reduction graph. We have three vertices. Each of these vertices intersects the other vertex in exactly three edges. Furthermore, these vertices all have weights $1$. The intersection graph can be found in Figure \ref{21eplaatje}.\\
Let us now consider the Galois action of $\tau$ on this graph. Note that for the intersection graph of $D$, we have that $\tau$ is trivial on all the edges. In the quotient, we have that these edges become smooth points. This happens because the morphism we created from $D$ to $\mathbb{P}^{1}$ is not disjointly branched. 
We have that $\tau$ fixes $\Gamma'_{0}$ and switches the other two vertices $\Gamma_{+,0}$ and $\Gamma_{-,0}$.  One can see this using the fundamental equality
\begin{equation*}
n=e_{\mathfrak{p}}f_{\mathfrak{p}}g_{\mathfrak{p}}
\end{equation*} 
from Equation \ref{FundGalEq1} in Section \ref{SerLocFields}. 
This then gives that the decomposition group of $\Gamma'_{0}$ is $S_{3}$ and the decomposition group of $\Gamma_{+,0}$ and $\Gamma_{-,0}$ are both the normal subgroup $\mathbb{Z}/3\mathbb{Z}$. The reduction graph in the quotient is then a graph on two vertices, intersecting each other in 3 edges. One of these vertices has weight $1$ and is obtained as the quotient of $\Gamma_{+,0}$ and $\Gamma_{-,0}$. The corresponding Galois diagram can be found in Figure \ref{21eplaatje}.

\end{exa}

\begin{exa}\label{LaatsteVoorbeeld1}
Consider the curve $C$ defined by
\begin{equation*}
z^3+p(x)z+q(x)=0
\end{equation*}
with
\begin{eqnarray*}
p(x)&=&x^3,\\
q(x)&=&x^4+\pi^{4}.
\end{eqnarray*}
This is again a genus 3 curve with a nonabelian morphism
\begin{equation*}
\phi:(x,z)\longmapsto{x}
\end{equation*}
of degree 3. The intermediate curve $D$ defined by
\begin{equation*}
y^2=4p^3+27q^2=4x^9+27(x^{4}+\pi^{4})
\end{equation*} 
then has genus 4 as before. Taking the semistable model corresponding to $xt=\pi$, we obtain the equation
\begin{equation*}
(y')^2=4x+27(1+t^4)^{2}
\end{equation*} 
with $y'=y/x^4$. This has the following reduction graph: we have two vertices above $x=0$, corresponding to
\begin{equation*}
(y')=\pm{\sqrt{27}(1+t^4)}.
\end{equation*}
These components $\Gamma_{+}$ and $\Gamma_{-}$ intersect each other 4 times, corresponding to the roots of $1+t^{4}$. Above $t=0$ we have one component $\Gamma_{0}$ of genus 0.\\
We now check the divisor of $f=y-\sqrt{27}q$ as before. We again find
\begin{equation*}
\text{div}_{\eta}(f)=9(P_{1})-9(\infty),
\end{equation*}
where $P_{1}=(x,y-\sqrt{27}\pi^{4})$. Note that $P_{1}$ reduces to $\Gamma_{+}$. We then again have the divisor on graphs
\begin{equation*}
\rho(\text{div}_{\eta}(f))=9\Gamma_{+}-9\Gamma_{0}.
\end{equation*}
Note that the length of every edge in $\Sigma(\mathcal{D})$ is $1$ by the identity
\begin{equation*}
(y'-\sqrt{27}(1+t^{4}))(y'+\sqrt{27}(1+t^{4}))=4\pi/t
\end{equation*}
and the fact that $\phi_{\Sigma}$ is \'{e}tale above the edge corresponding to $x=t=0$. We then find the following Laplacian as a solution:
\begin{eqnarray*}
\phi(\Gamma_{0})&=&0,\\
\phi(\Gamma_{0})&=&4,\\
\phi(\Gamma_{0})&=&5.
\end{eqnarray*}
\begin{figure}[h!]
\centering
\includegraphics[scale=0.5]{{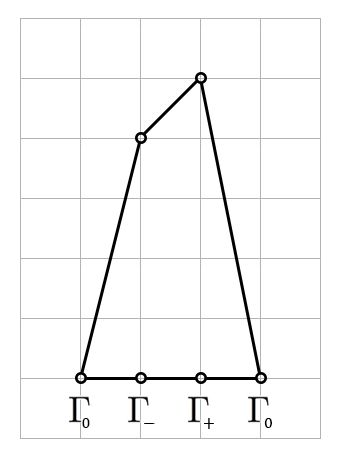}}
\caption{\label{28eplaatje} {\it{The Laplacian in Example \ref{LaatsteVoorbeeld1}.}}}
\end{figure}
See also Figure \ref{28eplaatje}.
Note that the slope between every pair of vertices is {\it{not}} divisible by 3. We therefore find that the reduction graph of the Galois closure has the same reduction graph as $\mathcal{D}$, but with different weights.\\
For $\Gamma_{-}$ and $\Gamma_{+}$, we find that they both have 5 branch points (corresponding to the intersection points), so that their genera are $3$. For $\Gamma_{0}$, we find that $g(\Gamma_{0})=0$. This determines the reduction graph.\\
\begin{figure}[h!]
\centering
\includegraphics[scale=0.6]{{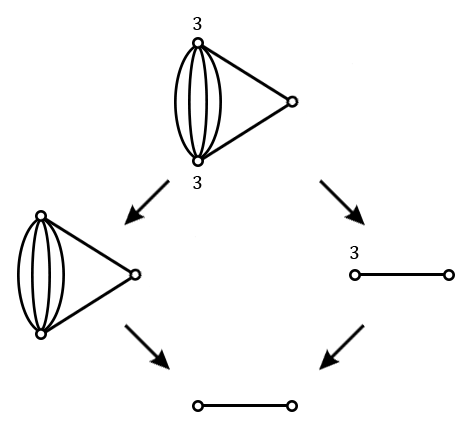}}
\caption{\label{29eplaatje} {\it{The Galois closure of graphs in Example \ref{LaatsteVoorbeeld1}.}}}
\end{figure}
If we now take the invariants under the automorphism of order $\tau$ corresponding to $C$, we then obtain the graph consisting of two vertices, with one vertex having genus $3$ and the other having genus $0$. We thus see that $C$ has potential good reduction. For the Galois diagram of graphs, see Figure \ref{29eplaatje}. 
\end{exa}

\section{Higher genus}

Let us now quickly say something about curves of higher genus. We will adopt the notation from \cite{Hart1}. We say that a curve has a $g^{1}_{d}$ if there exists a linear system of degree $d$ and dimension $1$ on $C$. This then automatically gives a morphism
\begin{equation*}
\phi:C\longrightarrow\mathbb{P}^{1}
\end{equation*}
of degree $d$. We will state the result in \cite[Page 345, Remark 5.5.1]{Hart1} again.
\begin{pro}
For any $d\geq{\dfrac{1}{2}g+1}$, any curve of genus $g$ has a $g^{1}_{d}$. For $d<\dfrac{1}{2}g+1$, there exist curves of genus $g$ having no $g^{1}_{d}$.
\end{pro}
Let us now set $d=4$. We will explain why in a moment. At any rate, we then find that any curve of genus $4,5,6$ admits a $g^{1}_{4}$. Furthermore, for higher genus there exist curves having no $g^{1}_{4}$.\\
For any curve of genus $4,5$ or $6$, we thus have a morphism of degree $4$
\begin{equation*}
\phi:C\longrightarrow{\mathbb{P}^{1}}.
\end{equation*}
The Galois group of this morphism is then a subgroup of $S_{4}$. Since $S_{4}$ is solvable, we can again use our techniques to find the reduction type of any curve of genus $4,5$ or $6$.\\
Problems arise for our method for Galois morphisms
\begin{equation*}
C\longrightarrow{\mathbb{P}^{1}}
\end{equation*}
that have $A_{5}$ as its Galois group. The techniques developed in this thesis are then no longer applicable, since $A_{5}$ is not solvable. If one can find a different morphism that has a solvable Galois group, then one can still find the reduction type of $C$.

\chapter{Conclusion}\label{Conclusion}
In Chapter \ref{Introduction}, we posed the following questions:

\begin{enumerate}
\item There exist criteria for the Berkovich skeleta of elliptic curves and genus two curves in terms of coordinates on their coarse moduli spaces, see \cite[Chapter VII]{Silv1} and \cite{Igusa}. Can they be generalized to curves of higher genus? 
\item Is there a fast algorithm for finding the Berkovich skeleton of a genus three curve?
\item Are there fast algorithms for finding the Berkovich skeleton of other types of curves?  
\end{enumerate}

In a joint paper \cite{supertrop} with Madeline Brandt, the author studied the moduli space of tropical superelliptic curves. We found that this tropical moduli space is a stacky polyhedral fan, as is the case for the usual moduli space of tropical curves of genus $g$, see \cite{trophyp}. This description is given in terms of the branch locus of the associated morphism $C\rightarrow{\mathbb{P}^{1}}$. The author however believes that by symmetrizing as in the genus one and two cases, one can obtain criteria similar to the ones in \cite[Chapter VII]{Silv1} and \cite{Igusa}.

Regarding the second and third question: we have found algorithms for the Berkovich skeleton of curves $C$ that admit a solvable covering $C\rightarrow{\mathbb{P}^{1}}$, see Algorithms \ref{AlgorithmAbelian} and \ref{AlgorithmSolvable}. Every curve of genus three admits either a degree two or a degree three covering to the projective line and the Galois closures of these give solvable Galois groups (namely $S_{2}$ and $S_{3}$). In other words, our algorithm is applicable to all genus three curves. 
Although we haven't done any analyses on the running time of the algorithm, we believe that this algorithm is much faster than the ones alluded to in the introduction, where costly Jacobian computations have to be done. It would be very interesting to obtain a precise statement on the running time of the algorithm.

There are also some questions left that were avoided on purpose throughout this thesis. One of the big assumptions on our coverings is that they are \emph{tame}, in the sense that the covering degree is prime to the characteristic of the residue field. It would be very interesting to see what kind of tropical criteria one could obtain for these "bad" characteristics.  A good place to start would be "wildly ramified" superelliptic curves, i.e. curves with a cyclic abelian covering $C\rightarrow{\mathbb{P}^{1}}$ of degree not coprime to $\text{char}(k)$.

Another thing the reader might have noticed is that the algorithm only works for solvable coverings. The covering data, as given in this thesis, work even for nonsolvable groups, but the twisting data are somewhat harder to produce for nonsolvable coverings. That is, we cannot give a direct description of the components for a "decomposable" nonsolvable covering. For solvable coverings, we can use the machinery of $2$-cocycles to give a nice and concise description, but for nonsolvable coverings, we can only say that one has to "link intersection points according to the components that one obtains from normalizing". That is to say, in this thesis we restricted ourselves to $2$-cocycles and cohomology groups and a general solution in terms of fundamental groups of punctured graphs is probably not far away.  





\appendix

\chapter{Normalizations for $S_{3}$-coverings}\label{Appendix1}

In this section, 
we will give the proofs for Propositions \ref{InertS3} and \ref{DivisorDegree3}. 
We first give a short review of the equations defining the Galois closure of a degree three separable field extension with Galois group $S_{3}$, after which we turn to normalizing discrete valuation rings in these extensions. 

The set-up is as follows. Let $R$ be a discrete valuation ring with quotient field $K$, residue field $k$, uniformizer $\pi$ and valuation $v$. Note that 
the residue field $k$ is not assumed to be algebraically closed, since we want to use these results for discrete valuation rings coming from irreducible components in a semistable model $\mathcal{C}$. 
We will denote the maximal ideal $(\pi)$ in $R$ by $\mathfrak{p}$. We will assume that the characteristic of $K$ is zero 
and that the characteristic of the residue field is coprime to six. Furthermore, we will assume that $K$ contains a primitive third root 
of unity $\zeta_{3}$ and a primitive fourth root of unity $\zeta_{4}$. 
The third root of unity allows us to use Kummer theory for abelian coverings of degree $3$. The fourth root of unity allows us to change the sign in the discriminant of a cubic equation, see the equations below. Note that it also implies that $\sqrt{3}\in{K}$ and thus $\sqrt{27}\in{K}$. 


\section{The Galois closure of an irreducible degree three extension}


Let $K\subset{L}$ be a field extension of degree $3$. Let $z$ be any element in $L\backslash{K}$. After a translation, its minimal polynomial in $K[x]$ is given by $f(x):=x^3+p\cdot{}x+q$, leading to the equation  
\begin{equation}\label{MainEq2}
z^3+p\cdot{z}+q=0,
\end{equation}
where $p,q\in{K}$. Its discriminant is then given by
\begin{equation}
\Delta_{f}:=-(4p^3+27q^2).
\end{equation} 
We first have the following
\begin{lemma}
Let $\overline{L}$ be the Galois closure of $L/K$. Then $\text{Gal}(\overline{L}/K)=S_{3}$ if and only if the discriminant $\Delta_{f}$ is not a square in $K$.
\end{lemma} 
\begin{proof}
See \cite[Proposition 22.4]{stewartgalois}. 
\end{proof}
We now assume that the Galois group of $\overline{L}/K$ is $S_{3}$. We have the following chain of subgroups
\begin{equation}
(1)\vartriangleleft{H}\vartriangleleft{S_{3}},
\end{equation} 
where $H$ has order $3$ and index $2$. In other words, $S_{3}$ is solvable. Using this fact, the famous Cardano formulas then express the roots of $f(x)$ in terms of radicals. 
 Let us quickly summarize the procedure. One considers the following equation: 
\begin{equation}\label{EquationW}
w^2-3wz-p=0.
\end{equation}  
If it has a root in $L$, we take that root and call it $w$. Otherwise, we take a quadratic extension to obtain the desired $w$. It will later turn out that this is exactly the extension to the Galois closure. Note that $w\neq{0}$. Indeed, otherwise we would have $p=0$ and this would imply that $L/K$ is abelian, a contradiction. One can also assume that $w\neq{0}$ in the abelian case, because at least one of the roots of Equation \ref{EquationW} has to be nonzero. 

  At any rate, this $w$ then satisfies the (probably more familiar) equation
\begin{equation}
z=w-\dfrac{p}{3w}.
\end{equation}
This is also known as {\it{Vieta's substitution}}. Plugging this into Equation \ref{MainEq2} then quickly leads to to a sextic equation
\begin{equation}\label{EquationFull}
w^6+qw^3-\dfrac{p^3}{27}=0,
\end{equation}
which is quadratic in $w^3$. We can now consider the element
\begin{equation}
\tau(w):=\dfrac{p}{3w}.
\end{equation}
The reader can immediately check that if $w$ satisfies Equation \ref{EquationFull}, then $\tau(w)$ also satisfies the same equation. 
 We now have
\begin{lemma}
$\tau(w)\neq{w}$.
\end{lemma}
\begin{proof}
Suppose that $\tau(w)=w$. Then $w^2=\dfrac{p}{3}$ and $w^6=\dfrac{p^3}{27}$. Substituting this into Equation \ref{EquationFull}, we find
\begin{equation}
\dfrac{p^3}{27}+qw^3-\dfrac{p^3}{27}=qw^3=0.
\end{equation}
In other words, either $q=0$ or $w=0$. If $q=0$, then Equation \ref{MainEq2} is reducible, a contradiction. We already saw that $w=0$ is impossible, so we arrive at the desired result. 
\end{proof}
Completing the square in Equation \ref{EquationFull}, we now obtain: 
\begin{equation}
(w^3+q/2)^2=\dfrac{p^3}{27}+\dfrac{q^2}{4}=\Delta':=\dfrac{\Delta}{4\cdot{27}},
\end{equation}
where $\Delta=4p^3+27q^2$. We thus see that the quadratic subfield $K(y):=K[y]/(y^2-\Delta')$ is contained in $K(w,z)=K(w)$ (where $w$ is a root of Equation \ref{EquationW}). Note that $K(y)$ is indeed a field, since $\Delta'=\dfrac{\Delta}{4\cdot{27}}$ is minus the discriminant, which is not a square in $K$ by assumption on the Galois group. Using the fact that field degrees are multiplicative, we then see that $K(w)\supset{K(y)}$ has degree three. This then implies that $K(w)$ has degree six over $K$, which also gives the irreducibility of Equation \ref{EquationFull}.
We now have
\begin{lemma}
The field extension $K(w)\supset{K(z)}\supset{K}$ is Galois of order six. As such, it is the Galois closure of $K(z)/K$. The two automorphisms given by
\begin{align*}
\sigma(w)&=\zeta_{3}\cdot{w},\\
\tau(w)&=\dfrac{p}{3w}
\end{align*}    
generate the Galois group. 
Here $\sigma$ has order three and $\tau$ has order two.
\end{lemma}
\begin{proof}
By basic field theory, $\tau$ defines an automorphism of order two on $K(w)$. One then also easily finds that $\sigma$ is an automorphism of order three. Note that they both fix the underlying field $K$. We now have two automorphisms that generate a group $<\sigma,\tau>=:H\subset{\text{Aut}(K(w))}$ with order equal to the degree of the field extension (namely six). This implies that $K(w)/K$ is Galois with Galois group $H$, as desired. 
\end{proof}

Let us now perform some cosmetic changes that remove the fractions from the equations. We scale Equation \ref{EquationFull} slightly using the variable change
\begin{equation}
w'=\dfrac{w}{c\sqrt{3}},
\end{equation}
where $c^3=2$. Writing $w$ for $w'$,  
this then leads to the equation
\begin{equation}
w^6+2\sqrt{27}qw^3-4p^3=0.
\end{equation} 
Completing the square and taking the extension $K\subset{K[y]/(y^2-\Delta)}$ with $\Delta=4p^3+27q^2$, we find that
\begin{equation}
w^3=\pm{}y-\sqrt{27}q.
\end{equation}
Throughout this thesis, we in fact take the extension
\begin{equation}
w^3=y-\sqrt{27}q.
\end{equation}
The other extension (namely $w^3=-y-\sqrt{27}q$) is just the extension corresponding to $\tau(w)$, where $\tau$ is now given by $\tau(w)=\dfrac{cp}{w}$.
\begin{cor}\label{GalClos1}
The Galois closure of $L\supset{K}$ is given by the two extensions
\begin{equation}
K\subset{K(y)}\subset{K(w)},
\end{equation}
where
\begin{equation}
w^3=y-\sqrt{27}q
\end{equation}
and
\begin{equation}
y^2=\Delta.
\end{equation}
\end{cor}



\section{Normalizations}\label{Normalizations}

Now let $R\subset{K}$ be as in the beginning of the Appendix. Let $L\supset{K}$ be a degree three field extension and $z\in{L}\backslash{K}$. After a translation, $z$ satisfies
\begin{equation}
z^3+p\cdot{z}+q=0,
\end{equation}
where $p$ and $q$ are in $K$. By scaling, we can even assume that $p,q\in{R}$. 
Let $K'\subset{\overline{L}}$ be the quadratic subfield, $A$ the integral closure of $R$ in $K'$ and $B$ the integral closure of $R$ in $\overline{L}$. Let $\mathfrak{q}$ be any prime in $B$ lying above $\mathfrak{p}=(\pi)$. We will give explicit equations for the ring $B$ and use those to give formulas for $|I_{\mathfrak{q}}|$, the inertia group of $\mathfrak{q}$.

We consider three cases:

\begin{enumerate}
\item $3v(p)>2v(q)$,
\item $3v(p)<2v(q)$,
\item $3v(p)=2v(q)$.
\end{enumerate}

In every case, we start with a computation of the integral closure $A$ and then deduce $B$ from $A$. From Corollary \ref{GalClos1}, we see that the extension $K\subset{K'}$ is given by
\begin{equation}
y^2=\Delta=4p^3+27q^2.
\end{equation}

\subsection{Case I}
Suppose that $3v(p)>2v(q)$. We let $p=\pi^{k_{1}}u_{1}$ and $q=\pi^{k_{2}}u_{2}$ for units $u_{i}$. We then find the integral equation
\begin{equation}\label{EquationCase1}
(\dfrac{y}{\pi^{k_{2}}})^2=4\pi^{3k_{1}-2k_{2}}u_{1}^3+27u_{2}^2.
\end{equation}
Let $y'=\dfrac{y}{\pi^{k_{2}}}$.
Reducing Equation \ref{EquationCase1} modulo $\mathfrak{p}$ yields the equation
\begin{equation*}
\overline{(y')^2}=\overline{27u_{2}^2}.
\end{equation*}
Or in other words: $\overline{y'}=\pm{\sqrt{27}u_{2}}$. In other words: $A$ is completely split over $R$. The primes are then given by:
\begin{eqnarray}
\mathfrak{q}_{1}&=&(\pi,\sqrt{27}u_{2})\\
\mathfrak{q}_{2}&=&(\pi,-\sqrt{27}u_{2}).
\end{eqnarray}
Note that this implies that $\pi$ is again a uniformizer of $A_{\mathfrak{q}_{i}}$ for both $i$.  
We then have the following Lemma:
\begin{lemma}\label{caseone}
Let $3v(p)>{2v(q)}$. Then
\begin{eqnarray*}
v_{\mathfrak{q}_{1}}(y-\sqrt{27}q)&=&3v(p)-v(q)\\
v_{\mathfrak{q}_{1}}(y+\sqrt{27}q)&=&v(q).\\
v_{\mathfrak{q}_{2}}(y+\sqrt{27}q)&=&3v(p)-v(q)\\
v_{\mathfrak{q}_{2}}(y-\sqrt{27}q)&=&v(q).
\end{eqnarray*}
\end{lemma}
\begin{proof}
We write $y-\sqrt{27}q=\pi^{k_{2}}\cdot{(y'-\sqrt{27}u_{2})}$ and use the relation
\begin{equation}
(y'-\sqrt{27}u_{2})(y'+\sqrt{27}u_{2})=4\pi^{3k_{1}-2k_{2}}u_{1}^3.
\end{equation}
Note that $y'-\sqrt{27}u_{2}$ and $y'+\sqrt{27}u_{2}$ are coprime, so that $y'+\sqrt{27}u_{2}$ is invertible at $\mathfrak{q}_{1}$.  We then see that the desired valuation is given by
\begin{equation}
v_{\mathfrak{q}_{1}}(y-\sqrt{27}q)=v_{\mathfrak{q}_{1}}(\pi^{k_{2}})+v_{\mathfrak{q}_{1}}(4\pi^{3k_{1}-2k_{2}}u_{1}^3)=k_{2}+3{k_{1}}-2k_{2}=3{k_{1}}-k_{2}.
\end{equation}
Using again that $y'+\sqrt{27}u_{2}$ is invertible at $\mathfrak{q}_{1}$, we obtain that $v_{\mathfrak{q}_{1}}(y+\sqrt{27}q)=v(\pi^{k_{2}})=k_{2}$, as desired.

The other two cases for $\mathfrak{q}_{2}$ follow in a completely analogous way and are left to the reader.  
\end{proof}

We can now give the order of the inertia groups for primes $\mathfrak{q}$ in $B$. 
\begin{lemma}
Let $3v(p)>2v(q)$. Then
\begin{enumerate}
\item $|I_{\mathfrak{q}}|=3$ $\iff$ $3\nmid{v(q)}$.
\item  $|I_{\mathfrak{q}}|=1$ $\iff$ $3| v(q)$. 
\end{enumerate}
\end{lemma}
\begin{proof}
This follows from Lemma \ref{caseone}.
\end{proof}

\subsection{Case II}
Suppose that $3v(p)<2v(q)$ and let $p=\pi^{k_{1}}u_{1}$ and $q=\pi^{k_{2}}u_{2}$ for units $u_{i}$, as before. We subdivide this case into two subcases:
\begin{enumerate}
\item [{\bf{(A)}}] $v(p)$ is divisible by $2$,
\item [{\bf{(B)}}] $v(p)$ is not divisible by $2$.
\end{enumerate}

Suppose that $v(p)$ is divisible by $2$. We start with the equation $(y-\sqrt{27}q)(y+\sqrt{27}q)=4p^3$. Dividing by $\pi^{3k_{1}}$, we obtain
\begin{equation}
(\dfrac{y-\sqrt{27}q}{\pi^{3k_{1}/2}})(\dfrac{y+\sqrt{27}q}{\pi^{3k_{1}/2}})=4u_{1}^3.
\end{equation}
We thus see that $\dfrac{y-\sqrt{27}q}{\pi^{3k_{1}/2}}$ and $\dfrac{y+\sqrt{27}q}{\pi^{3k_{1}/2}}$ are invertible. Note that the reduced equation is
\begin{equation}
\overline{y'^2}=\overline{4u_{1}^3}, 
\end{equation}
which might be reducible or irreducible, depending on whether $u_{1}$ is a square in the residue field $k$. In either case, we have the following
\begin{lemma}\label{Case2A}
Let $3v(p)<2v(q)$ and suppose that $v(p)$ is divisible by $2$. Then $|I_{\mathfrak{q}}|=1$. 
\end{lemma} 

This concludes the determination of the inertia groups for the first case. We would now also like to give the valuation of $y-\sqrt{27}q$ at a prime in $\text{Spec}(A)$ lying above $\mathfrak{p}$. As noted above, there are two cases to consider: the case where $\mathfrak{p}$ is split in $A$ and the case where $\mathfrak{p}$ is not split in $A$. We saw that being split in $A$ is equivalent to $\overline{u}_{1}$ being a square in $k$.

Let us consider the case where $\mathfrak{p}$ is split in $A$. We can then write $\overline{u_{1}}=h^2$, where $\overline{h}\in{k}$. Let $h$ be a lift of $h$ to $A$. Then there are two primes lying above $\mathfrak{p}$:
\begin{align*}
\mathfrak{q}_{1}&=(y'-2h^3,\pi)\\
\mathfrak{q}_{2}&=(y'+2h^3,\pi).
\end{align*}
We can now give $v_{\mathfrak{q}_{i}}(y\pm\sqrt{27}q)$:
\begin{lemma}
Suppose that $\mathfrak{p}$ is split in $A$. Let $\mathfrak{q}_{i}$ be the primes in $\text{Spec}(A)$ lying above $\mathfrak{p}$. Then
\begin{eqnarray*}
v_{\mathfrak{q}_{i}}(y\pm\sqrt{27}q)&=&3v_{}(p)/2.
\end{eqnarray*}
\end{lemma}
\begin{proof}
Using 
\begin{equation}\label{EquationPDivisible}
(\dfrac{y-\sqrt{27}q}{\pi^{3k_{1}/2}})(\dfrac{y+\sqrt{27}q}{\pi^{3k_{1}/2}})=4u_{1}^3,
\end{equation}
we see that $\dfrac{y-\sqrt{27}q}{\pi^{3k_{1}/2}})$ and $(\dfrac{y+\sqrt{27}q}{\pi^{3k_{1}/2}})$ are invertible and thus $v_{\mathfrak{q}_{i}}(\dfrac{y\pm\sqrt{27}q}{\pi^{3k_{1}/2}})=0$. 
Since $R\rightarrow{A_{\mathfrak{q}_{i}}}$ is \'{e}tale for both $i$, we obtain that $\pi$ is again a uniformizer. This quickly gives the lemma.
\end{proof}

Suppose now that $\mathfrak{p}$ is not split in $A$. There is one prime lying above $\mathfrak{p}$, namely
\begin{equation}
\mathfrak{q}=(y'^2-4u_{1}^3,\pi).
\end{equation}
We then have
\begin{lemma}
Let $\mathfrak{q}$ be the only prime lying above $\mathfrak{p}$. Then
\begin{equation}
v_{\mathfrak{q}}(y\pm\sqrt{27}q)=3v(p)/2.
\end{equation}
\end{lemma}
\begin{proof}
Using Equation \ref{EquationPDivisible} again, we see that $v_{\mathfrak{q}}(\dfrac{y\pm\sqrt{27}q}{\pi^{3k_{1}/2}})=0$. Since $R\rightarrow{A_{\mathfrak{q}}}$ is \'{e}tale, we have that $\pi$ is again a uniformizer and the result follows. 
\end{proof}

This concludes the case where $v(p)$ is divisible by $2$.
Now suppose that $v(p)$ is not divisible by $2$. We claim that $|I_{\mathfrak{q}}|=2$. We write $3v(p)=2k+1$ and find the equation
\begin{equation}\label{EquationRamifiedP}
(\dfrac{y-\sqrt{27}q}{\pi^{k}})(\dfrac{y+\sqrt{27}q}{\pi^{k}})=4\pi\cdot{}u_{1}^3.
\end{equation}
Writing $y'=\dfrac{y}{\pi^{k}}$ and $q'=\dfrac{q}{\pi^{k}}$, we see that there is only one prime lying above $\mathfrak{p}$ in this algebra, namely
\begin{equation}
\mathfrak{q}'=(y'-\sqrt{27}q', \pi)=(y'-\sqrt{27}q').
\end{equation}
Note that $\mathfrak{q}'$ is principal and thus $A$ is normal. The fiber over $\pi$ is of the form $\overline{y'^2}=\overline{0}$, showing that the extension is ramified.
Since the inertia group is cyclic inside $S_{3}$ and we already know that $|I_{\mathfrak{q}}|$ is greater than or equal to $2$, we find that $|I_{\mathfrak{q}}|=2$. We summarize this in a lemma:
\begin{lemma}\label{Case2B}
Let $3v(p)<2v(q)$ and suppose that $v(p)$ is not divisible by $2$. Then $|I_{\mathfrak{q}}|=2$.
\end{lemma} 

Note now that there are only two options for $v(p)$: it is either divisible by two or it is not. Using this observation, we then also obtain the reverse statement of Lemmas \ref{Case2A} and \ref{Case2B}. We could also obtain this from the following lemma:
\begin{lemma}
Let $\mathfrak{q}'$ be the only prime lying above $\mathfrak{p}\in\text{Spec}(A)$. Then
\begin{eqnarray*}
v_{\mathfrak{q}'}(y\pm\sqrt{27}q)&=&3v_{}(p).
\end{eqnarray*}
\end{lemma}
\begin{proof}
This follows from Equation \ref{EquationRamifiedP}, noting that $v_{\mathfrak{q}'}(\pi)=2$ and $v_{\mathfrak{q}'}(y\pm\sqrt{27}q)=1$.
\end{proof}

\subsection{Case III}
Suppose that $3v(p)=2v(q)$. Let $\Delta:=4p^3+27q^2$.  We again consider two cases:
\begin{enumerate}
\item [{\bf{(A)}}] $v(\Delta)$ is divisible by $2$,
\item [{\bf{(B)}}] $v(\Delta)$ is not divisible by $2$.
\end{enumerate}

Suppose first that $v(\Delta)$ is divisible by $2$. We then see that the extension is unramified in the quadratic subfield.
\begin{lemma}
Suppose that $3v(p)=2v(q)$ and that $v(\Delta)$ is divisible by $2$. Then $|I_{\mathfrak{q}}|=1$.
\end{lemma}
\begin{proof}
 We consider the equation
\begin{equation*}
(y'-\dfrac{\sqrt{27}q}{\pi^{v(q)}})(y'+\dfrac{\sqrt{27}q}{\pi^{v(q)}})=\dfrac{4p^3}{\pi^{2v(q)}}.
\end{equation*}
Note that the righthand side is invertible, implying that the elements on the lefthand side are also invertible. 
From $v(\Delta)\equiv{0}\mod{2}$, we obtain that $\pi$ is again a uniformizer at the two points lying above it. We denote them by $\mathfrak{q}_{i}$. 
This gives
\begin{lemma}
 Let the $\mathfrak{q}_{i}$ be the two primes lying above $\mathfrak{p}$. 
Then
\begin{eqnarray*}
v_{\mathfrak{q}_{i}}(y\pm\sqrt{27}q)&=&v(q).
\end{eqnarray*}
\end{lemma} 
\begin{proof}
This follows as before, noting that $(y'-\dfrac{\sqrt{27}q}{\pi^{v(q)}})$ and $(y'+\dfrac{\sqrt{27}q}{\pi^{v(q)}})$ are invertible and that $v_{\mathfrak{q}_{i}}(\pi)=1$.
\end{proof}
This then also quickly gives the rest of the lemma: since $3|3v(p)$, we find that $3|2v(q)$, implying that $3|v(q)=v_{\mathfrak{q}_{i}}(y\pm\sqrt{27}q)$. As always, this implies that the abelian extension $K'\subset{\overline{L}}$ is unramified above the $\mathfrak{q}_{i}$. 
\end{proof}
Now for the second case with $v(\Delta)\equiv{1}\mod{2}$. We find that there is only one point (denoted by $\mathfrak{q}$) that lies above $\mathfrak{p}$. The valuation of $\pi$ in $\mathfrak{q}$ is then $2$. This then gives
\begin{lemma}
Suppose that $3v(p)=2v(q)$, $v(\Delta)\equiv{1}\mod{2}$ and let $\mathfrak{q}$ be the only prime lying above $\mathfrak{p}$. 
\begin{equation}
v_{\mathfrak{q}}(y\pm\sqrt{27}q)=2v(q)=3v(p).
\end{equation}
\end{lemma}
\begin{proof}
This follows exactly as before: we have the equation
\begin{equation}
(y'-\dfrac{\sqrt{27}q}{\pi^{v(q)}})(y'+\dfrac{\sqrt{27}q}{\pi^{v(q)}})=\dfrac{4p^3}{\pi^{2v(q)}},
\end{equation}
implying that the lefthand side is invertible. Since $v_{\mathfrak{q}}(\pi)=2$, we obtain the lemma by a simple calculation.
\end{proof}

\bibliographystyle{alpha}
\bibliography{bibfiles}{}

\end{document}